\documentclass{amsart}
\usepackage{verbatim}
\usepackage{amssymb}
\usepackage{amsmath}
\usepackage[usenames]{color}
\usepackage{graphicx}
\usepackage{amscd}
\usepackage[numbers,sort,square]{natbib}
\usepackage[T1]{fontenc}

\usepackage{enumerate}

\theoremstyle{plain}
\newtheorem{theorem}{Theorem}

\newtheorem{proposition}[theorem]{Proposition}
\newtheorem{lemma}[theorem]{Lemma}
\newtheorem{corollary}[theorem]{Corollary}

\newtheorem{definition}[theorem]{Definition}

\theoremstyle{definition}

\newtheorem{remark}[theorem]{Remark}

\newtheorem{example}[theorem]{Example}
\numberwithin{equation}{section}
\numberwithin{theorem}{section}
\allowdisplaybreaks

\usepackage[colorinlistoftodos,prependcaption,color=yellow,textsize=tiny]{todonotes}

\newcommand{\eps}{\varepsilon}

\newcommand{\sign}{\mathrm{sign}\text{ }}

\newcommand{\dd}[0]{\mathrm{d}}
\newcommand{\ud}[0]{\,\mathrm{d}}

\newcommand{\vertiii}[1]{{\left\vert\kern-0.25ex\left\vert\kern-0.25ex\left\vert #1
    \right\vert\kern-0.25ex\right\vert\kern-0.25ex\right\vert}}

\begin{document}

\title[Local characteristics and tangency]
{Local characteristics and tangency\\ of vector-valued martingales}

\author{Ivan S.\ Yaroslavtsev}
\address{Max Planck Institute for Mathematics in the Sciences\\
Inselstra{\ss}e 22\\
04103 Leipzig\\
Germany}
\address{Delft Institute of Applied Mathematics\\
Delft University of Technology\\
P.O. Box 5031\\
2600 GA Delft\\
The Netherlands}
\email{yaroslavtsev.i.s@yandex.ru}

\begin{abstract}
This paper is devoted to tangent martingales in Banach spaces. We provide the definition of tangency through local characteristics, basic $L^p$- and $\phi$-estimates, a precise construction of a decoupled tangent martingale, new estimates for vector-valued stochastic integrals, and several other claims concerning tangent martingales and local characteristics in infinite dimensions. This work extends various real-valued and vector-valued results in this direction e.g.\ due to Grigelionis, Hitczenko, Jacod, Kallenberg, Kwapie\'{n}, McConnell, and Woyczy\'{n}ski. The vast majority of the assertions presented in the paper is done under the necessary and sufficient UMD assumption on the corresponding Banach space.
\end{abstract}



\maketitle

\tableofcontents

\section{Introduction}

This paper is devoted to {\em tangent martingales}. Let us start with the discrete setting. Which martingale difference sequences do we call tangent? For a Banach space $X$ two $X$-valued martingale difference sequences $(d_n)_{n\geq 1}$ and $(e_n)_{n\geq 1}$ are {\em tangent} if for every $n\geq 1$ a.s.\footnote{see Subsection \ref{subsec:ConexponPSCondProbCondIndep} for the definition of a conditional probability}
\begin{equation}\label{eq:introtandisdagdef}
  \mathbb P(d_n | \mathcal F_{n-1})=\mathbb P(e_n | \mathcal F_{n-1}),
\end{equation}
where $\mathbb P(d_n | \mathcal F_{n-1})(A) := \mathbb E(\mathbf 1_{A}(d_n)|\mathcal F_{n-1})$ and $\mathbb P(e_n | \mathcal F_{n-1})(A):= \mathbb E(\mathbf 1_{A}(e_n)|\mathcal F_{n-1})$  for any Borel set $A \subset X$.
This notion was first introduced by Zinn in \cite{Zinn85} where he proved that if $X = \mathbb R$, then for any $p\geq 2$ the $L^p$ moments of $\sum_{n} d_n$ and $\sum_n e_n$ are comparable given $(d_n)$ and $(e_n)$ are conditionally symmetric\footnote{i.e.\ the distributions \eqref{eq:introtandisdagdef} are symmetric a.s., equivalently $(d_n)$, $(-d_n)$, $(e_n)$, and $(-e_n)$ are tangent altogether} (the general case $1\leq p<\infty$ was obtained by Hitczenko in \cite{Hit88}). The estimates of Zinn and Hitczenko have been extended by McConnell in \cite{MC} and Hitczenko in \cite{HitUP} to infinite dimensions. It turned out that such estimates characterize a certain condition concerning the geometry of a Banach space, namely, the {\em UMD property} (see Subsection \ref{subsec:prelimUMD} for the definition).

\begin{theorem}[Hitczenko, McConnell]\label{thm:intromccnnll}
 Let $X$ be a Banach space, $1\leq p<\infty$. Then $X$ is UMD if and only if for any $X$-valued tangent martingale difference sequences $(d_n)_{n\geq 1}$ and $(e_n)_{n\geq 1}$ one has that
 \begin{equation}\label{eq:strongLpfortangentmccnnll}
  \mathbb E \sup_{0\leq N < \infty} \Bigl\| \sum_{n=1}^N d_n \Bigr\|^p \eqsim_{p, X}  \mathbb E \sup_{0\leq N < \infty} \Bigl\| \sum_{n=1}^N e_n \Bigr\|^p.
 \end{equation}
\end{theorem}
(Note that the paper \cite{MC} did not cover the case $p=1$, and  \cite{HitUP} was never published. Nevertheless, the reader can find this case in \cite[pp.\ 424--425]{CV07} and in Theorem \ref{thm:Ephifordiscretetantdandee}). 

A classical example of tangent martingale different sequences is provided by independent mean zero random variables. Let $(\xi_n)_{n\geq 1}$ be real-valued mean zero independent random variables, let $(v_n)_{n\geq 1}$ be $X$-valued bounded predictable (i.e.\ $v_n$ depends only on $\xi_1,\ldots,\xi_{n-1}$). Then $(v_n \xi_n)_{n\geq 1}$ is a martingale difference sequence. Moreover, then $(v_n \xi'_n)_{n\geq 1}$ is a tangent martingale difference sequence for $(\xi'_n)_{n\geq 1}$ being an independent copy of $(\xi_n)_{n\geq 1}$ (see Example \ref{ex:tangstandarddecxinvn-->xi'nvn}), so in the UMD case \eqref{eq:strongLpfortangentmccnnll} yields
\begin{equation}\label{eq:introv_nxi_neqsimv_nxi'_mb}
   \mathbb E \sup_{0\leq N < \infty} \Bigl\| \sum_{n=1}^N v_n \xi_n \Bigr\|^p \eqsim_{p, X}  \mathbb E \sup_{0\leq N < \infty} \Bigl\| \sum_{n=1}^N v_n\xi'_n \Bigr\|^p.
\end{equation}
It turned out that \eqref{eq:introv_nxi_neqsimv_nxi'_mb} characterizes the UMD property if one sets $(\xi_n)_{n\geq 1}$ to be Rademachers\footnote{see Definition \ref{def:ofRadRV}} (see Bourgain \cite{Bour83} and Garling \cite{Gar90,Gar85}), Gaussians (see Garling \cite{Gar85} and McConnell \cite{MC}), or Poissons (see Proposition \ref{prop:XUMDiffdecouplforPois}). In the Gaussian and Poisson cases the equivalence of \eqref{eq:introv_nxi_neqsimv_nxi'_mb} and the UMD property basically says that the following estimates hold for $X$-valued stochastic integrals
\begin{equation}\label{eq:inftoUMDintwrtBrandindBr}
  \mathbb E \sup_{t\geq 0} \Bigl\|  \int_0^t \Phi \ud W\Bigr\|^p \eqsim_{p, X} \mathbb E \sup_{t\geq 0} \Bigl\|  \int_0^t \Phi \ud \widetilde W\Bigr\|^p,
\end{equation}
\begin{equation}\label{eq:inftoUMDintwrtPoisbyindPois}
 \mathbb E \sup_{t\geq 0} \Bigl\|  \int_0^t F \ud \widetilde N\Bigr\|^p \eqsim_{p, X}  \mathbb E \sup_{t\geq 0} \Bigl\|  \int_0^t F \ud \widetilde N_{\rm ind}\Bigr\|^p
\end{equation}
(here $\Phi$ and $F$ are $X$-valued elementary predictable, $W$ is a Brownian motion, $\widetilde N$ is a compensated standard Poisson process, $\widetilde W$ and $\widetilde N_{\rm ind}$ are independent copies of $W$ and $\widetilde N$ respectively), which allows one to change the driving Brownian or Poisson noise in a stochastic integral by an independent copy without losing the information about strong $L^p$-norms of the stochastic integral, are {\em equivalent} to your Banach space $X$ having the UMD property. Estimates of the form \eqref{eq:inftoUMDintwrtBrandindBr} turned out to be exceptionally important in vector-valued stochastic integration theory as the right-hand side of \eqref{eq:inftoUMDintwrtBrandindBr} is nothing but a {\em $\gamma$-norm} (see Subsection \ref{subsec:gammanorm}) of $\Phi$ which is a natural extension of the Hilbert-Schmidt norm to general Banach spaces (see McConnell \cite{MC} and van Neerven, Veraar, and Weis \cite{NVW}, see also \cite{VY16,Ver,vNW08} for a general continuous martingale case and Dirksen \cite{Dirk14} for the Poisson case). Estimates \eqref{eq:inftoUMDintwrtBrandindBr} and \eqref{eq:inftoUMDintwrtPoisbyindPois} justify that tangent martingales are extremely important for vector-valued stochastic integration.

The procedure of changing the noise by an independent copy (in our case this was $(\xi_n) \mapsto (\xi'_n)$) together with extending the filtration in the corresponding way (i.e.\ $\mathcal F_n' := \sigma(\mathcal F_n, \xi'_1, \ldots, \xi'_n)$) creates a special tangent martingale difference sequence, namely a {\em decoupled} one which can be defined in the following way: $(e_n)$ is a decoupled tangent martingale difference sequence to $(d_n)$ if $(e_n)$ are conditionally independent given $\mathcal G:= \sigma \bigl((d_n)\bigr)$, i.e.\ for any Borel $B_1, \ldots, B_N \subset X$ a.s.\
\[
\mathbb P(e_1 \in B_1,\ldots, e_N \in B_N|\mathcal G) = \mathbb P(e_1 \in B_1|\mathcal G)\cdot \ldots \cdot \mathbb P( e_N \in B_N|\mathcal G),
\]
and $\mathbb P(e_n|\mathcal F_{n-1}) = \mathbb P(e_n | \mathcal G)$ for any $n\geq 1$. Note that such a martingale difference sequence might not exist on the probability space with the original filtration, so one may need to extend the probability space and filtration in such a way that $(d_n)$ preserves its martingale property.
Existence and uniqueness of such a decoupled $(e_n)$ was proved by Kwapie\'{n} and Woyczy\'{n}ski in \cite{KwW91} (see also de la Pe\~{n}a \cite{dlP94}, de la Pe\~{n}a and Gin\'{e} \cite{dlPG}, especially  \cite[Section 6.1]{dlPG} for a detailed proof, Kallenberg \cite{Kal17}, and S.G.\ Cox and Geiss \cite{CG}). The goal of the present paper is to extend Theorem \ref{thm:intromccnnll} to the continuous-time setting and to discover in this case the explicit form of a decoupled tangent local martingale.

Let us start with explaining what continuous-time tangent local martingales are. To this end we will need L\'evy martingales.
What do we know about them? Well, one of the most fundamental features of  L\'evy processes is the {\em L\'evy-Khinchin formula} which is the case of a L\'evy martingale $L$ with $L_0=0$ has the following form (see e.g.\ \cite{Sato,JS})
\begin{equation}\label{eq:LHformulaforLevyptocINTRO}
\mathbb E e^{i\theta L_t} = \exp\Bigl\{ t\Bigl(-\frac 12 \sigma^2 \theta^2 + \int_{\mathbb R} e^{i\theta x} - 1 -i\theta x \ud \nu(\dd x)\Bigr)\Bigr\},\;\;\; t\geq 0,\;\;\theta\in \mathbb R,
\end{equation}
for some fixed $\sigma \geq 0$ and for some fixed measure $\nu$ on $\mathbb R$. It turns out that the pair $(\sigma, \nu)$ characterizes the distribution of a L\'evy martingale, and it has the following analogue for a general real-valued martingale $M$: $([M^c], \nu^M)$, where $M^c$ is the {\em continuous part} of $M$ (see Subsection \ref{subsec:candec}) with $[M^c]$ being the quadratic variation of $M^c$ (see Subsection \ref{subsec:quadrvar}), and $\nu^M$ is a {\em compensator} of a random measure $\mu^M$ defined on $\mathbb R_+ \times \mathbb R$ by
\begin{equation}\label{eq:defofmuMINTRO}
\mu^M([0, t] \times B) := \sum_{0\leq s\leq t} \mathbf 1_{B\setminus\{0\}} (\Delta M_t),\;\;\; t\geq 0, \;\; B \in \mathcal B(\mathbb R)
\end{equation}
(see Subsection \ref{subsec:ranmeasures}). In the case $M=L$ we have that $[M^c]_t = \sigma^2 t$ and $\nu^M = \lambda \otimes \nu$, where $\lambda$ is the Lebesgue measure on $\mathbb R_+$. This pair $([M^c], \nu^M)$ is called to be the {\em local characteristics} (a.k.a.\ {\em Grigelionis characteristics} or {\em Jacod-Grigelionis characteristics}) of $M$, and two continuous-time local martingales are called {\em tangent} if their local characteristics coincide.  Continuous-times tangent martingales and local characteristics 
were intensively studied by Jacod \cite{Jac79,Jac83,Jac84}, Jacod and Shiryaev \cite{JS}, Jacod and Sadi \cite{JH87}, Kwapie\'{n} and Woyczy\'{n}ski \cite{KwW91,KwW86,KwW92,KwW89}, and Kallenberg \cite{Kal17} (see also \cite{Now02,Now03,MC,NVW}).
In particular, Kallenberg proved in \cite{Kal17} that for any real-valued continuous-time tangent martingales $M$ and $N$ one has that
\begin{equation}\label{eq:introKalLpest}
 \mathbb E \sup_{t\geq 0} |M_t|^p \eqsim_{p} \mathbb E \sup_{t\geq 0} |N_t|^p,\;\;\;\;\; 1\leq p<\infty,
\end{equation}
with more general inequalities (including concave functions of moderate growth) under additional assumptions on $M$ and $N$ (e.g.\ conditional symmetry).
Furthermore, in \cite{Jac84,JH87,KwW91,Kal17} it was shown that any real-valued martingale $M$ has a {\em decoupled tangent local martingale} $N$, i.e.\ a tangent local martingale $N$ defined on an enlarged probability space with an enlarged filtration such that $N(\omega)$ is a martingale with independent increments and with local characteristics $([M^c](\omega), \nu^M(\omega))$ for a.e.\ $\omega \in \Omega$ from the original probability space. Moreover, in the quasi-left continuous setting it was shown in  \cite{Jac84,JH87,KwW91} that such a martingale can be obtained via the following procedure: if we discretize $M$ on $[0, T]$, i.e.\ consider a discrete martingale $(f_k^n)_{k=1}^n = (M_{Tk/n})_{k=1}^n$,
and consider a decoupled tangent martingale $\tilde f^n := (\tilde f_k^n)_{k=1}^n$, then $\tilde f^n$ converges in distribution to $N$ as random variables with values in the Skorokhod space $\mathcal D([0, T], \mathbb R)$ (see Definition \ref{def:defofskotokhodspaxc}) as $n\to \infty$. This in particular justifies the definition of a continuous-time decoupled tangent martingale.

\bigskip

In the present paper we are going to explore various facts concerning vector-valued continuous-time tangent martingales. 
We will mainly focus on the following three questions:
\begin{itemize}
 \item How do local characteristics look like in Banach spaces?
 \item What is a decoupled tangent martingale in this case?
 \item Can we extend decoupling inequalities \eqref{eq:introKalLpest} to infinite dimensions?
\end{itemize}
We will also try to answer all the supplementary and related problems appearing while working on these three questions.
Let us outline the structure of the paper section-by-section.

In Section \ref{sec:prelim} we present some preliminaries to the paper, i.e.\ certain assertions (e.g.\ concerning martingales, random measures, stochastic integration, et cetera) which we will heavily need throughout the paper. 

Our main Section \ref{sec:tangmartconttime} is devoted to the definition of vector-valued continuous-time tangent martingales, basic $L^p$-estimates for these martingales, and the construction of a decoupled tangent martingale. How do we define tangent martingales in the vector-valued case? As we saw in Theorem \ref{thm:intromccnnll}, a Banach space $X$ having the UMD property plays an important r\^ole for existence of $L^p$-bounds for discrete tangent martingales. This also turned out to be equivalent to existence of local characteristics of a general $X$-valued martingale $M$. Namely, due to \cite{Y17MartDec,Y17GMY} $X$ has the UMD property if and only if a general $X$-valued martingale $M$ has the {\em Meyer-Yoeurp decomposition}, i.e.\ it can be uniquely decomposed into a sum of a continuous local martingale $M^c$ and a purely discontinuous local martingale $M^d$ (see Remark \ref{rem:MYdecBanach}). In this case we define the {\em local characteristics} of $M$ to be the pair $([\![M^c]\!], \nu^M)$, where $[\![M^c]\!]$ is a {\em covariation bilinear form}, i.e.\ a symmetric bilinear form-valued process satisfying 
\[
[\![M^c]\!]_t(x^*, x^*) = [\langle M, x^*\rangle]_t,\;\;\; t\geq 0,
\]
for any $x^*\in X^*$ a.s.\ (such a process exists because of Remark \ref{rem:ifUMDthencovbilform}), and $\nu^M$ is a compensator of a random measure $\mu^M$ defined on $\mathbb R_+ \times X$ analogously to \eqref{eq:defofmuMINTRO} (see Subsection \ref{subsec:quadrvar} and \ref{subsec:ranmeasures}). Similarly to the real-valued case, two $X$-valued martingales are {\em tangent} if they have the same local characteristics.

Next, we present $L^p$-estimates for UMD-valued tangent martingales.
In Theorem \ref{thm:tangentgencaseUMDuhoditrotasoldat} we extend the result \eqref{eq:introKalLpest} of Kallenberg to any UMD Banach space $X$, i.e.\ we prove that for any UMD Banach space $X$ and for any $X$-valued tangent martingales $M$ and $N$ one has that
\begin{equation}\label{eq:LpineqfortangmartinUMDINTRO}
\mathbb E \sup_{t\geq 0} \|M_t\|^p \eqsim_{p, X}\mathbb E \sup_{t\geq 0} \|N_t\|^p,\;\;\; 1\leq p<\infty.
\end{equation}
 Let us say a couple of words about how do we gain \eqref{eq:LpineqfortangmartinUMDINTRO}. To this end we need {\em the canonical decomposition}. Thanks to Meyer \cite{Mey76} and Yoeurp \cite{Yoe76} any real-valued martingale $M$ can be uniquely decomposed into a sum of a continuous local martingale $M^c$ (the Wiener-like part), a purely discontinuous quasi-left continuous  local martingale $M^q$ (the Poisson-like part), and a purely discontinuous local martingale $M^a$ with accessible jumps (the discrete-like part). It turned out that this decomposition can be expanded to the vector-valued case if and only if $X$ has the UMD property (see \cite{Y17MartDec,Y17GMY}). Moreover, as it is shown in Subsection \ref{subsec:charandthecandecbe} if $M=M^c+ M^q + M^a$ and $N=N^c + N^q + N^a$ are the canonical decompositions of tangent martingales $M$ and $N$, then $M^i$ and $N^i$ are tangent for any $i\in \{c,q,a\}$, and thus by strong $L^p$-estimates for the canonical decomposition presented in \cite{Y18BDG} (see Theorem \ref{thm:candecXvalued}) we need to show \eqref{eq:LpineqfortangmartinUMDINTRO} separately for each of these three cases. Then the continuous case immediately follows from weak differential subordination inequalities obtained in \cite{Y17MartDec,OY18,Y18BDG} and the discrete-like case can be shown via a standard discretization trick (see Subsection \ref{subsec:appforMarAppPDMAJ}) and Theorem \ref{thm:intromccnnll}.
 
 The most complicated and the most interesting mathematically is the Poisson-like case. First we show that \eqref{eq:inftoUMDintwrtPoisbyindPois} holds true not just for a compensated Poisson process, but for any stochastic integral with respect to a Poisson random measure (see Proposition \ref{prop:XUMDiffdecouplforPois}). Next we prove that any UMD-valued quasi-left continuous purely discontinuous martingale can be presented as a stochastic integral with respect to a quasi-left continuous compensated random measure (see Theorem \ref{thm:XisUMDiffMisintxwrtbarmuMvain}). Finally, by exploiting a certain approximation argument, we may assume that this random measure is defined over a finite jump space, and hence this is a time-changed Poisson random measure thanks to a fundamental result by Meyer \cite{Mey71} and Papangelou \cite{Pap72} (see e.g.\ also \cite{Kal,BN88,AH78}) which says that {\em any} quasi-left continuous integer random measure after a certain time change becomes a Poisson random measure. As this time change depends only on the compensator measure (which is one of local characteristics and which is the same for $M^q$ and $N^q$), \eqref{eq:inftoUMDintwrtPoisbyindPois} immediately yields \eqref{eq:LpineqfortangmartinUMDINTRO} for the quasi-left continuous purely discontinuous case.
 
 \smallskip

Another highlight point of Section \ref{sec:tangmartconttime} is existence, uniqueness, and construction of a {\em decoupled tangent martingale}. First, in Theorem \ref{thm:mainforDTMgencasenananana} we extend the result of Jacod \cite{Jac84}, Kwapie\'{n} and Woyczy\'{n}ski \cite{KwW91}, and Kallenberg \cite{Kal17} on existence of a decoupled tangent martingale to general UMD-valued martingales (recall that they have shown this existence only in the real-valued
case). Next in Subsection \ref{subsec:uniqofDTM} we show that a decoupled tangent martingale is unique in distribution (which extends the discrete case, see \cite{dlPG,KwW91}). Finally, in Subsection \ref{subsec:indincgivenLC} we prove that if $N$ is a decoupled tangent martingale of $M$, then $N$ has independent increments given the local characteristics $([\![M^c]\!], \nu^M)$ of $M$ which e.g.\ generalizes \cite[Theorem 3.1]{Kal17}. 

It is of interest to take a closer look at the structure of tangent martingales. Let us consider a particular case of \eqref{eq:inftoUMDintwrtBrandindBr} and \eqref{eq:inftoUMDintwrtPoisbyindPois}.
Intuitively it seems that stochastic integrals $\int \Phi \ud \widetilde W$ and $\int F \ud \widetilde N_{\rm ind}$ occurring in \eqref{eq:inftoUMDintwrtBrandindBr} and \eqref{eq:inftoUMDintwrtPoisbyindPois} should be decoupled tangent martingales to $\int \Phi \ud W$ and $\int F \ud \widetilde N$ respectively. And this is true as $\int \Phi (\omega) \ud \widetilde W$ is a.s.\ a martingale with independent increments and with the local characteristics $(\Phi (\omega) \Phi^*(\omega), 0)$ (here we can consider $\Phi\in \mathcal L(L^2(\mathbb R_+), X)$ instead of $\Phi:\mathbb R_+ \to X$ a.s.\ as $\Phi$ is elementary predictable, see Subsection \ref{subsec:prelimstint} and Section~\ref{sec:intwrtgenmart}), and $\int F(\omega) \ud \widetilde N_{\rm ind}$ has a.s.\ independent increments and the local characteristics $(0, \nu^F(\omega))$ with the measure $\nu^F(\omega)$ defined on $\mathbb R_+\times X$ by
\[
 \nu^F(\omega)([0, t]\times B) = \int_0^t \mathbf 1_{B}\bigl(F(s,\omega) \bigr)\ud s,\;\;\;\;\; t\geq 0,\;\; B\in \mathcal B(X).
\]
For a general martingale we have an expanded version the this construction. Recall that for a given UMD Banach space $X$ any $X$-valued martingale $M$ has the canonical decomposition $M = M^c + M^q + M^a$. Let us present a corresponding decoupled tangent martingale $N^c$, $N^q$, and $N^a$ for each of the cases separately (in the end we can simply sum up $N := N^c + N^q + N^a$ these cases, see Subsection \ref{subsec:CIprocess}). It turns out that by Subsection \ref{subsec:dectangcontmart} we have that $M^c \circ \tau^c = \int \Phi \ud W_H$ for some time-change $\tau^c$, some Hilbert space $H$, some $H$-cylindrical Brownian motion $W_H$ (see Subsection \ref{subsec:prelimstint}), and some $\Phi:\Omega \to \gamma(L^2(\mathbb R_+; H), X)$ (see Subsection \ref{subsec:gammanorm}; we are allowed to integrate such functions due to \cite{NVW}). Then it is sufficient to set $N^c:= \int \Phi \ud \widetilde W_H \circ A^c$ (where $A^c$ is the {\em inverse} time change to $\tau$, i.e.\ $\tau \circ A_t = A \circ \tau_t = t$ a.s.\ for any $t\geq 0$) to be the corresponding decoupled tangent martingale $N^c$ to $M^c$ for some independent $H$-cylindrical Brownian motion $\widetilde W_H$. Therefore $N^c(\omega)$ is a time-changed Wiener integral with a deterministic integrator, which agrees with \eqref{eq:inftoUMDintwrtBrandindBr}. The construction of a decoupled tangent martingale $N^a$ to $M^a$ simply copies the one done in the discrete case due to the approximation argument presented in Proposition \ref{prop:MmapproxMinLpforacccase} (see \cite{dlPG,KwW91,KwW92,dlP94} and Subsection \ref{subsec:dectanPDwithAJ}).

The most intriguing thing happens in the quasi-left continuous case. Recall that $M^q$ can be presented as an integral with respect to a compensated random measure, namely
\begin{equation}\label{eq:Mq=intxdbarmuMqINTRO}
M^q_t = \int_{[0, t]\times X} x \ud\bar{\mu}^{M^q}(\cdot, x),\;\;\; t\geq 0,
\end{equation} 
where $\mu^{M^q}$ is defined by \eqref{eq:defofmuMINTRO}, $\nu^{M^q}$ is the corresponding compensator, $\bar{\mu}^{M^q} = {\mu}^{M^q} - {\nu}^{M^q}$ (see Theorem \ref{thm:XisUMDiffMisintxwrtbarmuMvain}). It turns out that in this case
\begin{equation}\label{eq:Nq=intxdbarmuMqCoxINTRO}
N^q_t := \int_{[0, t]\times X} x \ud\bar{\mu}^{M^q}_{\rm Cox}(\cdot, x),\;\;\; t\geq 0,
\end{equation} 
is a decoupled tangent martingale to $M^q$, where $\mu^{M^q}_{\rm Cox}(\cdot, x)$ is a {\em Cox process} directed by $\nu^{M^q}$, $\bar{\mu}^{M^q}_{\rm Cox} = {\mu}^{M^q}_{\rm Cox} - {\nu}^{M^q}$. Cox processes were introduced by D.R.\ Cox in \cite{Cox55}, and in the present case this is a random measure on an enlarged probability space such that $\mu^{M^q}_{\rm Cox} (\omega)$ is a {\em Poisson random measure} on $\mathbb R_+ \times X$ with the {\em intensity} (or {\em compensator}, see Subsection \ref{subsec:PoissRMprelim}) $\nu^{M^q}(\omega)$ for a.e.\ $\omega \in \Omega$ from the original probability space. Thus $N^q(\omega)$ is a Poisson integral with deterministic integrator, which corresponds to \eqref{eq:inftoUMDintwrtPoisbyindPois}. The idea of employing Cox processes for creating decoupled tangent processes here is not new (see e.g.\ \cite{Kal17}), but what is the most difficult in the vector-valued case is to show that both integrals \eqref{eq:Mq=intxdbarmuMqINTRO} and \eqref{eq:Nq=intxdbarmuMqCoxINTRO} make sense and tangent (see Subsection \ref{subsec:dectanPDQLC}).

It is worth noticing that in Subsection \ref{subsec:RMstochintinBanspaxcewithCOx} we are discussing $L^p$-estimates for {\em general} vector-valued integrals with respect to {\em general} random measures. Recall that this type of estimates goes back to Novikov \cite{Nov75}, where he upper bounded an $L^p$-moment of a real-valued stochastic integral $\int F \ud \bar{\mu}$ by integrals in terms of $F$ and the compensator $\nu$ of $\mu$ (here $\bar{\mu} = \mu-\nu$; see Lemma \ref{subsec:RMstochintinBanspaxcewithCOx}). Later on sharp estimates of this form have been proven by Marinelli and R\"ockner \cite{MarRo} in the Hilbert space case and by Dirksen and the author \cite{DY17} in the $L^q$ case ($1<q<\infty$). In Theorem \ref{thm:mumuCoxcomparable} we show that for any UMD-valued elementary predictable $F$ and for any quasi-left continuous random measure $\mu$ one has that 
\begin{equation}\label{eq:intwrtmuandmuCOXINTRO}
\mathbb E \sup_{t\geq 0} \Bigl\| \int_{[0,t]\times J} F \ud \bar\mu \Bigr\|^p \eqsim_{p, X} \mathbb E  \Bigl\| \int_{\mathbb R_+\times J} F \ud \bar \mu_{\rm Cox} \Bigr\|^p,\;\;\;\; 1\leq p<\infty,
\end{equation}
where $\nu$ is a compensator of $\mu$, $\bar{\mu}:= \mu -\nu$, $ \mu_{\rm Cox}$ is a Cox process directed by $\nu$, and $\bar \mu_{\rm Cox} :=  \mu_{\rm Cox} -\nu$. Note that though it seems that the right-hand side of \eqref{eq:intwrtmuandmuCOXINTRO} depends on $F$ and $\mu_{\rm Cox}$, the distribution of the Cox process entirely depends on $\nu$ (in particular, $\mu_{\rm Cox}(\omega)$ is a Poisson random measure with the intensity $\nu(\omega)$), and so on the right-hand side of \eqref{eq:intwrtmuandmuCOXINTRO} we in fact have $\mathbb E \|F\|_{p, X, \nu}^p$, where $\|F(\omega)\|_{p, X, \nu(\omega)}$ is the $L^p$-norm of a stochastic integral of a deterministic function $F(\omega)$ with respect to the corresponding compensated Poisson random measure (see Subsection \ref{subsec:PoissRMprelim} and \cite{App07,ApRi}). Thus even though \eqref{eq:intwrtmuandmuCOXINTRO} does not provide an explicit formula for a stochastic integral in terms of $F$ and $\nu$, as it was  done in \cite{Nov75,MarRo,DY17}, nevertheless it semigeneralizes the papers \cite{Nov75,MarRo,DY17} as it tells us that in order to get $L^p$ bounds for UMD-valued stochastic integrals with respect to a general random measure we need only to prove the corresponding estimates for the Poisson case with deterministic integrands (see e.g.\ Remark \ref{rem:CoxidlikePoisson}).

\smallskip

 In Section \ref{sec:Uppbdsanddecprop} we show that if $X$ satisfies the so-called {\em decoupling property} (e.g.\ if $X = L^1$), then inequalities of the form 
 \begin{equation}\label{eq:INTROdecpropreno}
   \mathbb E \sup_{0\leq t \leq T} \|M_t\|^p \lesssim_{p, X}   \mathbb E  \|N_T\|^p ,\;\;\; T>0, \;\; p\in[1,\infty),
 \end{equation}
are possible for an $X$-valued martingale $M$ satisfying broad assumptions (see e.g.\ Remark \ref{rem:stochinforgenmartfordecpropet}), where $N$ is a corresponding decoupled tangent local martingale. Recall that the decoupling property was introduced by S.G.\ Cox and Veraar in \cite{CV07,CV} as a natural property while working with discrete decoupled tangent martingales and stochastic integrals.

\smallskip

In \cite{Kal17} Kallenberg also has shown $\phi$-inequalities for tangent continuous martingales (where $\phi$ is a convex function of moderate growth; recall that one can even omit the convexity assumption for conditionally symmetric martingales). In Section \ref{sec:Cfwithmodg} we extend these inequalities to full generality (i.e.\ general martingales in UMD Banach spaces). Though \cite{Kal17} also treats the semimartingale case, it is not known to the author how to prove such inequalities for vector-valued semimartingales.

In Section \ref{sec:intwrtgenmart} we present estimates for vector-valued stochastic integrals with respect to a general martingale which extend both \eqref{eq:inftoUMDintwrtBrandindBr} and \eqref{eq:inftoUMDintwrtPoisbyindPois}. Namely, we show that for a general $H$-valued martingale $\widetilde M$ (where $H$ is  a Hilbert space) and an $\mathcal L(H, X)$-valued elementary predictable process $\Phi$ one has that for any $1\leq p<\infty$
\begin{align}\label{eq:stochintestwithpredRHSINTRO}
 \mathbb E \sup_{t\geq 0} \Bigl\| \int_0^t \Phi \ud \widetilde M \Bigr\|^p &\eqsim_{p, X} \mathbb E \|\Phi q_{ \widetilde M^c}^{1/2}\|_{\gamma(L^2(\mathbb R_+, [ M^c]; H), X)}^p\nonumber\\ 
 &\quad\quad \quad+ \mathbb E  \Bigl\|\int_{\mathbb R_+ \times H}\Phi(s) h \ud \bar \mu^{\widetilde M^q}_{\rm Cox}(s, h) \Bigr\|^p\\
 &\quad \quad \quad \quad\quad \quad+ \mathbb E \Bigl\|\sum_{0\leq t < \infty} \Phi \Delta \widetilde N_t^{a}\Bigr\|^p,\nonumber
\end{align}
where $\widetilde M = \widetilde M^c + \widetilde M^q + \widetilde M^a$ is the canonical decomposition, $q_{\widetilde M^c}$ is a {\em quadratic variation derivative} of $\widetilde M^c$ (see Subsection \ref{subsec:quadrvar}), and $\widetilde N^a$ is a decoupled tangent martingale to $\widetilde M^a$.
Note that the right-hand side of \eqref{eq:stochintestwithpredRHSINTRO} in fact can be seen as an $L^p$ moment of a {\em predictable} process. Such estimates are in the spirit of works of Novikov \cite{Nov75} and Dirksen and the author in \cite{DY17}, and they are very different from the classical vector-valued Burkholder-Davis-Gundy inequalities presented e.g.\ in \cite{Y18BDG,VY18,BDG,MarRo16}. Note that the upper bound of \eqref{eq:stochintestwithpredRHSINTRO} characterizes the decoupling property (see Section \ref{sec:Uppbdsanddecprop} and Remark \ref{rem:stochinforgenmartfordecpropet}).

As it was discussed above, the notion of tangency heavily exploits the Meyer-Yoeurp decomposition, which existence for a general $X$-valued martingale is equivalent to $X$ obtaining the UMD property. But what if we have {\em weak tangency}, i.e.\ what if for a given Banach space $X$ and a pair of $X$-valued martingales $M$ and $N$ we have that $\langle M, x^*\rangle$ and $\langle N, x^*\rangle$ are tangent for any $x^*\in X^*$? How does this correspond to the tangency property and will we then have $L^p$-estimates for a family of Banach spaces different from the UMD one? In Section \ref{sec:WTversusT} we show that in the UMD case weak tangency and tangency coincide. Moreover, in the non-UMD setting no estimate of the form \eqref{eq:LpineqfortangmartinUMDINTRO} for weakly tangent martingales is possible.

In Section \ref{sec:LCharandRecPro} we discuss for which Banach spaces it is possible
to extend the definition of
decoupled tangent local martingales (and prove their existence)
via using weak local characteristics. It turns out that this is possible for Banach spaces with the so-called {\em recoupling property} which is dual to the decoupling property \eqref{eq:INTROdecpropreno} and which occurs to be equivalent to the well-discovered {\em UMD$^+$ property}. Moreover, the converse holds true, i.e.\ a Banach space $X$ having the recoupling property is necessary for any $X$-valued local martingale to have a decoupled tangent local martingale (see Theorem \ref{thm:existofdecoupltagivenerqw} and Remark \ref{rem:decweaktangmartopenq}). It remains open whether recoupling and UMD are identical (see e.g.\ \cite[Section O]{HNVW1}).

In Section \ref{sec:indincrements} we consider vector-valued martingales with independent increments. First recall that one of the inventors of local characteristics was Grigelionis (that is why local characteristics are sometimes called {\em Grigelionis characteristics}). In particular, in \cite{Grig77} he proved that a real-valued martingale has independent increments if and only if it has deterministic local characteristics (this result was extended by Jacod and Shiryaev in \cite{JS} to multi dimensions). In Section \ref{sec:indincrements} we extend this celebrated result to infinite dimensions. In preliminary Subsection \ref{subsec:weakloccharandMII} we show that for any Banach space $X$, an $X$-valued local martingale $M$ has independent increments if and only if it has deterministic {\em weak local characteristics}, i.e.\ the family $([\langle M, x^* \rangle^c, \nu^{\langle M, x^* \rangle}])_{x^*\in X^*}$ is deterministic (such an object always exists since $\langle M, x^* \rangle$ has local characteristics as a real-valued local martingale). Next in Subsection \ref{subsec:martwithindincrgenform} we prove that if this is the case, then $M$ actually has local characteristics (which are of course deterministic), and moreover, $M$ has the canonical decomposition $M = M^c + M^q + M^a$ so that $M^c$, $M^q$, and $M^a$ are mutually independent, and there exists a deterministic time-change $\tau^c$ such that $M^c\circ \tau^c = \int \Phi \ud W_H$ is a stochastic integral of some deterministic $\Phi \in \gamma(L^2(\mathbb R_+; H), X)$ with respect to some $H$-cylindrical Brownian motion $W_H$, $M^q  = \int x \ud \widetilde N(\cdot, x)$ for some fixed Poisson random measure $N$ on $\mathbb R_+ \times X$, and $M^a$ is a sum of its independent jumps which occur at deterministic family of times $(t_n)_{n\geq 1}$.
Note that throughout Section~\ref{sec:indincrements} $X$ is a general Banach space and there is no need in the UMD property.

Recall that Jacod \cite{Jac84} and Kwapie\'{n} and Woyczy\'{n}ski \cite{KwW91} proved that for a real-valued quasi-left continuous martingale $M$ a decoupled tangent martingale $N$ on $[0, T]$ is nothing but a limit in distribution of discrete decoupled tangent martingales $\tilde f^n$ as $n\to \infty$, where for each $n\geq 1$ a martingale $\tilde f^n := (\tilde f_k^n)_{k=1}^n$ is a decoupled tangent martingale to a discrete martingale $(f_k^n)_{k=1}^n = (M_{Tk/n})_{k=1}^n$, and the limit is considered as a limit in distribution of random variables with values in the Skorokhod space $\mathcal D([0, T], \mathbb R)$ (see Definition \ref{def:defofskotokhodspaxc}). In Section \ref{sec:appJKW} we extend this result to general UMD-valued martingales (thus somehow mixing together the discrete works of McConnell \cite{MC}, Hitczenko \cite{HitUP}, and de la Pe\~{n}a \cite{dlP94} and quasi-left continuous works of Jacod  \cite{Jac84,Jac83} and Kwapie\'{n} and Woyczy\'{n}ski \cite{KwW91}). In our setting such a limit theorem is possible since we know what the limiting object is (i.e.\ how does a decoupled tangent martingale look like) due to Section \ref{sec:tangmartconttime}, because of certain approximation techniques, and thanks to properties of stochastic integrals and the canonical decomposition.

Section \ref{sec:LevyHinchinFormula} is devoted to a characterization of the local characteristics of a general UMD-valued martingale via an exponential formula which can be considered as an extension of the L\'evy-Khinchin formula. There we show that for any UMD-valued martingale $M$ with the local characteristics $([\![M^c]\!], \nu^M)$ and for any $x^*\in X^*$
\begin{equation}\label{eq:LHforgenmartstrangeformINTRO}
t\mapsto e^{i\langle M_t, x^*\rangle} / {G_t(x^*)},\;\;\; t\geq 0,
\end{equation}
is a local martingale on $[0,\tau_{G(x^*)})$, where
\begin{multline*}
A_t(x^*) := - \frac 12 [\![M^c]\!]_t(x^*, x^*) + \int_{[0,t]\times X}(e^{i\langle x, x^*\rangle} - 1 - i\langle x, x^*\rangle) \ud \nu^M(s, x),\;\;\; t\geq 0,
\end{multline*}
\[
G_t(x^*):= e^{A_t(x^*)}\Pi_{0\leq s\leq t}(1 + \Delta A_s(x^*)) e^{-\Delta A_s(x^*)},\;\;\; t\geq 0,
\]
and $\tau_{G(x^*)} := \inf \{t\geq 0: G_t(x^*) = 0\} =  \inf \{t\geq 0: \Delta A_t(x^*) = -1\}$.
Moreover, $([\![M^c]\!], \nu^M)$ are unique bilinear form-valued  predictable process and predictable random measure such that \eqref{eq:LHforgenmartstrangeformINTRO} is a local martingale on $[0,\tau_{G(x^*)})$. This is a natural generalization of the L\'evy-Khinchin formula \eqref{eq:LHformulaforLevyptocINTRO} as if we set $M$ to be quasi-left continuous with independent increments, then $\tau_{G(x^*)} = \infty$ and $G(x^*)$ is deterministic, and consequently  \eqref{eq:LHforgenmartstrangeformINTRO} being a local martingale implies \eqref{eq:LHformulaforLevyptocINTRO}. 
The proof of the fact that \eqref{eq:LHforgenmartstrangeformINTRO} is a local martingale  on $[0,\tau_{G(x^*)})$ presented in Section \ref{sec:LevyHinchinFormula} follows directly from the multidimensional case shown by Jacod and Shiryaev in \cite{JS}.

In Section \ref{sec:charsubandchardomofmart} we discover $L^p$-inequalities for {\em characteristically subordinated} and {\em characteristically dominated} martingales. These notions are predictable versions of {\em weak differential subordination} of martingales (see \cite{Y17FourUMD,Y17MartDec,Y17GMY,OY18}) and {\em martingale domination} (see \cite{Y18BDG,Os11b,Burk87}) and have the following form: for a Banach space $X$ an $X$-valued martingale $N$ is characteristically subordinate to an $X$-valued martingale $M$ if for any $x^*\in X^*$ we have that a.s.\ 
\begin{enumerate}[(i)]
\item $|\langle N_0, x^* \rangle| \leq |\langle M_0, x^* \rangle|$,
\item $[\langle N^c, x^* \rangle]_t - [\langle N^c, x^* \rangle]_s \leq [\langle M^c, x^* \rangle]_t - [\langle M^c, x^* \rangle]_s$ for any $0\leq s\leq t$, and
\item $\nu^{\langle N, x^* \rangle} \leq \nu^{\langle M, x^* \rangle}$,
\end{enumerate}
and $N$ is characteristically dominated by $M$ if a.s.\
\begin{enumerate}[(i)]
\item $|\langle N_0, x^* \rangle| \leq |\langle M_0, x^* \rangle|$ for any $x^*\in X^*$,
\item $[\![N^c]\!]_{\infty} \leq [\![M^c]\!]_{\infty}$, and
\item $\nu^{N}(\mathbb R_+ \times \cdot) \leq \nu^{M}(\mathbb R_+ \times \cdot)$
\end{enumerate}
(here $M^c$ and $N^c$ are the continuous parts of $M$ and $N$, see Subsection \ref{subsec:candec}). In Subsection \ref{subsec:charsub} we compare weak differential subordination and characteristic subordination (these properties turn out to be incomparable) and show inequalities \eqref{eq:LpineqfortangmartinUMDINTRO} for characteristically subordinated martingales. In Subsection \ref{subsec:chardom} we show  inequalities \eqref{eq:LpineqfortangmartinUMDINTRO} for {\em quasi-left continuous} characteristically dominated martingales (both estimates are proven in the UMD setting). $L^p$-estimates for general characteristically dominated martingales remain open (see Remark \ref{rem:chardomfordiscretecase}) as the author does not know how to gain such estimates in the discrete case, though this case is very much in the spirit of the original work of Zinn  \cite{Zinn85}.

\smallskip

In the end of the present paper we have appendix Sections \ref{sec:tangunderlinoper} and \ref{sec:appmartapprox} where we collect some technical facts concerning tangency and martingale approximations.

\medskip

All over this section we used to talk about some mysterious UMD spaces. Recall that UMD spaces were introduced by Burkholder in 1980's while working with martingale transforms (see e.g.\ \cite{Burk84,Burk81,Burk83,Burk01}), and nowadays these spaces are used abundantly in vector-valued stochastic and harmonic analysis (see e.g.\ \cite{HNVW1,Bour83,Rubio86,NVW,Y18BDG,Y17FourUMD,GM-SS}). Let us shortly outline here where exactly the UMD property is needed/used in the present paper.
\begin{itemize}
\item Theorem \ref{thm:intromccnnll} due to Hitczenko and McConnell,
\item Burkholder's works \cite{Burk84,Burk81} on martingale transforms,
\item existence of the Meyer-Yoeurp and the canonical decomposition and the corresponding $L^p$- and $\phi$-estimates (see \cite{Y17MartDec,Y17GMY,Y18BDG}),
\item vector-valued stochastic integration with respect to a cylindrical Brownian noise thanks to van Neerven, Veraar, and Weis \cite{NVW},
\item Burkholder-Davis-Gundy inequalities (see \cite{Y18BDG,VY18}),
\item existence of a covariation bilinear form $[\![M]\!]$ (see \cite{Y18BDG}).
\end{itemize}
On the other hand, we obtain several new characterizations of the UMD property, such as
\begin{itemize}
\item estimate \eqref{eq:LpineqfortangmartinUMDINTRO} for continuous-time tangent martingales,
\item existence of a decoupled tangent martingale (see Theorem \ref{thm:mainforDTMgencasenananana} and Section \ref{sec:LCharandRecPro}),
\item estimate \eqref{eq:inftoUMDintwrtPoisbyindPois},
\item the fact that for a purely discontinuous quasi-left continuous martingale $M^q$ the integral $\int x \ud \bar{\mu}^{M^q}(\cdot, x)$ exists and coincides with $M^q$ (see Theorem \ref{thm:XisUMDiffMisintxwrtbarmuMvain}),
\item estimate \eqref{eq:intwrtmuandmuCOXINTRO},
\item $L^p$-estimates for characteristically subordinated and characteristically dominated martingales (see Section \ref{sec:charsubandchardomofmart}).
\end{itemize}
This demonstrates once again that the UMD property is not just a technical assumption, but a key player in any game involving martingales in Banach spaces.

\medskip

\emph{Acknowledgment} -- The author thanks Sjoerd Dirksen, Stefan Geiss, Jan van Neerven, and Mark Veraar for fruitful discussions and helpful comments. The author is grateful to the reviewer for their careful reading and detailed comments, which helped to significantly improve the presentation of the paper.

\section{Preliminaries}\label{sec:prelim}

Throughout the present article any Banach space is considered to be over the scalar field $\mathbb R$. (This is done as we are going to work with continuous-time martingales, which properties are well discovered only in the case of the real scalar field, see e.g.\ \cite{Kal,JS,Prot}.)

\smallskip

Let $X$ be a Banach space, $B\subset X$ be Borel. Then we denote the $\sigma$-algebra of all Borel subsets of $B$ by $\mathcal B(B)$.

\smallskip

For $a, b\in \mathbb R$ we write $a \lesssim_A b$ if there exists a constant $c$ depending only on $A$ such that
$a \leq cb$. $\gtrsim_A$ is defined analogously. We write $a \eqsim_A b$ if both $a \lesssim_A b$ and $a \gtrsim_A b$ hold simultaneously.

\smallskip

We will need the following definitions.

\begin{definition}\label{def:ofRadRV}
 A random variable $\xi:\Omega \to \mathbb R$ is called {\em Rademacher} if $\mathbb P(\xi=1) = \mathbb P(\xi = -1) = 1/2$.
\end{definition}

\begin{definition}\label{def:defofskotokhodspaxc}
 Let $X$ be a Banach space, $A \in \mathbb R$ be an interval (finite or infinite). The linear space $\mathcal D(A, X)$ of all $X$-valued c\`adl\`ag (i.e.\ right continuous with left limits) functions on $A$ is called the {\em Skorokhod space}.
\end{definition}

Recall that $\mathcal D(A, X)$ endowed with the sup-norm is a Banach space (see e.g.\ \cite{Y17FourUMD,VerPhD}).

\smallskip

For a Banach space $X$ and for a measurable space $(S, \Sigma)$ a function $f:S \to X$ is called {\em strongly measurable} if there exists a sequence $(f_n)_{n\geq 1}$ of simple functions such that $f_n \to f$ pointwise on $S$ (see \cite[Section 1.1]{HNVW1}). In the sequel we will call a function $f$ {\em strongly predictable} if it is strongly measurable with respect to the predictable $\sigma$-algebra (which is either $\mathcal P$, see Subsection \ref{subsec:BSvmart}, or $\widetilde{\mathcal P}$, see Subsection~\ref{subsec:ranmeasures}, depending on the underlying $S$).

 For a Banach space $X$ and a function $A:\mathbb R_+  \to X$ we set $A^*\in \overline{\mathbb R}_+$ to be $A^* := \sup_{t\geq 0} \|A_t\|$.

 Throughout the paper, unless stated otherwise, the probability space and filtration are assumed to be generated by all the processes involved.

\subsection{Enlargement of a filtered probability space}\label{subsec:enlargofFsPS}

We will need the following definition of an enlargement of a filtered probability space (see e.g.\ \cite[pp.\ 172--174]{KwW91}).

\begin{definition}\label{def:enlagoffiltprobspace}
Let $(\Omega, \mathcal F, \mathbb P)$ be a probability space with a filtration $\mathbb F = (\mathcal F_t)_{t\geq 0}$. Then a probability space $(\overline{\Omega}, \overline{\mathcal F}, \overline{\mathbb P})$ with a filtration $\overline{\mathbb F} = (\overline {\mathcal F}_t)_{t\geq 0}$ is called to be {\em an enlargement} of $(\Omega, \mathcal F, \mathbb P)$ and  $\mathbb F$ if there exists a measurable space $(\widehat{\Omega}, \widehat{\mathcal F})$ such that $\overline{\Omega} = \Omega \times \widehat{\Omega}$ and $ \overline{\mathcal F} = \mathcal F \otimes \widehat{\mathcal F}$, if there exists a family of probability measures $(\widehat {\mathbb P}_{\omega})_{\omega\in \Omega}$ such that $\omega \mapsto \widehat {\mathbb P}_{\omega}(B)$ is $\mathcal F$-measurable for any $B \in \widehat{\mathcal F}$ and
\[
\overline{\mathbb P}(A \times B) = \int_{A} \widehat{\mathbb P}_{\omega}(B) \ud \mathbb P(\omega),\;\;\; A\in \mathcal F, \;B\in  \widehat{\mathcal F},
\]
and if for any $\omega\in \Omega$ there exists a filtration $\widehat{\mathbb F}^{\omega} = (\widehat{ \mathcal F}_t^{\omega})_{t\geq 0}$ such that for any $B \in \widehat{\mathcal F}$ the process 
$$
(t, \omega) \mapsto \mathbf 1_{\widehat{\mathcal F}^{\omega}_t}(B),\;\;\; t\geq 0,\;\; \omega\in\Omega,
$$ 
is $\mathbb F$-adapted, and such that $\overline {\mathcal F}_t = \mathcal F_t \otimes \widehat{ \mathcal F}^{\cdot}_t$ for any $t\geq 0$, i.e.\
\[
A \times B \in \overline {\mathcal F}_t \;\; \text{if}\;\; A \in \mathcal F_t  \;\; \text{and}\;\; B \in \widehat{ \mathcal F}^{\omega}_t \;\; \text{for any}\;\; \omega \in A.
\]
\end{definition}

\begin{example}
A classical example of an enlargement of a filtered probability space can be a {\em product space}, i.e.\ the case when $\widehat {\mathbb P}_{\omega} = \widehat {\mathbb P}$ and $\widehat{ \mathcal F}_t^{\omega} = \widehat{ \mathcal F}_t$, $t\geq 0$, for any $\omega\in \Omega$ for some fixed measure $\widehat {\mathbb P}$ and some fixed filtration $\widehat{ \mathbb F} = (\widehat{ \mathcal F}_t)_{t\geq 0}$.
\end{example}

\subsection{Conditional expectation on a product space. Conditional probability and conditional independence}\label{subsec:ConexponPSCondProbCondIndep}

Let $(\Omega, \mathcal F, \mathbb P)$ be a probability space, and assume that there exist probability spaces $(\Omega', \mathcal F', \mathbb P')$ and $(\Omega'', \mathcal F'', \mathbb P''_{\omega'})_{\omega'\in \Omega'}$ (where $\mathbb P''_{\omega'}$ depends on $\omega'\in \Omega'$ in $\mathcal F'$-measurable way, see Subsection \ref{subsec:enlargofFsPS}) such that
\begin{equation}\label{eq:OmegaFPistheprodofOmega'F'P'Omega''F''P''}
(\Omega, \mathcal F, \mathbb P) = (\Omega' \times \Omega'', \mathcal F' \otimes \mathcal F'', \mathbb P' \otimes \mathbb P ''),
\end{equation}
i.e.\  $\mathbb P''_{\omega'}(A_2)$ is $\mathcal F'$-measurable for any $A_2 \in \mathcal F''$ and
$$
\mathbb P(A_1 \times A_2) = \int_{A_1} \mathbb P''_{\omega'}(A_2) \ud \mathbb P'(\omega'),\;\;\; A_1 \in \mathcal F',\;\; A_2 \in \mathcal F''.
$$
A particular example would be if $\mathbb P''_{\omega'} = \mathbb P''$ is a probability measure which does not depend on $\omega'\in \Omega'$.
Let $X$ be a Banach space, and let $f\in L^1(\Omega; X)$ (see \cite[Section 1.2]{HNVW1} for the definition of $L^p(\Omega; X)$). Then $\mathbb E (f| \mathcal F')$ is well defined (see \cite[Section 2.6]{HNVW1}; by $\mathbb E (\cdot |\mathcal F')$ here we mean $\mathbb E (\cdot | \mathcal F'\otimes \{\Omega '', \varnothing\})$), and moreover, by Fubini's theorem $f(\omega', \cdot)$ exists and strongly measurable for a.e.\ $\omega'\in \Omega'$ (the proof is analogous to the one provided by \cite[Section 3.4]{BMT07}). It is easy to see that for a.e.\ $\omega'\in \Omega'$
\begin{equation}\label{eq:defofmathbbEOmega''orE}
\mathbb E (f| \mathcal F')(\omega', \cdot) = \int_{\Omega''} f(\omega',\omega '') \ud \mathbb P''_{\omega'}(\omega'') =: \mathbb E_{\Omega ''} f(\omega',\cdot),
\end{equation}
where the notation $\mathbb E_{\Omega''}$ means {\em averaging for every fixed $\omega'\in \Omega '$ over $\Omega ''$}. Indeed, for any $A\in \mathcal F'$ by Fubini's theorem we have that
\[
\int_{A \times \Omega''} f \ud \mathbb P = \int_A \mathbb E_{\Omega''} f(\omega',\cdot) \ud \mathbb P'(\omega'),
\]
so \eqref{eq:defofmathbbEOmega''orE} follows by the definition of a conditional expectation. 

\begin{example}\label{ex:defofExiforrvxis}
 If there exists an $\mathcal F''$-measurable $\xi:\Omega'' \to \mathbb R$ such that $\mathcal F'' = \sigma(\xi)$ for a.e.\ $\omega' \in \Omega'$, then we will often write $\mathbb E_{\xi} :=\mathbb  E_{\Omega''} = \mathbb E( \cdot | \mathcal F' )$ (i.e.\ {\em averaging over all the values of $\xi$}).
\end{example}

Let $(\Omega, \mathcal F, \mathbb P)$ be a probability space, $(S, \Sigma)$ be a measurable space, $\xi:\Omega \to S$ be a random variable. Let $\mathcal G\subset \mathcal F$ be a sub-$\sigma$-algebra. Then we define the {\em conditional probability} $\mathbb P(\xi|\mathcal G):\Sigma \to L^1(\Omega)$ to be as follows
\begin{equation}\label{eq:condprob}
\mathbb P(\xi|\mathcal G)(A) := \mathbb E \bigl(\mathbf 1_{A}(\xi)|\mathcal G\bigr),\;\;\; A\in \Sigma.
\end{equation}

Now let $N\geq 1$, $(\xi_n)_{n= 1}^N$ be $S$-valued random variables. Then $\xi_1, \ldots, \xi_N$ are called {\em conditionally independent given $\mathcal G$} if for any sets $B_1,\ldots, B_N\in \Sigma$ we have that
\begin{equation}\label{eq:defofcondinderp}
 \mathbb P\bigl( (\xi_n)_{n=1}^N \big| \mathcal G \bigr)(B_1 \times \cdots \times B_N)= \Pi_{n=1}^N \mathbb P(\xi_n |\mathcal G)(B_n).
\end{equation}

In the sequel we will need the following proposition.

\begin{proposition}\label{prop:Omega'Omega';arecondindepifondforanyomega'as}
Let $(\Omega, \mathcal F, \mathbb P)$ be defined by \eqref{eq:OmegaFPistheprodofOmega'F'P'Omega''F''P''} for some $(\Omega', \mathcal F', \mathbb P')$ and for some family $(\Omega'', \mathcal F'', \mathbb P''_{\omega'})_{\omega'\in\Omega'}$. Let $(\xi_n)_{n=1}^N$ be as above. Assume that for almost any fixed $\omega'\in \Omega'$, $\bigl(\xi_n(\omega',\cdot)\bigr)_{n=1}^N$ are independent. Then $(\xi_n)_{n=1}^N$ are conditionally independent given $\mathcal F'$.
\end{proposition}

\begin{proof}
By the definition of conditional independence we need to show that for any sets $B_1,\ldots, B_n \in \Sigma$
\[
 \mathbb P\bigl( (\xi_n)_{n=1}^N \big| \mathcal F' \bigr)(B_1 \times \cdots \times B_N)= \Pi_{n=1}^N \mathbb P(\xi_n |\mathcal F')(B_n).
\]
To this end note that by \eqref{eq:defofmathbbEOmega''orE} for $\mathbb P'$-a.e.\ $\omega'\in \Omega'$
\begin{align*}
 \mathbb P\bigl( (\xi_n)_{n=1}^N \big| \mathcal F' \bigr)(B_1 \times \cdots \times B_N)&(\omega', \cdot) = \mathbb E\bigl (\mathbf 1_{B_1 \times \cdots \times B_N}(\xi_1,\ldots, \xi_N)\big|\mathcal F'\bigr)(\omega',\cdot)\\
 &=\int_{\Omega''} \Pi_{n=1}^N \mathbf 1_{B_n} \bigl(\xi_n(\omega', \omega '')\bigr)\ud \mathbb P''_{\omega'}(\omega'')\\
 &= \Pi_{n=1}^N\int_{\Omega''} \mathbf 1_{B_n} \bigl(\xi_n(\omega', \omega '')\bigr)\ud \mathbb P''_{\omega'}(\omega'')\\
 &= \Pi_{n=1}^N \mathbb P(\xi_n|\mathcal F')(\omega', \cdot)(B_n),
\end{align*}
which terminates the proof.
\end{proof}

We will also need the following consequence of the proposition.

\begin{corollary}\label{cor:condindgivenRVsuffcond}
Let $(S, \Sigma)$ and $(T, \mathcal{T})$ be measurable spaces, let $(\Omega, \mathcal F, \mathbb P)$ be defined by \eqref{eq:OmegaFPistheprodofOmega'F'P'Omega''F''P''}, and let $\xi:\Omega'\to S$ and $\eta:\Omega \to T$ be measurable. Assume that $\eta$ is measurable with respect to $\sigma(\xi)\otimes \mathcal F''$. Let $F_1, \ldots, F_N:S\times T \to \mathbb R$. Then $\bigl(F_n(\xi, \eta)\bigr)_{n=1}^N$ are conditionally independent given $\sigma(\xi)$ if there exists $A\in \Sigma$ with $\mathbb P(\xi \in A) = 1$ such that $\bigl(F_n(a, \eta(a, \cdot))\bigr)_{n=1}^N$ are independent for any $a\in A$.
\end{corollary}

\begin{proof}
The corollary follows from Proposition \ref{prop:Omega'Omega';arecondindepifondforanyomega'as} if one sets $\Omega'  := A$, $\mathbb P' := \mathcal L(\xi)$, and 
$$
\mathbb P''_{\omega'} := \mathcal L\bigl(\eta(\omega')\bigr),\;\;\; \omega'\in\Omega',
$$ 
where the latter exists by \cite[Theorem 10.2.2 and pp.\ 344, 386]{Dud89}
(here $\mathcal L$ means the distribution).
\end{proof}

We refer the reader to \cite{HNVW1} for further details on vector-valued integration and vector-valued conditional expectation.

\subsection{The UMD property}\label{subsec:prelimUMD}
A Banach space $X$ is called a {\em UMD\footnote{UMD stands for {\em unconditional martingale differences}} space} if for some (equivalently, for all)
$p \in (1,\infty)$ there exists a constant $\beta>0$ such that
for every $N \geq 1$, every martingale
difference sequence $(d_n)^N_{n=1}$ in $L^p(\Omega; X)$, and every $\{-1,1\}$-valued sequence
$(\varepsilon_n)^N_{n=1}$
we have
\[
\Bigl(\mathbb E \Bigl\| \sum^N_{n=1} \varepsilon_n d_n\Bigr\|^p\Bigr )^{\frac 1p}
\leq \beta \Bigl(\mathbb E \Bigl \| \sum^N_{n=1}d_n\Bigr\|^p\Bigr )^{\frac 1p}.
\]
The least admissible constant $\beta$ is denoted by $\beta_{p,X}$ and is called the {\em UMD$_p$ constant} of $X$ (or just the {\em UMD constant} of $X$ if the value of $p$ is understood). It is well known (see \cite[Chapter 4]{HNVW1}) that $\beta_{p, X}\geq p^*-1$ and that $\beta_{p, H} = p^*-1$ for a Hilbert space $H$ and any $1<p<\infty$ (here $p^* := \max\{p, p/(p-1)\}$). 

We will also frequently use the following equivalent definition of the UMD property. $X$ is UMD if and only if for any $1\leq p<\infty$ and for any $(d_n)_{n=1}^N$ and $(\eps_n)_{n=1}^N$ as above we have that
\[
 \mathbb E \sup_{1\leq m\leq N} \Bigl\| \sum^m_{n=1} \varepsilon_n d_n\Bigr\|^p
\eqsim_{p, X} \mathbb E\sup_{1\leq m\leq N} \Bigl \| \sum^m_{n=1}d_n\Bigr\|^p.
\]
Note that a similar definition of the UMD property can be provided for a general convex function of moderate growth (see e.g.\ \cite[p.\ 1000]{Burk81}).
 We refer the reader to \cite{Burk81,Burk01,HNVW1,Rubio86,Pis16,LVY18,HNVW2,GM-SS,dlPG,Y18BDG,GY19} for details on UMD Banach spaces.
 
 \subsection{Stopping times}\label{subsec:prelimsttimes}

A stopping time $\tau$ is called {\em predictable} if there exists a sequence of stopping times $(\tau_n)_{n\geq 1}$ such that $\tau_n<\tau$ a.s.\ on $\{\tau > 0\}$ and $\tau_n \nearrow \tau$ a.s.\ as $n\to \infty$. A stopping time $\tau$ is called {\em totally inaccessible} if $\mathbb P(\tau=\sigma \neq \infty)=0$ for any predictable stopping time $\sigma$. 

With a predictable stopping time $\tau$ we associate a $\sigma$-field $\mathcal F_{\tau-}$ which has the following form
\begin{equation}\label{eq:defofftau-}
 \mathcal F_{\tau-} := \sigma\{\mathcal F_0 \cup (\mathcal F_{t}\cap \{t<\tau\}), t> 0\} = \sigma\{\mathcal F_{\tau_n}, n\geq 1\},
\end{equation}
where $(\tau_n)_{n\geq 1}$ is a sequence of stopping time announcing $\tau$ (see \cite[p.\ 491]{Kal} for details).

Later on we will work with different types of martingales based on the properties of their jumps, and in particular we will frequently use the following definition (see e.g.\ Subsection \ref{subsec:candec}). Recall that for a c\`adl\`ag process $A$ and for a stopping time $\tau$ we set $\Delta A_{\tau}:= A_{\tau} - \lim_{\eps \searrow 0} A_{0 \vee (\tau - \eps)} $ on $\{\tau<\infty\}$.

\begin{definition}
 Let $X$ be a Banach space, $A:\mathbb R_+ \times \Omega \to X$ be a c\`adl\`ag process. Then $A$ is called {\em quasi-left continuous} if $\Delta A_{\tau}=0$ a.s.\ on $\{t<\infty\}$ for any predictable stopping time $\tau$. $A$ is called to have {\em accessible jumps} if $\Delta A_{\tau}=0$ a.s.\ on $\{t<\infty\}$ for any totally inaccessible stopping time $\tau$.
\end{definition}
 We refer the reader to \cite{Kal,JS,DY17,Y17MartDec,Y17GMY} for further details.

\subsection{Martingales: real- and Banach space-valued}\label{subsec:BSvmart}

Let $(\Omega,\mathcal F, \mathbb P)$ be a probability space with a filtration $\mathbb F = (\mathcal F_t)_{t\geq 0}$ which satisfies the usual conditions (see \cite{Prot,Kal,JS}). Then particularly $\mathbb F$ is right-continuous. A {\em predictable $\sigma$-algebra $\mathcal P$} is a $\sigma$-algebra on $\mathbb R_+\times \Omega$ generated by all predictable rectangles of the form $(s, t] \times B$, where $0\leq s < t$ and $B \in \mathcal F_s$.

Let $X$ be a Banach space. An adapted process $M:\mathbb R_+ \times \Omega \to X$ is called a {\em martingale} if $M_t\in L^1(\Omega; X)$ and $\mathbb E (M_t|\mathcal F_s) = M_s$ for all $0\leq s\leq t$. $M$ is called a {\em local martingale} if there exists a nondecreasing sequence $(\tau_n)_{n\geq 1}$ of stopping times such that $\tau_n \nearrow \infty$ a.s.\ as $n\to \infty$ and $M^{\tau_n}$ is a martingale for any $n\geq 1$ (recall that for a stopping time $\tau$ we set $M^{\tau}_t := M_{\tau\wedge t}$, $t\geq 0$, which is a local martingale given $M$ is a local martingale, see \cite{Kal,Prot,JS}). It is well known that in the real-valued case any local martingale is {\em c\`adl\`ag} (i.e.\ has a version which is right-continuous and that has limits from the left-hand side). The same holds for a general $X$-valued local martingale $M$ as well (see e.g.\ \cite{Y17FourUMD,VerPhD}), so for any stopping time $\tau$ one can define $\Delta M_{\tau} := M_{\tau} - \lim_{\eps \searrow 0} M_{0 \vee (\tau - \eps)}$ on $\{\tau<\infty\}$.

Let $1\leq p\leq \infty$. A martingale $M:\mathbb R_+ \times \Omega \to X$ is called an {\em $L^p$-bounded martingale} if $M_t \in L^p(\Omega; X)$ for each $t\geq 0$ and there exists a limit $M_{\infty} := \lim_{t\to \infty} M_t\in L^p(\Omega; X)$ in $L^p(\Omega; X)$-sense. 

Since $\|\cdot\|:X \to \mathbb R_+$ is a convex function, and $M$ is a martingale, $\|M\|$ is a submartingale by Jensen's inequality, and hence by Doob's inequality (see e.g.\ \cite[Theorem 1.3.8(i)]{KS}) we have that for all $1\leq p< \infty$
\begin{equation}\label{eq:DoobsineqXBanach}
 \mathbb E \|M_t\|^p \leq \mathbb E \sup_{0\leq s\leq t} \|M_s\|^p \leq \frac{p}{p-1}\mathbb E \|M_t\|^p,\;\;\; t\geq 0.
\end{equation}
In fact, the following theorem holds for martingales having strong $L^p$-moments (see e.g.\ \cite{Wei94,Wei92} for the real-valued case, the infinite dimensional case can be proven analogously, see e.g.\ \cite{VerPhD,YPhD,Y17FourUMD, DY17,Y18BDG,Y17MartDec}). Recall that Skorokhod spaces were defined in Definition \ref{def:defofskotokhodspaxc}.
\begin{theorem}\label{thm:strongLpmartforamaBanacoace}
Let $X$ be a Banach space, $1\leq p <\infty$. Then the family of all martingales $M:\mathbb R_+ \times \Omega \to X$ satisfying $\mathbb E \sup_{t\geq 0}\|M_t\|^p<\infty$ forms a closed subspace of $L^p(\Omega; \mathcal D(\mathbb R_+, X))$.
\end{theorem}

\begin{remark}\label{rem:locmartislocallyL1DR+X}
 Recall that any local martingale $M:\mathbb R_+ \times \Omega \to X$ is locally in $L^1(\Omega; \mathcal D(\mathbb R_+, X))$. Indeed, set $(\tau_n)_{n\geq 1}$ be a localizing sequence and for each $n\geq 1$ set $\sigma_n := \inf\{t\geq 0: \|M_t\| \geq n\}$. Then $\sigma_n\to \infty$ as $n\to \infty$ a.s.\ since $M$ has c\`adl\`ag paths, and thus $\tau_n\wedge\sigma_n\wedge n\to \infty$ as $n\to \infty$ a.s.\ as well. On the other hand we have that for each $n\geq 1$
 \begin{align*}
  \mathbb E \sup_{t\geq 0}\|M^{\tau_n\wedge \sigma_n\wedge n}_t\|& =  \mathbb E \sup_{ 0 \leq t\leq \tau_n\wedge \sigma_n\wedge n}\|M_t\| \leq \mathbb E n \wedge \| M_{\tau_n\wedge \sigma_n\wedge n}\|\\
  &\leq  n \wedge  \mathbb E \|M_{\tau_n\wedge \sigma_n\wedge n}\| = \leq  n \wedge  \mathbb E \|M^{\tau_n\wedge \sigma_n}_n\| <\infty,
 \end{align*}
where we used the fact that $M^{\tau_n\wedge \sigma_n}$ is a martingale as $M^{\tau_n}$ is a martingale (see e.g.\ \cite{Kal}).
\end{remark}

 Later we will need the following lemma proven e.g.\ in \cite[Subsection 5.3]{DY17} (see also \cite{Kal,Y18BDG}).
 \begin{lemma}\label{lem:DeltaMtaugiventau-=0}
  Let $X$ be a Banach space, $M:\mathbb R_+ \times \Omega \to X$ be a martingale such that $\limsup_{t\to \infty}\mathbb E \|M_t\|<\infty$. Let $\tau$ be a finite predictable stopping time. Then $\Delta M_{\tau}$ is integrable and
  \[
   \mathbb E (\Delta M_{\tau}|_{\mathcal F_{\tau-}}) = 0,
  \]
  where $\mathcal F_{\tau-}$ is defined by \eqref{eq:defofftau-}. Equivalently, $t\mapsto \Delta M_{\tau} \mathbf 1_{[\tau ,\infty)}(t)$, $t\geq 0$, is a martingale.
 \end{lemma}
 
 We refer the reader to \cite{HNVW1,Pis16,YPhD,Kal,Prot,MP,MetSemi,VerPhD,Os12} for further information on martingales.

\subsection{Quadratic variation}\label{subsec:quadrvar}

Let $H$ be a Hilbert space, $M:\mathbb R_+ \times \Omega \to H$ be a local martingale. We define a {\em quadratic variation} of $M$ in the following way:
\begin{equation}\label{eq:defquadvar}
 [M]_t  := \mathbb P-\lim_{{\rm mesh}\to 0}\sum_{n=1}^N \|M(t_n)-M(t_{n-1})\|^2,
\end{equation}
where the limit in probability is taken over partitions $0= t_0 < \ldots < t_N = t$. Note that $[M]$ exists and is nondecreasing a.s.
The reader can find more information on quadratic variations in \cite{MetSemi,MP,VY16} for the vector-valued setting, and in \cite{Kal,Prot,MP} for the real-valued setting.

As it was shown in \cite[Proposition 1]{Mey77} (see also \cite[Theorem 2.13]{Roz90} and \cite[Example 3.19]{VY16} for the continuous case), for any $H$-valued martingale $M$ there exists an adapted process $q_M:\mathbb R_+ \times \Omega \to \mathcal L(H)$ which we will call a {\em quadratic variation derivative}, such that the trace of $q_M$ does not exceed $1$ on $\mathbb R_+ \times \Omega$, $q_M$ is self-adjoint nonnegative on $\mathbb R_+ \times \Omega$, and for any $h,g\in H$ a.s.\
\begin{equation*}\label{eq:whyq_Misinportant5element}
  [\langle M, h\rangle,\langle M, g\rangle]_t =\int_0^t \langle q_M^{1/2}(s)h, q_M^{1/2}(s) g \rangle \ud [M]_s,\;\;\; t\geq 0.
\end{equation*}

\smallskip

For any martingales $M, N:\mathbb R_+ \times \Omega \to H$ we can define a {\em covariation} $[M,N]:\mathbb R_+ \times \Omega \to \mathbb R$ as $[M,N] := \frac{1}{4}([M+N]-[M-N])$.
Since $M$ and $N$ have c\`adl\`ag versions, $[M]$ and $[M,N]$ have c\`adl\`ag versions as well (see \cite[Theorem I.4.47]{JS} and \cite{MetSemi,Kal}).

\medskip

\begin{definition}\label{def:prelimdefcovbilform}
Let $X$ be a Banach space, $M:\mathbb R_+ \times \Omega \to X$ be a local martingale. Fix $t\geq 0$. Then $M$ is said to have a {\em covariation bilinear from} $[\![M]\!]_t$ at $t\geq 0$ if there exists a continuous bilinear form-valued random variable $[\![M]\!]_t:X^* \times X^* \times \Omega \to \mathbb R$ such that for any fixed $x^*, y^*\in X^*$ a.s.\ $[\![M]\!]_t(x^*, y^*) = [\langle M, x^*\rangle, \langle M, y^*\rangle]_t$. 
\end{definition}

\begin{remark}\label{rem:ifUMDthencovbilform}
 It is known due to \cite{Y18BDG} that if $X$ has the UMD property, then any $X$-valued local martingale $M$ has a covariation bilinear form $[\![M]\!]$. Moreover, $[\![M]\!]$ has a c\`adl\`ag adapted version, and if $M$ is continuous, then $[\![M]\!]$ has a continuous version as well, and for a general local martingale $M$ one has that $\gamma([\![M]\!]_{t})<\infty$ a.s., where for a bilinear form $V:X^* \times X^* \to X$ we set the {\em Gaussian characteristic} $\gamma(V)$ to be
 \begin{itemize}
 \item the $L^2$-norm of a Gaussian random variable $\xi$ having $V$ as its bilinear covariance form, i.e.\ $\mathbb E\langle \xi, x^*\rangle \langle \xi, y^*\rangle = V(x^*, y^*)$ for any $x^*, y^*\in X^*$, if such $\xi$ exists,
 \item $\infty$, if such $\xi$ does not exist.
 \end{itemize}
 We refer the reader to \cite{Y18BDG} for further details.
\end{remark}

\subsection{The canonical decomposition}\label{subsec:candec}

In this subsection we discuss the so-called {\em canonical decomposition} of martingales. First let us start with the following technical definitions. Recall that a c\`adl\`ag function $A:\mathbb R_+ \to X$ is called {\em pure jump} if $A_t = A_0+ \sum_{0<s\leq t} \Delta A_s$ for any $t\geq 0$, where the latter sum converges absolutely.

\begin{definition}\label{def:purelydiscmart}
 Let $X$ be a Banach space. A local martingale $M:\mathbb R_+\times \Omega \to X$ is called {\em purely discontinuous} if $[\langle M, x^*\rangle]$ is pure jump a.s.\ for any $x^*\in X^*$.
\end{definition}

\begin{definition}\label{def:candec}
 Let $X$ be a Banach space, $M:\mathbb R_+ \times \Omega \to X$ be a local martingale. Then $M$ is called to have {\em the canonical decomposition} if there exist local martingales $M^c, M^q, M^a:\mathbb R_+ \times \Omega \to X$ such that $M^c$ is continuous, $M^q$ is purely discontinuous quasi-left continuous, $M^a$ is purely discontinuous with accessible jumps, $M^c_0=M^q_0=0$ a.s., and $M = M^c + M^q + M^a$.
\end{definition}

\begin{remark}\label{rem:candec==>candecforanyx*}
Note that if $M = M^c + M^q + M^a$ is the canonical decomposition, then $\langle M, x^*\rangle = \langle M^c, x^*\rangle + \langle M^q, x^*\rangle + \langle M^a, x^*\rangle$ is the canonical decomposition for any $x^*\in X^*$ (see e.g.\ \cite{Y17MartDec,Y17GMY,DY17}).
\end{remark}

 \begin{remark}\label{rem:candecsplitsjumps}
 Note that by \cite{Y17MartDec,Y17GMY,Kal,JS} if the canonical decomposition of a local martingale $M$ exists, then $M^q$ and $M^a$ collect different jumps of $M$, i.e.\ a.s.\
 \begin{equation}\label{eq:candecsplitsjumps}
 \begin{split}
  \{t\geq 0: \Delta M^q_t\neq 0\} \cup \{t\geq 0: \Delta M^a_t\neq 0\} &= \{t\geq 0: \Delta M_t\neq 0\},\\
\{t\geq 0: \Delta M^q_t\neq 0\} \cap \{t\geq 0: \Delta M^a_t\neq 0\} &= \varnothing.
 \end{split}
 \end{equation}
\end{remark}

Then the following theorem holds, which was first proved in \cite{Mey76,Yoe76} in the real-valued case, and in \cite{Y17MartDec,Y17GMY,Y18BDG} in the vector-valued case (see also \cite[Chapter~25]{Kal}). 

\begin{theorem}[The canonical decomposition]\label{thm:candecXvalued}
 Let $X$ be a Banach space. Then $X$ is UMD if and only if any local martingale $M:\mathbb R_+ \times \Omega \to X$ has the canonical decomposition $M = M^c + M^q + M^a$. Moreover, if this is the case, then the canonical decomposition is unique, and for any $1\leq p<\infty$
 \begin{equation}\label{eq:candecstrongLpestmiaed}
  \mathbb E \sup_{t\geq 0} \|M^c_t\|^p + \mathbb E \sup_{t\geq 0} \|M^q_t\|^p + \mathbb E \sup_{t\geq 0} \|M^a_t\|^p \eqsim_{p, X} \mathbb E \sup_{t\geq 0} \|M_t\|^p.
 \end{equation}
\end{theorem}

If we will have a closer look on each of the parts of the canonical decomposition, then we will figure out that $M^c$ is in fact a time changed stochastic integral with respect to a cylindrical Brownian motion (see Subsection \ref{subsec:dectangcontmart}), $M^q$ is a time changed stochastic integral with respect to a Poisson random measure (see Subsection \ref{subsec:PoissRMprelim}), while $M^a$ can be represented as a discrete martingale if it has finitely many jumps (see Subsection \ref{subsec:dectanPDwithAJ} and \ref{subsec:appforMarAppPDMAJ}; see also \cite{DY17,Y17MartDec,Kal}). Thus we often call $M^c$ the {\em Wiener-like part}, $M^q$ the {\em Poisson-like part}, while $M^a$ is often called a {\em discrete-like part} of $M$: in many cases the corresponding techniques help in finding required inequalities for $M^c$, $M^q$, and $M^a$.

Note that the canonical decomposition plays an important r\^ole in stochastic integration theory (see e.g.\ \cite{DY17,DMY18,Y18BDG}).

\begin{remark}\label{rem:MYdecBanach}
Often we will use the so-called {\em Meyer-Yoeurp decomposition} which splits a local martingale $M$ into a continuous part $M^c$ and a purely discontinuous part $M^d$. This decomposition is unique if it exists, and in the case of existence of the canonical decomposition $M= M^c + M^q + M^a$ one obviously has $M^d = M^q + M^a$. Analogously to Theorem \ref{thm:candecXvalued} one can show that for a given Banach space $X$ every $X$-valued local martingale has the Meyer-Yoeurp decomposition if and only if $X$ has the UMD property (see \cite{Y17MartDec,Y17GMY,YPhD}). 
\end{remark}

Later we will need the following lemma shown  in \cite[Subsection 5.1]{DY17} (see \cite{Kal} for the real-valued version). Recall that two stopping times $\tau$ and $\sigma$ have {\em disjoint graphs} if $\mathbb P(\tau = \sigma<\infty) =0$.

\begin{lemma}\label{lem:PDmAJhasjumpsatprsttimes}
 Let $X$ be a Banach space, $M:\mathbb R_+ \times \Omega \to X$ be a purely discontinuous local martingale with accessible jumps. Then there exist a sequence $(\tau_n)_{n\geq 1}$ of finite predictable stopping times with disjoint graphs such that a.s.\
 \[
  \{t\geq 0: \Delta M_t \neq 0\} \subset \{\tau_1,\ldots, \tau_n,\ldots\}.
 \]
\end{lemma}

\subsection{Random measures}\label{subsec:ranmeasures}
Let $(J, \mathcal J)$ be a measurable space so that $\mathcal J$ is countably generated. A family 
$$
\mu = \{\mu(\omega; \ud t, \ud x), \omega \in \Omega\}
$$
of nonnegative measures on $(\mathbb R_+ \times J, \mathcal
B(\mathbb R_+)\otimes \mathcal J)$ is called a {\em random
measure}. A~random measure $\mu$ is called {\em integer-valued} if it takes values in
$\mathbb N\cup\{\infty\}$, i.e.\ for each $A \in \mathcal
B(\mathbb R_+)\otimes \mathcal J$ one has that
$\mu(A) \in \mathbb N\cup\{\infty\}$ a.s., and if $\mu(\{t\}\times
J)\in \{0,1\}$ a.s.\ for all $t\geq 0$ (so $\mu$ is a sum of atoms with a.s.\ disjoint supports, see \cite[Proposition II.1.14]{JS}). We say that $\mu$ is \emph{non-atomic in
time} if $\mu(\{t\}\times J) = 0$ a.s.\ for all $t\geq 0$.

Let $\mathcal
O$ be the {\em optional $\sigma$-algebra}  on $\mathbb R_+ \times \Omega$, i.e.\  the $\sigma$-algebra generated by all c\`adl\`ag adapted processes. Let $\widetilde {\mathcal O} := \mathcal O \otimes \mathcal J$, $\widetilde {\mathcal P} := \mathcal P \otimes \mathcal J$ (see Subsection \ref{subsec:BSvmart} for the definition of $\mathcal P$). A random measure $\mu$ is called {\em optional} (resp.\ {\em predictable}) if for any
$\widetilde{\mathcal O}$-measurable (resp.\ $\widetilde{\mathcal
P}$-measurable) nonnegative $F:\mathbb R_+ \times \Omega \times J
\to \mathbb R_+$ the stochastic integral
\[
  (t, \omega) \mapsto \int_{\mathbb R_+\times J}\mathbf 1_{[0,t]}(s)F(s,\omega,x) \mu(\omega;\ud s, \ud x),\;\;\; t\geq 0, \ \omega \in \Omega,
\]
as a function from $\mathbb R_+ \times \Omega$ to
$\overline{\mathbb R}_+$ is optional (resp.\ predictable).

Let $X$ be a Banach space. Then we can extend stochastic
integration with respect to random measures to \mbox{$X$-valued} processes in the following way.
Let $F:\mathbb R_+ \times \Omega \times J  \to X$ be elementary predictable, i.e.\ there exists partition $B_1, \ldots, B_N$ of $J$, $0= t_0 < t_1 \ldots < t_L$, and simple $X$-valued random variables $(\xi_{n,\ell})_{n=1, m=1}^{N,L}$ such that $\xi_{n,\ell}$ is $\mathcal F_{t_{\ell-1}}$-measurable for any $1\leq \ell \leq L$ and $1\leq n \leq N$ and
\[
F(t, \cdot, j) = \sum_{n=1}^N \sum_{\ell=1}^L \mathbf 1_{(t_{\ell-1}, t_{\ell}]}(t) \mathbf 1_{B_n}(j) \xi_{n,\ell}.
\]
Let $\mu$ be a
random measure. The integral
\begin{equation}\label{eq:stochintwrtranmeasdefof}
\begin{split}
 t\mapsto \int_{\mathbb R_+ \times J} F(s,\cdot, x)&\mathbf 1_{[0,t]}(s)\mu(\cdot; \ud s, \ud x)\\
 &:= \sum_{n=1}^N \sum_{\ell=1}^L  \mu\bigl( (t_{\ell-1}\wedge t, t_{\ell}\wedge t] \times B_n \bigr) \xi_{n,\ell},\;\;\; t\geq 0,
 \end{split}
\end{equation}
is well-defined and optional (resp.\ predictable) if $\mu$ is
optional (resp.\ predictable), and $\int_{\mathbb R_+ \times
J}\|F\|\ud\mu$ is a.s.\ bounded.

A random measure $\mu$ is called $\widetilde{\mathcal
P}$-$\sigma$-finite if there exists an increasing sequence of sets
$(A_n)_{n\geq 1}\subset \widetilde{\mathcal P}$ such that
$\int_{\mathbb R_+ \times J} \mathbf
1_{A_n}(s,\omega,x)\mu(\omega; \ud s, \ud x)$ is finite a.s.\ and
$\cup_n A_n = \mathbb R_+ \times \Omega \times J$. According to
\cite[Theorem II.1.8]{JS} every $\widetilde{\mathcal
P}$-$\sigma$-finite optional random measure $\mu$ has a {\em
compensator}: a~{\em unique} $\widetilde{\mathcal P}$-$\sigma$-finite
predictable random measure $\nu$ such that
\begin{equation}\label{eq:defofcompofrandsamdsaer}
\mathbb E \int_{\mathbb R_+ \times J}F \ud \mu = \mathbb E \int_{\mathbb R_+ \times J}F \ud \nu
\end{equation}
for each $\widetilde{\mathcal P}$-measurable
real-valued nonnegative $F$. For any optional
$\widetilde{\mathcal P}$-$\sigma$-finite measure $\mu$ we define
the associated compensated random measure by $\bar{\mu}: = \mu
-\nu$.

For each $\widetilde{\mathcal P}$-strongly-measurable $F:\mathbb
R_+ \times \Omega \times J \to X$ such that 
$$
\mathbb E \int_{\mathbb R_+ \times J}\|F\|\ud
\mu< \infty
$$ 
(or, equivalently, $\mathbb E \int_{\mathbb R_+ \times J}\|F\|\ud
\nu< \infty$, see the definition of a compensator above)
we can define a process
\begin{equation}\label{eq:defofstochintwrtbarmu}
t\mapsto \int_{[0,t]\times J}F\ud \bar{\mu} := \int_{[0,t]\times J}F\ud \mu -  \int_{[0,t]\times J}F\ud \nu,\;\;\;t\geq 0,
\end{equation}
which turns out to be a purely discontinuous martingale (see Proposition \ref{prop:defofstochintwrtranmcaseodFoffinitemuintds}, \cite[Theorem II.1.8]{JS}, and \cite{DY17}). 

\smallskip

We will need  the following classical result of Novikov \cite[Theorem 1]{Nov75}.
\begin{lemma}[A.A.\ Novikov]\label{lemma:Nov}
 Let $\mu$ be an integer-valued optional random measure on $\mathbb R_+ \times J$ with a compensator $\nu$ being non-atomic in time, $F:\mathbb R_+ \times \Omega \times J \to \mathbb R$ be $\widetilde{\mathcal P}$-measurable. Then
 \begin{equation}\label{eq:NovRV}
 \begin{split}
  \mathbb E \sup_{t\geq 0} \Bigl| \int_{[0,t]\times J}f\ud \bar{\mu}\Bigr|^p &\lesssim_{p} \mathbb E  \int_{\mathbb R_+ \times \Omega}|f|^p\ud \nu \text{ if }\; 1\leq p\leq 2,\\
\mathbb E \sup_{t\geq 0}  \Bigl| \int_{[0,t] \times J}f\ud \bar{\mu}\Bigr|^p  &\lesssim_{p} \Bigl(\mathbb E \int_{\mathbb R_+ \times J}|f|^2 \ud\nu\Bigr)^{\frac p2}+ \mathbb E  \int_{\mathbb R_+ \times \Omega}|f|^p\ud \nu \text{ if }\;  p\geq 2.
\end{split}
 \end{equation}
\end{lemma}

\smallskip

For an $X$-valued martingale $M$ we associate a {\em jump measure} $\mu^M$ which is a random measure on $\mathbb R_+ \times X$ that counts the jumps of $M$
\begin{equation}\label{eq:defofmuM}
\mu^M([0, t] \times B) := \sum_{0\leq s\leq t} \mathbf 1_{B\setminus\{0\}} (\Delta M_t),\;\;\; t\geq 0, \;\; B \in \mathcal B(X).
\end{equation}
Note that $\mu^M$ is $\widetilde {\mathcal P}$-$\sigma$-finite and we will frequently use the following fact which was proved in \cite[Corollary II.1.19]{JS} (see also \cite{Kal,DY17,KalRM}).
\begin{lemma}\label{lem:nuMisnonatomiffMqlc}
Let $X$ be a Banach space, $M:\mathbb R_+ \times \Omega \to X$ be a local martingale. Let $\mu^M$ be the associated jump measure. Then $M$ is quasi-left continuous if and only if the corresponding compensator $\nu^M$ of  $\mu^M$ is non-atomic in time.
\end{lemma}

We refer the reader to \cite{JS,Kal,DY17,Y18BDG,KalRM,MarRo,Marinelli13,Jac84,Now02,Nov75,Grig71} for details on random measures and stochastic integration with respect to random measures.

\subsection{Poisson random measures}\label{subsec:PoissRMprelim}

An important example of random measures is a {\em Poisson random measure}.

 Let $(S, \Sigma, \rho)$ be  a measure space, $\rho$ be $\sigma$-finite. Then we can define a {\em Poisson random measure} (a.k.a.\ {\em Poisson point process}) $N_{\rho}$ with {\em intensity} (or {\em compensator}) $\rho$, i.e.\ a function $\Sigma\mapsto L^0(\Omega, \mathbb N_0 \cup\{+\infty\})$ satisfying the following properties
\begin{enumerate}[(i)]
\item $N_{\rho}(A)$ has the Poisson distribution with a parameter $\rho(A)$ for any $A \in \Sigma$ such that $\rho(A)<\infty$,
\item $N_{\rho}(A_1), \ldots,N_{\rho}(A_n)$ are independent for any disjoint $A_1,\ldots, A_n \in \Sigma$,
\item $N_{\rho}$ is a.s.\ a measure on $\Sigma$
\end{enumerate}
(see \cite[Chapter 4]{Sato} and \cite{KingPois} for details).  We can also define the {\em compensated Poisson random measure} $\widetilde N_{\rho}$ to be $\widetilde N_{\rho}(A) := N_{\rho}(A) - \rho(A)$ for any $A\in \Sigma$ satisfying $\rho(A)<\infty$. 

\begin{remark}
 If we set $S = \mathbb R_+ \times J$ and $\rho = \nu = \lambda \otimes \nu_0$ (so that we have the setting which we used to work above) with $\lambda$ being the Lebesgue measure on $\mathbb R_+$ and $\nu_0$ being some fixed $\sigma$-finite measure  on $J$, then we come up with Poisson measures that are often exploited as a noncontinuous noise for SPDE's (see e.g.\ \cite{Dirk14,BH09,FTT10,MPR10,Grig71,Nov75,PZ,ZBL19} and references therein).
\end{remark}

In the sequel we will need the following definition of an integral with respect to a Poisson random measure.

\begin{definition}\label{def:defofstochintwrtPoisspoirnprod}
Let $X$ be a Banach space, $(S, \Sigma, \rho)$ be a measure space, $N_{\rho}$ be a Poisson random measure on $S$ with the intensity $\rho$. Then a strongly $\Sigma$-measurable function $F:S\to X$ is called {\em integrable} with respect to $\widetilde N_{\rho} = N_{\rho} - \rho$ if there exist an increasing family of sets $(A_n)_{n\geq 1} \in \Sigma$ such that $\cup_{A_n} = S$, $\int_{A_n} \|F\| \ud \rho <\infty$, and $\int_{A_n} F \ud \widetilde N_{\rho}$ converges in $L^1(\Omega; X)$ as $n\to \infty$.
\end{definition}

\begin{remark}\label{rem:defofdtochintwrtpoisrmforFoffinitenumes}
Let $G:S \to X$ be strongly $\Sigma$-measurable such that $\int_{S} \|G\| \ud \rho<\infty$. Then $G\in L^1(S, \rho;X)$, and as  for any step function $H\in L^1(S, \rho;X)$ we have that $\mathbb E \int_S \|H\| \ud N_{\rho} = \|H\|_{L^1(S, \rho;X)}$ by the definition of $N_{\rho}$ (in particular, $\mathbb E N_{\rho}(A) = \rho(A)$ for any $A \in \Sigma$), we can extend the stochastic integral definition to $G$ by a standard expanding operator procedure. Thus $\int_{A_n} F \ud \widetilde N_{\rho} := \int_{A_n} F \ud N_{\rho} - \int_{A_n} F \ud \rho $ in the definition above is well defined.
\end{remark}

\begin{remark}\label{rem:sticintwrtPoisrmonercanchosreeanysetws}
Definition \ref{def:defofstochintwrtPoisspoirnprod} is quite different from the one given in Subsection \ref{subsec:dectanPDQLC} as we do not have a time scale (so there is no martingale structure) and since we are working with Poisson random measures. Moreover, notice that if such a family $(A_n)_{n\geq 1}$ exists, then for any other increasing family $(A'_n)_{n\geq 1}$ having the same properties as $(A_n)_{n\geq 1}$ we will have that  $\int_{A'_n} F \ud \widetilde N_{\rho}$ converges in $L^1(\Omega;X)$ as $n\to \infty$. Indeed, let 
\begin{equation}\label{eq:defofintwrtPoisrndmars}
\xi = \int_{S} F \ud \widetilde N_{\rho}:= L^1(\Omega; X)-\lim_{n\to \infty} \int_{A_n} F \ud \widetilde N_{\rho}.
\end{equation}
 Then 
$$
(\xi_n)_{n\geq 1} := \Bigl( \int_{A'_n} F\ud \widetilde N_{\rho}\Bigr)_{n\geq 1},
$$
is a martingale with independent increments as $\xi_{n+1} -\xi_n = \int_{A'_n \setminus A'_{n-1}} F  \ud \widetilde N_{\rho}$ is independent of $\sigma(N_{\rho}|_{A'_{n}})$, and hence independent of $\xi_1, \ldots, \xi_n$. Thus we have that for any $x^*\in X^*$, $\mathbb E (\langle \xi , x^* \rangle|\sigma(N_{\rho}|_{A'_n})) = \langle \xi_n , x^* \rangle$ for any $n\geq 1$ (which follows from the fact that 
$$
\int_{A'_n \cap A_m} \langle F, x^* \rangle \ud \widetilde N_{\rho} \to \int_{A'_n} \langle F, x^* \rangle \ud \widetilde N_{\rho}\; \;\;\text{in}\;\;\; L^1(\Omega)\; \;\;\text{as}\; \;\;m\to \infty,
$$
from \cite[Theorem 3.3.2]{HNVW1}, and from the definition \eqref{eq:defofintwrtPoisrndmars} of $\xi$), so $\langle \xi_n , x^* \rangle$ converges to $\langle \xi , x^* \rangle$ by \cite[Theorem 3.3.2]{HNVW1}, thus by the It\^o-Nisio theorem \cite[Theorem 6.4.1]{HNVW2} we have that $\xi_n$ converges to $\xi$ in $L^1(\Omega; X)$.
\end{remark}

\subsection{Stochastic integration}\label{subsec:prelimstint}

Let $H$ be a Hilbert space, $X$ be a Banach space. For each $x\in X$ and $h \in H$ we denote the linear operator $g\mapsto \langle g, h\rangle x$, $g\in H$, by $h\otimes x$. The process $\Phi: \mathbb R_+ \times \Omega \to \mathcal L(H,X)$ is called  \textit{elementary predictable}
with respect to the filtration $\mathbb F = (\mathcal F_t)_{t \geq 0}$ if it is of the form
\begin{equation}\label{eq:elprogPhi}
 \Phi(t,\omega) = \sum_{k=1}^K\sum_{\ell=1}^L \mathbf 1_{(t_{k-1},t_k]\times B_{\ell k}}(t,\omega)
\sum_{n=1}^N h_n \otimes x_{k\ell n},\;\;\; t\geq 0,\;\; \omega \in \Omega,
\end{equation}
where $0 = t_0 < \ldots < t_K <\infty$, for each $k = 1,\ldots, K$ the sets
$B_{1k},\ldots,B_{Lk}$ are in $\mathcal F_{t_{k-1}}$, the vectors $h_1,\ldots,h_N$ are in $H$, and $(x_{k\ell n})_{k,\ell,n=1}^{K,L,M}$ are elements of $X$.
Let $\widetilde M:\mathbb R_+ \times \Omega \to H$ be a local martingale. Then we define the {\em stochastic integral} $\Phi \cdot \widetilde M:\mathbb R_+ \times \Omega \to X$ of $\Phi$ with respect to $\widetilde M$ as follows:
\begin{equation*}\label{eq:defofstochintwrtM}
 (\Phi \cdot \widetilde M)_t := \sum_{k=1}^K\sum_{\ell=1}^L \mathbf 1_{B_{\ell k}}
\sum_{n=1}^N \langle(\widetilde M(t_k\wedge t)- \widetilde M(t_{k-1}\wedge t)), h_n\rangle x_{k\ell n},\;\; t\geq 0.
\end{equation*}

A map $W_H:\mathbb R_+  \times H\to L^2(\Omega) $ is called an {\em $H$-cylindrical Brownian motion} (see \cite[Chapter 4.1]{DPZ}) if 
\begin{itemize}
 \item $W_H(\cdot)h$ is a Brownian motion for any $h\in H$,
 \item $\mathbb E W_H(t)h\,W_H(s)g = \langle h,g\rangle \min\{t,s\}$
for all $h$, $g \in H$ and $t$, $s \geq 0$.
\end{itemize}

For an $H$-cylindrical Brownian motion $W_H$ we can define a stochastic integral of $\Phi$ of the form \eqref{eq:elprogPhi} in the following way
\begin{equation*}
 (\Phi \cdot W_H)_t := \sum_{k=1}^K\sum_{\ell=1}^L \mathbf 1_{B_{k}}
\sum_{n=1}^N (W_H(t_k\wedge t)h_n- W_H(t_{k-1}\wedge t)h_n) x_{k\ell n},\;\; t\geq 0.
\end{equation*}
Further, if $X=\mathbb R$, then due to \cite[Theorem 4.12]{DPZ} (see also \cite{Kal,VY16,NVW}) it is known that a.s.\
\begin{equation}\label{eq:qvvarofstintwrtHcylbrmot}
 [\Phi \cdot W_H]_t = \int_{0}^t \|\Phi\|^2 \ud s,
\end{equation}
and in particular by the Burkholder-Davis-Gundy inequalities \cite[Theorem 17.7]{Kal} we have that for any $0<p<\infty$
\begin{equation}\label{eq:LpboundsforstochintwrtHcylbrm}
\mathbb E \sup_{t\geq 0}\bigl\|(\Phi \cdot W_H)_t\bigr\|^p \eqsim_{p} \mathbb E \Bigl(\int_{0}^t \|\Phi\|^2 \ud s \Bigr)^{p/2}.
\end{equation}

We refer the reader to \cite{Kal,JS,DPZ,VY16,Y18BDG,DY17,MP,MetSemi,Mey76,NVW} for further details on stochastic integration and cylindrical Brownian motions.

\subsection{$\gamma$-radonifying operators}\label{subsec:gammanorm}

 Let $H$ be a separable Hilbert space and let $X$ be a Banach space. Let $T\in \mathcal L(H, X)$. Then $T$ is called {\em $\gamma$-radonifying} if
\begin{equation}\label{eq:defofgammanormsnove}
 \|T\|_{\gamma(H,X)} := \Bigl(\mathbb E \Bigl\|\sum_{n=1}^{\infty} \gamma_n Th_n\Bigr\|^2\Bigr)^{\frac 12} <\infty,
\end{equation}
where $(h_n)_{n\geq 1}$ is an orthonormal basis of $H$, and $(\gamma_n)_{n\geq 1}$ is a sequence of independent standard Gaussian random variables (if the series on the right-hand side of \eqref{eq:defofgammanormsnove} does not converge, then we set $\|T\|_{\gamma(H,X)}:=\infty$). Note that $\|T\|_{\gamma(H,X)}$ does not depend on the choice of the orthonormal basis $(h_n)_{n\geq 1}$ (see \cite[Section 9.2]{HNVW2} and \cite{Ngamma} for details). Often we will call $\|T\|_{\gamma(H, X)}$ the {\em $\gamma$-norm} of $T$. Note that if $X$ is a Hilbert space, then $\|T\|_{\gamma(H,X)}$ coincides with the Hilbert-Schmidt norm of $T$.
$\gamma$-norms are exceptionally important in analysis as they are easily computable and enjoy a number of useful properties such as the ideal property, $\gamma$-multiplier theorems, Fubini-type theorems, etc., see \cite{HNVW2,Ngamma}.

\subsection{Tangent martingales: the discrete case}\label{subsec:tanmartdiscase}

Let $X$ be a Banach space, $(d_n)_{n\geq 1}$ and $(e_n)_{n\geq 1}$ be $X$-valued martingale difference sequences. 
\begin{definition}\label{def:preldefoftangdiscrXvalmart}
$(d_n)_{n\geq 1}$ and $(e_n)_{n\geq 1}$ are called {\em tangent} if
\begin{equation}\label{eq:disctang}
\mathbb P(d_n|\mathcal F_{n-1}) = \mathbb P(e_n|\mathcal F_{n-1}),\;\;\; n\geq 1.
\end{equation}
\end{definition}
(Recall that conditional probabilities have been defined in Subsection \ref{subsec:ConexponPSCondProbCondIndep}.)
\begin{example}\label{ex:tangstandarddecxinvn-->xi'nvn}
Let $(v_n)_{n\geq 1}$ be a predictable uniformly bounded $X$-valued sequence, $(\xi_n)_{n\geq 1}$ and $(\xi'_n)_{n\geq 1}$ be adapted sequences of mean-zero real-valued independent random variables such that $\xi_n$ and $\xi_n'$ are equidistributed, integrable, and independent of $\mathcal F_{n-1}$ for any $n\geq 1$.
Then martingale difference sequences $(\xi_nv_n)_{n\geq 1}$ and $(\xi'_nv_n)_{n\geq 1}$ are tangent. Indeed, for any $n\geq 1$ and $A\in \mathcal B(X)$ we have that a.s.
\begin{multline*}
\mathbb P(\xi_nv_n|\mathcal F_{n-1})(A) = \mathbb E(\mathbf 1_A(\xi_nv_n)(A)|\mathcal F_{n-1}) = \mathbb E(\mathbf 1_{A/v_n}(\xi_n)(A)|\mathcal F_{n-1})\\
 \stackrel{(*)}= \mathbb E(\mathbf 1_{A/v_n}(\xi'_n)(A)|\mathcal F_{n-1}) = \mathbb E(\mathbf 1_A(\xi'_nv_n)(A)|\mathcal F_{n-1}) = \mathbb P(\xi'_nv_n|\mathcal F_{n-1})(A),
\end{multline*}
where $(*)$ follows from the fact that $\xi_n$ and $\xi'_n$ are i.i.d.\ and independent from $\mathcal F_{n-1}$, and the fact that $v_n$ is $\mathcal F_{n-1}$ measurable, where for $A \subset X$ and $x\in X$ we define $A/x\subset \mathbb R$ by
\[
A/x := \{t\in \mathbb R:tx\in A\}.
\]
\end{example}
It was shown by Hitczenko in \cite{HitUP} (see also \cite{HNVW1,KwW91,dlP94,dlPG,DY17,CG}) that any $X$-valued martingale difference sequence $(d_n)_{n\geq 1}$ has a {\em decoupled} tangent martingale difference sequence on an enlarged probability space with an enlarged filtration, i.e.\ there exists an enlarged filtration $\overline {\mathbb F}$ w.r.t.\ which $(d_n)$ remains being a martingale difference sequence, an $\overline {\mathbb F}$-adapted martingale difference sequence $(e_n)_{n\geq 1}$, and a $\sigma$-algebra $\mathcal G \subset {\mathcal F}_{\infty}$ such that
\[
\mathbb P(e_n|\mathcal F_{n-1}) = \mathbb P(e_n|\mathcal G) ,\;\;\; n\geq 1,
\]
and $(e_n)_{n\geq 1}$ are conditionally independent given $\mathcal G$ (see Subsection \ref{subsec:ConexponPSCondProbCondIndep}). Moreover, $(e_n)_{n\geq 1}$ is unique up to probability.
Later in Section \ref{sec:tangmartconttime} we will extend a construction of such a martingale to the continuous-time case.

\begin{remark}\label{rem:discrdefislikecontdafa}
 Note that due to Proposition \ref{prop:Omega'Omega';arecondindepifondforanyomega'as}, the construction of a decoupled tangent martingale difference sequence  \cite{HitUP,dlP94,dlPG}, and the uniqueness of its distribution  we can give the following equivalent definition: $(e_n)_{n\geq 1}$ is a  decoupled tangent martingale difference sequence to $(d_n)_{n\geq 1}$ if and only if for a.e.\ $\omega \in \Omega$ the sequence $(e_n(\omega))_{n\geq 1}$ is a sequence of mean zero random variables so that $c_n(\omega)$ has $\mathbb P(d_n|\mathcal F_{n-1})(\omega)$ as its distribution (see \cite{dlP94,dlPG} or the proof of Theorem \ref{thm:detangmartforMXVpdwithaccjumps} for the construction of $\mathbb P(d_n|\mathcal F_{n-1})(\omega)$).
\end{remark}

\section{Tangent martingales: the continuous-time case}\label{sec:tangmartconttime}

This section is devoted to continuous-time tangent martingales and their properties. As the notion of tangency in the continuous-times case (see Definition \ref{def:tangconttimecase} below) only cares about the jumps of a process and the quadratic variation of its continuous part, throughout this section we will assume that any martingale starts in zero. Also, in the sequel we will frequently use the {\em stopping times argument} which is allowed by Theorem \ref{thm:MNtangent==>MtauNtautangent+ifNCIINtauCII}. In particular, while talking about tangent local martingales $M$ and $N$ we can automatically assume that these martingales have finite strong $L^1$-moment, i.e.\ $\mathbb E \sup_{t\geq 0}\|M_t\|$ and $\mathbb E \sup_{t\geq 0}\|N_t\|$ can be presumed to be finite unless stated otherwise (see Remark \ref{rem:locmartislocallyL1DR+X}).

\subsection{Local characteristics and tangency}\label{subsec:loccharandtangds}

In order to define tangent martingales in the continuous-time case we need {\em local characteristics}.

Let $M:\mathbb R_+ \times \Omega \to \mathbb R$ be a local martingale, $M = M^c + M^d$ be the Meyer-Yoeurp decomposition of $M$ (see Remark \ref{rem:MYdecBanach}). Then the pair $([M^c], \nu^M)$, where $[M^c]$ is the quadratic variation of $M^c$ (see Subsection \ref{subsec:quadrvar}), and $\nu^M$ is a compensator of the random measure $\mu^M$ defined by \eqref{eq:defofmuM}, is called the {\em local characteristics} of $M$.

Let $X$ be a Banach space, $M$ be an $X$-valued local martingale. Assume that $M$ admits the Meyer-Yoeurp decomposition $M = M^c + M^d$  (see Remark \ref{rem:MYdecBanach}) and that $M^c$ has a covariation bilinear form $[\![M^c]\!]$ (see Subsection \ref{subsec:quadrvar}). Then we define the local characteristics of $M$ to be the pair $([\![M^c]\!], \nu^M)$, where $\nu^M$ is a compensator of the random measure $\mu^M$ defined by \eqref{eq:defofmuM}.

\begin{definition}\label{def:tangconttimecase}
 Two $X$-valued local martingales $M$ and $N$ are called {\em tangent} if both local martingales have local characteristics which coincide. 
\end{definition}

\begin{remark}\label{rem:disctangagreesconttime}
Note that this definition of tangency agrees with the one for discrete martingales given in Subsection \ref{subsec:tanmartdiscase}. Indeed, let $(d_n)_{n\geq 1}$ and $(e_n)_{n\geq 1}$ be tangent martingale difference sequences. Then they are tangent in the continuous-time case if for any $n\geq 1$ compensators of random measures $\mu^{d_n}$ and $\mu^{e_n}$ defined on $X$ by 
\[
\mu^{d_n}(A) = \mathbf 1_{A}(d_n),\;\;\mu^{e_n}(A) = \mathbf 1_{A}(e_n),\;\;\; A \in \mathcal B(X),
\]
coincide. But these compensators exactly coincide with 
$\mathbb P(d_n|\mathcal F_{n-1})$  and with $\mathbb P(e_n|\mathcal F_{n-1})$ respectively
as by the definition \eqref{eq:condprob} of $\mathbb P(d_n|\mathcal F_{n-1})$ and $\mathbb P(e_n|\mathcal F_{n-1})$ and by \eqref{eq:disctang} for any Borel set $A\subset X$ one has that
\begin{equation}\label{eq:dnentangentindiscrete-->incont}
\begin{split}
\mathbb E ( \mu^{d_n}(A) | &\mathcal F_{n-1}) = \mathbb E ( \mathbf 1_{A}(d_n) | \mathcal F_{n-1}) = \mathbb P(d_n|\mathcal F_{n-1})(A) \\ &= \mathbb P(e_n|\mathcal F_{n-1})(A) = \mathbb E ( \mathbf 1_{A}(e_n) | \mathcal F_{n-1}) = \mathbb E ( \mu^{e_n}(A) | \mathcal F_{n-1}).
\end{split}
\end{equation}
The converse direction can be shown similarly.
\end{remark}

Now we are ready to define a decoupled tangent local martingale. Recall that conditional independence was defined in \eqref{eq:defofcondinderp} and an enlargement of a filtered probability space was defined in Definition \ref{def:enlagoffiltprobspace}.

\begin{definition}\label{def:dectanglocmartcontimecased}
Let $(\Omega, \mathcal F, \mathbb P)$ be a probability space endowed with a filtration $\mathbb F = (\mathcal F_{t})_{t\geq 0}$. Let $M:\mathbb R_+ \times \Omega \to X$ be a local martingale. A process $N:\mathbb R_+ \times \overline{\Omega} \to X$ over an enlarged probability space $(\overline{\Omega}, \overline{\mathcal F},\overline{ \mathbb P})$ with an enlarged filtration $\overline{\mathbb F} = (\overline{\mathcal F}_t)_{t\geq 0}$ is called  a {\em decoupled tangent local martingale} of $M$ if $M$ is a local $\overline{\mathbb F}$-martingale with the same local characteristics $([\![M^c]\!], \nu^M)$, $N$ is a local martingale, $M$ and $N$ are tangent, and $N(\omega)$ is a local martingale with independent increments and local characteristics $([\![M^c]\!](\omega), \nu^M(\omega))$ for a.e.\ $\omega\in \Omega$.
\end{definition}

Note that we can always set $\mathcal F :=\mathcal F_{\infty}$, and that $N$ may be assumed to have independent increments given $\mathcal F$ due to Proposition \ref{prop:Omega'Omega';arecondindepifondforanyomega'as}.
We refer the reader to \cite[p.\ 174]{KwW91} and \cite{KwW92,KwW86,JH87,CG,dlPG,dlP94,Kal17} for further details on decoupled tangent local martingales.

\begin{remark}\label{rem:Mmightchangelocalcharacp}
 Note that the martingale properties of $M$ in Definition \ref{def:dectanglocmartcontimecased} could change while enlarging the filtration, so it is extremely important to presume that $M$ is a local $\overline{\mathbb F}$-martingale having the same local characteristics as it used to obtain. Indeed, first $M$ can lose its martingality, e.g.\ if $\overline{\mathcal F}_t$ contains ${\mathcal F}_{\infty}$ for any $t\geq 0$ (then $M_t$ is $\overline{\mathcal F}_0$-measurable). But even if $M$ remains a local martingale, it could change its local characteristics. Though $[\![M^c]\!]$ would stay the same for any filtration (due to the fact that $[\![M^c]\!]$ is a continuous part of $[\![M]\!]$, see Subsection \ref{subsec:candec}, \cite[Theorem 26.14]{Kal}, and \cite{Y17MartDec}, and the fact that $[\![M]\!]$ does not depend on the enlargement by its definition \eqref{eq:defquadvar}), $\nu^M$ can change. E.g.\ let $X=\mathbb R$ and let $M = N^1-N^2$, where $N^1$ and $N^2$ are independent standard Poisson processes. Let $\mathbb F = ({\mathcal F}_t)_{t\geq 0}$ be generated by $M$ and let $\overline{\mathcal F}_t$ be generated by ${\mathcal F}_t$ and $\sigma\bigl((\tau_n)_{n\geq 1}\bigr)$, where $(\tau_n)_{n\geq 1}$ are all jump times of $M$. Then $M$ is both an $\mathbb F$- and $\overline{\mathbb F}$-martingale, but for any $A\subset \mathbb R$ and $I \subset \mathbb R_+$ in the first case we have that $\nu^M(I\times A) = \lambda_{\mathbb R_+}(I) (\mathbf 1_{A}(-1) + \mathbf 1_{A}(1))$, but in the second case we obtain $\nu^M(I\times A) = \frac 12\sum_{n\geq 1} \mathbf 1_{I}(\tau_n)(\mathbf 1_{A}(-1) + \mathbf 1_{A}(1))$.
\end{remark}

\begin{remark}\label{rem:locmartwithIIisamartwithII}
Note that every local martingale with independent increments is a martingale by \cite[Theorem 2.5.1]{dlPG}. Indeed, let $M:\mathbb R_+ \times \Omega \to X$ be a local martingale with independent increments. Then there exists a sequence $(\tau_n)_{n\geq 1}$ of stopping times such that $\tau_n \nearrow \infty$ a.s.\ as $n\to \infty$ and $M^{\tau_n}$ is a martingale. Moreover, by strengthening stopping times $(\tau_n)_{n\geq 1}$ we can assume that $\mathbb E \sup_{t\geq 0}\|M^{\tau_n}_{t}\|<\infty$. Then by \cite[Theorem 2.5.1]{dlPG} we have that for any $n\geq 1$
\[
\mathbb E \sup_{t\geq 0}\|M^{\tau_n}_{t}\| \eqsim \mathbb E \sup_{t\geq 0}\|\widetilde M^{\tau_n}_{t}\|<\infty,
\]
where $\widetilde M$ is an independent copy of $M$. As $\tau_n$ and $\widetilde M$ are independent, we have that
\begin{equation}\label{eq:forMlocwithindincismartcsmama}
 \mathbb E \sup_{0\leq s\leq t}\|\widetilde M_{s}\| = \mathbb E \sup_{0\leq s\leq t}\|M_{s}\| <\infty,
\end{equation}
for any $t\geq 0$ safistying $\mathbb P(\tau_n > t)>0$. Since $\tau_n \to \infty$ a.s.\ as $n\to \infty$, \eqref{eq:forMlocwithindincismartcsmama} holds for any $t\geq 0$, and thus $M$ is a martingale by a standard argument. Consequently, $N(\omega)$ in the definition above is can be assume to be a martingale instead of a local martingale.
\end{remark}

\begin{remark}\label{rem:defofdectanglmartidiscreteandfac,q}
Let us show that this definition of a decoupled tangent martingale agrees with the one given Subsection \ref{subsec:tanmartdiscase}, i.e.\ if we have two martingales $M, N:\mathbb R_+ \times\Omega \to X$ which are purely discontinuous and have jumps only at natural points, then $N$ s a decoupled tangent martingale to $M$ if and only if the same holds for the corresponding differences. Let $(d_n)_{n\geq 1}$ be an $X$-valued martingale difference sequence, $(c_n)_{n\geq 1}$ be a decoupled tangent martingale difference sequence. Let $M$ and $N$ be martingales with respect to the filtration $\mathbb F = (\mathcal F_t)_{t\geq 0} := (\mathcal F_{[t]})_{t\geq 0}$ (here $(\mathcal F_n)_{n\geq 1}$ is the filtration where $(d_n)_{n\geq 1}$ lives and $[a]$ is the integer part of $a \geq 0$) of the form 
$$
M_t = \sum_{n\leq t} d_n,\;\;N_t = \sum_{n\leq t} c_n,\;\;\; t\geq 0.
$$
Then $M$ and $N$ are tangent by \eqref{eq:dnentangentindiscrete-->incont}, and $N(\omega)$ is a martingale with independent increments and local characteristics $(0, \nu^M(\omega))$ as the same holds for $(c_n)_{n\geq 1}$ thanks to Remark \ref{rem:discrdefislikecontdafa}, so $N$ is a decoupled tangent martingale to $M$. The converse can be shown analogously.
\end{remark}

Now we are going to state two main results of the paper.

\begin{theorem}\label{thm:tangentgencaseUMDuhoditrotasoldat}
 Let $X$ be a Banach space. Then $X$ is UMD if and only if any $X$-valued local martingale has local characteristic. Moreover, if this is the case, then for any tangent $X$-valued local martingales $M$ and $N$ and for any $1\leq p < \infty$ we have that
 \begin{equation}\label{eq:Lpboundsforgentangmarraz}
  \mathbb E \sup_{t\geq 0} \|M_t\|^p \eqsim_{p, X}  \mathbb E \sup_{t\geq 0} \|N_t\|^p.
 \end{equation}
\end{theorem}

\begin{theorem}\label{thm:mainforDTMgencasenananana}
 Let $X$ be a UMD Banach space. Then for any $X$-valued local martingale there exists a decoupled tangent local martingale.
\end{theorem}

In order to prove these theorems we will need to treat each of the cases of the canonical decompositions separately in Subsection \ref{subsec:dectangcontmart}, \ref{subsec:dectanPDQLC}, \ref{subsec:dectanPDwithAJ}, and then combine them using Subsection \ref{subsec:charandthecandecbe} in Subsection \ref{subsec:CIprocess}.

\subsection{Local characteristics and canonical decomposition}\label{subsec:charandthecandecbe}

Let us first show that different parts of the canonical decomposition are responsible for different parts of the corresponding local characteristics, and in particular that if two martingales are tangent, then the same holds for the corresponding parts of the canonical decomposition.

\begin{theorem}\label{thm:candecfortangaretang}
Let $X$ be a Banach space, $M, N:\mathbb R_+ \times \Omega \to X$ be local martingales that have the canonical decompositions $M=M^c + M^q + M^a$ and $N=N^c + N^q + N^a$. Assume also that both $M$ and $N$ have local characteristics $([\![M^c]\!], \nu^{M})$ and $([\![N^c]\!], \nu^{N})$ respectively, and that $M$ and $N$ are tangent. Then the corresponding terms of the canonical decomposition have local characteristics and are tangent as well.
\end{theorem}

We will prove the theorem by using the following elementary propositions.

\begin{proposition}\label{prop:Grigcharcontcase}
Let $X$ be a Banach space, $M:\mathbb R_+ \times \Omega \to X$ be a continuous local martingale which has a covariation bilinear form $[\![M]\!]$. Then the local characteristics of $M$ are $([\![ M ]\!], 0)$.
\end{proposition}

\begin{proof}
As $M$ does not have jumps, $M=M+0$ is the Meyer-Yoeurp decomposition; moreover, $\mu^M = 0$ a.s., and hence $\nu^M=0$ a.s.
\end{proof}

\begin{proposition}\label{prop:GigcharforPDQLC}
Let $X$ be a Banach space, $M:\mathbb R_+ \times \Omega \to X$ be a purely discontinuous quasi-left continuous local martingale. Then the local characteristics of $M$ are $(0, \nu^M)$, where $\nu^M$ is non-atomic in time.
\end{proposition}

\begin{proof}
First notice that $M$ is purely discontinuous, hence $M^c = 0$, and thus $[\![M^c]\!] =0$. The fact that $\nu^M$ is non-atomic in time follows from Lemma \ref{lem:nuMisnonatomiffMqlc}.
\end{proof}

\begin{proposition}\label{prop:GrigcharforPDmAJ}
Let $X$ be a Banach space, $M:\mathbb R_+ \times \Omega \to X$ be a purely discontinuous local martingale with accessible jumps. Then the local characteristics of $M$ are $(0, \nu^M)$, where for $\nu^M$ there exist predictable stopping times $(\tau_n)_{n\geq 1}$ such that for any $A\in \widetilde {\mathcal P}$ we have that a.s.\
\begin{equation}\label{eq:nuMhasfinitesuppforMPDwAJ}
\int_{\mathbb R_+ \times X} \mathbf 1_A(t,\cdot,x) \ud \nu^M(t,x) = \int_{\mathbb R_+ \times X} \mathbf 1_A(t,\cdot,x)\mathbf 1_{\{\tau_1, \ldots, \tau_n,\ldots\}}(t) \ud \nu^M(t,x).
\end{equation}
In other words,  $\{\tau_1,\ldots,\tau_n,\ldots\}\times X$ is a set of full $\nu^M$-measure a.s.
\end{proposition}

For the proof we will need the following lemma, which follows from \cite[Subsection 5.3]{DY17}.

\begin{lemma}\label{lem:onejumpformu-->onejumpfornu}
Let $(J, \mathcal J)$ be a measurable space, $\mu$ be an integer-valued random measure on $\mathbb R_+ \times \Omega$, and let $\tau$ be a predictable stopping times such that $\mu \mathbf 1_{\tau} = \mu$. Then for  the corresponding compensator $\nu$ we have we have that  $\nu \mathbf 1_{\tau} = \nu$, i.e.\ $\{\tau\}\times X$ is a set of full $\nu^M$-measure a.s.
\end{lemma}

\begin{proof}[Proof of Proposition \ref{prop:GrigcharforPDmAJ}]
First notice that $[\![M^c]\!] =0$ analogously to Proposition~\ref{prop:GigcharforPDQLC}. As $M$ is purely discontinuous with accessible jumps, by Lemma \ref{lem:PDmAJhasjumpsatprsttimes} there exists a sequence $(\tau_n)_{n\geq 1}$ of finite predictable stopping times with disjoint graphs such that a.s.\
\[
\{t\geq 0: \Delta M_{t} \neq 0\} \subset \{\tau_1,\ldots, \tau_n,\ldots\}.
\]
Then the desired follows from the fact that $\mu^M = \sum_{n\geq 1} \mu^{M}\mathbf 1_{\tau_n}$ a.s., Lemma \ref{lem:onejumpformu-->onejumpfornu}, and the fact that $\nu^M = \sum_{n\geq 1} \nu^{M}\mathbf 1_{\tau_n}$, which follows e.g.\ from \cite[Chapter II.1]{JS}. Indeed, $\mu^M$ is $\widetilde {\mathcal P}$-$\sigma$-finite, and if we fix $A\in \widetilde {\mathcal P}$ such that $\mathbb E\int_{\mathbb R_+ \times X} \mathbf 1_{A} \ud \mu^M <\infty$, then by the monotone convergence theorem and by the definition of a compensator measure (see Subsection \ref{subsec:ranmeasures} and \cite[Chapter II.1]{JS})
\begin{multline*}
\mathbb E\int_{\mathbb R_+ \times X} \mathbf 1_{A} \ud \nu^M = \mathbb E\int_{\mathbb R_+ \times X} \mathbf 1_{A} \ud \mu^M= \sum_{n\geq 1}\mathbb E\int_{\mathbb R_+ \times X} \mathbf 1_{A} \mathbf 1_{\tau_n} \ud \mu^{M}\\
 =\mathbb E\int_{\mathbb R_+ \times X} \mathbf 1_{A} \mathbf 1_{\{\tau_1, \ldots,\tau_n\ldots\}} \ud \mu^{M}= \mathbb E\int_{\mathbb R_+ \times X} \mathbf 1_{A} \mathbf 1_{\{\tau_1, \ldots,\tau_n\ldots\}} \ud \nu^{M},
\end{multline*}
so \eqref{eq:nuMhasfinitesuppforMPDwAJ} follows immideately.
\end{proof}

\begin{proposition}\label{prop:Grigcharforcandechowdoesitlooklike}
Let $X$ be a Banach space, $M:\mathbb R_+ \times \Omega \to X$ be a local martingale. Assume that $M$ admits the canonical decomposition $M = M^c + M^q + M^a$. Then we have that $\nu^M = \nu^{M^q} + \nu^{M^a}$.
\end{proposition}

\begin{proof}
Note that by Definition \ref{def:candec} of the canonical decomposition and by Remark~\ref{rem:candecsplitsjumps}, $M^q$ and $M^a$ collect different jumps of $M$, i.e.\ \eqref{eq:candecsplitsjumps} holds true, and hence $\mu^{M^q} + \mu^{M^a} = \mu^M$, so $\nu^{M^q} + \nu^{M^a} = \nu^M$ by the definition of a compensator (see Subsection \ref{subsec:ranmeasures}).
\end{proof}

We will also need the following lemma, which follows from \cite[Theorem 9.22]{KalRM}. Recall that a random measure $\mu$ with a compensator $\nu$ is called {\em quasi-left continuous} if any stochastic integral with respect to $\bar{\mu}=\mu -\nu$ defined by \eqref{eq:defofstochintwrtbarmu} is quasi-left continuous. A random measure $\mu$ with a compensator $\nu$ is called {\em accessible} if any stochastic integral with respect to $\bar{\mu}$ has accessible jumps.

\begin{lemma}\label{lem:decompofmeasuresontwoparts}
 Let $(J,\mathcal J)$ be a measurable space, $\mu$ be an integer-valued $\widetilde P$-$\sigma$-finite random measure on $\mathbb R_+ \times \Omega$. Then there exist unique random measures $\mu^{q}$ and $\mu^a$ such that $\mu = \mu^q + \mu^a$ and such that $\mu^q$ is quasi-left continuous and $\mu^a$ is accessible.
\end{lemma}

\begin{remark}\label{rem:whichwhatpartofrmmeans}
 Note that analogously to Lemma \ref{lem:nuMisnonatomiffMqlc} one can show that $\mu$ is quasi-left continuous if and only if $\nu$ is non-atomic in time. Similarly, by the fact that any accessible measure is supported by countably many predictable stopping times (see \cite[Theorem 9.22]{KalRM}) and by applying techniques from the proof of Proposition~\ref{prop:GrigcharforPDmAJ} it follows that $\mu$ is accessible if and only if its compensator $\nu$ has a.s.\ a support of countably many points on $\mathbb R_+$.
\end{remark}

\begin{proof}[Proof of Theorem \ref{thm:candecfortangaretang}]
The theorem follows from Proposition \ref{prop:Grigcharcontcase}, \ref{prop:GigcharforPDQLC}, \ref{prop:GrigcharforPDmAJ}, and \ref{prop:Grigcharforcandechowdoesitlooklike}, Lemma \ref{lem:decompofmeasuresontwoparts}, and Remark \ref{rem:whichwhatpartofrmmeans}.
\end{proof}

\subsection{Continuous martingales}\label{subsec:dectangcontmart}

First let us consider continuous martingales. The theory of continuous martingales is already classical and in particular due to \cite{Y18BDG,NVW,Gar85} the following proposition holds true.

\begin{proposition}\label{prop:domcontcase}
Let $X$ be a Banach space, $0< p<\infty$. Then $X$ is UMD if and only if for any pair $M, N:\mathbb R_+ \times \Omega \to X$ of continuous martingales with $[\![N]\!]_{\infty} \leq  [\![M]\!]_{\infty}$ a.s.\ one has that
\begin{equation}\label{eq:domcontcase}
\mathbb E \sup_{1\leq t < \infty} \|N_t\|^p \lesssim_{p, X}\mathbb E \sup_{1\leq t < \infty} \|M_t\|^p.
\end{equation}
\end{proposition}

What we are interested in is constructing for an $X$-valued continuous martingale $M$ a decoupled tangent martingale $N$ (see Definition \ref{def:dectanglocmartcontimecased}), which we will need later in Subsection \ref{subsec:CIprocess}.

\begin{theorem}\label{thm:DTMforcontcase}
Let $X$ be a UMD Banach space, $M:\mathbb R_+ \times \Omega \to X$ be a continuous local martingale. Then there exists an enlarged probability space $(\overline {\Omega}, \overline{\mathcal F}, \overline {\mathbb P})$ endowed with an enlarged filtration $\overline {\mathbb F} = (\overline {\mathcal F}_t)_{t\geq 0}$, and an $\overline {\mathbb F}$-adapted continuous local martingale $N:\mathbb R_+ \times \overline {\Omega} \to X$ such that $M$ is a local $\overline {\mathbb F}$-martingale with the same local characteristics, $M$ and $N$ are tangent, and $N(\omega)$ is a martingale with independent increments and with local characteristics $([\![M]\!](\omega), 0)$ for a.e.\ $\omega \in \Omega$.
\end{theorem}

For the proof we will need the following statement concerning {\em Brownian representations}.

\begin{lemma}[Brownian representation]\label{lem:brrepres}
 Let $(M^n)_{n\geq 1}$ be a family of continuous martingales and $(a_n)_{n\geq 1}$ be a real-valued sequence such that a.s.\ for any $n\geq 1$ and any $0\leq s\leq t$
 \[
  [M^n]_t - [M^n]_s \leq a_n(t-s).
 \]
 Then there exists a Hilbert space $H$, an enlarged probability space $(\overline {\Omega}, \overline{\mathcal F}, \overline {\mathbb P})$ endowed by an enlarged filtration $\overline {\mathbb F} = (\overline {\mathcal F}_t)_{t\geq 0}$, an $\overline {\mathbb F}$-adapted cylindrical Brownian motion $W_H$, and a family of predictable $H$-valued processes $(f_n)_{n\geq 1}$ depending only on the family of processes $([M^n, M^m])_{n, m\geq 1}$ such that
 \[
  M^n = f_n \cdot W_H,\;\;\;\; n\geq 1.
 \]
\end{lemma}

\begin{proof}
 Let $H$ be a separable Hilbert space, $(h_n)_{n\geq 1}$ be an orthonormal basis of $H$. Let us define an $H$-valued process $M := \sum_{n\geq 1} \frac{1}{\sqrt {a_n}n}M^nh_n$. First let us show that this process is well defined. To this end we need to show that $\sum_{n=1}^{\infty} \bigl|\frac{1}{\sqrt{a_n}n}M^n_t\bigr|^2$ converges a.s.\ for any $t\geq 0$. Note that by the monotone convergence theorem one has that a.s.\
  \begin{multline*}
 \mathbb E \sum_{n=1}^{\infty} \left|\frac{1}{\sqrt{a_n}n}M^n_t\right|^2  = \lim_{N\to \infty} \mathbb E \sum_{n=1}^{N}\frac{1}{a_nn^2}|M^n_t|^2\\
  \stackrel{(*)}= \lim_{N\to \infty}\mathbb E  \sum_{n=1}^{N}\frac{1}{a_nn^2} [M^n]_t \leq \mathbb E  \sum_{n=1}^{\infty}\frac{1}{n^2} t \lesssim {t},
 \end{multline*}
where $(*)$ holds by \cite[Theorem 26.12]{Kal}. Therefore $\sum_{n=1}^{\infty} \bigl|\frac{1}{\sqrt{a_n}n}M^n_t\bigr|^2$ converges in measure, and as this is a sum of nonnegative random variables, it converges a.s. For the similar reason and by the continuity of the conditional expectation \cite[Section 2.6]{HNVW1} one has that $M$ is an $H$-valued martingale so that by e.g.\ \cite[(3.8)]{VY16} for any $0\leq s\leq t$ a.s.\
\begin{align*}
 [M]_t - [M]_s = \sum_{n=1}^{\infty} [\langle M, h_n\rangle]_t - [\langle M, h_n\rangle]_s &= \sum_{n=1}^{\infty} \frac{1}{a_nn^2}\bigl([M^n]_t - [M^n]_s\bigr)\\
 &\leq \sum_{n=1}^{\infty} \frac {1}{n^2}(t-s) \lesssim t-s.
\end{align*}
Consequently, by \cite[Example 3.18 and Proposition 4.7]{VY16} (see also \cite[Theorem 2]{Ond}, \cite{JKFR,Yar16Br}, \cite[Theorem 8.2]{DPZ}, and \cite[Theorem 3.4.2]{KS}) there exist an enlarged probability space $(\overline {\Omega}, \overline{\mathcal F}, \overline {\mathbb P})$ endowed by an enlarged filtration $\overline {\mathbb F} = (\overline {\mathcal F}_t)_{t\geq 0}$, an $\overline {\mathbb F}$-adapted cylindrical Brownian motion $W_H$, and a predictable process $\Phi :\mathbb R_+ \times \Omega \to \mathcal L(H)$ such that $M = \Phi \cdot W_H$. Notice that by the construction of $\Phi$ presented in \cite[Proposition 4.7]{VY16} it depends only on the family of processes $([\langle M,  h\rangle, \langle M,  g\rangle])_{h, g\in H}$, and the latter by the fact that covariation is bilinear depends only on the family of processes
\[
 ([\langle M,  h_n\rangle, \langle M,  h_m\rangle])_{n,m\geq 1} = ([M^n, M^m])_{n, m\geq 1}.
\]
The desired follows by setting $f_n := \sqrt{a_n} n\Phi^*h_n$ for any $n\geq 1$, so
\[
 f_n \cdot W_H =\sqrt{a_n} n \, \Phi^*h_n \cdot W_H = \sqrt{a_n} n\langle \Phi \cdot W_H, h_n \rangle = \sqrt{a_n} n\langle M, h_n \rangle = M^n.
\]
\end{proof}

\begin{proof}[Proof of Theorem \ref{thm:DTMforcontcase}]
Without loss of generality by the Pettis measurability theorem \cite[Theorem 1.1.20]{HNVW1} we may assume that $X$ is separable (and as $X$ is UMD, it is reflexive, so $X^*$ is separable as well). By a stopping time argument we may assume that $M$ is uniformly bounded a.e.\ on $\mathbb R_+ \times \Omega$. Let $(x^*_n)_{n\geq 1}$ be a dense subset of a unit ball in $X^*$, and let us define a random time-change
\begin{equation}\label{eq:defoftautforcontcase}
\tau_t := \inf\Bigl\{s\geq 0: \sum_{n=1}^{\infty} \tfrac{1}{n^2} [ \langle M, x_n^*\rangle ]_s + s\geq t\Bigr\},\;\;\; t\geq 0
\end{equation}
(the latter time-change is well defined as $([\langle M, x_n^* \rangle])_{n\geq 1}$ are a.s.\ uniformly bounded since by \cite{Y18BDG} 
$$
[ \langle M, x^*\rangle ]_s \leq \|[\![M]\!]_s\| \leq \gamma([\![M]\!]_s)<\infty,\;\;\; s\geq 0,
$$ 
a.s.\ for any $x^*\in X^*$, $\|x^*\|\leq 1$, see Remark \ref{rem:ifUMDthencovbilform} for the definition of $\gamma(\cdot)$). Note that as stricktly increasing, this time-change is invertible, i.e.\ for a time changed filtration $\mathbb G = (\mathcal G_t)_{t\geq 0} := (\mathcal F_{\tau_t})_{t\geq 0}$ there exists a strictly increasing $\mathbb F$-predictable continuous process $A:\mathbb R_+ \times \Omega \to X$ such that $(A\circ \tau)_t = (\tau \circ A)_t = t$ a.s.\ for any $t\geq 0$. (Note that in fact $A_t = \sum_{n=1}^{\infty} \tfrac{1}{n^2} [ \langle M, x_n^*\rangle ]_t + t$ as it is defined by \eqref{eq:defoftautforcontcase}, see e.g.\ \cite[Theorem 2.16]{Y17MartDec} or \cite[Subsection 2.4.4]{YPhD}). Let $\widetilde M := M \circ \tau$. Then by the Kazamaki theorem \cite[Theorem 17.24]{Kal} for any $n\geq 1$ a.s.\
\begin{align*}
[\langle \widetilde M,x_n^* \rangle ]_t - [\langle \widetilde M,x_n^* \rangle ]_s  &= [\langle  M,x_n^* \rangle ]_{\tau_t} - [\langle M,x_n^* \rangle ]_{\tau_s}\\
&\leq n^2(t-s),\;\;\; 0\leq s \leq t,
\end{align*}
thus by Lemma \ref{lem:brrepres} we can show that there exist an enlarged probability space $(\overline {\Omega}, \overline{\mathcal F}, \overline {\mathbb P})$, an enlarged filtration $\overline{\mathbb G} = (\overline{\mathcal G}_t)_{t\geq 0}$, a Hilbert space $H$, a $\overline {\mathbb G}$-adapted cylindrical Wiener process $W_H$, and a family of $H$-valued $\overline {\mathbb G}$-predictable  processes $(f_n)$ such that 
\begin{equation}\label{eq:Mxn*isrepresentable}
 \langle \widetilde M_t, x_n^*\rangle = \int_0^{t} f_n\ud W_H,\;\;\;\; n\geq 1.
\end{equation}
Let us first show that for any $x^*\in X^*$ there exists a $\overline {\mathbb G}$-predictable process $f_{x^*}:\mathbb R_+ \times \Omega \to H$ such that $\langle \widetilde M_t, x^*\rangle = f_{x^*}\cdot W_H$. Let $Y = {\text {span}} \{x_n^*, n\geq 1\}\subset X^* $. By the definition of $(x_n^*)_{n\geq 1}$, $Y$ is dense in $X^*$, so there exists a sequence $(y_m)_{m\geq 1}\in Y$ that converges to $x^*$. By the definition of $Y$, by the linearity of a stochastic integral, and by \eqref{eq:Mxn*isrepresentable} for each $m\geq 1$ there exists a $\overline {\mathbb G}$-predictable process $f_{y_m}:\mathbb R_+ \times \Omega \to H$ such that $\langle \widetilde M_t, y_m\rangle = f_{y_m}\cdot W_H$. Moreover, $f_{y_m}\cdot W_H$ converges to $\langle \widetilde M, x^*\rangle$ in strong $L^p$ sense for any $1\leq p <\infty$ as
\begin{align*}
  \mathbb E \sup_{t\geq 0} \Bigl|\int_0^tf_{y_m}\ud W_H - \langle \widetilde M_t, x^*\rangle\Bigr|^p &= \mathbb E \sup_{t\geq 0} |\langle \widetilde M, x^*-y_m\rangle|^p\\
  &\leq\|x^*-y_m\|^p \mathbb E \sup_{t\geq 0} \| \widetilde M\|^p  \to 0,\;\;\; m\to \infty,
\end{align*}
where the latter follows from the fact that $\widetilde M$ is uniformly bounded. Therefore the existence of the desired $f_{x^*}$ follows e.g.\ from \eqref{eq:LpboundsforstochintwrtHcylbrm} and the fact that the space $L^p(\overline{\Omega}; L^2(\mathbb R_+; H))$ restricted to a proper predictable $\sigma$-algebra is a Banach space. Moreover, it follows from \eqref{eq:qvvarofstintwrtHcylbrmot} that  for any $x^* \in X^*$ for a.e.\ $\omega \in \overline\Omega$
\begin{equation}\label{eq:L2notmoffx*}
 \int_0^{\infty}\|f_{x^*}\|^2 \ud s  = [ \widetilde M, x^* ]_{\infty}  = [\![ \widetilde M ]\!]_{\infty} (x^*, x^*)\lesssim_{M,\omega} \|x^*\|^2.
\end{equation}
Let us now show that for a.e.\ $\omega \in\overline \Omega$ the mapping $x^*\mapsto f_{x^*}(\omega)$ can be assumed to be linear. As $Y$ is a linear subspace of $X^*$ generated by a countable set, it has a countable Hamel basis $(z_n)_{n\geq 1}\subset X^*$. Let $Z$ be a $\mathbb Q$-span of $(z_n)_{n\geq 1}$. Then by the linearity of a stochastic integral and by the fact that $Z$ is countable for any $z^1,\ldots,z^k\in Z$ and for any $r_1,\ldots, r_k\in \mathbb Q$ we can  assume that $f_{r_1z^1 + \ldots + r_kz^k} = r_1 f_{z^1} + \ldots + r_k f_{z^k}$ everywhere on $\overline{\Omega}$. Moreover, without loss of generality since $Z$ is countable we know that \eqref{eq:L2notmoffx*} holds a.s.\ for any $z\in Z$. Thus there exists $\overline\Omega_0\subset\overline \Omega$  of full measure such that for any $\omega \in \overline\Omega_0$ we have a bounded linear operator $T(\omega):Z \to L^2(\mathbb R_+ ;H)$ with $T(\omega)z = f_{z}(\omega)$. By extending the operator $T(\omega)$ to the whole $X^*$ and by the construction of $f_{x^*}$ for a general $x^*\in X^*$ we have that $T(\omega)x^* = f_{x^*}(\omega)$ for a.e.\ $\omega\in \overline\Omega$, and the desired holds true.

Now note that by \eqref{eq:L2notmoffx*}, $T^*$ corresponds to a bilinear form $[\![M]\!]_{\infty}$ having a finite Gaussian characteristic (i.e.\ a.s.\ $[\![M]\!]_{\infty}(x^*, x^*) = \langle T x^*, T x^* \rangle$ for any $x^*\in X^*$ with $\gamma([\![M]\!]_{\infty})<\infty$, see \cite[Subsection 3.2]{Y18BDG} and Remark \ref{rem:ifUMDthencovbilform}), so by \cite[Subsection 3.2]{Y18BDG} a.s.\ there exists $\Phi = T^* \in \gamma(L^2(\mathbb R_+; H), X)$ such that a.s.\ 
\[
 \langle \widetilde M_t, x^*\rangle = \int_0^{t} f_{x^*} \ud W_H = \int_0^{t}\Phi^* x^* \ud W_H,\;\;\; x^*\in X^*,
\]
and thus by \cite[Theorem 3.5]{NVW} and the fact that $\Phi^* x^* = Tx^* = f_{x^*}$ is a predictable process for any $x^* \in X^*$ we have that $\widetilde M = \Phi \cdot W_H$. 

Finally, let us construct $N$. Let $W'_H$ be an independent cylindrical Brownian motion, $\widetilde N  := \Phi \cdot W'_H$, and let $N := \widetilde N \circ A$. Then $N$ is a martingale on an enlarged probability space $(\widehat {\Omega}, \widehat{\mathcal F}, \widehat{\mathbb P})$ and an enlarged filtration $\widehat {\mathbb F} = (\widehat {\mathcal F}_t)_{t\geq 0}$ (which is now generated by the original filtration and time changed cylindrical Brownian motions $W_H$ and $W'_H$) by Kazamaki theorem \cite[Theorem 17.24]{Kal} and the stochastic integration theory \cite{NVW}, so we need to show that $M$ is a local martingale with the same local characteristics $([\![M]\!], 0)$, $[\![N]\!] = [\![ M]\!]$ a.s., and that $N(\omega)$ is a martingale with independent increments and with the local characteristics $([\![ M]\!](\omega), 0)$ for a.e.\ $\omega \in \Omega$. $M$ is a martingale because of the fact that $M$ is independent of $W_H'$ and the fact that $M\circ \tau = \Phi \cdot W_H$ is a $\overline {\mathbb G}$-martingale, so when we time-change it back by using $A$ we get a martingale with respect to the filtration generated by the original filtration  $\mathbb F$ and a time-changed cylidrical process $W_H' \circ A$ (see Kazamaki theorem \cite[Theorem 17.24]{Kal}), while $M$ does not change its local characteristics as it remains continuous and due to the definition of a quadratic variation which does not depend on enlargement of filtration (see Subsection \ref{subsec:quadrvar}). The second holds by \eqref{eq:qvvarofstintwrtHcylbrmot} and due to the fact that for any $x^*\in X^*$ a.s.\
\[
[\![N]\!]_t(x^*, x^*) = [\![\widetilde N]\!]_{A_t}(x^*, x^*) = \int_0^{A_t} \|\Phi^*x^*\|^2 \ud s =  [\![\widetilde M]\!]_{A_t}(x^*, x^*) = [\![M]\!]_t(x^*, x^*).
\]
Now let us prove that $N(\omega)$ is a martingale with independent increments and with the local characteristics $([\![ M]\!](\omega), 0)$. 
This directly follows from the construction of a stochastic integral (see Subsection \ref{subsec:prelimstint} and \cite{NVW}), the fact that $\Phi(\omega)$ is deterministic and is in $\gamma(L^2(\mathbb R_+; H), X)$ for a.e.\ $\omega \in \Omega$, the fact that the time change $(\tau_t(\omega))_{t\geq 0}$ is deterministic for a.e.\ $\omega \in \Omega$, and the fact that $W'_H$ does not depend on $\omega \in \Omega$.
\end{proof}

\begin{remark}\label{rem:condectangmarthasindinrccond[[M]]}
One can straighten the latter proposition in the following way: $N$ has independent increments given $[\![M^c]\!]$. Indeed, due to the construction of $\Phi$ (see Lemma \ref{lem:brrepres} and its proof) and $(\tau_t)_{t\geq 0}$ we have that these random elements are $\sigma([\![M^c]\!])$-measurable, so the desired follows from Corollary \ref{cor:condindgivenRVsuffcond} as for a.e.\ fixed $[\![M^c]\!]$ both $\Phi$ and $(\tau_t)_{t\geq 0}$ are fixed.
\end{remark}

\subsection{Stochastic integrals with respect to random measures}\label{subsec:RMstochintinBanspaxcewithCOx}

Before treating the case of purely discontinuous quasi-left continuous martingales in Subsection \ref{subsec:dectanPDQLC} we will need to prove a similar result for stochastic integrals with respect to random measures (see Subsection \ref{subsec:ranmeasures}). This case will be done via {\em Cox processes}.

\subsubsection{Cox process}\label{subsubsec:Coxprocess}
Let $(J, \mathcal J)$ be a measurable space, $\mu$ be an optional integer-valued random measure on $\mathbb R_+ \times J$ with a compensator $\nu$ which is non-atomic in time (which is equivalent to the fact that $\mu$ is quasi-left continuous, see \cite[Theorem 9.22]{KalRM} or \cite[Corollary II.1.19]{JS}).
 Due to Cox \cite{Cox55} (see also \cite{Kal,KalRM,CI80,JS,KingPois}) it is known that there exists an enlarged probability space $(\overline \Omega, \overline {\mathcal F}, \overline {\mathbb P}) = (\Omega \times \widehat {\Omega}, \mathcal F \otimes \widehat {\mathcal F}, \mathbb P \otimes \widehat {\mathbb P})$ (where $(\widehat{\Omega}, \widehat {\mathcal F}, \widehat {\mathbb P})$ is an independent probability space where the corresponding Poisson random measure lives, see Example \ref{ex:CoxprocessfiniteJ}), an enlarged filtration $\overline {\mathbb F}= (\overline{\mathcal F}_t)_{t\geq 0}$, and a unique up to distribution integer-valued random measure $\mu_{\rm Cox}$ on $\mathbb R_+ \times J$ optional with respect to $\overline {\mathbb F}$ having $\nu$ as a compensator so that $\mu_{\rm Cox}$ is conditionally Poisson given $\mathcal F$, i.e.\ for any $C\in \widetilde {\mathcal P}$ (see Subsection \ref{subsec:ranmeasures} for the definition of $\widetilde {\mathcal P}$) with $\mathbb E \int_{\mathbb R_+ \times J} \mathbf 1_C \ud \nu<\infty$
 \begin{itemize}
 \item processes $\mu_{\rm Cox}\bigl(([0,\cdot]\times A_1)\cap C\bigr)$ and $\mu_{\rm Cox}\bigl(([0,\cdot]\times A_2)\cap C\bigr)$ are conditionally independent given $\mathcal F$ for any disjoint $A_1, A_2 \in \mathcal J$,
 \item the random measure $\mu_{\rm Cox}(\omega)$ is a Poisson random measure and for almost any fixed $\omega \in \Omega$ (see Subsection \ref{subsec:PoissRMprelim}).
 \end{itemize} 
 Such a random measure $\mu_{\rm Cox}$ is called a {\em Cox process directed by $\nu$}.

\begin{example}\label{ex:CoxprocessfiniteJ}
Let $J$ be finite, $J = \{1, \ldots, n\}$, $\mathcal J$ be a $\sigma$-algebra generated by all atoms of $J$. Let $\mu$ be a random measure on $\mathbb R_+ \times J$ with a compensator $\nu$ such that $\nu([0,t] \times J)<\infty$ a.s.\ for any $t\geq 0$. Then the construction of the Cox process $\mu_{\rm Cox}$ directed by $\nu$ has the following form. Let $N$ be a standard Poisson random measure on $\mathbb R_+ \times J$, i.e.\ $\int \mathbf1_{\{0\}} \ud N, \ldots, \int \mathbf1_{\{n\}} \ud N$ are independent Poissons and 
\[
\mathbb E \int_{[0, t]\times J} \mathbf1_{\{m\}} \ud N = t,\;\;\;\;\; 1\leq m\leq n.
\]
Then the desired measure has the following form
\begin{equation}\label{eq:CoxifJfinite}
\mu_{\rm Cox}([0,t]\times \{m\}) := N\bigl([0,\nu([0,t]\times \{m\})]\times \{m\}\bigr),\;\;\;\; 1\leq m\leq n.
\end{equation}
In the case of a general $\widetilde P$-$\sigma$-finite compensator $\nu$ the latter can be expressed as the sum $\nu = \sum_{k\geq 1} \nu^k$ of compensators $(\nu^k)_{k\geq 1}$ with disjoint domains in $\widetilde {\mathcal P}$, where each of $\nu^k$ satisfies $\nu^k(\mathbb R_+ \times J)<\infty$ a.s., and then the Cox process $\mu_{\rm Cox}$ will have the form $\mu_{\rm Cox} = \sum_{k\geq 1}\mu^k_{\rm Cox}$, where each of $\mu^k_{\rm Cox}$ is constructed analogously \eqref{eq:CoxifJfinite}, but with using independent standard Poisson random measures $N^k$ on $\mathbb R_+ \times J$ and compensators $\nu^k$ respectively.
\end{example}

\subsubsection{Random measures: tangency and decoupling}\label{subsubsec:RMtanganddec}

It turns out that Cox processes play an important r\^ole in random measure theory and in particular if one changes a random measure by the corresponding Cox process then the strong $L^p$-norm does not change much. Recall that a stochastic integral with respect to a random measure is defined by \eqref{eq:stochintwrtranmeasdefof}.

\begin{theorem}\label{thm:mumuCoxcomparable}
Let $X$ be a Banach space, $1\leq p<\infty$. Then $X$ is UMD if and only if for any measurable space $(J, \mathcal J)$ and any integer-valued random measure $\mu$ on $\mathbb R_+ \times J$ with a compensator measure $\nu$ which is  non-atomic in time one has that for any elementary predictable $F:\mathbb R_+ \times \Omega \times J\to X$
\begin{equation}\label{eq:CoxmeasiffUMD}
\mathbb E \sup_{t\geq 0} \Bigl\| \int_{[0,t]\times J} F \ud \bar\mu \Bigr\|^p \eqsim_{p, X} \mathbb E  \Bigl\| \int_{\mathbb R_+\times J} F \ud \bar \mu_{\rm Cox} \Bigr\|^p,
\end{equation}
where  $\mu_{\rm Cox}$ is the Cox process directed by $\nu$, $\bar \mu = \mu - \nu$, and $\bar \mu_{\rm Cox} = \mu_{\rm Cox} - \nu$.
\end{theorem}

For the proof we will need the following proposition. Recall that a Poisson measure $N$ is called {\em nontrivial} if its compensator in nonzero (equivalently, if it is nonzero itself).

\begin{proposition}\label{prop:XUMDiffdecouplforPois}
Let $X$ be a Banach space, $1\leq p<\infty$, $(J, \mathcal J)$ be a measurable space, $N$ be a nontrivial Poisson random measure on $\mathbb R_+ \times J$ with a compensator $\nu = \lambda \otimes \kappa$ with $\lambda$ being Lebesgue on $\mathbb R_+$ and $\kappa$ being a $\sigma$-finite measure on $(J, \mathcal J)$, $\widetilde N:= N - \nu$ be the corresponding compensated Poisson measure. Then $X$ has the UMD property if and only if for any elementary predictable $F:\mathbb R_+ \times \Omega \times J \to X$
\begin{equation}\label{eq:decthmniceRHS}
\mathbb E \sup_{t\geq 0} \Bigl\| \int_{[0,t]\times J} F \ud \widetilde N \Bigr\|^p \eqsim_{p, X} \mathbb E  \Bigl\| \int_{\mathbb R_+\times J} F \ud \widetilde N_{\rm ind} \Bigr\|^p,
\end{equation}
where $\widetilde N_{\rm ind}$ is an independent copy of $\widetilde N$.
\end{proposition}

\begin{proof}
First notice that as $\nu = \lambda \otimes \kappa$, $\widetilde N$ is time-homogeneous, i.e.\ the distributions of $\widetilde N$ and shifted $\widetilde N(\cdot, \cdot +t,\cdot)$ are the same for any $t\geq 0$.

The ``only if'' part follows from the inequalities \eqref{eq:strongLpfortangentmccnnll} for discrete tangent martingales, Remark \ref{ex:tangstandarddecxinvn-->xi'nvn}, the definition of a stochastic integral with reapect to a random measure \eqref{eq:stochintwrtranmeasdefof}, and from the fact that by \cite[Proposition 6.1.12]{HNVW2}
\begin{multline*}
 \mathbb E  \Bigl\| \int_{\mathbb R_+\times J} F \ud \widetilde N_{\rm ind} \Bigr\|^p = \mathbb E \mathbb E_{N_{\rm ind}} \Bigl\| \int_{\mathbb R_+\times J} F \ud \widetilde N_{\rm ind} \Bigr\|^p\\
 \eqsim_p \mathbb E \mathbb E_{N_{\rm ind}} \sup_{t\geq 0} \Bigl\| \int_{[0,t]\times J} F \ud \widetilde N_{\rm ind} \Bigr\|^p=\mathbb E \sup_{t\geq 0} \Bigl\| \int_{[0,t]\times J} F \ud \widetilde N_{\rm ind} \Bigr\|^p
\end{multline*}
(here $\mathbb E_{N_{\rm ind}}$ is defined by Example \ref{ex:defofExiforrvxis}).
Let us show the ``if'' part. Let \eqref{eq:decthmniceRHS} be satisfied for some $1\leq p<\infty$ and for any elementary predictable $F$. Then
\begin{align*}
\mathbb E \sup_{t\geq 0} \Bigl\| \int_{[0,t]\times J} F \ud \widetilde N \Bigr\|^p& \eqsim_{p, X} \mathbb E  \Bigl\| \int_{\mathbb R_+\times J} F \ud \widetilde N_{\rm ind} \Bigr\|^p\\
& \stackrel{(*)}\eqsim_p \mathbb E \Bigl\| \int_{\mathbb R_+\times J} F \ud \widetilde N^1_{\rm ind}  -  \int_{\mathbb R_+\times J} F \ud \widetilde N^2_{\rm ind}\Bigr\|^p,
\end{align*}
where $\widetilde N^1_{\rm ind}$ and $\widetilde N^2_{\rm ind}$ are independent copies of $\widetilde N$, and $(*)$ follows by a triangle inequality and the $L^p$-contractility of a conditional expectation \cite[Corollary 2.6.30]{HNVW1}. Therefore for any predictable process $a:\mathbb R_+ \times \Omega \to \{-1,1\}$ independent of $\widetilde N^1_{\rm ind}$ and $\widetilde N^2_{\rm ind}$ one has that
\begin{align}\label{eq:stochintaFandFcomp}
\mathbb E \sup_{t\geq 0} \Bigl\| \int_{[0,t]\times J}  a F \ud \widetilde N \Bigr\|^p &\eqsim_{p, X} \mathbb E \Bigl\| \int_{\mathbb R_+\times J} aF \ud \widetilde N^1_{\rm ind}  -  \int_{\mathbb R_+\times J} aF \ud \widetilde N^2_{\rm ind}\Bigr\|^p\nonumber\\
&\stackrel{(*)}= \mathbb E \Bigl\| \int_{\mathbb R_+\times J} F \ud \widetilde N^1_{\rm ind}  -  \int_{\mathbb R_+\times J} F \ud \widetilde N^2_{\rm ind}\Bigr\|^p\\
&\eqsim_{p, X} \mathbb E \sup_{t\geq 0} \Bigl\| \int_{[0,t]\times J}  F \ud \widetilde N \Bigr\|^p,\nonumber
\end{align}
where $(*)$ follows from the fact that we are integrating both $aF$ and $F$ with respect to a symmetric independent noise, so $a$ does not play any role. Let us show that \eqref{eq:stochintaFandFcomp} implies the UMD property. Without loss of generality by assuming that $J:=A$ for some fixed $A \subset \mathcal J$ with $0<\kappa(A)<\infty$ and that $F$ has only steps of the form $\mathbf 1_A$, we may assume that $J$ consists only of one point and thus $\widetilde N$ is a standard compensated Poisson process with the rate parameter $\kappa(A)$ (so, by a time-change argument the rate can be assumed $1$), and thus \eqref{eq:stochintaFandFcomp} implies that for any elementary predictable $F:\mathbb R_+ \times \Omega \to X$
\begin{equation}\label{eq:atransfforstochintwrtwidetN}
\mathbb E \sup_{t\geq 0} \Bigl\| \int_{0}^t aF \ud \widetilde N \Bigr\|^p \eqsim_{p, X}\mathbb E \sup_{t\geq 0} \Bigl\| \int_0^t F \ud \widetilde N \Bigr\|^p.
\end{equation}
First assume that $p>1$. Then due to Doob's maximal inequality \eqref{eq:DoobsineqXBanach}, \eqref{eq:atransfforstochintwrtwidetN} is equivalent to
\begin{equation}\label{eq:aFandFLpineq}
\mathbb E \Bigl\| \int_{0}^{\infty} aF \ud \widetilde N \Bigr\|^p \eqsim_{p, X}\mathbb E  \Bigl\| \int_0^{\infty} F \ud \widetilde N \Bigr\|^p.
\end{equation}
Let $(r_n)_{n=1}^N$ be a sequence of independent Rademacher random variables (see Definition \ref{def:ofRadRV}), $\phi_1 \in X$, $\phi_n:\{-1,1\}^{n-1} \to X$ for $2\leq n\leq N$. Let $(\eps_n)_{n=1}^N$ be a $\{-1,1\}$-valued sequence. By the definition of the UMD property and by \cite[Theorem 4.2.5]{HNVW1} we only need to show that
\begin{equation}\label{eq:aFtoFimpliesUMDUMDprop}
\mathbb E \Bigl\| \sum_{n=1}^N \eps_n r_n \phi_n(r_1,\ldots,r_{n-1}) \Bigr\|^p \lesssim_{p, X} \mathbb E \Bigl\| \sum_{n=1}^N r_n \phi_n(r_1,\ldots,r_{n-1}) \Bigr\|^p.
\end{equation}
To this end we approximate in $L^p$-sense distributions of $\sum_{n=1}^N r_n \phi_n(r_1,\ldots,r_{n-1})$ and $\sum_{n=1}^N \eps_n r_n \phi_n(r_1,\ldots,r_{n-1})$ by $\int_0^{\infty} F \ud \widetilde N$ and $\int_0^{\infty} aF \ud \widetilde N$ respectively by finding appropriate $F$ and $a$.

Fix $\eps>0$. Let $A>0$ be such that for a stopping time
\[
\tau:= \inf\{t\geq 0:|\widetilde N_t| \geq A\}
\]
one has that 
\begin{equation}\label{eq:AsignNtauisalmostNtau}
\| \sign \widetilde N_{\tau} - \widetilde N_{\tau}/A\|_{L^p(\Omega)}<\eps.
\end{equation}
Such $A$ exists since $|\Delta \widetilde N_\tau| \leq 1$ and $\tau<\infty$ a.s.\ as $\widetilde N$ is unbounded a.s., so a.s.\ 
\begin{equation}\label{eq:valofNtau/A}
 \widetilde N_{\tau} /A\in \bigl[-1-\tfrac{1}{A}, -1\bigr] \cup \bigl[1, 1 + \tfrac{1}{A}\bigr],
\end{equation}
and since by \cite[Theorem 25.14]{Kal}
\begin{equation}\label{eq:widetildeNtaumeanzero}
\mathbb E \widetilde N_{\tau} = 0.
\end{equation}
Let $\tau_0=0$, $\tau_1 = \tau$, and for any $2\leq n\leq N$ define
\[
\tau_n := \inf\{t\geq \tau_{n-1}:\bigl|\widetilde N_t -\widetilde N_{\tau_{n-1}}\bigr| \geq A\}.
\]
By strong Markov property of L\'evy processes we have that $(\widetilde N_{\tau_n} - \widetilde N_{\tau_{n-1}})_{n=1}^N$ are independent, and thus by \eqref{eq:AsignNtauisalmostNtau} and \eqref{eq:widetildeNtaumeanzero} there exists a sequence of independent Rademacher random variables which we without loss of generality can denote by $(r_n)_{n=1}^N$ such that
\begin{equation}\label{eq:rnapproxwidetN-widetN/A}
\bigl\| r_n- (\widetilde N_{\tau_n} - \widetilde N_{\tau_n-1})/A\bigr\|_{L^p(\Omega)}\lesssim_p\eps,\;\;\;\;\;\; 1\leq n\leq N.
\end{equation}
(One just needs to use the fact that by \eqref{eq:AsignNtauisalmostNtau} and \eqref{eq:widetildeNtaumeanzero}, $|\mathbb E\sign(N_{\tau_n} - N_{\tau_{n-1}})|\lesssim_p\eps$, so by \eqref{eq:valofNtau/A} one can approximate $(\widetilde N_{\tau_n} - \widetilde N_{\tau_n-1})/A$ by a Rademacher.)
Now let us define appropriate $F$ and $a$ in the following way:
\[
F(t) := 
\begin{cases}
\phi_1/A, &\text{if}\;\;0\leq t\leq \tau_1,\\
\phi_n(r_1,\ldots,r_{n-1})/A, &\text{if}\;\;\tau_{n-1} <t\leq \tau_n,\;2\leq n\leq N,\\
0,&\text{if}\;\; t>\tau_N,
\end{cases}
\]
\[
a(t) := 
\begin{cases}
\eps_n, &\text{if}\;\;\tau_{n-1} <t\leq \tau_n,\;1\leq n\leq N,\\
1,&\text{if}\;\; t>\tau_N.
\end{cases}
\]
Then one has that $F$ and $a$ are predictable by \cite[Theorem I.2.2]{JS}, and moreover
\begin{align*}
\Bigl\| \int_0^{\infty} F \ud \widetilde N &- \sum_{n=1}^N r_n \phi_n(r_1,\ldots,r_{n-1})  \Bigr\|_{L^p(\Omega; X)} \\
& = \Bigl\| \sum_{n=1}^N \bigl(r_n - (\widetilde N_{\tau_n} -\widetilde N_{\tau_n-1})/A \bigr) \phi_n(r_1,\ldots,r_{n-1})  \Bigr\|_{L^p(\Omega; X)}\\
&\leq \sum_{n=1}^N \Bigl\| \bigl(r_n - (\widetilde N_{\tau_n} -\widetilde N_{\tau_n-1}) /A\bigr) \phi_n(r_1,\ldots,r_{n-1})  \Bigr\|_{L^p(\Omega; X)}\\
&\leq L\sum_{n=1}^N \Bigl\| r_n - (\widetilde N_{\tau_n} -\widetilde N_{\tau_n-1}) /A\Bigr\|_{L^p(\Omega; X)} \stackrel{(*)}\lesssim_p NL \eps,
\end{align*}
where $L>0$ is such that $\|\phi_n\|_{\infty}<L$ for any $1\leq n\leq N$, and $(*)$ follows from \eqref{eq:rnapproxwidetN-widetN/A}. For the same reason we have that
\[
\Bigl\| \int_0^{\infty} aF \ud \widetilde N - \sum_{n=1}^N \eps_nr_n \phi_n(r_1,\ldots,r_{n-1})  \Bigr\|_{L^p(\Omega; X)}  \lesssim_p NL \eps.
\]
By letting $\eps \to 0$ and by \eqref{eq:aFandFLpineq} we obtain \eqref{eq:aFtoFimpliesUMDUMDprop}.

\smallskip

Now let $p=1$. Then we need to use good-$\lambda$ inequalities in order to show that \eqref{eq:atransfforstochintwrtwidetN} holds for any $p>1$ (see Section \ref{sec:Cfwithmodg}). Let $M = \int_{[0, \cdot] \times J} F \ud \widetilde N$ and $L = \int_{[0, \cdot] \times J} aF \ud \widetilde N$ for some elementary predictable $F:\mathbb R_+ \times \Omega \to \mathbb R$ and $a:\mathbb R_+ \times \Omega \to \{0, 1\}$. Let us fix $\beta>1$, $\delta >0$, and $\lambda>0$, and let us define stopping times
\begin{align*}
\sigma&:= \inf\{t\geq 0: \|L_t\|>\lambda\},\\
\tau&:= \inf\{t\geq 0: \|M_t\|>\delta\lambda\},\\
\rho&:= \inf\{t\geq 0: \|F_t\|>\delta\lambda\}.
\end{align*}
 Define
$$
\widehat M_t:= \int_{[0, t] \times J} F \mathbf 1_{(\tau \wedge \sigma \wedge \rho, \tau \wedge \rho]} \ud \widetilde N,\;\;\; t\geq 0,
$$
$$
\widehat L_t:= \int_{[0, t] \times X} aF \mathbf 1_{(\tau \wedge \sigma \wedge \rho, \tau \wedge \rho]} \ud \widetilde N,\;\;\; t\geq 0.
$$ 
Note that $\widehat M$ coincides with $M - M^{\tau \wedge \sigma \wedge \rho}$ on $[0, \tau\wedge \rho]$, so by the definition of $\tau$ and $\rho$ we have that $\widehat M \leq 2\delta\lambda$ (note that $F$ is elementary predictable, so $\rho$ is predictable, and so $\Delta M_{\rho} = \Delta L_{\rho} = 0$ as $M$ and $L$ are quasi-left continuous), and thus by \eqref{eq:atransfforstochintwrtwidetN} for $p=1$ we have that
\begin{equation}\label{eq:MhatdominatesLhatforstochintwrtPoiss}
\mathbb E \sup_{t\geq 0} \|\widehat L_t\|\lesssim_{X}\mathbb E \sup_{t\geq 0} \|\widehat M_t\| \leq 2^p \delta^p\lambda^p.
\end{equation}
Therefore, as $\|\Delta M_t\| \leq \|F_t\|$ a.s.\ for any $t\geq 0$,
\begin{equation*}
\begin{split}
\mathbb P(L^*>\beta\lambda, \Delta M^* \vee M^*\leq \delta \lambda)  &\leq \mathbb P(L^*>\beta\lambda, \tau = \rho = \infty) \stackrel{(*)}\leq \mathbb P(\widehat L^*>(\beta - \delta-1)\lambda)\\
& \leq \frac{1}{(\beta - \delta-1)\lambda} \mathbb E \widehat L^* \stackrel{(**)}\lesssim_{X}\frac{1}{(\beta - \delta-1)\lambda} \mathbb E \widehat M^*,
\end{split}
\end{equation*}
where $(*)$ follows from the fact that if $\tau = \rho  = \infty$, then $\widehat L$ coincides with $L-L^{\tau \wedge \sigma \wedge \rho}$ on $\mathbb R_+$, and the fact that $\mathbb P(\sigma = \rho) = 0$ as $\rho$ is predictable and $\sigma$ is totally inaccessible (see \cite[Chapter 25]{Kal}), while $(**)$ holds by \eqref{eq:MhatdominatesLhatforstochintwrtPoiss}. On the other hand as $\widehat M \leq 2\delta\lambda$ a.s.
\[
 \mathbb E \widehat M^* = \mathbb E \widehat M^* \mathbf 1_{\tau \wedge \sigma \wedge \rho < \infty} \leq 2\delta\lambda \mathbb P(\sigma<\infty)  = 2\delta\lambda \mathbb P(L^* > \lambda).
\]
Consequently,
\[
 \mathbb P(L^*>\beta\lambda, \Delta M^* \vee M^*\leq \delta \lambda) \lesssim_{X}\frac{2\delta\lambda}{(\beta - \delta-1)\lambda}\mathbb P(L^* > \lambda),
\]
and
one has that \eqref{eq:atransfforstochintwrtwidetN} holds for any $p>1$ by Lemma \ref{lem:goodlgivesphi} and by the fact that $\Delta M^* \leq 2M^*$ a.s., so the UMD property follows from the case $p>1$ considered above.
\end{proof}

For the proof of Theorem \ref{thm:mumuCoxcomparable} we will also need the following technical lemma on approximation of continuous increasing predictable functions which simpler form was proven in \cite[Subsection 5.5]{DY17}.

\begin{lemma}\label{lem:Lipshfuncapprox}
Let $F:\mathbb R_+ \times \Omega \to \mathbb R_+$ be a nondecreasing continuous predictable process such that $F(t) - F(s) \leq C(t-s)$ a.s.\ for any $0\leq s \leq t$ and any $C>0$, and such that $F(0)=0$ a.s. Let $0<p<\infty$. Then for any $p\in(0, 1]$, for any $T\geq 0$, and for any $\delta>0$ these exists natural $K_0>0$ such that for any $K>K_0$ and for $(t_k)_{k=0}^K = (Tk/K)_{k=0}^K$ we have that
\begin{equation}\label{eq:lemLipshfuncapprox}
 \mathbb E \Bigl(\sum_{k=1}^K  \bigl|F(t_k)- \mathbb E (F(t_k)| \mathcal F_{t_{k-1}})\bigr|\Bigr)^p<\delta.
\end{equation}
\end{lemma}

\begin{proof}
Let us first show the lemma for $p= 1$.
As it was shown in \cite[Subsection 5.5]{DY17}, there exists a predictable process $f:\mathbb R_+ \times \Omega \to [0, C]$ such that a.s.\ 
\begin{equation}\label{eq:thereexistsfwvalin[0,C]sithatF=intfas}
F_t = \int_0^t f(s)\ud s,\;\;\; t\geq 0.
\end{equation}
For each $K>0$ define
\begin{equation}\label{eq:defofTKftdsahg}
 T_K f(t) := \mathbb E \bigl(f(t)\big|\mathcal F_{t_k}\bigr),\;\;\;\;\;\;t_{k-1}<t\leq t_k,\;\; k=1,\ldots, K.
\end{equation}
Then it is sufficient to show that $T_K f$ converges to $f$ in $L^1([0, T]\times \Omega, \lambda|_{[0, T]} \otimes \mathbb P)$ (where $\lambda$ is the Lebesgue measure on $\mathbb R_+$) as
\begin{align*}
 \mathbb E \sum_{k=1}^K  \bigl|F(t_k)- \mathbb E (F(t_k)| \mathcal F_{t_{k-1}})\bigr| &= \mathbb E \sum_{k=1}^K  \Bigl| \int_{t_{k-1}}^{t_k} f(t) - \mathbb E \bigl(f(t)\big|\mathcal F_{t_k}\bigr) \ud t\Bigr|\\
 &\leq \mathbb E \sum_{k=1}^K   \int_{t_{k-1}}^{t_k} \bigl|f(t) - \mathbb E \bigl(f(t)\big|\mathcal F_{t_k}\bigr)\bigr| \ud t\\
 &= \mathbb E \int_{0}^T \bigl|f(t) -  \mathbb E \bigl(f(t)\big|\mathcal F_{t_k}\bigr)\bigr| \ud t \\
 & =  \|f -T_K f\|_{L^1([0, T]\times \Omega, \lambda|_{[0, T]} \otimes \mathbb P)}.
\end{align*}
Note that $T_K$ is a bounded linear operator on $L^1([0, T]\times \Omega, \lambda|_{[0, T]} \otimes \mathbb P)$ of norm $1$ as for any $g\in L^1([0, T]\times \Omega, \lambda|_{[0, T]} \otimes \mathbb P)$  by the Fubini theorem and by the fact that a conditional expectation is a contraction on $L^1(\Omega)$
\begin{align*}
  \|T_Kg\|_{L^1([0, T]\times \Omega, \lambda|_{[0, T]} \otimes \mathbb P)} &= \mathbb E \int_0^T \bigl|\mathbb E \bigl(g(t)\big|\mathcal F_{t_k}\bigr)\bigr| \ud t =  \int_0^T \mathbb E \bigl|\mathbb E \bigl(g(t)\big|\mathcal F_{t_k}\bigr)\bigr| \ud t\\
 &\leq \int_0^T \mathbb E |g(t)| \ud t = \mathbb E \int_0^T|g(t)| \ud t = \|g\|_{L^1([0, T]\times \Omega, \lambda|_{[0, T]}}.
\end{align*}
Therefore by \cite[Lemma 9.4.7]{EMT04} it is sufficient to show that $T_k f_n \to f_n$ for a converging to $f$ sequence $(f_n)_{n\geq 1}$. To this end we need to set $f_n(\cdot) := f(\cdot - 1/n)$ on $[1/n, \infty]$ and $f(\cdot) = 0$ on $[0, 1/n]$. Then $f_n$ converges to $f$ in $L^1([0, T]\times \Omega, \lambda|_{[0, T]} \otimes \mathbb P)$ as $n\to \infty$ by \cite[Lemma 9.4.7]{EMT04} (translation operators have a strong limit, namely the identity operator, see e.g.\ \cite[Theorem 1]{AP11}) and by the dominated convergence theorem, while $T_Kf_n = f_n$ for $K \geq n$ as in this case $f_n(t)$ is $\mathcal F_{t-1/K}$-measurable. Therefore the desired follows.

Let us show the case $p\neq1$. In this case it is sufficient to notice that for any $K\geq 1$ a.s.\
\begin{align*}
\sum_{k=1}^K  \bigl|F(t_k)- \mathbb E (F(t_k)| \mathcal F_{t_{k-1}})\bigr|  &= \sum_{k=1}^K  \bigl|F(t_k) - F(t_{k-1})- \mathbb E (F(t_k) - F(t_{k-1})| \mathcal F_{t_{k-1}})\bigr| \\
&\stackrel{(i)}\leq \sum_{k=1}^K  |F(t_k) - F(t_{k-1})|+| \mathbb E (F(t_k) - F(t_{k-1})| \mathcal F_{t_{k-1}})|\\
&\stackrel{(ii)}= \sum_{k=1}^K  F(t_k) - F(t_{k-1})+ \mathbb E (F(t_k) - F(t_{k-1})| \mathcal F_{t_{k-1}})\\
&\stackrel{(iii)}= \int_0^T f(t) \ud t + \int_0^T T_K f(t) \ud t \stackrel{(iv)}\leq 2CT,
\end{align*}
where $f$ is defined by \eqref{eq:thereexistsfwvalin[0,C]sithatF=intfas}, $T_K$ is defined by \eqref{eq:defofTKftdsahg}, $(i)$ follows from a triangle inequality, $(ii)$ follows from the fact that $F$ is nondecreasing (and the same holds for the conditional expectations), $(iii)$ follows from the definition of $f$ and $T_kf$, and $(iv)$ holds by the fact that $f\in [0, C]$ a.e.\ on $\mathbb R_+ \times \Omega$, by the definition \eqref{eq:thereexistsfwvalin[0,C]sithatF=intfas} of $T_K$, and the fact that a conditional expectation is a contraction on  $L^{\infty}$ (so $T_Kf\in [0, C]$ a.e.\ on $\mathbb R_+ \times \Omega$). Therefore \eqref{eq:lemLipshfuncapprox} for $p\neq 1$ follows by the dominated convergence theorem and the case $p=1$.
\end{proof}

\begin{proof}[Proof of Theorem \ref{thm:mumuCoxcomparable}]
The ``if'' part follows from Proposition \ref{prop:XUMDiffdecouplforPois}. Let us show the ``only if'' part. As $F$ is elementary predictable we may assume that $J$ is finite, $J= \{1,\ldots,n\}$, $\mathcal J$ is generated by all atoms, $X$ is finite dimensional, and $F$ has the following form
\begin{equation}\label{eq:formofFforJsmall}
 F(t, \cdot, j) = \sum_{k=1}^K \mathbf 1_{[t_{k-1}, t_k)}(t) \xi_{k, j},\;\;\; t\geq 0,\;1\leq j\leq n,
\end{equation}
where $0\leq t_0 \leq \ldots \leq t_K$, and $\xi_{k, j}$ is elementary $X$-valued $\mathcal F_{t_{k-1}}$-measurable for any $k=1,\ldots,K$ and $1\leq j\leq n$. 

Let $\mu_{\rm Cox}$ be as constructed in Example \ref{ex:CoxprocessfiniteJ}. Then we need to show that
\begin{equation}\label{eq:mubarbyNexactform}
\begin{split}
\mathbb E &\sup_{t\geq 0} \Bigl\|\sum_{k=1}^K \sum_{j=1}^n\bar{\mu}([t_{k-1}\wedge t, t_k \wedge t) \times \{j\}) \xi_{k, j} \Bigr\|^p\\
&\eqsim_{p, X}\mathbb E \mathbb E_N \sup_{t\geq 0} \Bigl\|\sum_{k=1}^K \sum_{j=1}^n\widetilde N\bigl([\nu^j([0,t_{k-1}\wedge t)), \nu^j([0,t_k \wedge t))) \times \{j\}\bigr) \xi_{k, j} \Bigr\|^p,
\end{split}
\end{equation}
where $\nu_N$ is a compensator of $N$, $\widetilde N:= N - \nu_N$, $\mathbb E_N$ denotes expectation in $\Omega_N$ (i.e.\ the expectation taken for a fixed $\omega\in \Omega$, see Example \ref{ex:defofExiforrvxis}), and $\nu^j$ is a random measure on $\mathbb R_+$ of the form
\begin{equation}\label{eq:defofnujddsd}
\nu^j (A):= \nu(A\times \{j\}),\;\;\; A \in \mathcal B(\mathbb R_+),\;\; j=1,\ldots, n.
\end{equation}

In order to derive \eqref{eq:mubarbyNexactform} we will use the fact that any random measure is Poisson after a certain time-change (see \cite[Corollary 25.26]{Kal}) and the decoupling inequality \eqref{eq:decthmniceRHS}. The proof will be done in four steps.

{\em Step 1: $\nu([s, t) \times \{j\}) \leq t-s$, $\nu(\mathbb R_+ \times \{j\}) = \infty$ , $1\leq p\leq 2$.} First assume that a.s.
\begin{equation}\label{eq:nuisdombyLebesgue}
 \nu([s,t]\times J) \leq t-s,\;\;\;\; 0\leq s \leq t,
\end{equation}
that $ \nu(\mathbb R_+ \times \{j\}) = \infty$ a.s.\ for any $j=1,\ldots,n$, and that $1\leq p\leq 2$. By the fact that any martingale has a c\`adl\`ag version (see Subsection \ref{subsec:BSvmart}) and by adding knots to the mesh we may assume that $K$ is so big that (or the mesh is so small that)
\begin{equation*}
\mathbb E \max_{k=1}^K\Bigl\| \int_{[0, t_k]\times J} F \ud \bar\mu \Bigr\|^p \leq \mathbb E \sup_{t\geq 0}\Bigl\| \int_{[0, t]\times J} F \ud \bar\mu \Bigr\|^p \leq 2\mathbb E \max_{k=1}^K\Bigl\| \int_{[0, t_k]\times J} F \ud \bar\mu \Bigr\|^p,
\end{equation*}
so instead of \eqref{eq:CoxmeasiffUMD} it is sufficient to show that for $K$ big enough
\begin{equation}\label{eq:CoxmeasiffUMDp>1}
\mathbb E \max_{k=1}^K\Bigl\| \int_{[0, t_k]\times J} F \ud \bar\mu \Bigr\|^p \eqsim_{p, X} \mathbb E  \Bigl\| \int_{\mathbb R_+\times J} F \ud \bar \mu_{\rm Cox} \Bigr\|^p.
\end{equation}
By \cite[Corollary 25.26]{Kal} the random measure $\chi$ defined on $\mathbb R_+ \times \Omega$ by
\[
 \chi([0,s)\times \{j\}) := \mu([0, \tau_s^j)\times \{j\}),
 \;\;\; s\geq 0,\;\; 1\leq j \leq n,
\]
with
\begin{equation}\label{eq:defoftaujs}
\tau_s^j := \inf\{s\geq 0: \nu([0, s)\times \{j\}) \geq t\},\;\;\; s\geq 0,
\end{equation}
is a standard Poisson random measure with a compensator
\[
 \nu_{\chi}([0,s)\times \{j\}) = s,\;\;\; s\geq 0,\;\;1\leq j\leq n.
\]
Without loss of generality by an approximation argument we may assume that $K$ in \eqref{eq:formofFforJsmall} is so large so that there exists $T>0$ such that $t_0,\ldots,t_K$ in \eqref{eq:formofFforJsmall} are $0, \tfrac{T}{K},\ldots,\tfrac{T(k-1)}{K}, T$. Moreover, by considering a smaller mash for any $\delta>0$ we can assume that $K$ is so large that by Lemma \ref{lem:Lipshfuncapprox}, by predictability and continuity of the process $t\mapsto \nu([0,t))$, and by \eqref{eq:nuisdombyLebesgue}
\begin{equation}\label{eq:deltaforcondexpect}
 \mathbb E \max_{k=1}^K \sum_{j=1}^n |\nu([0,t_k)\times \{j\}) - \mathbb E (\nu([0,t_k)\times \{j\})| \mathcal F_{t_{k-1}})|<\delta.
\end{equation}
 Therefore the integral on the left-hand side of \eqref{eq:mubarbyNexactform} becomes
\[
 t\mapsto \sum_{k=1}^K \sum_{j=1}^n \bar{\chi}\Bigl(\bigl[\nu^j([0, t_{k-1} \wedge t)), \nu^j([0, t_{k} \wedge t))\bigr) \times \{j\}\Bigr) \xi_{k, j},
\]
where $\bar{\chi} = \chi-\nu_{\chi}$. As $\chi$ is a standard Poisson random measure, by \eqref{eq:nuisdombyLebesgue}, by adding some pieces of standard Poisson random measure within stopping times, and by using the fact that Poisson process is strong Markov and stationary, without loss of generality we may assume that there exists a standard Poisson random measure $\eta$ on $\mathbb R_+\times J$ with a compensator measure $\nu_{\eta} = \nu_{\chi}$ such that 
$$
\eta|_{\bigl[t_{k-1}, t_{k-1} +  \nu^j[t_{k-1}, t_{k})\bigr)} = \chi|_{\bigl[\nu^j([0, t_{k-1})), \nu^j([0, t_{k}))\bigr)},\;\;\;\; k=1,\ldots, K,
$$ 
and $\eta|_{[t_{k-1} +  \nu^j[t_{k-1}, t_{k}), t_k)}$ is copied from an independent from $\chi$ standard Poisson random measure.
Then the integral above becomes as follows
\begin{equation}\label{eq:Mqformhalohalolastline}
 M_t =  \sum_{k=1}^K \sum_{j=1}^n \bar{\eta}\Bigl(\bigl[t_{k-1},  t_{k-1}+ \nu^j[t_{k-1} \wedge t, t_{k} \wedge t)\bigr) \times \{j\}\Bigr) \xi_{k, j},\;\;\; t\geq 0,
\end{equation}
where $\xi_{k, j}$ is $\mathcal F_{t_{k-1}}\otimes \sigma(\eta|_{[0,t_{k-1}]})$-measurable and $\bar{\eta} := \eta - \nu_{\eta}$.
Let $L:\mathbb R_+ \times \Omega \to X$ be a process defined for every $t\geq 0$ by
\begin{multline}\label{eq:defofLapproxMq}
 L_t =  \sum_{k=1}^K \sum_{j=1}^n \bar{\eta}\Bigl(\bigl[t_{k-1} \wedge t, ( t_{k-1}
 +\mathbb E ( \nu^j[t_{k-1} , t_{k})| \mathcal F_{t_{k-1}}))\wedge t\bigr ) \times \{j\}\Bigr) \xi_{k, j}.
\end{multline}
Note that $L$ is a martingale with respect to an enlarged filtration $\mathbb F' = (\mathcal F'_t)_{t\geq 0}$ of the following form
\[
\mathcal F'_t = 
\begin{cases}
 \sigma(\eta|_{[0,t]}, \mathcal F_0),\;\;\;& 0\leq t<t_1,\\
\sigma(\eta|_{[0,t]}, \mathcal F'_{t_k}),\;\;\;& 1\leq k < K, \;\; t_{k} < t < t_{k+1},\\
\sigma(\eta|_{[0,t_k]}, \mathcal F_{t_k}),\;\;\; &1\leq k \leq K, \;\; t = t_k,\\
\mathcal F'_{t_K},\;\;\;& t>t_K.
\end{cases}
\]
$L$ is a martingale with respect to $\mathbb F'$, but it can be decomposed into two parts $L^1$ and $L^2$ which are martingales in different filtrations, in the following way. First we introduce a stopping time
\begin{equation}\label{eq:defofsigmajkrassiyane}
\sigma^j_k  := \tau^j_{\nu^j[0,t_{k-1})+ \mathbb E ( \nu^j[t_{k-1}, t_{k})| \mathcal F_{t_{k-1}})},
\end{equation}
where $\tau^j_s$ is as defined by \eqref{eq:defoftaujs}. Then let us define for any $t\geq 0$
\begin{align*}
L^1_t &:= \sum_{k=1}^K \sum_{j=1}^n \bar{\eta}\Bigl(\bigl[t_{k-1},  t_{k-1}+\mathbb E ( \nu^j[t_{k-1} , t_{k} )| \mathcal F_{t_{k-1}}) \wedge \nu^j[t_{k-1}\wedge t , t_{k} \wedge t)\bigr) \times \{j\}\Bigr)\xi_{k, j}\\
& = \sum_{k=1}^K \sum_{j=1}^n\bar{\mu}\bigl([t_{k-1}\wedge t, \sigma^j_k\wedge t) \times \{j\}\bigr) \xi_{k, j},
\end{align*}
which is a martingale with respect to the original filtration, and 
\[
L^2_t:= \sum_{k=1}^K \mathbf 1_{t\geq t_k}
 \sum_{j=1}^n \bar{\eta}\Bigl(\bigl[(t_{k-1} + \nu^j[t_{k-1}, t_{k}))\wedge t, ( t_{k-1}+\mathbb E ( \nu^j[t_{k-1}, t_{k} )| \mathcal F_{t_{k-1}}))\wedge t \bigr) \times \{j\}\Bigr)\mathbf 1_{\sigma^j_k\geq t_k} \xi_{k, j},
 \]
which is a martingale with respect to the enlarged filtration 
$$
\mathbb F'' = (\mathcal F_t'')_{t\geq 0} := \bigl( \mathcal F_{\infty} \otimes \sigma(\eta|_{A_{\eta} \cap [0, t] \times J})\bigr)_{t\geq 0},
$$ 
where $A_{\eta}:= \cup_{k=1}^K \cup_{j\in J} \bigl[t_{k-1},  t_{k-1}+ \nu^j[t_{k-1} \wedge t, t_{k} \wedge t)\bigr) \times \{j\}$ is a $\sigma$-field depending on $\Omega$ and $\otimes$ does not mean a direct product, see Subsection \ref{subsec:ConexponPSCondProbCondIndep}.
Note that $L^1$ and $L^2$ are martingales in different scales, so $L = L^1 + L^2$ not necessarily in general (unless $\nu((s, t] \times \{j\}) = t-s$), but $L_{\infty} = L^1_{\infty} + L^2_{\infty}$ and $\sup_{t\geq 0} \|L_t\| \leq \sup_{t\geq 0} \|L^1_t\| + \sup_{t\geq 0} \|L^2_t\|$ a.s.

Next, by Novikov's inequalities \eqref{eq:NovRV}, the fact that $X$ can be assumed finite dimensional, the fact that $F$ is uniformly bounded, the definition \eqref{eq:defofsigmajkrassiyane} of $\sigma^j_k$, and \eqref{eq:deltaforcondexpect}
 \begin{equation}\label{eq:L1approxintFwrtmubar}
  \begin{split}
    \mathbb E \max_{k=1}^K &\Bigl \| \int_{[0, t_k] \times J} F \ud \bar {\mu} - L^1_{t_k}\Bigr\|^p\\
 &= \mathbb E \max_{k=1}^K\Bigl \| \sum_{\ell=1}^k \sum_{j=1}^n\bar{\mu}([\sigma^j_{\ell}, t_{\ell}) \times \{j\}) \mathbf 1_{\sigma^j_{\ell}\leq t_{\ell}} \xi_{\ell, j}\Bigr\|^p\\
 &\lesssim_{p,F}\mathbb E\sum_{k=1}^K \sum_{j=1}^n\nu\bigl([\sigma^j_k, t_{k}) \times \{j\}\bigr) \mathbf 1_{\sigma^j_k\leq t_k}\\
 & \leq\mathbb E \sum_{k=1}^K \sum_{j=1}^n |\nu^j([0,t_k)) - \mathbb E (\nu^j([0,t_k))| \mathcal F_{t_{k-1}})|<\delta.
  \end{split}
 \end{equation}
On the other hand for a similar reason and the fact that $\nu_{\eta}(\cdot \times \{j\})$ is a standard Lebesgue measure on $\mathbb R_+$ for any $j=1,\ldots,n$
\begin{equation}\label{eq:L2issmall<delta}
\begin{split}
 \mathbb E \sup_{t\geq 0} \| L^2_{t}\|^p
 & = \mathbb E \sup_{t\geq 0} \mathbb E_{\eta}\Big\|\sum_{k=1}^K \mathbf 1_{t\geq t_k}
 \sum_{j=1}^n \bar{\eta}([t_{k-1} + \nu^j[t_{k-1}, t_{k}),  t_{k-1}\\
 &\quad\quad\quad\quad\quad\quad\quad+\mathbb E ( \nu^j[t_{k-1} , t_{k} )| \mathcal F_{t_{k-1}}) ) \times \{j\})\mathbf 1_{\sigma^j_k\geq t_k} \xi_{k, j} \Big\|^p\\
 &\stackrel{(*)}\lesssim_{p, F}\mathbb E \sum_{k=1}^K  \sum_{j=1}^n \nu_{\eta}([t_{k-1} + \nu^j[t_{k-1}, t_{k}),  t_{k-1}\\ &\quad\quad\quad\quad\quad\quad\quad+ \mathbb E ( \nu^j[t_{k-1}, t_{k})| \mathcal F_{t_{k-1}}) ) ) \times \{j\})\mathbf 1_{\sigma^j_k\geq t_k}\\
 &\leq \mathbb E \sum_{k=1}^K  \sum_{j=1}^n \Bigl| \mathbb E ( \nu^j[t_{k-1}, t_{k})| \mathcal F_{t_{k-1}}) ) ) -\nu^j[t_{k-1}, t_{k})\Bigr|\mathbf 1_{\sigma^j_k\geq t_k}\\
 &\leq  \mathbb E \sum_{k=1}^K \sum_{j=1}^n |\nu^j([0,t_k)) - \mathbb E (\nu^j([0,t_k))| \mathcal F_{t_{k-1}})|<\delta,
 \end{split}
\end{equation}
where $\mathbb E_{\eta}$ is defined by Example \ref{ex:defofExiforrvxis}, $(*)$ follows from the fact that $F$ is uniformly bounded, \eqref{eq:NovRV}, and the fact that the random measure constructed from
$$
\Bigl(\eta|_{\cup_{j\in J}[t_{k-1} + \nu^j[t_{k-1}, t_{k}),  t_{k})\times \{j\}}\Bigr)_{k=1}^K
$$ 
is standard Poisson with the compensator measure $\bigl(\nu_{\eta}|_{\cup_{j\in J}[t_{k-1} + \nu^j[t_{k-1}, t_{k}),  t_{k})\times \{j\}}\bigr)_{k=1}^K$.
As we can choose $K$ big enough (and $\delta$ small enough), it is sufficient to show that
\[
\mathbb E \max_{k=1}^K \|L^1_{t_k}\|^p \eqsim_{p, X} \mathbb E  \Bigr\| \int_{\mathbb R_+\times J} F \ud \bar{\mu}_{\rm Cox} \Bigl\|^p.
\]
To this end first notice that analogously to \eqref{eq:L2issmall<delta}
\begin{equation}\label{eq:forwildetildeNlikeL2also<delta}
\begin{split}
&\mathbb E \mathbb E_{N}\sup_{t\geq 0} \Big\|\sum_{k=1}^K \mathbf 1_{t\geq t_k}
 \sum_{j=1}^n \widetilde N([t_{k-1} + \nu^j[t_{k-1}\wedge t, t_{k}\wedge t),  t_{k-1}\\
 &\quad\quad\quad\quad\quad\quad\quad+\mathbb E ( \nu^j[t_{k-1} \wedge t, t_{k} \wedge t)| \mathcal F_{t_{k-1}}) ) \times \{j\})\mathbf 1_{\sigma^j_k\geq t_k} \xi_{k, j} \Big\|^p\lesssim_{p,F}\delta,
 \end{split}
\end{equation}
where $N$ is defined as in \eqref{eq:mubarbyNexactform}.
Next note that by Theorem \ref{thm:intromccnnll} (see also the proof of Proposition \ref{prop:XUMDiffdecouplforPois}), \eqref{eq:forwildetildeNlikeL2also<delta}, and the fact that $\mathbb E ( \nu^j[t_{k-1} \wedge t, t_{k} \wedge t)| \mathcal F_{t_{k-1}})$ is $ \mathcal F'_{t_{k-1}}$-measurable for any $t\geq 0$
\begin{align*}
\mathbb E  \max_{k=1}^K \|L^1_{t_k}\|^p &=\mathbb E \max_{k=1}^K\mathbb E_{\eta}  \Big\| \sum_{\ell=1}^k \sum_{j=1}^n \bar{\eta}([t_{\ell-1},  t_{\ell-1}+\mathbb E ( \nu^j[t_{\ell-1} , t_{k} )| \mathcal F_{t_{\ell-1}}) ) \times \{j\})\mathbf 1_{A_{\ell, j}} \xi_{\ell, j}\\
&\quad\quad\quad\quad + \sum_{\ell=1}^k \sum_{j=1}^n \bar{\eta}([t_{\ell-1},  t_{\ell}) \times \{j\})\mathbf 1_{\overline{A_{\ell, j}}} \xi_{\ell, j} \Big\|^p\\
&\eqsim_{p, X}\mathbb E\mathbb E_{N}  \max_{k=1}^K \Big\| \sum_{\ell=1}^k \sum_{j=1}^n \widetilde N([t_{\ell-1},  t_{\ell-1}+\mathbb E ( \nu^j[t_{\ell-1} , t_{k} )| \mathcal F_{t_{\ell-1}}) ) \times \{j\})\mathbf 1_{A_{\ell, j}} \xi_{\ell, j}\\
&\quad\quad\quad\quad + \sum_{\ell=1}^k \sum_{j=1}^n \widetilde N([t_{\ell-1},  t_{\ell}) \times \{j\})\mathbf 1_{\overline{A_{\ell, j}}} \xi_{\ell, j} \Big\|^p\\
&\stackrel{(*)} \eqsim_{p}\mathbb E \mathbb E_{N}  \Big\| \sum_{k=1}^K \sum_{j=1}^n \widetilde N ([t_{k-1},  t_{k-1}+\mathbb E ( \nu^j[t_{k-1} , t_{k} )| \mathcal F_{t_{k-1}}) ) \times \{j\})\mathbf 1_{A_{k, j}} \xi_{k, j}\\
&\quad\quad\quad\quad + \sum_{k=1}^K \sum_{j=1}^n \widetilde N ([t_{k-1},  t_{k}) \times \{j\})\mathbf 1_{\overline{A_{k, j}}} \xi_{k, j} \Big\|^p\\
&\stackrel{(**)}\eqsim_{\delta} \mathbb E \mathbb E_N  \Big\| \sum_{k=1}^K \sum_{j=1}^n \widetilde N\Bigl(\bigl[t_{k-1},  t_{k-1}+ \nu^j[t_{k-1} , t_{k} )\bigr) \times \{j\}\Bigr) \xi_{k, j} \Big\|^p\\
&= \mathbb E  \Bigl\| \int_{\mathbb R_+\times J} F \ud \bar \mu_{\rm Cox} \Bigr\|^p,
\end{align*}
where $A_{k,j} := \{\mathbb E ( \nu^j[t_{k-1}, t_{k} )| \mathcal F_{t_{k-1}}) \leq  \nu^j[t_{k-1} , t_{k} )\} \subset \Omega$, $(*)$ holds by Lemma \ref{lem:MhasIIthemathbbEsupphiMeqsimphiEphiM}, and  $(**)$ holds for $\delta$ small enough by \eqref{eq:deltaforcondexpect} e.g.\ analogously \eqref{eq:L1approxintFwrtmubar}.  Therefore \eqref{eq:CoxmeasiffUMDp>1}, and hence \eqref{eq:CoxmeasiffUMD}, follows. This terminates the proof.

{\em Step 2: $\nu([s, t) \times \{j\}) \leq t-s$, $\nu(\mathbb R_+ \times \{j\}) = \infty$ , general $1\leq p < \infty$.} In the case of a general $1\leq p<\infty$ we will have exactly the same proof as in Step 1, but with applying more complicated Novikov inequalities \eqref{eq:NovRV} for the case $p>2$.

{\em Step 3: $\nu([s, t) \times \{j\}) <\infty$, $\nu(\mathbb R_+ \times \{j\}) = \infty$ , general $1\leq p < \infty$.} Assume now that $\nu$ is infinite but finite on finite intervals. Then by a standard time change argument (see \cite[Theorems 10.27 and 10.28]{Jac79} or \cite[Subsection 5.5]{DY17}) we can assume that a.s.\
\begin{equation*}
 \nu([s,t]\times J) \leq t-s,\;\;\;\; 0\leq s \leq t,
\end{equation*}
which was considered in Step 2.

{\em Step 4: $\nu$ is general, $1\leq p < \infty$.} If we have a general measure $\nu$, then we make the following two tricks. First, instead of considering $\mu$, we consider $\mu^m :=\mu|_{A_m}$, where $(A_m)_{m\geq 1}$ is an increasing family of elements of $\widetilde {\mathcal P}$ such that $\cup_{m} A_m = \mathbb R_+\times \Omega \times J$ and $\mathbb E \mu(A_m)<\infty$ for any $m\geq 1$ (such a family exists as $\mu$ is $\widetilde {\mathcal P}$-$\sigma$-finite). Note that by Step 3 for any $m\geq 1$ we have that
 \begin{equation}\label{eq:CoxmeasiffUMDapproxbyAm}
\mathbb E \sup_{t\geq 0} \Bigl\| \int_{[0,t]\times J} F \ud \bar\mu^m \Bigr\|^p \eqsim_{p, X} \mathbb E  \Bigl\| \int_{\mathbb R_+\times J} F \ud \bar \mu^m_{\rm Cox} \Bigr\|^p.
\end{equation}
Indeed, though $\mu^m$ is finite a.s., we can add to $\mu^m$ another independent Poisson random measure $\eps \zeta$, $\eps>0$, where $\zeta$ is a standard Poisson random measure with a~compensator $\nu_{\zeta}$ satisfying $\nu_{\zeta}((s, t]\times \{j\}) = t-s$ for all $0\leq s\leq t$ and $j\in J$. Then by Step 3 we have that
 \begin{equation*}
\mathbb E \sup_{t\geq 0} \Bigl\| \int_{[0,t]\times J} F \ud (\bar\mu^m + \eps\bar{\zeta}) \Bigr\|^p \eqsim_{p, X} \mathbb E  \Bigl\| \int_{\mathbb R_+\times J} F \ud ( \bar\mu^m_{\rm Cox} + \eps\bar{\zeta}_{\rm Cox}) \Bigr\|^p,
\end{equation*}
and \eqref{eq:CoxmeasiffUMDapproxbyAm} follows by letting $\eps\to 0$, by a triangle inequality, and by the fact that
\[
 \mathbb E \sup_{t\geq 0} \Bigl\| \int_{[0,t]\times J} F \ud \eps\bar{\zeta} \Bigr\|^p \eqsim_{p, X} \mathbb E  \Bigl\| \int_{\mathbb R_+\times J} F \ud\eps\bar{\zeta}_{\rm Cox} \Bigr\|^p \lesssim_{F,X,p} \eps^p.
\]

Now notice that by Burkholder-Davis-Gundy inequalities \cite[Subsection 7.2]{Y18BDG}, by $\gamma$-dominated convergence \cite[Theorem 9.4.2]{HNVW2}, and by monotone convergence theorem (see Subsection \ref{subsec:gammanorm} for the definition of a $\gamma$-norm)
\begin{align*}
  \mathbb E \sup_{t\geq 0} \Bigl\| \int_{[0,t]\times J} F \ud \bar\mu^m - \int_{[0,t]\times J} F \ud \bar\mu \Bigr\|^p &= \mathbb E \sup_{t\geq 0} \Bigl\| \int_{[0,t]\times J} F \mathbf 1_{\mathbb R_+ \times \Omega \times J \setminus A_m} \ud \bar\mu \Bigr\|^p \\
  &\eqsim_{p, X} \mathbb E \|F \mathbf 1_{\mathbb R_+ \times \Omega \times J \setminus A_m}\|^p_{\gamma(L^2(\mathbb R_+ \times J; \mu), X)}\\
  &\to 0,\;\;\; m\to \infty
\end{align*}
(see also Section \ref{sec:appmartapprox}),
and for the same reason and the fact that we can set $\mu^m_{\rm Cox}$ to be $\mu_{\rm Cox}|_{A_m}$ (as they are equidistributed)
\[
 \mathbb E  \Bigl\| \int_{\mathbb R_+\times J} F \ud \bar \mu^m_{\rm Cox} - \int_{\mathbb R_+\times J} F \ud \bar \mu_{\rm Cox} \Bigr\|^p \to 0,\;\;\; m\to \infty.
\]
Thus \eqref{eq:CoxmeasiffUMD} follows as a limit of \eqref{eq:CoxmeasiffUMDapproxbyAm}.
\end{proof}

\begin{remark}\label{rem:intFdmuisamaringeiththesameloccharenlfi}
Note that $\int F \ud \bar{\mu}_{\rm Cox}$ is a decoupled tangent martingale to $\int F \ud \bar{\mu}$. Indeed, let $(\Omega, \mathcal F, \mathbb P)$ and $\mathbb F$ be the original probability space and filtration, let $(\overline{\Omega},\overline{ \mathcal F},\overline{ \mathbb P})$ be the extended by $\mu_{\rm Cox}$ probability space, and let $\overline{ \mathbb F} = (\overline{ \mathcal F}_t)_{t\geq 0}$ be such that $\overline{ \mathcal F}_t := \sigma(\mathcal F_t, {\mu}_{\rm Cox}|_{[0, t]})$ for any $t\geq 0$. One needs to show that $M = \int F \ud \bar{\mu}$ is an $\overline{ \mathbb F}$-martingale with the same local characteristics $(0, \nu^M)$ with $\nu^M$ defined by 
\begin{equation}\label{eq:nuMforremthatFmuCoxisdecpocw}
\nu^M(I\times B) = \int_{I} \mathbf 1_{B\setminus \{0\}}(F) \ud \nu,\;\;\; I\subset \mathbb R_+, \;\;\; B\subset X \; \text{Borel}.
\end{equation}
As $F$ is elementary predictable (but the same can be proven for any strongly $\widetilde {\mathcal P}$-measurable $F$ via exploiting an approximation argument, see Proposition \ref{prop:defofstochintwrtranmcaseodFoffinitemuintds} and Definition \ref{def:Fisintagrablewrtbarmugendef}) and as $\mathcal J$ is countably generated, there exists an increasing family of finite $\sigma$-algebras $(\mathcal J^m)_{m\geq 1}$ on $J$ such that $F$ is $\mathcal B(\mathbb R_+) \otimes \mathcal J^m \otimes \mathcal F$-measurable for any $m\geq 1$ and $\mathcal J = \sigma((\mathcal J^m)_{m\geq 1})$. Let $\overline{ \mathbb F}^m = (\overline{ \mathcal F}^m_t)_{t\geq 0}$ be such that $\overline{ \mathcal F}^m_t := \sigma(\mathcal F_t, {\mu}_{\rm Cox}|_{[0, t]\times \mathcal J^m})$ for any $t\geq 0$, where $\sigma({\mu}_{\rm Cox}|_{[0, t]\times \mathcal J^m})$ means that we are considering $\sigma$-algebras generated by processes $t\mapsto {\mu}_{\rm Cox}([0, t]\times A)$ for all $A\in \mathcal J^m$ (by approximating $F$ as it was done in Step 4 of the proof of Theorem \ref{thm:mumuCoxcomparable} we may assume that  $\mathbb E{\mu}_{\rm Cox}(\mathbb R_+\times J) = \mathbb E \nu(\mathbb R_+\times J) = \mathbb E{\mu}(\mathbb R_+\times J)<\infty$), and let us first show that $M$ is an $\overline{ \mathbb F}^m$-martingale. Let $(A_n^m)_{n=1}^{N_m}\subset J$ be the partition generating $\mathcal J^m$ and let $\tau^n_s = \inf \{t\geq 0:\nu([0, t]\times A_n^m) \geq s\}$. Then thanks to Example \ref{ex:CoxprocessfiniteJ} and the fact that the distribution of a Cox process is uniquely determined by its compensator we have that there exist independent standard Poisson processes $(N^{n})_{n=1}^{N_m}$ which are also independent of $\mathcal F$ such that $ {\mu}_{\rm Cox}([0, \tau^n_s]\times A) = N^n_s$. Fix any $t\geq r\geq 0$.
Let $\mathcal N^m := \sigma((N^n)_{n=1}^{N_m})$. Then
\begin{equation}\label{eq:Mt-MrconexpdirmuXpcsa}
\mathbb E(M_t-M_r| \overline{ \mathcal F}_r^m)\stackrel{(*)}= \mathbb E \Bigl(\mathbb E(M_t-M_r|  \mathcal F_r\otimes \mathcal N^m)\Big|\overline{ \mathcal F}_r^m\Bigr) \stackrel{(**)}= 0,
\end{equation}
where $(*)$ follows from the fact that $\overline{ \mathcal F}_r^m \subset \mathcal F_r\otimes \mathcal N^m$, and $(**)$ holds as $M_t-M_r$ is independent of $\mathcal N^m$ and as $\mathbb E(M_t-M_r|  \mathcal F_r) = 0$. Now it suffices to notice that $\mathbb E(M_t-M_r| \overline{ \mathcal F}_r) = \lim_{m\to \infty}\mathbb E(M_t-M_r| \overline{ \mathcal F}_r^m) =0$ due to the martingale convergence theorem \cite[Theorem 3.3.2]{HNVW1}. The fact that $M$ keeps the same local characteristics $(0, \nu^M)$ can be shown by \eqref{eq:nuMforremthatFmuCoxisdecpocw} via proving that $\mu$ has the same compensator after enlarging the probability space and filtration. Assume that $\mu$ has a different $\overline{ \mathbb F}$-compensator $\tilde {\nu}$. Then for any $I\subset \mathbb R_+$ and $A \subset J$ we have that $t\mapsto \mu([0, t]\cap I\times J) - \nu([0, t]\cap I\times J)$ is a martingale (as an integral w.r.t.\ $\bar{\mu}$ by the first part of the remark) and $t\mapsto \mu([0, t]\cap I\times J) - \tilde {\nu}([0, t]\cap I\times J) $ is a martingale (by the definition of a compensator), so a predictable finite variation process $t\mapsto \nu([0, t]\cap I\times J)- \tilde{\nu}([0, t]\cap I\times J) $ is a martingale thus it is zero by \cite[Lemma 25.11]{Kal} and hence $\tilde \nu = \nu$ a.s.
\end{remark}

\begin{remark}\label{rem:CoxidlikePoisson}
Inequality \eqref{eq:CoxmeasiffUMD} has an equivalent formulation in terms of Poisson random measures. Indeed, as $\nu$ is $\widetilde P$-$\sigma$-finite, it is a.s.\ $\sigma$-finite, so by Subsection \ref{subsec:PoissRMprelim} a.s.\ there exists a Poisson random measure $N_{\nu}$, which distribution by the definition coincides with the distribution of the Cox process directed by $\nu$, so we have that for a.e.\ $\omega\in \Omega$
the distributions of 
$\int_{\mathbb R_+\times J} F \ud \bar \mu_{\rm Cox}$ and $ \int_{\mathbb R_+\times J} F \ud \widetilde N_{\nu}$ coincide,
and in particular by \eqref{eq:CoxmeasiffUMD}
\begin{equation}\label{eq:Coxispoissonforfixedomwfag}
\mathbb E \sup_{t\geq 0} \Bigl\| \int_{[0,t]\times J} F \ud \bar\mu \Bigr\|^p \eqsim_{p, X} \mathbb E  \Bigl\|  \int_{\mathbb R_+\times J} F \ud \widetilde N_{\nu} \Bigr\|^p,
\end{equation}
where the right-hand side is very much in the spirit of $\gamma$-radonifying operators (see Subsection \ref{subsec:gammanorm}; see also \cite{ApRi,RieGaa}), but here instead of considering Gaussian random variables we have Poisson random measures. Note that this parallel with $\gamma$-radonifying operators might mislead as e.g.\ one has that $ \bigl\|  \int_{\mathbb R_+\times J} F \ud \widetilde N_{\nu} \bigr\|_{L^p(\Omega; X)}$ and $ \bigl\|  \int_{\mathbb R_+\times J} F \ud \widetilde N_{\nu} \bigr\|_{L^q(\Omega; X)}$ are incomparable for different $p$ and $q$ (a classical example would be $\alpha$-stable processes which can be represented as such integrals) which is of course not the case for Wiener integrals, see e.g.\ \cite[Proposition 6.3.1]{HNVW2}.

Assume that for a Banach space $X$, for some $1\leq p<\infty$, and for any Poisson random measure $N$ on $\mathbb R_+ \times J$ with a compensator $\nu$ there exists a function $R_{\nu}$ acting on deterministic $X$-valued functions on $\mathbb R_+\times J$ such that $\mathbb E \bigl\|  \int_{\mathbb R_+\times J} F \ud \widetilde N \bigr\|^p \lesssim_{p, X} R_{\nu}(F)$ (resp.\ $\mathbb E \bigl\|  \int_{\mathbb R_+\times J} F \ud \widetilde N \bigr\|^p \gtrsim_{p, X} R_{\nu}(F)$.) In this case, if $X$ is a UMD Banach space, we can conclude from \eqref{eq:Coxispoissonforfixedomwfag} that
\[
\mathbb E \sup_{t\geq 0} \Bigl\| \int_{[0,t]\times J} F \ud \bar\mu \Bigr\|^p \lesssim_{p, X} \mathbb E R_{\nu} (F)\;\; \left(\text{resp. }\mathbb E \sup_{t\geq 0} \Bigl\| \int_{[0,t]\times J} F \ud \bar\mu \Bigr\|^p \gtrsim_{p, X} \mathbb E R_{\nu} (F)\right).
\]
An example of such $R_{\nu}$ in the case of martingale type $r$ spaces was presented in \cite{ZBL19,Hau11}:
\[
R_{\nu}(F):=
\begin{cases}
\min\left\{\int_{\mathbb R_+ \times J} \|F\|^p \ud \nu, \Bigl(\int_{\mathbb R_+ \times J} \|F\|^r \ud \nu \Bigr)^{\frac{p}{r}}\right\},\;\;\; 1\leq p\leq r,\\
\int_{\mathbb R_+ \times J} \|F\|^p \ud \nu+ \Bigl(\int_{\mathbb R_+ \times J} \|F\|^r \ud \nu \Bigr)^{\frac{p}{r}},\;\;\; p>r
\end{cases}
\]
(though  \cite{ZBL19,Hau11} work only with Poisson random measures with the compensator of the form $\lambda_{\mathbb R_+} \otimes \nu_0$, in the case of deterministic $F$ these estimates can be generalized to a general compensated Poisson measure). Therefore \eqref{eq:Coxispoissonforfixedomwfag} allows us to extend \cite{ZBL19,Hau11} to stochastic integrals with respect to general random measures for martingale type $r$ Banach spaces with the UMD property.
\end{remark}

\subsection{Purely discontinuous quasi-left continuous martingales}\label{subsec:dectanPDQLC}

The present subsection is devoted to the purely discontinuous quasi-left continuous case. Our goal is to show that any $X$-valued purely discontinuous quasi-left continuous martingale $M$ coincides with $\int x \ud \bar{\mu}^M(\cdot, x)$ (where $\mu^M$ is defined by \eqref{eq:defofmuM}), so that we can reduce this case to the one considered in Subsection \ref{subsec:RMstochintinBanspaxcewithCOx}. To this end we need to define an integral of a {\em general predictable} (i.e.\ not necessarily elementary predictable) process with respect to a random measure (as $(t, x) \mapsto x$, $t\geq 0$, $x\in X$ is {\em not} elementary predictable). Let us start with the following proposition.

\begin{proposition}\label{prop:defofstochintwrtranmcaseodFoffinitemuintds}
Let $(J, \mathcal J)$ be a measurable space, $\mu$ be an integer-valued optional random measure over $\mathbb R_+ \times J$, $\nu$ be its compensator, $\bar{\mu} := \mu-\nu$. Let $X$ be a Banach space. Let $F:\mathbb R_+ \times \Omega \times J \to X$ be strongly  $\widetilde {\mathcal P}$-measurable so  that $ \mathbb E \int_{\mathbb R_+ \times J}\|F\| \ud \mu<\infty$. Then 
\begin{equation}\label{eq:defofintFwermubarforFwithgoodcond}
M_t :=\int_{[0,t] \times J}F \ud \bar{\mu} = \int_{[0,t] \times J}F \ud {\mu} - \int_{[0,t] \times J}F \ud {\nu},\;\;\; t\geq 0,
\end{equation}
is well defined and is a martingale. Moreover,
\begin{equation}\label{eq:ineqforsupnormforgoosmart}
\mathbb E \sup_{t\geq 0} \|M_t\| \leq 2\mathbb E \int_{\mathbb R_+ \times J}\|F\| \ud \mu.
\end{equation}
\end{proposition}

\begin{proof}
First of all note that $M$ is well defined by formula \eqref{eq:defofintFwermubarforFwithgoodcond} as by Fubini's theorem $F$ is a.s.\ $\mathcal B(\mathbb R_+) \otimes \mathcal J$-measurable and a.s.\ integrable with respect to $\mu$ and $\nu$ (the a.s.\ integrability w.r.t.\ $\nu$ holds since $\mathbb E \int_{\mathbb R_+ \times J}\|F\| \ud \nu = \mathbb E \int_{\mathbb R_+ \times J}\|F\| \ud \mu<\infty$, see \eqref{eq:defofcompofrandsamdsaer}). Also notice that \eqref{eq:ineqforsupnormforgoosmart} follows directly from \eqref{eq:defofintFwermubarforFwithgoodcond}, from a triangle inequality, and from \eqref{eq:defofcompofrandsamdsaer}.

As $F$ is strongly  $\widetilde {\mathcal P}$-measurable, as $F\in L^1(\Omega \times \mathbb R_+ \times J, \mathbb P \otimes \nu; X)$, and as step functions are dense in $L^1(\Omega \times \mathbb R_+ \times J,\widetilde {\mathcal P}, \mathbb P \otimes \nu; X)$ (here we choose the measure $\nu$ so that $\mathbb P \otimes \nu$ is a measure on $\widetilde {\mathcal P}$), there exist elementary predictable processes $(F^n)_{n\geq 1}$, $F^n:\mathbb R_+ \times \Omega \times J \to X$ for any $n\geq 1$, such that 
$$
\mathbb E \int_{\mathbb R_+ \times J}\|F^n\| \ud \mu = \mathbb E \int_{\mathbb R_+ \times J}\|F^n\| \ud \nu<\infty
$$ 
for any $n\geq 1$ and 
$$
\mathbb E \int_{\mathbb R_+ \times J}\|F-F^n\| \ud \mu = \mathbb E \int_{\mathbb R_+ \times J}\|F-F^n\| \ud \nu \to 0,\;\;\;\;n\to \infty.
$$ 
For each $n\geq 1$ let
\[
M^n_t := \int_{[0,t] \times J}F^n \ud \bar{\mu} = \int_{[0,t] \times J}F^n \ud {\mu} - \int_{[0,t] \times J}F^n \ud {\nu},\;\;\; t\geq 0.
\]
Then $M^n$ is a martingale. On the other hand by \eqref{eq:ineqforsupnormforgoosmart} we have that
\[
\mathbb E \sup_{t\geq 0} \|M_t - M^n_t\| \leq 2 \mathbb E \int_{\mathbb R_+ \times J}\|F-F^n\| \ud \mu \to 0,\;\;\; n\to \infty,
\]
and thus, as martingales form a closed subset of $L^1(\Omega; \mathcal D(\mathbb R_+, X))$ (see Definition \ref{def:defofskotokhodspaxc} and Theorem \ref{thm:strongLpmartforamaBanacoace}), $M$ is a martingale as well.
\end{proof}

Now we are ready to define an integral of a general process with respect to a random measure.

\begin{definition}\label{def:Fisintagrablewrtbarmugendef}
 Let $(J, \mathcal J)$ be a measurable space, $\mu$ be an integer-valued optional random measure on $\mathbb R_+ \times J$, $\nu$ be its compensator, $\bar{\mu} := \mu-\nu$. Let $X$ be a Banach space. A general strongly  $\widetilde {\mathcal P}$-measurable process $F:\mathbb R_+ \times \Omega \times J \to X$ is called to be {\em integrable} with respect to $\bar{\mu}$ if for any increasing family $(A_n)_{n\geq 1}$ of elements of $\widetilde {\mathcal P}$ satisfying $\mathbb E \int_{A_n} \|F\|\ud \mu <\infty$ for any $n\geq 1$ and $\cup_{n\geq 1}A_n = \mathbb R_+ \times \Omega \times J$, we have that the processes $\int_{A_n} F \ud 
\bar{\mu}$ converge in $L^1(\Omega; \mathcal D(\mathbb R_+, X))$ as $n\to \infty$.

$F$ is called to be {\em locally integrable} with respect to $\bar{\mu}$ if there exists an increasing sequence of stopping times $(\tau_n)_{n\geq 1}$ such that $\tau_n \to \infty$ as $n\to \infty$ and $A \mathbf 1_{[0,\tau_n]}$ is integrable with respect to $\bar{\mu}$ for any $n\geq 1$.
\end{definition}

 This definition is very much in the spirit of Lebesgue integration (a function $f$ can be integrable only if its restrictions $f|_{B_n}$ are integrable and the corresponding integrals converge as the restriction domains $B_n$'s blow up) or vector-valued stochastic integration with respect to a Brownian motion, see \cite{NVW}.
 
\begin{remark}\label{rem:intAnFudmubariswelldefoinedforgoodAn}
 Notice that in this case
 \begin{equation}\label{eq:intAnFudmubar=remark}
  t \mapsto \int_{A_n\cap[0, t]\times J} F \ud \bar{\mu} = \int_{A_n\cap[0, t]\times J} F \ud 
{\mu} - \int_{A_n\cap[0, t]\times J} F \ud {\nu},\;\;\; t\geq 0,
 \end{equation}
 are well defined martingales by Proposition \ref{prop:defofstochintwrtranmcaseodFoffinitemuintds}.
 \end{remark}

Now let us formulate the main theorem of the present subsection.

\begin{theorem}\label{thm:XisUMDiffMisintxwrtbarmuMvain}
 Let $X$ be a Banach space. Then $X$ is UMD if and only if for any purely discontinuous quasi-left continuous martingale $M:\mathbb R_+ \times \Omega \to X$ with $\mathbb E \sup_{t\geq 0} \|M_t\| <\infty$ we have that $x$ is integrable with respect to $\bar{\mu}^M$. If this is the case, then $M_t = \int_{[0, t]\times X} x \ud \bar{\mu}^M$ a.s.
\end{theorem}

\begin{proof}
 We will separately prove the ``if'' and the ``only if'' parts.
 
 {\em The ``only if'' part.} Let $X$ be a UMD Banach space, $M:\mathbb R_+ \times \Omega \to X$ be a purely discontinuous quasi-left continuous martingale with $\mathbb E \sup_{t\geq 0} \|M_t\| <\infty$, and let $(A_n)_{n\geq 1}$ be some increasing family from $\widetilde {\mathcal P}$ satisfying the properties from Definition~\ref{def:Fisintagrablewrtbarmugendef}. 
 For every $n\geq 1$ define an $X$-valued martingale 
 $$
 M^n_t := \int_{A_n \cap [0, t] \times X} x\ud \bar{\mu}^M,\;\;\;\; t\geq 0.
 $$
 We need to show that $(M^n)_{n\geq 1}$ converges in $L^1(\Omega; \mathcal D(\mathbb R_+, X))$ as $n\to \infty$.
 Note that by the definition of $M^n$ we have that $\Delta M^n_t(\omega) = \Delta M_t \mathbf 1_{A_n}(t, \omega, \Delta M_t(\omega))$ for a.e.\ $\omega \in \Omega$ for any $t\geq 0$, and thus, as $M$ and $M^n$ are purely discontinuous ($M^n$ is purely discontinuous as it is an integral with respect to a random measure, see e.g.\ \cite[$\S$II.1d]{JS} or \cite{Y17GMY,Y17MartDec}), by \cite[Subsection 6.1]{Y18BDG} we have that (see Subsection \ref{subsec:gammanorm} for the definition of a $\gamma$-norm)
 \[
   \infty > \mathbb E \sup_{t\geq 0} \|M_t\| \eqsim_{X} \mathbb E \bigl\|(\Delta M_t)_{t\geq 0}\bigr\|_{\gamma(\ell^2(\mathbb R_+), X)};
 \]
 also note that for any $n\geq 1$ by Remark \ref{rem:intAnFudmubariswelldefoinedforgoodAn} we have that $\mathbb E \sup_{t\geq 0} \|M^n_t\| < \infty$.
 Thus we have that
\[
  \mathbb E \sup_{t\geq 0} \|M_t - M^n_t\| \eqsim_{X} \mathbb E \bigl\|(\Delta M_t \mathbf 1_{\overline A_n}(t, \cdot, \Delta M_t(\cdot)))_{t\geq 0}\bigr\|_{\gamma(\ell^2(\mathbb R_+), X)},
\]
(here $\overline A_n$ means the complement of $A$ in $\mathbb R_+ \times \Omega \times J$). As $(\overline A_n)_{n\geq 1}$ is a vanishing family, by $\gamma$-dominated convergence \cite[Theorem 9.4.2]{HNVW2} and by the monotone convergence theorem we have that $ \mathbb E \sup_{t\geq 0} \|M_t - M^n_t\|  \to 0$ as $n\to \infty$. Consequently, $M^n$ converges to $M$ in $L^1(\Omega; \mathcal D(\mathbb R_+, X))$.

\smallskip

{\em The ``if'' part.} This part of the proof is based on the tricks from \cite[Subsection~4.4]{Y17GMY}. Assume that $X$ is not UMD. Our goal is to find such a purely discontinuous quasi-left continuous martingale $M$ and such an increasing family of sets $(A_n)_{n\geq 1}$ in $\widetilde {\mathcal P}$ that $\int_{A_n} \|x\| \ud \mu^M<\infty$ for any $n\geq 1$, but $\int x \mathbf 1_{A_n} \ud \bar{\mu}^M$ diverges in $L^1(\Omega; \mathcal D(\mathbb R_+, X))$. 

Due to the formula \cite[(1.7)]{Burk81} and \cite[Subsection~4.4]{Y17GMY} we have that $X$ is UMD if and only if
\begin{equation}\label{eq:XisUMDiffEf*==Eg*}
 \mathbb E \sup_n \|g_n\| \lesssim_X \mathbb E \sup_n \|f_n\|
\end{equation}
for any $X$-valued discrete martingales $f = (f_n)_{n\geq 0}$ and $g = (g_n)_{n\geq 0}$ with $g_0 = f_0=0$ and with
\begin{equation}\label{eq:notUMSDimpliesnoXwrtbarmudeffofgn}
 g_n - g_{n-1} = \eps_n(f_n-f_{n-1}),\;\;\; n\geq 1,
\end{equation}
for any fixed $\{0,1\}$-valued sequence $(\eps_n)_{n\geq 1}$. As $X$ is not UMD, \eqref{eq:XisUMDiffEf*==Eg*} does not hold. Therefore there exists a Paley-Walsh martingale $(f_n)_{n\geq 1}$ (see \cite{Burk86,HNVW1} why we can restrict to the Paley-Walsh case), i.e.\ a martingale $(f_n)_{n\geq 0}$ such that there exists a sequence $(r_n)_{n\geq 1}$ of Rademachers (see Definition \ref{def:ofRadRV}) so that $f_n - f_{n-1} = r_n \phi_n(r_1,\ldots,r_n)$ for some $\phi_n:\{-1,1\}^{n-1} \to X$ for every $n\geq 1$ and $f_0=0$, and a $\{0,1\}$-valued sequence $(\eps_n)_{n\geq 1}$, such that we have that $\mathbb E \sup_{n} \|f_n\| = 1$ and $ \mathbb E\sup_{n} \|g_n\| = \infty$ for $(g_n)_{n\geq 0}$ satisfying \eqref{eq:notUMSDimpliesnoXwrtbarmudeffofgn} and $g_0=0$.

Let $N^1$ and $N^2$ be two independent standard Poisson processes (Note that  in this case $N^1-N^2$ has a zero compensator and thus it is a martingale). Let $\tau_0=0$, and for each $n\geq 1$ define
\[
 \tau_n := \inf\{t\geq \tau_{n-1}: (N^1_t - N^1_{\tau_{n-1}}) \vee (N^2_t - N^2_{\tau_{n-1}}) \neq 0\},\;\;\; n\geq 1.
\]
Note that $\tau_n < \infty$ a.s.\ and that $\tau_n \to \infty$ a.s.\ as $n\to \infty$ (for the construction of the standard Poisson process we refer the reader e.g.\ to \cite{KingPois,Kal,Sato,ShPr2e} or in any other standard probability textbook), and that as Poisson processes have strong Markov property, 
$$
(\sigma_n)_{n\geq 1} = \bigl(N^1_{\tau_n} - N^2_{\tau_n}  - (N^1_{\tau_{n-1}} - N^2_{\tau_{n-1}})\bigr)_{n\geq 1}
$$
are i.i.d.\ random variables. Moreover, as $N^1$ and $N^2$ can have jumps of size $1$ a.s., $\sigma_n \in \{-1,1\}$ a.s., and as $N^1$ and $N^2$ are independent equidistributed, $(\sigma_n)_{n\geq 1}$ are independent Rademachers. In particular, for a simplicity of the proof we identify $(\sigma_n)_{n\geq 1}$ with $(r_n)_{n\geq 1}$.

Now let us consider a martingale $M:\mathbb R_+ \times \Omega \to X$ of the form $M = \int \Phi \ud (N^1 - N^2)$, where
\[
 \Phi(t) := \sum_{n=1}^{\infty}  \phi_n(r_1,\ldots,r_n) \mathbf 1_{(\tau_{n-1}, \tau_n]}(t),\;\;\; t\geq 0,
\]
and the integral is defined in the Riemann-Stieltjes way ($N^1-N^2$ is a.s.\ locally of finite variation).
First of all, $M$ is a local martingale since for any $n\geq 1$ we have that $\Phi^{\tau_n}$ is bounded and takes values in a finite-dimensional subspace of $X$, so the stochastic integral $M^{\tau_n} = \int \Phi^{\tau_n} \ud (N^1 - N^2)$ is well defined by a classical finite dimensional theory (see \cite{Kal,JS}). Moreover, as $M_{\tau_n} = f_n$ a.s.\ and as $M$ is a.s.\ a constant on $[\tau_{n-1}, \tau_n)$ for any $n\geq 1$, $\mathbb E M^* = \mathbb E \sup_{n} \|f_n\| = 1$, and thus by the dominated convergence theorem and by the fact that a conditional expectation is a contraction on $L^1(\Omega; X)$ (see \cite[Section 2.6]{HNVW1})
\[
 \mathbb E (M_t|\mathcal F_s) = \lim_{n\to \infty}  \mathbb E (M_{\tau_n \wedge t}|\mathcal F_s) = \lim_{n\to \infty}  M_{\tau_n \wedge s} = M_s,\;\;\; 0\leq s\leq t,
\]
so $M$ is a martingale with $\mathbb E M^* <\infty$.

Since $\mathbb E \sup_{n}\|g_n\| = \infty$, there exists a sequence $0 = k_1 < \ldots < k_m < \ldots$ such that $\mathbb E \sup_{k_m < n \leq k_{m+1}} \|g_n-g_{k_m}\|>1$ for each $m\geq 1$. Set $J = X$. Define $\mathbb R_+ \times \Omega \times J \supset A_0 =\varnothing$ and
\begin{equation}\label{eq:defofA2m-1A2msadas}
\begin{split}
 A_{2m-1}&:= A_{2m-2} \cup_{k_m < n \leq k_{m+1}, \eps_n = 1}( \tau_{n-1}, \tau_{n}] \in \widetilde {\mathcal P},\;\;\; m\geq 1, \\
 A_{2m}&:= A_{2m-1} \cup_{k_m < n \leq k_{m+1}, \eps_n = 0}( \tau_{n-1}, \tau_{n}] \in \widetilde {\mathcal P},\;\;\; m\geq 1.
\end{split}
\end{equation}
Then we have that $\cup_m A_m = \mathbb R_+ \times \Omega \times J$ and that $(A_m)_{m\geq 1}$ in increasing. Let us first show that $\mathbb E \int_{A_m} \|x\| \ud \mu^M<\infty$ for any $m\geq 1$
\begin{multline*}
 \mathbb E \int_{A_m} \|x\| \ud \mu^M \leq  \mathbb E \int_{[0, \tau_{k_{m+1}}]\times J} \|x\| \ud \mu^M \stackrel{(*)}= \mathbb E \sum_{0\leq s\leq \tau_{k_{m+1}}} \|\Delta M_s\|\\ 
 \stackrel{(**)}= \mathbb E \sum_{n=1}^{k_{m+1}} \|\Delta M_{\tau_n}\|
 = \mathbb E \sum_{n=1}^{k_{m+1}} \|f_n-f_{n-1}\| \leq 2k_{m+1} \mathbb E \sup_n \|f_n\| \leq 2k_{m+1},
\end{multline*}
where $(*)$ follows from the definition \eqref{eq:defofmuM} of $\mu^M$ and $(**)$ follows from the fact that $M$ is a.s.\ a constant on $[\tau_{n-1}, \tau_n)$ for any $n\geq 1$ and from the definition of $M$. Therefore $x \mathbf 1_{A_m}$ is integrable with respect to $\bar{\mu}^M$ by Proposition \ref{prop:defofstochintwrtranmcaseodFoffinitemuintds}.
Let us now show that $\int_{A_m} x \ud \bar {\mu}^M$ does not converge in $L^1(\Omega; \mathcal D(\mathbb R_+, X))$. It is sufficient to show that $\int_{A_{2m-1}} x \ud \bar {\mu}^M - \int_{A_{2m-2}} x \ud \bar {\mu}^M$ is big enough for any $m\geq 1$:
\begin{align*}
  \mathbb E \sup_{t\geq 0} \Bigl \| \int_{A_{2m-1} \cap[0, t]\times J} x &\ud \bar {\mu}^M - \int_{A_{2m-2}\cap[0, t]\times J} x \ud \bar {\mu}^M\Bigr\|\\
  &=  \mathbb E \sup_{t\geq 0} \Bigl \| \int_{A_{2m-1}\setminus A_{2m-2} \cap[0, t]\times J} x \ud \bar {\mu}^M \Bigr\|\\
  &\stackrel{(i)}= \mathbb E \sup_{t\geq 0} \Bigl \| \sum_{n=k_{m}+1}^{k_{m+1}} \mathbf 1_{\eps_n=1} (M_{\tau_n \wedge t} - M_{\tau_{n-1}\wedge t}) \Bigr\|\\
  &\stackrel{(ii)}= \mathbb E \sup_{n=k_{m}+1}^{k_{m+1}}  \Bigl \| \sum_{k=k_{m}+1}^{n} \mathbf 1_{\eps_n=1} (M_{\tau_n } - M_{\tau_{n-1}}) \Bigr\|\\
  &\stackrel{(iii)}= \mathbb E \sup_{n=k_{m}+1}^{k_{m+1}} \|g_n - g_{k_{m}}\| \geq 1,
\end{align*}
where in $(i)$ we used the definition \eqref{eq:defofmuM} of $\mu^M$, the fact that $M$ is pure jump, and \eqref{eq:defofA2m-1A2msadas}, $(ii)$ follows from the fact that $M$ is pure jump and has its jumps at $\{\tau_1, \ldots, \tau_n, \ldots\}$, and $(iii)$ follows from the definition of $M$ and $g$. Thus $x$ is not integrable with respect to $\bar{\mu}^M$.
\end{proof}

\begin{remark}
 Though in the sequel we will need only the ``only if'' part of the theorem above, the author decided to include the ``if'' statement with such a complicated proof as well because he found such a nontrivial characterization of the UMD property rather important and elegant.
\end{remark}

\begin{remark}
 Note that under the so-called {\em Radon-Nikod\'ym property} (many spaces have this property, e.g.\ reflexive spaces, see \cite[Definition 1.3.9]{HNVW1} and the references therein) in Theorem \ref{thm:XisUMDiffMisintxwrtbarmuMvain} there is no difference between considering martingales over $\mathbb R_+$ and over $[0, 1]$. Indeed, if $M:\mathbb R_+ \times \Omega\to X$ is such that $\mathbb E \sup_{t\geq 0} \|M_t\| <\infty$ and if $X$ has the Radon-Nikod\'ym property, then by \cite[Theorem 3.3.16]{HNVW1} $M$ has an $L^1$-limit $M_{\infty}$. Therefore we can define a time-change $t \mapsto 2\arctan t/{\pi}$, a time-changed filtration $\mathbb G = (\mathcal G_t)_{t\geq 0}$ with $\mathcal G_t = \mathcal F_{\tan(\pi t /2)}$ for $0\leq t<1$ and $\mathcal G_{t} = \mathcal F_{\infty}$ for $t\geq 1$, and a $\mathbb G$-martingale $\widetilde M$ with $\widetilde M_t = M_{\tan(\pi t /2)}$ for $0 \leq t< 1$ and $\widetilde M = M_{\infty}$ for $t\geq 1$. In this case we have that $x$ in integrable with respect to $\bar{\mu}^M$ if and only if it is integrable with respect $\bar{\mu}^{\widetilde M}$.
\end{remark}

Let us now state the corollaries we were looking for.

\begin{corollary}\label{cor:LstropboundforPDQLStangmartUMD}
 Let $X$ be a Banach space, $1\leq p< \infty$. Then $X$ has the UMD property if and only if for any pair $M, N:\mathbb R_+ \times \Omega \to X$ of tangent purely discontinuous quasi-left continuous martingales one has that
 \begin{equation}\label{eq:MNtangentPDQLCPnormineqiIFFFUMD}
    \mathbb E \sup_{t\geq 0} \|M_t\|^p \eqsim_{p,  X} \mathbb E \sup_{t\geq 0} \|N_t\|^p.
 \end{equation}
\end{corollary}

For the proof we will need the following lemma.

\begin{lemma}\label{lem:Fintwrtmu1iffwrtmu2UMDcase}
Let $X$ be a UMD Banach space, $(J, \mathcal J)$ be a measurable space, $\mu^1$ and $\mu^2$ be optional quasi-left continuous random measures on $\mathbb R_+ \times J$. Assume that $\mu^1$ and $\mu^2$ have the same compensator $\nu$. Let $\bar{\mu}^1:= \mu^1-\nu$ and $\bar{\mu}^2 := \mu^2-\nu$. Let $F:\mathbb R_+ \times \Omega \to X$ be a strongly $\widetilde {\mathcal P}$-measurable process. Then $F$ is stochastically integrable with respect to $\bar{\mu}^1$ if and only if it is stochastically integrable with respect to $\bar{\mu}^2$. Moreover, if this is the case, then $\int F \ud\bar{\mu}^1$ and $\int F \ud\bar{\mu}^2$ are tangent.
\end{lemma}

\begin{proof}
Let us first show the ``if and only if'' statement. Recall that stochastic integrability of a general $X$-valued predictable process with respect to a random measure was defined in Definition \ref{def:Fisintagrablewrtbarmugendef}. Assume that $F$ is stochastically integrable with respect to $\bar{\mu}^1$. Let us show that $F$ is stochastically integrable with respect to $\bar{\mu}^2$. Let $(A_n)_{n\geq 1}$ be an increasing sequence of sets in $\widetilde {\mathcal P}$ with $\cup_{n}A_n = \Omega\times \mathbb R_+ \times J$ satisfying $\mathbb E \int_{A_n} \|F\| \ud \mu^1<\infty$ for any $n\geq 1$. Then by \eqref{eq:defofcompofrandsamdsaer} we have that
\[
\mathbb E \int_{A_n} \|F\| \ud \mu^2 = \mathbb E \int_{A_n} \|F\| \ud \nu = \mathbb E \int_{A_n} \|F\| \ud \mu^1 <\infty,\;\;\; n\geq 1,
\] 
so $F\mathbf 1_{A_n}$ is stochastically integrable with respect to $\bar{\mu}^1$ and $\bar{\mu}^2$ by Proposition \ref{prop:defofstochintwrtranmcaseodFoffinitemuintds}. For each $n\geq 1$ let us set
\[
M^n_t := \int_{[0, t]\times J} F \mathbf 1_{A_n} \ud \bar{\mu}^1,\;\; N^n_t := \int_{[0, t]\times J} F \mathbf 1_{A_n} \ud \bar{\mu}^2,\;\;\;\; t\geq 0.
\]
As $F$ is stochastically integrable with respect to $\bar{\mu}^1$, $M^n$ is a Cauchy sequence in $L^1(\Omega; \mathcal D(\mathbb R_+, X))$. By Theorem \ref{thm:mumuCoxcomparable} and by the fact that $\mu^1$ and $\mu^2$ have the same compensator we have that for any $m \geq n\geq 1$
\begin{multline*}
\mathbb E \sup_{t\geq 0}\|N^m_t - N^n_t \|  = \mathbb E \sup_{t\geq 0} \Bigl\| \int_{[0, t]\times J} F \mathbf 1_{A_m \setminus A_n} \ud\bar{\mu}^2 \Bigr\|\\
\eqsim_{X} \mathbb E \sup_{t\geq 0} \Bigl\| \int_{[0, t]\times J} F \mathbf 1_{A_m \setminus A_n} \ud\bar{\mu}^1 \Bigr\| = \mathbb E \sup_{t\geq 0}\|M^m_t - M^n_t \|,
\end{multline*}
so $N^n$ is a Cauchy sequence in $L^1(\Omega; \mathcal D(\mathbb R_+, X))$, and thus $F$ is integrable with respect to $\bar{\mu}^2$ by Definition \ref{def:Fisintagrablewrtbarmugendef}.

Let us show that $M := \int F \ud\bar{\mu}^1$ and $N:= \int F \ud\bar{\mu}^2$ are tangent, i.e.\ as $M$ and $N$ are purely discontinuous, we need to show that the compensators $\nu^M$ and $\nu^N$ of $\mu^M$ and $\mu^N$ respectively coincide. Fix a predictable set $A \subset \mathbb R_+ \times \Omega \times X$. Then for any $t\geq 0$ we have that a.s.\
\[
\int_{[0, t]\times X} \mathbf 1_{A}\ud \mu^M = \sum_{0\leq s\leq t} \mathbf 1_{A}(s, \cdot, \Delta M_s) = \int_{[0, t]\times J} \mathbf 1_{A}(\cdot, \cdot, F) \ud \mu^1
\]
(the latter can be infinite),
so by the definition of a compensator we have that
\[
\mathbb E \int_{\mathbb R_+\times X} \mathbf 1_A \ud \nu^M = \mathbb E \int_{\mathbb R_+ \times J} \mathbf 1_A(\cdot, \cdot, F) \ud \nu.
\]
The same can be shown for $\nu^N$. Therefore, as $A$ was arbitrary predictable, $\nu^M$ and $\nu^N$ coincide, so $M$ and $N$ are tangent.
\end{proof}

\begin{proof}[Proof of Corollary \ref{cor:LstropboundforPDQLStangmartUMD}]
 The ``only if'' part follows from the fact that $M= \int x \ud \bar{\mu}^M$ and $N= \int x \ud \bar{\mu}^N$ by Theorem \ref{thm:XisUMDiffMisintxwrtbarmuMvain}, the fact that $\nu^M = \nu^N$ as $M$ and $N$ are tangent, and so Theorem~\ref{thm:mumuCoxcomparable}, Definition \ref{def:Fisintagrablewrtbarmugendef}, and Lemma \ref{lem:Fintwrtmu1iffwrtmu2UMDcase}
 \begin{align*}
 \mathbb E \sup_{t\geq 0} \|M_t\|^p  &= \mathbb E  \sup_{t\geq 0}\Bigl\| \int_{[0, t]\times X}x \ud \bar{\mu}^M \Bigr\|^p \eqsim_{p, X} \mathbb E  \sup_{t\geq 0}\Bigl\| \int_{[0, t]\times X}x \ud \bar{\mu}_{\rm Cox} \Bigr\|^p\\
&  \eqsim_{p, X} \mathbb E  \sup_{t\geq 0}\Bigl\| \int_{[0, t]\times X}x \ud \bar{\mu}^N \Bigr\|^p  =  \mathbb E \sup_{t\geq 0} \|N_t\|^p 
 \end{align*}
 (here we use the fact that both $\mu^M$ and $\mu^N$, as they have the same compensator, can also share the same Cox process $\mu_{\rm Cox}$). The ``if'' part follows directly from Theorem~\ref{thm:mumuCoxcomparable} since we can simply set $M := \int F \ud \bar{\mu}$ and $N := \int F \ud \bar{\mu}_{\rm Cox}$. Then $M$ and $N$ are tangent, so in this case \eqref{eq:MNtangentPDQLCPnormineqiIFFFUMD} in equivalent to \eqref{eq:CoxmeasiffUMD}, and thus $X$ is UMD by Theorem~\ref{thm:mumuCoxcomparable}.
\end{proof}

\begin{corollary}\label{thm:XUMDPDQLCdectangmart}
 Let $X$ be a UMD Banach space, $M:\mathbb R_+ \times \Omega \to X$ be a purely discontinuous quasi-left continuous local martingale. Let $\mu^M$ be a random measure defined by \eqref{eq:defofmuM} with a compensator $\nu^M$. Then 
 $$
 N_t:= \int_{[0, t]\times X} x \ud \bar{\mu}_{\rm Cox},\;\;\; t\geq 0,
 $$
 is a decoupled tangent local martingale to $M$, where $\mu_{\rm Cox}$ is a Cox process directed by $\nu^M$, $\bar{\mu}_{\rm Cox} = \mu_{\rm Cox} - \nu^M$.
\end{corollary}

\begin{proof}
By a stopping time argument presented in Theorem \ref{thm:MNtangent==>MtauNtautangent+ifNCIINtauCII} and by Remark \ref{rem:locmartislocallyL1DR+X} we may assume that $\mathbb E \sup_{t\geq 0} \|M_t\|<\infty$. Then the corollary follows immediately from Theorem~\ref{thm:mumuCoxcomparable}, \ref{thm:XisUMDiffMisintxwrtbarmuMvain} (so $x$ is integrable with respect to $ \bar{\mu}_{\rm Cox}$), and Definition \ref{def:dectanglocmartcontimecased}, where we have that after the probability space and filtration enlargement $M$ remains a martingale with the same local characteristics thanks to Remark \ref{rem:intFdmuisamaringeiththesameloccharenlfi} and the fact that the Borel $\sigma$-algebra $\mathcal B(X)$ is countably generated (recall that thanks to the Pettis measurability theorem \cite[Theorem 1.1.6]{HNVW1} we may assume that $X$ is separable).
\end{proof}

\begin{remark}\label{rem:NPDQLChascondindeindcfsgivennuM}
Note that $N$ constructed above has independent increments given $\nu^M$. Indeed, for almost any fixed $\nu^M$ we have that the corresponding Cox process $\mu_{\rm Cox}$ has a deterministic compensator, so it is a deterministically time changed standard Poisson random measure (see e.g.\ Example \ref{ex:CoxprocessfiniteJ}), and so for almost any fixed $\nu^M$ we have that $N(\nu^M) = \int x \ud \mu_{\rm Cox}$ has independent increments. Therefore the desired follows from Corollary \ref{cor:condindgivenRVsuffcond}.
 \end{remark}

\subsection{Purely discontinuous martingales with accessible jumps}\label{subsec:dectanPDwithAJ}
The present subsection is devoted to $L^p$ estimates for purely discontinuous martingales with accessible jumps and to how a decoupled tangent martingale in this case look like.
First we will start with the following elementary proposition which will provide us with $L^p$-bounds for tangent martingales.

\begin{proposition}\label{prop:pdwajjusttang}
Let $X$ be a Banach space, $1\leq p<\infty$. Then $X$ is UMD if and only if for any pair $M$ and $N$ of $X$-valued tangent purely discontinuous local martingales with accessible jumps one has that
\begin{equation}\label{eq:PDMWAJtangentLpbounds}
\mathbb E \sup_{t\geq 0} \|M_t\|^p \eqsim_{p, X}\mathbb E \sup_{t\geq 0} \|N_t\|^p.
\end{equation}
\end{proposition}

For the proof we will need the following technical lemma.

\begin{lemma}\label{lem:predsttimetang}
Let $X$ be a Banach space, $M, N:\mathbb R_+ \times \Omega \to X$ be tangent local martingales. Let $\tau$ be a predictable stopping times with $\tau<\infty$ a.s. Then
$\mathbb P(\Delta M_{\tau}|\mathcal F_{\tau}) = \mathbb P(\Delta N_{\tau}|\mathcal F_{\tau})$.
\end{lemma}

\begin{proof}
As $M$ and $N$ are tangent, $\nu^{M} = \nu^{N}$. In particular, since $\tau$ is predictable (and hence a process $t\mapsto \mathbf 1_{\tau}(t)$ is predictable as well) we have that for any Borel set $A \in X$ a.s.
\begin{multline*}
\mathbb E(\mathbf 1_{A}(\Delta M_{\tau})|\mathcal F_{\tau-}) \stackrel{(*)}= \int_{\mathbb R_+} \mathbf 1_{\tau}(t) \mathbf 1_{A}(x) \ud \nu^M(t, x)\\
 = \int_{\mathbb R_+} \mathbf 1_{\tau}(t) \mathbf 1_{A}(x) \ud \nu^N(t, x) \stackrel{(**)}= \mathbb E(\mathbf 1_{A}(\Delta N_{\tau})|\mathcal F_{\tau-}),
\end{multline*}
where $(*)$ follows from the fact that a.s.\ $\mathbf 1_{A}(\Delta M_{\tau}) =  \int_{\mathbb R_+} \mathbf 1_{\tau}(t) \mathbf 1_{A}(x) \ud \mu^M(t, x)$, \cite[Subsection 5.3]{DY17}, the definition of a compensator \cite[Theorem I.3.17]{JS}, the fact that thus both 
$$
t\mapsto \mathbf 1_{t\geq \tau}\mathbb E(\mathbf 1_{A}(\Delta M_{\tau})|\mathcal F_{\tau-})  \;\;\text{and}\;\; t\mapsto \int_{[0, t]} \mathbf 1_{\tau}(\tau) \mathbf 1_{A}(x) \ud \nu^M(t, x),\;\;\;\;\; t\geq 0,
$$ 
are compensators of $\mathbf 1_{A}(\Delta M_{t}) $ and the uniqueness of a compensator \cite[Theorem I.3.17]{JS}; $(**)$ holds for the same reason.
\end{proof}

\begin{proof}[Proof of Proposition \ref{prop:pdwajjusttang}]
First notice that  the ``if'' part follows from Lemma \ref{lem:predsttimetang}, Theorem \ref{thm:intromccnnll}, and the fact that any discrete martingale can be represented as a continuous-time martingale having jumps at natural points (see e.g.\ Remark \ref{rem:disctangagreesconttime}).

Let us show the ``only if'' part. Let $X$ be a UMD Banach space, $M$ and $N$ be $X$-valued purely discontinuous tangent martingales with accessible jumps. Then by Lemma~\ref{lem:PDmAJhasjumpsatprsttimes} there exist sequences $(\tau_n^M)_{n\geq 1}$ and $(\tau_n^N)_{n\geq 1}$ of predictable stopping times with disjoint graphs such that a.s.
\[
\{t\geq 0: \Delta M \neq 0\} \subset \{\tau_1^M, \ldots, \tau_n^M,\ldots\},
\]
\[
\{t\geq 0: \Delta N \neq 0\} \subset \{\tau_1^N, \ldots, \tau_n^N,\ldots\}.
\]
Moreover, by a standard merging procedure concerning predictable stopping times (see e.g.\ \cite{Kal,JS} and \cite[Subsection 5.1]{DY17}) we may assume that there exists a sequence $(\tau_n)_{n\geq 1}$ of predictable stopping times with disjoint graphs such that a.s.\
\[
\{\tau_1^M, \ldots, \tau_n^M,\ldots\} \cup \{\tau_1^N, \ldots, \tau_n^N,\ldots\} \subset \{\tau_1, \ldots, \tau_n,\ldots\}.
\]
For any $m\geq 1$ define martingales $M^m$ and $N^m$ by \eqref{eq:defofMmforaccjase}. Fix $\eps>0$. By Proposition~\ref{prop:MmapproxMinLpforacccase} we can fix $m\geq 1$ to be such that
\[
\mathbb E \sup_{t\geq 0} \|M_t - M^m_t\|^p <\eps,\;\;\;\;\;\;\;\; \mathbb E \sup_{t\geq 0} \|N_t - N^m_t\|^p<\eps. 
\]
Let $\tau_1',\ldots,\tau_m'$ be an increasing rearrangement of $\tau_1,\ldots,\tau_m$ (see \cite[Subsection 5.3]{DY17}). Then sequences $(d_n)_{n=1}^{2m}$ and $(e_n)_{n=1}^{2m}$
\[
d_n = 
\begin{cases}
\Delta M_{\tau_{n/2}'},\;\;\; & n\;\text{is even},\;\; 1\leq n \leq 2m,\\
0,\;\;\; & n\;\text{is odd},\;\; 1\leq n \leq 2m,
\end{cases}
\]
\[
e_n = 
\begin{cases}
\Delta N_{\tau_{n/2}'},\;\;\; & n\;\text{is even},\;\; 1\leq n \leq 2m,\\
0,\;\;\; & n\;\text{is odd},\;\; 1\leq n \leq 2m,
\end{cases}
\]
are tangent martingale difference sequences with respect to the filtration 
$$
(\mathcal F_{\tau_1'-},\mathcal F_{\tau_1'}, \mathcal F_{\tau_2'-},\ldots,\mathcal F_{\tau_{m}'-}, \mathcal F_{\tau_m'}).
$$
Indeed, first of all the latter is a filtration by \cite[Lemma 25.2]{Kal}. Next notice that for any even $n=2,\ldots, 2m$
\[
\mathbb E \bigl(d_n|\mathcal F_{\tau_{n/2}'-}\bigr) =\mathbb E \bigl(e_n|\mathcal F_{\tau_{n/2}'-}\bigr) = 0,
\]
by Lemma \ref{lem:DeltaMtaugiventau-=0}. Finally for any even $n= 2,\ldots,2m$ we have that
\begin{multline*}
\mathbb P\bigl(d_n|\mathcal F_{\tau_{n/2}'-}\bigr) = \mathbb P\bigl(\Delta M_{\tau_{n/2}'}|\mathcal F_{\tau_{n/2}'-}\bigr)
\stackrel{(*)} = \mathbb P\bigl(\Delta N_{\tau_{n/2}'}|\mathcal F_{\tau_{n/2}'-}\bigr) = \mathbb P\bigl(e_n|\mathcal F_{\tau_{n/2}'-}\bigr),
\end{multline*}
where $(*)$ follows from Lemma \ref{lem:predsttimetang} and the fact that $M$ and $N$ are tangent. Therefore by the definition of $M^m$ and $N^m$, by Theorem \ref{thm:intromccnnll}, and by the fact that $\sup_{t\geq 0} \|M^m_t\| = \sup_{n=1}^{2m}\bigl\| \sum_{k=1}^n d_k \bigr\|$ and $\sup_{t\geq 0} \|N^m_t\| = \sup_{n=1}^{2m}\bigl\| \sum_{k=1}^n e_k \bigr\|$ (as both martingales $M^m$ and $N^m$ are pure jump processes which jumps coincide with $(d_n)_{n=1}^{2m}$ and $(e_n)_{n=1}^{2m}$ respectively) we have that
\[
\mathbb E\sup_{0\leq t <\infty}\|M^m_t\|^p \eqsim_{p, X}\mathbb E\sup_{0\leq t <\infty}\|N^m_t\|^p,
\]
and the desired follows by approaching $\eps$ to zero and by Proposition \ref{prop:MmapproxMinLpforacccase}.
\end{proof}

Let us now show that for any purely discontinuous martingale with accessible jumps taking values in UMD Banach spaces there exists a decoupled tangent martingale.

\begin{theorem}\label{thm:detangmartforMXVpdwithaccjumps}
Let $X$ be a UMD Banach space, $M:\mathbb R_+ \times \Omega \to X$ be a purely discontinuous local martingale with accessible jumps. Then there exist an enlarged probability space $(\overline {\Omega}, \overline{\mathcal F}, \overline {\mathbb P})$ endowed with an enlarged filtration $\overline {\mathbb F} = (\overline {\mathcal F}_t)_{t\geq 0}$, and an $\overline {\mathbb F}$-adapted purely discontinuous local martingale $N:\mathbb R_+ \times \overline {\Omega} \to X$ with accessible jumps such that $M$ is a local $\overline {\mathbb F}$-martingale with the same local characteristics, $M$ and $N$ are tangent and $N(\omega)$ is a martingale with independent increments and with the local characteristics $(0,\nu^M(\omega))$ for a.e.\ $\omega \in \Omega$.
\end{theorem}

\begin{proof}
By Remark \ref{rem:locmartislocallyL1DR+X} and Theorem \ref{thm:MNtangent==>MtauNtautangent+ifNCIINtauCII} The proof will be based on the construction of a CI\footnote{{\em CI} is for {\em conditionally independent}} tangent martingale difference sequence presented in the proof of \cite[Proposition 6.1.5]{dlPG}. Let $(\tau_n)_{n\geq 1}$ be a sequence of predictable stopping times with disjoint graphs such that a.s.\
\[
\{t\geq 0: \Delta M_t \neq 0\} \subset \{\tau_1,\ldots,\tau_n,\ldots\}
\]
(see Lemma \ref{lem:PDmAJhasjumpsatprsttimes}).
Let us define
\[
\overline \Omega := X^{\mathbb N}\times \Omega,\;\;\;\;\overline{\mathcal F} := \bigl(\otimes_{n\geq 1} \mathcal B(X)\bigr)\otimes \mathcal F,
\] 
where $\mathcal B(X)$ is the Borel $\sigma$-algebra of $X$, and let for any $t\geq 0$  a $\sigma$-algebra $\overline{\mathcal F}_t$ on $\overline{\Omega}$ be generated by all the sets 
\begin{equation}\label{eq:defofmathcalFtoverlinebysets}
 (B_n \mathbf 1_{\tau_n \leq t} \cup X \mathbf 1_{\tau_n > t} )_{n\geq 1} \times R,\;\;\; B_1 ,\ldots,B_n,\ldots \in \mathcal B(X),\;\;\; R\in \mathcal F_t,
\end{equation}
so $ \overline{\mathcal F}_t := S_t\otimes \mathcal F_t$ for any $ t\geq 0$ where
\begin{equation}\label{eq:defofStforPDMAJdecoupl}
S_t := \otimes_{n\geq 1} \bigl(\mathcal B(X) \mathbf 1_{\tau_n \leq t} \cup \{\varnothing, \Omega\} \mathbf 1_{\tau_n > t}\bigr).
\end{equation}
($\overline{\mathcal F}_t$ sees $x_n$ if $\tau_n \leq t$, and does not see otherwise. Note that $\otimes$ in $S_t\otimes \mathcal F_t$ does not mean the direct product of $\sigma$-algebras since $S_t$ by its definition \eqref{eq:defofStforPDMAJdecoupl} depends on $\Omega$, but in this case $\otimes$ means that the corresponding $\sigma$-algebra is generated by products of sets of the form \eqref{eq:defofmathcalFtoverlinebysets}). Let $\overline{\mathbb F}: = (\overline{\mathcal F}_t)_{t\geq 0}$.
As $(X, \mathcal B(X))$ is a {\em Polish space} (see \cite[pp.\ 344, 386]{Dud89}), by \cite[Theorem 10.2.2]{Dud89} for any $n\geq 1$ and for almost any $\omega \in \Omega$ there exists a probability measure $\mathbb P^n_{\omega}$ on $X$ such that for any $B \in \mathcal B(X)$ (see \eqref{eq:defofftau-} for the definition of $\mathcal F_{\tau-}$)
\begin{equation}\label{eq:forPDWAJdefofPomeganchm}
\mathbb E \bigl(\mathbf 1_{B}(\Delta M_{\tau_n})|\mathcal F_{\tau_n-}\bigr)(\omega) =\mathbb P^n_{\omega}(B),\;\;\; \omega\in \Omega.
\end{equation}
Then set 
\begin{equation}\label{eq:deofPbarforDTMwithaJmomente}
 \overline {\mathbb P} (A \times R) := \int_R \otimes_{n\geq 1} \mathbb P^n_{\omega}(A) d\mathbb P(\omega),\;\;\; A \in X^{\mathbb N}, \;\; R \in \mathcal F.
\end{equation}

Now let us construct a c\`adl\`ag process $N:\mathbb R_+ \times \overline{\Omega} \to X$ satisfying for a.e.\ $\omega \in \overline{\Omega}$
\begin{equation}\label{eq:DeltaNtaunforDTMforPDAJmart}
 \Delta N_{\tau_n}\bigl((x_i)_{i\geq 1}, \omega \bigr) = x_n, \;\;\; (x_i)_{i\geq 1} \in X^{\mathbb N}.
\end{equation}
(Spoiler: this is going to be our decoupled tangent martingale).
We need to show that such a process exists $\overline {\mathbb P}$-a.s.\ and that this is an $\overline{\mathbb F}$-martingale. For each $m\geq 1$ define $N^m:\mathbb R_+ \times \overline{\Omega} \to X$ to be 
\begin{equation}\label{eq:defofNmapproxNDTMforPDWAJ}
 N^m_t\bigl((x_n)_{n\geq 1}, \omega \bigr)  := \sum_{n=1}^m x_n \mathbf 1_{[0, t]}(\tau_n),\;\;\; t\geq 0, \;\; (x_n)_{n\geq 1} \in X^{\mathbb N},\;\;\omega \in \Omega.
\end{equation}
First note that $N^m$ is an $\overline{\mathbb F}$-adapted process with values in $X$ as for any fixed $t\geq 0$
$$
\bigl(t,(x_n)_{n\geq 1}, \omega \bigr)\mapsto x_n \mathbf 1_{[0, t]}(\tau_n),\;\; (x_n)_{n\geq 1} \in X^{\mathbb N},\;\;\omega \in \Omega,
$$
is  $\overline{\mathcal F}_t$-measurable since $\overline{\mathcal F}_t = S_t \otimes \mathcal F_t$ with $S_t$ defined by \eqref{eq:defofStforPDMAJdecoupl}, so $N^m$ is $\overline{\mathbb F}$-adapted as a sum of $\overline{\mathbb F}$-adapted processes.

Let us show that $N^m$ is a purely discontinuous martingale with accessible jumps. $N^m$ has accessible jumps as by the definition \eqref{eq:defofNmapproxNDTMforPDWAJ} of $N^m$ it jumps only at predictable stopping times $\{\tau_1, \ldots, \tau_m\}$ (which remain predictable stopping times with respect to the enlarged filtration $\overline{\mathbb F}$ as they remain being announced by the same sequences of stopping times, see Subsection \ref{subsec:prelimsttimes}). Note that for any $1\leq n\leq m$ we have that (here $\overline {\mathcal F}_{\tau_n-}$ and $S_{\tau_n-}$ are defined analogously $\mathcal F_{\tau_n-}$ through an announcing sequence as $\tau_n$ is a predictable stopping time, see Subsection \ref{subsec:prelimsttimes} and \eqref{eq:defofStaunStaun-})
\[
\mathbb E (\Delta N^m_{\tau_n}|\overline {\mathcal F}_{\tau_n-}) = \mathbb E (\Delta N^m_{\tau_n}|S_{\tau_n-} \otimes {\mathcal F}_{\tau_n-}) = \mathbb E \bigl(\mathbb E (\Delta N^m_{\tau_n}|S_{\tau_n-} \otimes {\mathcal F})\big| S_{\tau_n-} \otimes {\mathcal F}_{\tau_n-}\bigr).
\]
(Here $\otimes$ again {\em is not} a product of $\sigma$-algebras, but a $\sigma$-algebra generated by products of sets of the form familiar to \eqref{eq:defofmathcalFtoverlinebysets}).
We need to show that $\mathbb E (\Delta N^m_{\tau_n}|S_{\tau_n-} \otimes {\mathcal F}) = 0$. It is sufficient to show that for $\overline {\mathbb P}$-almost any fixed $\omega\in \overline{\Omega}$, $\mathbb E (\Delta N^m_{\tau_n(\omega)}(\omega)|S_{\tau_n-}) = 0$ because we have that for any $R\in \mathcal F$ and $A\times R \in S_{\tau_n-} \otimes {\mathcal F}$ (where $A$ depends on $\Omega$ in a predictable way so that $A \times R$ has the form \eqref{eq:defofmathcalFtoverlinebysets}) we have that
\[
\int_{A \times R} \Delta N^m_{\tau_n} \ud \overline{\mathbb P} = \int_R \int_A \Delta N^m_{\tau_n} \ud \otimes_{n\geq 1} \mathbb P_n^{\omega} \ud \mathbb P(\omega), 
\]
so the first integral equals zero if $\int_A \Delta N^m_{\tau_n} \ud \otimes_{n\geq 1} \mathbb P_n^{\omega} = 0$ for a.e.\ $\omega\in \Omega$ for any $A \in S_{\tau_n-}$.
By the definition \eqref{eq:defofStforPDMAJdecoupl} of $S_t$ we have that for almost any fixed $\omega \in \Omega$
\begin{equation}\label{eq:defofStaunStaun-}
\begin{split}
 S_{\tau_n} &= \otimes_{i\geq 1, \tau_i \leq \tau_n} \mathcal B(X) \otimes_{i\geq 1, \tau_i > \tau_n} \{\varnothing, \Omega\},\\
 S_{\tau_n-} &= \otimes_{i\geq 1, \tau_i < \tau_n} \mathcal B(X) \otimes_{i\geq 1, \tau_i \geq \tau_n} \{\varnothing, \Omega\}, 
\end{split}
\end{equation}
so we have that 
$$
S_{\tau_n} = \sigma(S_{\tau_{n}-},\Delta S_{\tau_n}), \;\;\; \Delta S_{\tau_n} := \otimes_{i\geq 1} \bigl(\mathcal B(X) \mathbf 1_{i=n}\bigr) \cup\bigl( \{\varnothing, \Omega\}\mathbf 1_{i\neq n}\bigr),
$$
(here we used the fact that $(\tau_n)_{n\geq 1}$ have a.s.\ disjoint graphs), so $S_{\tau_n}$ is a.s.\ generated by two independent $\sigma$-algebras $S_{\tau_n-}$ and $\Delta S_{\tau_n}$ (which are independent a.s.\ by the definition \eqref{eq:deofPbarforDTMwithaJmomente} of $\overline{\mathbb P}$),
and hence as $\Delta N^m_{\tau_n}$ is  a.s.\ $\Delta S_{\tau_n}$-measurable, $\mathbb E (\Delta N^m_{\tau_n}(\omega)|S_{\tau_n-}) = \mathbb E_{X^{\mathbb N}} (\Delta N^m_{\tau_n})(\omega)$. Finally note that $\Delta N^m_{\tau_n}(\omega)$ has $\mathbb P^n_{\omega}$ as its distribution by the definition \eqref{eq:deofPbarforDTMwithaJmomente} of $\overline {\mathbb P}$ and the definition \eqref{eq:defofNmapproxNDTMforPDWAJ} of $N^m$, and the latter distribution has a.s.\ a mean zero by the definition \eqref{eq:forPDWAJdefofPomeganchm} as $\int_{X} x \ud \mathbb P^n_{\omega} = \mathbb E (\Delta M_{\tau_n}|\mathcal F_{\tau_n-})(\omega) = 0$ for a.e.\ $\omega\in \Omega$. Therefore $\mathbb E (\Delta N^m_{\tau_n}|\overline {\mathcal F}_{\tau_n-}) = 0$, and hence $N^m$ is a martingale by \cite[Subsection 5.3]{DY17}.

 Let us now show that $M$ is an $\overline{\mathbb F}$-martingale with the same local characteristics $(0, \nu^M)$. Fix $n\geq 1$. Then for any Borel $B \subset X$ and any $\overline{\mathcal F}_{\tau_n-}$-measurable bounded $F:\overline{\Omega} \to \mathbb R$ we have that
\begin{equation}\label{eq:whywehavemarqoidq1oimeiqmw}
 \mathbb E \mathbf 1_{B}(\Delta M_{\tau_n}) F = \mathbb E \mathbb E_{\omega}\mathbf 1_{B}(\Delta M_{\tau_n})(\omega) F = \mathbb E\mathbf 1_{B}(\Delta M_{\tau_n})  \mathbb E_{\omega}F,
\end{equation}
where $\mathbb E_{\omega}$ denotes expectation w.r.t.\ $\otimes_{n\geq 1} \mathbb P^n_{\omega}$ for each fixed $\omega\in \Omega$. As $F$ is $\overline{\mathcal F}_{\tau_n-}$-measurable and as $\overline{\mathcal F}_{\tau_n-}\subset S_{\infty} \otimes  {\mathcal F}_{\tau_n-}$ (this follows due to that fact that $\overline{\mathcal F}_{\tau_n-} = S_{\tau_n-} \otimes {\mathcal F}_{\tau_n-}$) we have that $ \mathbb E_{\omega}F$ is ${\mathcal F}_{\tau_n-}$-measurable and hence
\begin{equation}\label{eq:whywehavemarqoidq2oimeiqmw}
\mathbb E\mathbf 1_{B}(\Delta M_{\tau_n})  \mathbb E_{\omega}F = \mathbb E \mathbb E \bigl(\mathbf 1_{B}(\Delta M_{\tau_n})\big| {\mathcal F}_{\tau_n-} \bigr)  \mathbb E_{\omega}F = \mathbb E \mathbb E \bigl(\mathbf 1_{B}(\Delta M_{\tau_n})\big| {\mathcal F}_{\tau_n-} \bigr) F,
\end{equation}
so by the definition of conditional expectation and freedom in choice of $F$ we have that
\begin{equation}\label{eq:PDelmaMtaun-givenoverFdao=PDelmcaco}
\begin{split}
 \overline{\mathbb P}\bigl((\Delta M_{\tau_n})|\overline{\mathcal F}_{\tau_n-}\bigr)(B) &= \mathbb E \bigl(\mathbf 1_{B}(\Delta M_{\tau_n})|\overline{\mathcal F}_{\tau_n-}\bigr)\\
 &= \mathbb E \bigl(\mathbf 1_{B}(\Delta M_{\tau_n})|{\mathcal F}_{\tau_n-}\bigr) = \overline{\mathbb P}\bigl((\Delta M_{\tau_n})|\overline{\mathcal F}_{\tau_n-}\bigr)(B),
\end{split}
\end{equation}
and hence $M$ has the same $\overline{\mathbb F}$-local characteristics $(0, \nu^M)$. The fact that $M$ is an $\overline{\mathbb F}$-martingale follows from Lemma \ref{lem:DeltaMtaugiventau-=0}, the fact that $(\tau_n)_{n\geq 1}$ exhausts all the jumps of $M$, and the fact that by \eqref{eq:PDelmaMtaun-givenoverFdao=PDelmcaco}
\[
\mathbb E \bigl(\Delta M_{\tau_n}|\overline{\mathcal F}_{\tau_n-}\bigr) = \mathbb E \bigl(\Delta M_{\tau_n}|{\mathcal F}_{\tau_n-}\bigr) = 0.
\]

Now let $M^m$ be defined by \eqref{eq:defofMmforaccjase} and let us show that $N^m$ is a decoupled tangent martingale to $M^m$. Note that $M^m$ is an $\overline {\mathbb F}$-martingale as well as $M$. First $M^m$ and $N^m$ have jumps only at $\{\tau_1,\ldots, \tau_m\}$, so they are tangent because for any $1\leq n\leq m$ for any $B\in \mathcal B(X)$ for a.e.\ $\omega\in \overline{\Omega}$
\begin{multline*}
 \mathbb P(\Delta M^m_{\tau_n}| \overline{\mathcal F}_{\tau_n-})(B)(\omega)  \stackrel{(i)}= \mathbb P(\Delta M_{\tau_n}| \overline{\mathcal F}_{\tau_n-})(B)(\omega) \\= \mathbb E (\mathbf 1_{B}(\Delta M_{\tau_n})| \overline{\mathcal F}_{\tau_n-})(\omega)
  \stackrel{(ii)}=\mathbb P^n_{\omega}(B) \stackrel{(iii)}= \mathbb P(\Delta N^m_{\tau_n}|\overline{\mathcal F}_{\tau_n-})( B)(\omega) ,
\end{multline*}
where $(i)$ follows from the definition of $M^m$, $(ii)$ holds by the definition \eqref{eq:forPDWAJdefofPomeganchm} of $\mathbb P^n_{\omega}$, and $(iii)$ follows from the definition of $\overline {\mathbb P}$ and the definition of $N^m$.

Let us show that $N^m$ is a decoupled tangent martingale to $M^m$, i.e.\ that $N^m(\omega)$ has independent mean-zero increments and local characteristics $(0, \nu^{M^m}(\omega))$ for a.e.\ $\omega \in \Omega$. This easily follows from the fact that for a.e.\ fixed $\omega \in \Omega$ the process $N^m(\omega)$ has fixed jumps at $\{\tau_1(\omega),\ldots,\tau_m(\omega)\}$ and for every $1\leq n\leq m$ we have that $\Delta N^m_{\tau_n}(\omega)$ is $\Delta S_{\tau_n(\omega)}$-measurable; as
$ S_{\infty} = \sigma(\Delta S_{\tau_n(\omega)}, n\geq 1) = \otimes_{n\geq 1}\mathcal B(X)$,
so $(\Delta N^m_{\tau_n}(\omega))_{n= 1}^m$ are independent since $(\Delta S_{\tau_n(\omega)})_{n\geq 1}$ are independent. The fact that $N^m(\omega)$ has local characteristics $(0, \nu^{M^m}(\omega))$ follows from the construction of $N^m$.

Now let us show that $N^m$ converges as $m\to \infty$, and that the limit coincides with the desired $N$ which thus exists. For any $m_2 \geq m_1 \geq 1$ by \eqref{eq:PDMWAJtangentLpbounds} and by the fact that $N^{m_1} - N^{m_2}$ is a decoupled tangent martingale to $M^{m_1} - M^{m_2}$ (which can be shown analogously to the considerations above) we have that 
$$
\mathbb E \sup_{t\geq 0} \|N^{m_1}_t - N^{m_2}_t\| \eqsim_{X}\mathbb E \sup_{t\geq 0} \|M^{m_1}_t - M^{m_2}_t\|.$$
Thus martingales $(N^m)_{m\geq 1}$ converge in $L^1(\overline{\Omega}; \mathcal D(\mathbb R_+, X))$ by \eqref{eq:approxfactforMmforPDwAJcase}. Let $\widetilde N$ be the limit. 
Note that by Theorem \ref{thm:strongLpmartforamaBanacoace} $\widetilde N$ is an $\overline{\mathbb F}$-martingales. Let us show that $\widetilde N$ coincides with the desired $N$. For any $n\geq 1$ we have that for $\Delta N_{\tau_n}$ defined by \eqref{eq:DeltaNtaunforDTMforPDAJmart} (note that we still need to prove that $N$ exists and that $\widetilde N = N$) and by the fact that $\Delta N_{\tau_n} = \Delta N^m_{\tau_n}$ for $m\geq n$
\begin{multline*}
 \mathbb E \|\Delta \widetilde N_{\tau_n} - \Delta N_{\tau_n}\| = \lim_{m\to \infty} \mathbb E \|\Delta \widetilde N_{\tau_n} - \Delta N^m_{\tau_n}\|\\
 \stackrel{(*)}\leq \lim_{m\to \infty} \mathbb E \Bigl(\| \widetilde N_{\tau_n} - N^m_{\tau_n}\| +  \| \widetilde N_{\tau_n-} - N^m_{\tau_n-}\|\Bigr)\\
 \leq 2\lim_{m\to \infty} \mathbb E \sup_{t\geq 0} \| \widetilde N_{t} - N^m_{t}\| = 0,
\end{multline*}
where $(*)$ follows by a triangle inequality.
For the same reason we have that $\Delta \widetilde N_{\tau} = 0$ a.s.\ on $\tau \notin \{\tau_1,\ldots, \tau_n,\ldots\}$ for any stopping time $\tau$. Therefore $\widetilde N$ coincides with the desired $N$, so such $N$ exists. $N$ is a decoupled tangent martingale to $M$ for the same reason as $N^m$ is a decoupled tangent martingale to $M^m$ for any $m\geq 1$.
\end{proof}

\begin{remark}
 To sum up Theorem \ref{thm:detangmartforMXVpdwithaccjumps}. Any purely discontinuous martingale $M$ with accessible jumps and with values in a UMD Banach space has a tangent martingale $N$ on an enlarged probability space with an enlarged filtration such that for a.e.\ $\omega$ from the original probability space $N$ is a martingale with fixed jump times coinciding with the jumps times of $M$ and with independent increments. 
\end{remark}

\begin{remark}\label{rem:NPDwAJhasindincrgivenNnuMdsa}
Note that $N$ constructed in the proof of Theorem \ref{thm:detangmartforMXVpdwithaccjumps} has independent increments given $\nu^M$. Indeed, for a.e.\ fixed $\nu^M$ we have that the set $(\tau_n(\cdot))_{n\geq 1}$ is fixed, the distributions $(\mathbb P_{\omega}^n)_{n\geq 1}$ are fixed and mean zero, so $(\Delta N_{\tau_n(\cdot)})_{n\geq 1}$ are independent mean-zero random variable. Consequently the desired independence follows from Corollary \ref{cor:condindgivenRVsuffcond}.
\end{remark}

\subsection{Proof of Theorem \ref{thm:tangentgencaseUMDuhoditrotasoldat} and \ref{thm:mainforDTMgencasenananana}}\label{subsec:CIprocess}

Let us finally prove Theorem \ref{thm:tangentgencaseUMDuhoditrotasoldat} and \ref{thm:mainforDTMgencasenananana}. First, Theorem \ref{thm:tangentgencaseUMDuhoditrotasoldat} follows from Theorem \ref{thm:candecXvalued}, Remark \ref{rem:MYdecBanach}, Proposition \ref{prop:domcontcase},  \ref{prop:pdwajjusttang}, and Corollary \ref{cor:LstropboundforPDQLStangmartUMD}.  Now let us show Theorem \ref{thm:mainforDTMgencasenananana}. 
This theorem follows from the fact that for a.e.\ fixed $\omega\in \Omega$ we have that $N^c(\omega)$, $N^q(\omega)$, and $N^a(\omega)$ are independent. Since each of them a.s.\ have independent increments and local characteristics $([\![M^c]\!](\omega), 0)$, $(0, \nu^{M^q}(\omega))$, and $(0, \nu^{M^a}(\omega))$ respectively, $N(\omega)$ has independent increments and local characteristics $([\![M^c]\!](\omega), \nu^M(\omega))$ (the letter follows from Proposition \ref{prop:Grigcharforcandechowdoesitlooklike}). 

 It remains to show that $M$ and $N$ are local $\overline{\mathbb F}$-martingales with $\overline{\mathbb F}$-local characteristics $([\![M^c]\!], \nu^M)$. Let us start with $M$. To this end recall that the new filtration $\overline {\mathbb F}$ over the enlarged probability space $(\overline{\Omega},\overline{\mathcal F}, \overline{\mathbb P})$ is generated by ${\mathbb F}$, time-changed independent cylindrical Wiener process $W_H'$ (see the proof of Theorem \ref{thm:DTMforcontcase}), the Cox process $\mu_{\rm Cox}$ (see Remark \ref{rem:intFdmuisamaringeiththesameloccharenlfi}), and the filtration $(S_t(\omega))_{t\geq 0}$ defined by \eqref{eq:defofStforPDMAJdecoupl}. Let $F:\overline{\Omega} \to \mathbb R$ be any bounded $\mathcal F$-measurable random variable. Then for any fixed $t\geq 0$
\begin{equation}\label{eq:FFmeasiconsexpwrtwodiasFtthesacaom}
\mathbb E (F|\overline{\mathcal F}_t) = \mathbb E\bigl( \mathbb E(F|\sigma(W_H',\mathcal N,S_{\infty},{\mathcal F}_t))\big|\overline{\mathcal F}_t \bigr) \stackrel{(*)}= \mathbb E\bigl( \mathbb E(F|{\mathcal F}_t)\big|\overline{\mathcal F}_t \bigr) = \mathbb E (F|{\mathcal F}_t),
\end{equation}
where $\mathcal N$ is a $\sigma$-algebra generated by independent sequence of standard Poisson processes (this sequence can be assumed finite thanks to Remark \ref{rem:intFdmuisamaringeiththesameloccharenlfi}), and where $(*)$ follows from the fact that $F$ is independent of $W_H'$ and $\mathcal N$ and the trick similar to the computations \eqref{eq:whywehavemarqoidq1oimeiqmw} and \eqref{eq:whywehavemarqoidq2oimeiqmw}. Hence, as $F$ was general, $M$ is a local $\overline {\mathbb F}$-martingale. 

In order to show that $M$ preserves its local characteristics we notice by Remark \ref{rem:Mmightchangelocalcharacp} that $[\![M^c]\!]$ stays the same, the predictable jumps $(\tau_n)_{n\geq 1}$ remain predictable (hence $M^a$ has accessible jumps and the local characteristics $(0, \nu^{M^a})$ do not change by \eqref{eq:whywehavemarqoidq1oimeiqmw} and \eqref{eq:whywehavemarqoidq2oimeiqmw}), and $M^q$ does not change its local characteristics as analogously to Remark \ref{rem:intFdmuisamaringeiththesameloccharenlfi} with exploiting \eqref{eq:FFmeasiconsexpwrtwodiasFtthesacaom} instead of \eqref{eq:Mt-MrconexpdirmuXpcsa} $\mu^{M^q}$ has the same compensator $\nu^{M^q}$, so $M^q$ has the same local characteristics $(0, \nu^{M^q})$, and thus $M$ keeps the local characteristics $([\![M^c]\!], \nu^M)$ by Subsection \ref{subsec:charandthecandecbe}.

Proving that $N$ is a local $\overline{\mathbb F}$-martingales with local characteristics $([\![M^c]\!], \nu^M)$ follows analogously. We will only show this for $N^c$ (the cases of $N^q$ and $N^a$ can be shown similarly). Let $\widetilde {\mathcal F}_t := \mathcal F_{t}\otimes W_H'|_{[0, A_t]}$ and $\widetilde {\mathbb F} := (\widetilde {\mathcal F}_t)_{t\geq 0}$, where the time-change $(A_t)_{t\geq 0}$ depending only on $\mathcal F_{t}$ defined in the proof of Theorem \ref{thm:DTMforcontcase}. Then for any bounded $ \widetilde {\mathcal F}_{\infty}$-measurable $F$ similarly to \eqref{eq:FFmeasiconsexpwrtwodiasFtthesacaom}
\begin{equation*}
\mathbb E (F|\overline{\mathcal F}_t) = \mathbb E\bigl( \mathbb E(F|\sigma(\mathcal N,S_{\infty},\widetilde{\mathcal F}_t))\big|\overline{\mathcal F}_t \bigr) = \mathbb E\bigl( \mathbb E(F|\widetilde{\mathcal F}_t)\big|\overline{\mathcal F}_t \bigr) = \mathbb E (F|\widetilde{\mathcal F}_t),
\end{equation*}
so as $N^c$ is a local $\widetilde {\mathbb F}$-martingale, it is a local $\overline {\mathbb F}$-martingale. The fact that $N$ has the local characteristics $([\![M^c]\!], \nu^M)$ follows in the same way as for $M$.

\subsection{Uniqueness of a decoupled tangent martingale}\label{subsec:uniqofDTM}

This subsection is devoted to showing that a decoupled tangent local martingale, if exists, is {\em unique up to the distribution}.

\begin{proposition}\label{prop:dectanglocmaathasuniqerdistr}
Let $X$ be a Banach space, $M:\mathbb R_+  \times \Omega \to X$ be a local martingale. Assume that $M$ has two decoupled tangent local martingales $N^1$ and $N^2$ on possibly different enlarged probability spaces with enlarged filtrations. Then $N^1$ and $N^2$ are equidistributed as random elements with values in $\mathcal D(\mathbb R_+, X)$.
\end{proposition}

\begin{proof}
Suppose that $N^1$ and $N^2$ live on probability spaces $(\overline{\Omega}^1, \overline{\mathcal F}^1, \overline{\mathbb P}^1)$ and $(\overline{\Omega}^2, \overline{\mathcal F}^2, \overline{\mathbb P}^2)$ respectively, where both $(\overline{\Omega}^1, \overline{\mathcal F}^1, \overline{\mathbb P}^1)$ and $(\overline{\Omega}^2, \overline{\mathcal F}^2, \overline{\mathbb P}^2)$ are enlargements of $(\Omega, \mathcal F, \mathbb P)$ (see Definition \ref{def:enlagoffiltprobspace}). Then by Definition \ref{def:dectanglocmartcontimecased} for a.e.\ fixed $\omega\in \Omega$ processes $N^1(\omega)$ and $N^2(\omega)$ are local martingales with independent increments and local characteristics $([\![M^c]\!](\omega), \nu^M(\omega))$. Thus $N^1(\omega)$ and $N^2(\omega)$ are equidistributed by Corollary \ref{cor:ifMIIthendirstofMisundetbylocchas}, and thus $N^1$ and $N^2$ are equidistributed as we have that for any Borel set $B \in \mathcal D(\mathbb R_+, X)$
\[
\overline{\mathbb P}^1(N^1 \in B) = \int_{\Omega} \widehat{\mathbb P}_{\omega}^1(N^1(\omega) \in B) \ud \mathbb P(\omega) = \int_{\Omega} \widehat{\mathbb P}_{\omega}^2(N^2(\omega)\in B) \ud \mathbb P(\omega) = \overline{\mathbb P}^2(N^2 \in B),
\]
where $ \widehat{\mathbb P}_{\omega}^1$ and $ \widehat{\mathbb P}_{\omega}^2$ are as in Definition \ref{def:enlagoffiltprobspace}. This terminates the proof.
\end{proof}

\subsection{Independent increments given the local characteristics}\label{subsec:indincgivenLC}

In fact, we can make Definition \ref{def:dectanglocmartcontimecased} stronger by proving the following theorem.

\begin{theorem}
Let $X$ be a UMD Banach space, $M:\mathbb R_+ \times \Omega \to X$ be a local martingale, $N:\mathbb R_+ \times \overline{\Omega} \to X$ be a decoupled tangent local martingale to $M$. Then $N$ has independent increments given $([\![M^c]\!], \nu^M)$.
\end{theorem}

This theorem extends e.g.\ \cite[Example 6.1.7]{dlPG} and \cite[Theorem 3.1]{Kal17}.

\begin{proof}
The theorem follows directly from the construction of a decoupled tangent local martingale presented in Theorem \ref{thm:DTMforcontcase}, \ref{thm:mumuCoxcomparable}, and \ref{thm:detangmartforMXVpdwithaccjumps}, from Remark \ref{rem:condectangmarthasindinrccond[[M]]}, \ref{rem:NPDQLChascondindeindcfsgivennuM}, and \ref{rem:NPDwAJhasindincrgivenNnuMdsa}, from that fact that we can consider an enlargement of $(\Omega, \mathcal F, \mathbb P)$ generated by $W_H'$, $\mu_{\rm Cox}$, and $\overline{\mathbb P}$ defined by \eqref{eq:deofPbarforDTMwithaJmomente}, and from Corollary \ref{cor:condindgivenRVsuffcond} on condtioinal independence with respect to a random variable.
\end{proof}

\section{Upper bounds and the decoupling property}\label{sec:Uppbdsanddecprop}

As it was shown in Theorem \ref{thm:tangentgencaseUMDuhoditrotasoldat}, if $X$ is UMD, then for any local martingale $M$ and for a decoupled tangent local martingale $N$ we have that for any $1\leq p<\infty$

\begin{equation}\label{eq:Esup|Mt|peqsimEN_T|P}
\mathbb E \sup_{0\leq t\leq T} \|M_t\|^p \eqsim_{p, X} \mathbb E \sup_{0\leq t\leq T} \|N_t\|^p \stackrel{(*)}\eqsim_{p}\mathbb E  \|N_T\|^p,\;\;\; T>0,
\end{equation}
where $(*)$ follows from Lemma \ref{lem:MhasIIthemathbbEsupphiMeqsimphiEphiM} and the fact that $N(\omega)$ is a martingale with independent increments for a.e.\ $\omega \in \Omega$. But what if we are interested only in the upper bound of \eqref{eq:Esup|Mt|peqsimEN_T|P} (this is often the case, see Remark \ref{rem:stochinforgenmartfordecpropet} on stochastic integration)? Can we have such estimates for non-UMD Banach spaces? Inequalities of such form have been discovered by Cox and Veraar in \cite{CV07,CV} (see also \cite{CG,HNVW1,MC}) and they turn out to characterize the so-called {\em decoupling property}.

\begin{definition}\label{def:decproperty}
Let $X$ be a Banach space. Then $X$ is said to have the {\em decoupling property} if for any $1\leq p <\infty$, for any $X$-valued martingale difference sequence $(d_n)_{n\geq 1}$ and for a decoupled tangent martingale difference sequence $(e_n)_{n\geq 1}$ one has that
\begin{equation}\label{eq:defofdecpropqdsao}
\mathbb E \sup_{N\geq 1} \Bigl\| \sum_{n=1}^N d_n \Bigr\|^p \lesssim_{p, X} \mathbb E \Bigl\| \sum_{n=1}^{\infty} c_n\Bigr\|^p.
\end{equation}
\end{definition}

Unlike the UMD property, Banach spaces with the decoupling property might not enjoy reflexivity. For example, $L^1$ spaces has the decoupling property. Moreover, {\em quasi-}Banach spaces can also satisfy \eqref{eq:defofdecpropqdsao} (e.g.\ $L^q$ for $q\in (0,1)$, see \cite{CV}).

The goal of the present section is to extend \eqref{eq:defofdecpropqdsao} to the continuous-times case. Of course for a general Banach space $X$ with a decoupling property and for a general $X$-valued martingale we will not have a decoupled tangent local martingale thanks to Theorem \ref{thm:mainforDTMgencasenananana}, but nonetheless, we are able to provide a continuous-time analogue of \eqref{eq:defofdecpropqdsao} in some spacial cases when such a decouple tangent local martingale exists. Let us start with the continuous case which is an elementary consequence of \cite[Theorem 5.4]{CV}. Recall that for any time change $(\tau_s)_{s\geq 0}$ we have the {\em inverse} time change $(A_t)_{t\geq 0}$ defined by $A_t := \inf\{s\geq 0: \tau_s \geq t\}$, and that a process is in $\gamma_{loc}(L^2(\mathbb R_+; H), X)$ if it is locally in $\gamma(L^2(\mathbb R_+; H), X)$ (see Subsection \ref{subsec:gammanorm}).

\begin{theorem}\label{thm:decpropfordonmartupperbasv}
Let $X$ be a Banach space with the decoupling property, $M^c:\mathbb R_+\times \Omega \to X$ be a continuous local martingale. Assume that there exists a time change $(\tau^c_s)_{s\geq 0}$, a Hilbert space $H$, an $H$-cylindrical Brownian motion $W_H$ adapted with respect to (possibly, enlarged) filtration $\mathbb G := (\mathcal F_{\tau_s})_{s\geq 0}$, and a strongly $\mathbb G$-predictable process $\Phi:\Omega \to \gamma_{loc}(L^2(\mathbb R_+; H), X)$ such that for any $x^*\in X^*$ we have that
$\langle M^c, x^* \rangle \circ \tau^c = \Phi^*x^* \cdot W_H$ a.s. Then $M^c$ has a decoupled tangent local martingale $N^c$ which has the following form: $N^c = (\Phi \cdot \overline W_H) \circ A^c$, where $\overline W_H$ is an independent copy of $W_H$ and $(A^c_t)_{t\geq 0}$ is the time change inverse to $\tau^c$. Moreover, if this is the case then for any $1\leq p<\infty$
\begin{equation}\label{eq:upperbounddecforcontmar}
 \mathbb E \sup_{0\leq t\leq T} \|M^c_t\|^p \lesssim_{p, X} \mathbb E  \|N^c_T\|^p,\;\;\; T\geq 0. 
\end{equation}
\end{theorem} 

\begin{proof}
First note that by \cite[Theorem 5.4]{CV} $\Phi$ is integrable with respect to $W_H$ and $M^c \circ \tau^c = \Phi \cdot W_H$. Moreover, by  \cite[Theorem 5.4]{CV} we also have that $\Phi$ is integrable with respect to $\overline W_H$. Let $N^c := (\Phi \cdot \overline W_H) \circ A^c$. Then $N^c$ is a decoupled tangent local martingale to $M^c$ due to Definition \ref{def:dectanglocmartcontimecased} and \eqref{eq:qvvarofstintwrtHcylbrmot}. \eqref{eq:upperbounddecforcontmar} follows directly from \cite[(5.3)]{CV} and the fact that
\[
 \mathbb E_{\overline W_H}  \|N^c_T\|^p = \mathbb E_{\overline W_H} \Bigl\| \int_0^{A_t} \Phi \ud \overline W_H\Bigr\|^p  \eqsim_{p} \mathbb E \|\Phi\|^p_{\gamma(L^2([0, A_t]; H), X)}.
\]
\end{proof}

Now let us move to the quasi-left continuous case. Recall that a stochastic integral with respect to a random measure was defined in Definition \ref{def:Fisintagrablewrtbarmugendef}.

\begin{theorem}\label{thm:upperbdsforPSWLCmaerviacox}
 Let $X$ be a Banach space satisfying the decoupling property, $(J, \mathcal J)$ be a measurable space. Let $\mu$ be a $\widetilde {\mathcal P}$-$\sigma$-finite quasi-left continuous integer random measure on $\mathbb R_+ \times J$ with a compensator $\nu$, $\bar{\mu} := \mu-\nu$. Let $F:\mathbb R_+ \times \Omega \to X$ be strongly $\widetilde {\mathcal P}$-measurable. Assume that $F(\omega)$ is integrable with respect to $\bar{\mu}_{\rm Cox}(\omega)$ for a.e.\ $\omega \in \Omega$, where ${\mu}_{\rm Cox}$ is a Cox process directed by $\nu$, $\bar{\mu}_{\rm Cox} = {\mu}_{\rm Cox} - \nu$. Then $F$ is locally integrable with respect to $\bar{\mu}$ and for any $1\leq p<\infty$
 \begin{equation}\label{eq:qlcranmareanddecptreopgiceLpestforstofashint}
  \mathbb E \sup_{t\geq 0} \Bigl\|\int_{[0, t] \times J} F \ud \bar{\mu} \Bigr\|^p \lesssim_{p, X}  \mathbb E  \Bigl\|\int_{R_+ \times J} F \ud \bar{\mu}_{\rm Cox} \Bigr\|^p.
 \end{equation}
\end{theorem}

\begin{proof}
 For each $k\geq 1$ define a stopping time 
 $$
 \tau_k:= \inf\Bigl\{t\geq 0: \mathbb E_{\rm Cox} \Bigl \| \int_{[0, t]\times J} F  \ud \bar{\mu}_{\rm Cox}  \Bigr\|  = k\Bigr\}.
 $$
We can find such (possibly infinite) $t$ that $\mathbb E_{\rm Cox}  \| \int_{[0, t]\times J} F \ud \bar{\mu}_{\rm Cox}  \|  = k$ since the function $t\mapsto \mathbb E_{\rm Cox}  \| \int_{[0, t]\times J} F(\omega) \ud \bar{\mu}_{\rm Cox}(\omega)  \|$ is continuous in $t\geq 0$ for a.e.\ $\omega\in \Omega$ because $\int_{[t, t+\eps]\times J} F(\omega) \ud \bar{\mu}_{\rm Cox}(\omega) \to 0$ a.s.\ as $\eps\to 0$. Without loss of generality by a stopping time argument we can set $F := F \mathbf 1_{[0, \tau_k]}$ and so we may assume that $\mathbb E_{\rm Cox}  \| \int_{\mathbb R_+\times J} F \ud \bar{\mu}_{\rm Cox}  \|  <C$ a.s.\ for some $C>0$.

Let us show that $F$ is integrable with respect to $\bar{\mu}$. Let the sets $(A_n)_{n\geq 1}$ be as in Definition \ref{def:Fisintagrablewrtbarmugendef}. Then
\[
 t\mapsto M^n_t := \int_{[0, t] \times J} F \mathbf 1_{A_n} \ud \bar{\mu},\;\;\; t\geq 0,
\]
is a martingale for any $n\geq 1$. Due to \eqref{eq:defofdecpropqdsao} and by an approximating by step functions we have that 
\begin{equation}\label{eq:Mnisnaisdanabdfordecopropreaqwdjioqlcaas}
 \mathbb E \sup_{t\geq 0} \|M^n_t\|^p \lesssim_{p,X} \mathbb E \mathbb E_{\rm Cox} \Bigl\| \int_{\mathbb R_+ \times \Omega}  F \mathbf 1_{ A_n} \ud\bar{\mu}_{\rm Cox} \Bigr\|^p,\;\;\; n\geq 1,
\end{equation}
and for the same reason for any $m\geq n\geq 1$
\begin{equation}\label{eq:Mniscauchyfordecopropreaqwdjioqlcaas}
 \mathbb E \sup_{t\geq 0} \|M^m_t - M^n_t\| \lesssim_{ X} \mathbb E \mathbb E_{\rm Cox} \Bigl\| \int_{\mathbb R_+ \times \Omega}  F \mathbf 1_{A_m \setminus A_n} \ud \bar{\mu}_{\rm Cox}\Bigr\|.
\end{equation}
Thus we have that $(M^n)_{n\geq 1}$ is a Cauchy sequence in $L^1(\Omega;\mathcal D(\mathbb R_+, X))$ by \eqref{eq:Mniscauchyfordecopropreaqwdjioqlcaas} and Remark \ref{rem:sticintwrtPoisrmonercanchosreeanysetws}. Inequality \eqref{eq:qlcranmareanddecptreopgiceLpestforstofashint} follows from \eqref{eq:Mnisnaisdanabdfordecopropreaqwdjioqlcaas} by letting $n\to \infty$.
\end{proof}

The following corollary is a direct consequence of Theorem \ref{thm:upperbdsforPSWLCmaerviacox}.

\begin{corollary}\label{cor:decoaspropforQLCPDmart}
 Let $X$ be a Banach space with the decoupling property, $M^q : \mathbb R_+ \times \Omega \to X$ be a purely discontinuous quasi-left continuous local martingale. Let $\mu^{M^q}$ be defined by \eqref{eq:defofmuM}, $\nu^{M^q}$ be the corresponding compensator, $\mu^{M^q}_{\rm Cox}$ be a Cox process directed by $\nu^{M^q}$, $\bar{\mu}^{M^q}_{\rm Cox}:= \mu^{M^q}_{\rm Cox} - \nu^{M^q}$. Assume that $\int_{[0, t]\times X} x \ud \bar{\mu}^{M^q}_{\rm Cox}$ is well defined a.s.\ for any $t\geq 0$. Then an $X$-valued local martingale $N^q$ defined by
 \[
  N^q_t := \int_{[0, t]\times X}x \ud \bar{\mu}^{M^q}_{\rm Cox},\;\;\; t\geq 0,
 \]
is a decoupled tangent local martingale to $M^q$ and for any $1\leq p<\infty$
\begin{equation}\label{eq:upperbounddecforqlcontpdmar}
 \mathbb E \sup_{0\leq t\leq T} \|M^q_t\|^p \lesssim_{p, X} \mathbb E  \|N^q_T\|^p,\;\;\; T\geq 0. 
\end{equation}
\end{corollary}

Now let us move to the accessible jump case.

\begin{theorem}\label{thm:deppropforPDwiAJupperbds}
Let $X$ be a Banach space with the decoupling property, $M^a : \mathbb R_+ \times \Omega \to X$ be a purely discontinuous local martingale with accessible jumps. Assume that it has a decoupled tangent local martingale $N^a$. Then for any $1\leq p<\infty$
\begin{equation}\label{eq:upperbounddecfocontpdmaraccjudsa}
 \mathbb E \sup_{0\leq t\leq T} \|M^a_t\|^p \lesssim_{p, X} \mathbb E  \|N^a_T\|^p,\;\;\; T\geq 0. 
\end{equation}
\end{theorem}

\begin{proof}
 Fix $T>0$. Without loss of generality assume that $\mathbb E \|N^a_T\|^p<\infty$. Let $(\tau_n)_{n\geq 1}$ be finite predictable stopping times with disjoint graphs which exhaust jumps of $M^a$, $M^{a, m}$ and $N^{a, m}$ be defined analogously to \eqref{eq:defofMmforaccjase}. First notice that thanks to the proof of Proposition \ref{prop:pdwajjusttang} and \eqref{eq:defofdecpropqdsao} we have that for any $m\geq 1$
 \begin{equation*}
   \mathbb E \sup_{0\leq t\leq T} \|M^{a, m}_t\|^p \lesssim_{p, X} \mathbb E  \|N^{a, m}_T\|^p. 
 \end{equation*}
For the same reason we have that for any $\ell >m\geq 1$
 \begin{equation*}
   \mathbb E \sup_{0\leq t\leq T} \|M^{a, \ell}_t - M^{a, \ell}_t\|^p \lesssim_{p, X} \mathbb E  \|N^{a, \ell}_T - N^{a, m}_T\|^p. 
 \end{equation*}
Therefore in order to show \eqref{eq:upperbounddecfocontpdmaraccjudsa} it is sufficient to show that $\mathbb E  \|N^{a, m}_T - N^{a}_T\|^p \to 0$ as $m\to \infty$. This follows directly from the fact that $N^a(\omega)$ has independent increments, hence $N^{a, m}(\omega) = \mathbb E (N^{a}(\omega) | \sigma(N^{a, m}(\omega)))$ due to the construction of $N^{a, m}$, and so the desired holds true by \cite[Theorem 3.3.2]{HNVW1}.
\end{proof}

The following theorem sums up Theorem \ref{thm:decpropfordonmartupperbasv}, Corollary \ref{cor:decoaspropforQLCPDmart}, and Theorem \ref{thm:deppropforPDwiAJupperbds}.

\begin{theorem}\label{thm:deopropertandLpboubdsfordecoupltangma}
Let $X$ be a Banach space with the decoupling property, $M: \mathbb R_+ \times \Omega \to X$ be a local martingale. Assume additionally that $M$ has the canonical decomposition $M = M^c + M^q + M^a$. Assume that $M^c$, $M^q$, and $M^a$ satisfy the conditions of  Theorem \ref{thm:decpropfordonmartupperbasv}, Corollary \ref{cor:decoaspropforQLCPDmart}, and Theorem \ref{thm:deppropforPDwiAJupperbds} respectively. Then $M$ has a decoupled tangent local martingale $N$, and for any $1\leq p<\infty$ one has that
\[
  \mathbb E \sup_{0\leq t\leq T} \|M_t\|^p \lesssim_{p, X} \mathbb E  \|N_T\|^p,\;\;\; T\geq 0. 
\]
\end{theorem}

\begin{proof}
 Existence of a decoupled tangent local martingale $N$ follows directly from Theorem \ref{thm:decpropfordonmartupperbasv}, Corollary \ref{cor:decoaspropforQLCPDmart}, and Theorem \ref{thm:deppropforPDwiAJupperbds}. Let $N = N^c + N^q + N^a$ be the canonical decomposition of $N$. Then by \eqref{eq:upperbounddecforcontmar}, \eqref{eq:upperbounddecforqlcontpdmar}, and \eqref{eq:upperbounddecfocontpdmaraccjudsa}, and by a triangle inequality we have that
 \[
    \mathbb E \sup_{0\leq t\leq T} \|M_t\|^p \lesssim_{p, X}  \mathbb E  \|N^c_T\|^p + \|N^q_T\|^p + \|N^a_T\|^p,\;\;\; T\geq 0.
 \]
 It remains to show that
 \[
 \mathbb E  \|N_T\|^p \eqsim_p \mathbb E  \|N^c_T\|^p + \|N^q_T\|^p + \|N^a_T\|^p,\;\;\; T\geq 0,
 \]
which follows from the fact that $N^c_T(\omega)$, $N^q_T(\omega)$, and $N^a_T(\omega)$ are independent mean-zero for a.e.\ $\omega \in \Omega$ due to Definition \ref{def:dectanglocmartcontimecased} and Theorem \ref{thm:genformofmartwithindince}.
\end{proof}

\section{Convex functions with moderate growth}\label{sec:Cfwithmodg}

A function $\phi:\mathbb R_+ \to \mathbb R_+$ is called to have a {\em moderate growth} if there exists $\alpha>0$ such that $\phi(2x) \leq \alpha \phi(x)$ for any $x\geq 0$. The goal of the present section is to show the following result about tangent martingales and convex functions with moderate growth which extends Theorem \ref{thm:intromccnnll} to more general functions and to continuous-time martingales and also extends \cite[Theorem 4.2]{Kal17} to infinite dimensions.

\begin{theorem}\label{thm:XUMDiffphimaxtangent}
Let $X$ be a Banach space, $\phi:\mathbb R_+ \to \mathbb R_+$ be a convex function of moderate growth such that $\phi(0) =0$. Then $X$ is UMD if and only if we have that for any tangent local martingales $M, N:\mathbb R_+ \times \Omega \to X$
\[
\mathbb E \phi(M^*) \eqsim_{\phi, X}\mathbb E \phi(N^*),
\]
where $M^* := \sup_{t\geq 0} \|M_t\|$ and $N^* := \sup_{t\geq 0} \|N_t\|$.
\end{theorem}

In order to prove the theorem we will need two components: the canonical decomposition and {\em good-$\lambda$ inequalities} for each part of the canonical decomposition. Namely we will use the following lemma proven by Burkholder in \cite[Lemma 7.1]{Bur73} (see also \cite[pp.\ 88--90]{Burk86}, \cite[pp.\ 1000--1001]{Burk81}, and  \cite[Section 4]{OY18} for various forms of general good-$\lambda$ inequalities).

\begin{lemma}\label{lem:goodlgivesphi}
Let $f, g:\Omega\to \mathbb R_+$ be measurable such that for some $\beta>1$, $\delta >0$, and $\eps>0$
\[
\mathbb P(g>\beta\lambda, f\leq \delta \lambda) \leq \eps \mathbb P(g>\lambda),\;\;\; \lambda >0.
\]
Let $\phi:\mathbb R_+ \to \mathbb R_+$ be a convex function of moderate growth with $\phi(0) = 0$. Let $\gamma<\eps^{-1}$ and $\eta$ be such that 
\[
\phi(\beta \lambda) \leq \gamma \phi(\lambda),\;\;\;\;\phi(\delta^{-1} \lambda) \leq \eta \phi(\lambda),\;\;\;\lambda>0.
\]
Then
\[
\mathbb E \phi(g) \leq \frac{\gamma \eta}{1-\gamma\eps} \mathbb E \phi(f).
\]
\end{lemma}

\subsection{Good-$\lambda$ inequalities}

Let us start with good-$\lambda$ inequalities for tangent continuous and purely discontinuous quasi-left continuous martingales. 
The following good-$\lambda$ inequalities for continuous tangent martingales follow from $L^p$ estimates \eqref{eq:domcontcase} analogously good-$\lambda$ inequalities presented in \cite[Section 8 and~9]{Bur73}.

\begin{proposition}\label{prop:goodlforconttanglocalmart}
Let $X$ be a UMD Banach space, $M, N:\mathbb R_+ \times \Omega \to X$ be tangent continuous local martingales.
Then we have that for any $1<p<\infty$, $\delta>0$, and $\beta>1$ 
\begin{equation*}\label{eq:goodlforconttangmaty}
\mathbb P(N^*>\beta\lambda, M^*\leq \delta \lambda) \lesssim_{p,X}\frac{ \delta^p}{(\beta-1)^p}\mathbb P(N^*>\lambda),\;\;\; \lambda >0,
\end{equation*}
where $M^* := \sup_{t\geq 0} \|M_t\|$ and $N^* := \sup_{t\geq 0} \|N_t\|$.
\end{proposition}

Let us now show good-$\lambda$ inequalities for stochastic integrals with respect to a random measure. First we will need a definition of a conditionally symmetric martingale.

\begin{definition}
 Let $X$ be a Banach space. $M:\mathbb R_+ \times \Omega \to X$ is called {\em conditionally symmetric} if $M$ has local characteristics and if $M$ and $-M$ are tangent.
\end{definition}

\begin{remark}
 Note that in the discrete case, i.e.\ when we have an $X$-valued discrete martingale difference sequence $(d_n)_{n\geq 1}$, the latter definition is equivalent to $\mathbb P(d_n|\mathcal F_{n-1})$ being symmetric a.s.\ for any $n\geq 1$. 
\end{remark}

Now let us state and prove the desired good-$\lambda$ inequalities.

\begin{proposition}\label{prop:goodlambdapdQlc}
Let $X$ be a UMD Banach space, $M$ and $N$ be $X$-valued tangent purely discontinuous quasi-left continuous conditionally symmetric local martingales. Then for any $\delta> 0$ and any $\beta>\delta + 1$ we have that
\begin{equation}\label{eq:goodlforstochintwwrtrm}
\mathbb P(N^*>\beta\lambda, \Delta M^* \vee\Delta N^* \vee M^* \leq \delta \lambda) \lesssim_{p,X}\frac{ \delta^p}{(\beta-\delta - 1)^p}\mathbb P(N^*>\lambda),\;\;\; \lambda >0,
\end{equation}
where $\Delta M^* := \sup_{t\geq 0} \|\Delta M_t\|$, $\Delta N^* := \sup_{t\geq 0} \|\Delta N_t\|$, $M^* := \sup_{t\geq 0} \|M_t\|$, and $N^* := \sup_{t\geq 0} \|N_t\|$.
\end{proposition}

For the proof we will need the following elementary lemma.

\begin{lemma}\label{lem:smallerjumpsonlyforgoodl}
Let  $X$, $M$, and $N$ be as above, $a>0$. Let
\[
\rho := \inf\{t\geq 0: \|\Delta M_t\| \vee \|\Delta N_t\|> a\}
\]
be a stopping time. Then $t\mapsto \Delta M_{\rho}\mathbf 1_{t\geq \rho}$ and $t\mapsto \Delta N_{\rho}\mathbf 1_{t\geq \rho}$ are local martingales. Moreover, we have that a.s.\
\[
M_t = \int_{[0, t] \times X} \mathbf 1_{\|x\|\leq a} x \ud \bar{\mu}^M,\;\;\; N_t = \int_{[0, t] \times X} \mathbf 1_{\|x\|\leq a} x \ud \bar{\mu}^N,\;\;\; t\in [0, \rho),
\]
where $\mu^M$ and $\mu^N$ are as defined by \eqref{eq:defofmuM}.
\end{lemma}

\begin{proof}
As $M$ and $N$ are conditionally symmetric and tangent, we may set that $\nu = \nu^M = \nu^N$ is the compensator for both $\mu^M$ and $\mu^N$, and that $\nu(\cdot \times B) = \nu(\cdot \times -B)$ a.s.\ for any Borel set $B\in \mathcal B(X)$. Now let 
$$
M'_t = \int_{[0,t]\times X} \mathbf 1_{\|\cdot\|> a}(x) x \mathbf 1_{[0, \rho]}(s) \ud\bar{\mu}^M(s,x),\;\;\; t\geq 0,
$$
$$
N'_t = \int_{[0,t]\times X} \mathbf 1_{\|\cdot\|> a}(x) x \mathbf 1_{[0, \rho]}(s) \ud\bar{\mu}^N(s,x),\;\;\; t\geq 0.
$$
These processes are local martingales by the fact that $x$ is locally stochastically integrable with respect to $\bar{\mu}^M$ and $\bar{\mu}^N$ thanks to Theorem \ref{thm:XisUMDiffMisintxwrtbarmuMvain}, therefore $x \mathbf 1_{A}$ is also locally integrable with respect to $\bar{\mu}^M$ for any $A\subset \widetilde {\mathcal P}$ by \cite[Subsection 7.2]{Y18BDG} and $\gamma$-domination \cite[Theorem 9.4.1]{HNVW2}. On the other hand, as $\nu$ is symmetric in $x\in X$ and as the function $1_{\|\cdot\|> a}(x) x \mathbf 1_{[0, \rho]}(s)$ is antisymmetric in $x\in X$, by the definition~of~$\rho$ we have that
\begin{align*}
M'_t &= \int_{[0,t]\times X} \mathbf 1_{\|\cdot\|> a}(x) x \mathbf 1_{[0, \rho]}(s) \ud\bar{\mu}^M(s,x)\\
& = \int_{[0,t]\times X} \mathbf 1_{\|\cdot\|> a}(x) x \mathbf 1_{[0, \rho]}(s) \ud{\mu}^M(s,x)\\
 &\quad\quad- \int_{[0,t]\times X} \mathbf 1_{\|\cdot\|> a}(x) x \mathbf 1_{[0, \rho]}(s) \ud\nu(s,x)\\
& = \int_{[0,t]\times X} \mathbf 1_{\|\cdot\|> a}(x) x \mathbf 1_{[0, \rho]}(s) \ud{\mu}^M(s,x)\\
&= \sum_{0\leq s \leq t\wedge \rho} \Delta M_s \mathbf 1_{\| \Delta M_s\| > a} =   \Delta M_{\rho}\mathbf 1_{t\geq \rho},
\end{align*}
so the desired follows for $M$. The same can be done for $N$.

The second part of the lemma follows from the fact that 
$$
M = M' +\int_{[0, \cdot] \times X} \mathbf 1_{\|x\|\leq a} x \ud \bar{\mu}^M,\;\;\;N= N' +\int_{[0, \cdot] \times X} \mathbf 1_{\|x\|\leq a} x \ud \bar{\mu}^N 
$$
a.s.\ on $[0, \rho]$ and the fact that by the considerations above a.s.\ 
$$
M'_t = \Delta M_{\rho}\mathbf 1_{t\geq \rho},\;\;\; N'_t = \Delta N_{\rho}\mathbf 1_{t\geq \rho},\;\;\; t\geq 0.
$$
\end{proof}

\begin{proof}[Proof of Proposition \ref{prop:goodlambdapdQlc}]
The proof is based on approach of Kallenberg \cite[pp.\ 36--39]{Kal17}.
 Let us define stopping times
\begin{align*}
\sigma&:= \inf\{t\geq 0: \|N_t\|>\lambda\},\\
\tau&:= \inf\{t\geq 0: \|M_t\|>\delta\lambda\},\\
\rho&:= \inf\{t\geq 0: \|\Delta M_t\|\vee \|\Delta N_t\|>\delta\lambda\}.
\end{align*}
Let $\mu^M$ and $\mu^N$ be defined by \eqref{eq:defofmuM}, $\bar{\mu}^M= \mu^M - \nu$ and $\bar{\mu}^N= \mu^N - \nu$ be the corresponding compensated random measures (as $M$ and $N$ are tangent, $\mu^M$ and $\mu^N$ have the same compensator). Define
$$
\widehat M_t:= \int_{[0, t] \times X} \mathbf 1_{\|x\|\leq \delta\lambda} x \mathbf 1_{(\tau \wedge \sigma \wedge \rho, \tau \wedge \rho]}(s) \ud \bar{\mu}^M(s, x),\;\;\; t\geq 0,
$$
$$
\widehat N_t:= \int_{[0, t] \times X} \mathbf 1_{\|x\|\leq \delta\lambda} x \mathbf 1_{(\tau \wedge \sigma \wedge \rho, \tau \wedge \rho]} (s)\ud \bar{\mu}^N(s, x),\;\;\; t\geq 0.
$$ 
Note that by Lemma \ref{lem:smallerjumpsonlyforgoodl}, $\widehat M$ coincides with $M - M^{\tau \wedge \sigma \wedge \rho}$ on $[0, \tau\wedge \rho)$, so by the definition of $\tau$ and $\rho$ we have that $\widehat M \leq 2\delta\lambda$, and thus by the fact that $\widehat M$ and $\widehat N$ are tangent by Lemma \ref{lem:Fintwrtmu1iffwrtmu2UMDcase}, so by Corollary \ref{cor:LstropboundforPDQLStangmartUMD} for any $1<p<\infty$ we have that
\begin{equation}\label{eq:MhatdominatesNhat}
\mathbb E \sup_{t\geq 0} \|\widehat N_t\|^p\lesssim_{p, X}\mathbb E \sup_{t\geq 0} \|\widehat M_t\|^p \leq 2^p \delta^p\lambda^p.
\end{equation}
Therefore,
\begin{equation*}
\begin{split}
\mathbb P(N^*>\beta\lambda, \Delta M^* \vee \Delta N^* \vee M^*\leq \delta \lambda)  &\leq \mathbb P(N^*>\beta\lambda, \tau = \rho = \infty)\\
& \stackrel{(*)}\leq \mathbb P(\widehat N^*>(\beta-\delta-1)\lambda)\\
& = \mathbb P\bigl((\widehat N^*)^p>(\beta-\delta-1)^p\lambda^p\bigr)\\
& \leq \frac{1}{(\beta-\delta-1)^p\lambda^p} \mathbb E (\widehat N^*)^p\\
& \stackrel{(**)}\lesssim_{p,X}\frac{1}{(\beta-\delta-1)^p\lambda^p} \mathbb E (\widehat M^*)^p,
\end{split}
\end{equation*}
where $(*)$ follows from the fact that if $\tau = \rho  = \infty$, so by Lemma \ref{lem:smallerjumpsonlyforgoodl} $\widehat N$ coincides with $N-N^{ \sigma}$ on $\mathbb R_+$ and $\|N_{\sigma}\| \leq \|N_{\sigma-}\| + \|\Delta N_{\sigma}\| \leq (1+\delta )\lambda$ on $\{\tau=\rho = \infty\}$, while $(**)$ holds by \eqref{eq:MhatdominatesNhat}.
The desired  then follows by
\begin{equation*}
\begin{split}
\mathbb E (\widehat{M}^*)^p = \mathbb E (\widehat{M}^*)^p \mathbf 1_{\sigma \leq \tau\wedge \rho} &\leq \mathbb E 2^p \delta^p \lambda^p \mathbf 1_{\sigma <\infty} \\
&= 2^p \delta^p \lambda^p\mathbb P(\sigma <\infty) =2^p \delta^p \lambda^p\mathbb P(N^*>\lambda).
\end{split}
\end{equation*}
\end{proof}

\subsection{Proof of Theorem \ref{thm:XUMDiffphimaxtangent}}
First we will prove each case of the canonical decomposition separately, and then compile them using the following proposition.

\begin{proposition}\label{prop:candecandphiarecorrespondest}
 Let $X$ be a UMD Banach space, $\phi:\mathbb R_+ \to \mathbb R_+$ be convex of moderate growth such that $\phi(0) = 0$. Then for any local martingale $M:\mathbb R_+ \times \Omega \to X$ with the canonical decomposition $M = M^c + M^q + M^a$ we have that
 \begin{equation}\label{eq:candecphiineqforsled}
    \mathbb E \phi(M^*) \eqsim_{\phi, X} \mathbb E \phi\bigl((M^c)^*\bigr) + \mathbb E \phi\bigl((M^q)^*\bigr) +\mathbb E \phi\bigl((M^a)^*\bigr). 
 \end{equation}
\end{proposition}

\begin{proof}
Inequality $\lesssim_{\phi, X}$ of \eqref{eq:candecphiineqforsled} follows from the fact that $M = M^c + M^q + M^a$ a.s., so $M^* \leq (M^c)^* + (M^q)^* + (M^a)^*$ a.s., and the fact that $\phi$ has moderate growth, so a.s.\
$$
\phi(M^*) \leq \phi\bigl( (M^c)^* + (M^q)^* + (M^a)^*\bigr) \eqsim_{\phi}  \phi\bigl( (M^c)^*\bigr) +  \phi\bigl( (M^q)^*\bigr)  +  \phi\bigl(  (M^a)^*\bigr).
$$
Let us  show $\gtrsim_{\phi, X}$ of \eqref{eq:candecphiineqforsled}.
 As $X$ is UMD, each of $M$, $M^c$, $M^q$, and $M^a$ has a covariation bilinear form $[\![M]\!]$, $[\![M^c]\!]$, $[\![M^q]\!]$, and $[\![M^a]\!]$ respectively (see Remark \ref{rem:ifUMDthencovbilform}). Moreover, by \cite[Subsection 7.6]{Y18BDG} we have that $[\![M]\!]=[\![M^c]\!]+[\![M^q]\!]+[\![M^a]\!]$ a.s., and thus by \cite[Subsection 3.2 and Section 5]{Y18BDG} (see also Remark \ref{rem:ifUMDthencovbilform}) for any $i\in \{c, q, a\}$
 \[
    \mathbb E \phi(M^*) \eqsim_{\phi, X} \mathbb E \phi\bigl(\gamma([\![M]\!]_{\infty})\bigr) \geq \mathbb E \phi\bigl(\gamma([\![M^i]\!]_{\infty})\bigr) \eqsim_{\phi, X} \mathbb E \phi\bigl((M^i)^*\bigr).
 \]
This terminates the proof.
\end{proof}

Fix $\phi:\mathbb R_+ \to \mathbb R_+$ convex of moderate growth such that $\phi(0) = 0$. 

\begin{theorem}\label{thm:Ephifordiscretetantdandee}
Let $X$ be a UMD Banach space, $(d_n)_{n\geq 1}$ and $(e_n)_{n\geq 1}$ be tangent martingale difference sequences. Then we have that
\[
\mathbb E \phi\left( \sup_{n\geq 1} \Bigl\| \sum_{k=1}^n d_n \Bigr\|\right) \eqsim_{\phi, X}\mathbb E \phi\left( \sup_{n\geq 1} \Bigl\| \sum_{k=1}^n e_n\Bigr\|\right). 
\]
\end{theorem}

\begin{proof}
Let $(r_n)_{n\geq 1}$ be a sequence of independent Rademachers (see Definition \ref{def:ofRadRV}). Then by \cite[(8.22)]{Burk86} and by \cite[Section 2]{Y18BDG} we have that
\begin{equation}\label{eq:makingtangMDSsymnme}
 \begin{split}
  \mathbb E \phi\left( \sup_{n\geq 1} \Bigl\| \sum_{k=1}^n r_nd_n\Bigr\|\right) &\eqsim_{\phi, X}\mathbb E \phi\left( \sup_{n\geq 1} \Bigl\| \sum_{k=1}^n d_n\Bigr\|\right),\\
  \mathbb E \phi\left( \sup_{n\geq 1} \Bigl\| \sum_{k=1}^n r_ne_n\Bigr\|\right) &\eqsim_{\phi, X}\mathbb E \phi\left( \sup_{n\geq 1} \Bigl\| \sum_{k=1}^n e_n\Bigr\|\right). 
 \end{split}
\end{equation}
Finally, $(r_nd_n)_{n\geq 1}$ and $(r_n e_n)_{n\geq 1}$ are tangent martingale difference sequences with respect to an enlarged filtration $\overline {\mathbb F} = (\overline{\mathcal F}_n)_{n\geq 1}$ which is generated by the original filtration $({\mathcal F}_n)_{n\geq 1}$ and by Rademachers $(d_n)_{n\geq 1}$ as for any $n\geq 1$ and for any Borel set $A \in \mathcal B(X)$
\begin{equation}\label{eq:rndnandrnenaretagicdnenadfosim}
 \begin{split}
   \mathbb P(r_n d_n|\overline{\mathcal F}_{n-1})(A) &= \mathbb E (\mathbf 1_A (r_n d_n) | \overline{\mathcal F}_{n-1})\\
   &\stackrel{(i)}= \tfrac{1}{2}\mathbb E (\mathbf 1_A ( d_n) | \overline{\mathcal F}_{n-1}) + \tfrac{1}{2}\mathbb E (\mathbf 1_{-A} ( d_n) | \overline{\mathcal F}_{n-1})\\
 &\stackrel{(ii)}=\tfrac{1}{2}\mathbb E (\mathbf 1_A ( d_n) | {\mathcal F}_{n-1}) + \tfrac{1}{2}\mathbb E (\mathbf 1_{-A} ( d_n) | {\mathcal F}_{n-1})\\
  &\stackrel{(iii)}=\tfrac{1}{2}\mathbb E (\mathbf 1_A ( e_n) | {\mathcal F}_{n-1}) + \tfrac{1}{2}\mathbb E (\mathbf 1_{-A} ( e_n) | {\mathcal F}_{n-1})\\
  &\stackrel{(iv)}=  \mathbb P(r_n e_n|\overline{\mathcal F}_{n-1})(A),
 \end{split}
\end{equation}
where $(i)$ follows from the fact that $r_n$ is independent of $d_n$ and $\overline{\mathcal F}_{n-1}$, $(ii)$ follows from the fact that $d_n$ is independent of $\sigma(r_1, \ldots, r_{n-1})$, $(iii)$ holds as $(d_n)_{n\geq 1}$ and $(e_n)_{n\geq 1}$ are tangent, and finally $(iv)$ holds as $(i)$, $(ii)$, and $(iii)$ can analogously be shown for $e_n$.
Moreover, $r_nd_n$ and $r_ne_n$ are conditionally symmetric given $\mathcal F_{n-1}$ for any $n\geq 1$, so we have that
\begin{equation}\label{eq:phifosymmetrizedMDSprosh}
 \mathbb E \phi\left( \sup_{n\geq 1} \Bigl\| \sum_{k=1}^n r_nd_n\Bigr\|\right) \eqsim_{\phi, X}\mathbb E \phi\left( \sup_{n\geq 1} \Bigl\| \sum_{k=1}^n r_ne_n\Bigr\|\right)
\end{equation}
 by \cite{HitUP} (see \cite[pp.\ 424--425]{CV07}). The desired follows from \eqref{eq:makingtangMDSsymnme} and~\eqref{eq:phifosymmetrizedMDSprosh}.
\end{proof}

\begin{theorem}\label{thm:MNconttantgethenphiinrq}
Let $X$ be a UMD Banach space, $M, N:\mathbb R_+ \times \Omega \to X$ be tangent continuous local martingales. Then $\mathbb E \phi(M^*) \eqsim_{\phi, X} \mathbb E \phi(N^*)$.
\end{theorem}

\begin{proof}
The proof follows directly from Lemma \ref{lem:goodlgivesphi} and Proposition \ref{prop:goodlforconttanglocalmart}.
\end{proof}

\begin{theorem}\label{thm:tangentphiPDQLCtratata}
Let $X$ be a UMD Banach space, $M, N:\mathbb R_+ \times \Omega \to X$ be tangent purely discontinuous quasi-left continuous local martingales. Then $\mathbb E \phi(M^*) \eqsim_{\phi, X} \mathbb E \phi(N^*)$.
\end{theorem}

For the proof of the theorem we will need the following lemma. 

\begin{lemma}\label{lem:MNtangentthenthejumpsaresmall}
Let $X$ be a Banach space, $M, N:\mathbb R_+ \times \Omega \to X$ be tangent local martingales. Then
\begin{equation}\label{eq:probofsupofjumpsoftang}
\mathbb P(\Delta N^* \geq \lambda) \leq 6 \mathbb P(\Delta M^* \geq \lambda),\;\;\; \lambda>0,
\end{equation}
where we set $\Delta M^* := \sup_{t\geq 0} \|\Delta M_t\|$  and $\Delta N^* := \sup_{t\geq 0} \|\Delta N_t\|$.
\end{lemma}

\begin{proof}
By a standard restriction to finite dimensions argument (see e.g.\ the proof of \cite[Theorem 3.3]{Y17FourUMD}) and by the fact that $AM$ and $AN$ ate tangent for any linear operator $A \in \mathcal L(X, Y)$ (see Theorem \ref{thm:MNtangent==>TMTNaretangentforanylinop}) we may assume that $X$ is finite dimensional.
Due to Theorem \ref{thm:candecfortangaretang} we may assume that both $M$ and $N$ are purely discontinuous. Let $M = M^q + M^a$ and $N=N^q + N^a$ be the canonical decompositions of $M$ and $N$. Then by \eqref{eq:candecsplitsjumps} we have that a.s.\
\[
\{ t\geq 0: \Delta M_t \neq 0\} = \{ t\geq 0: \Delta M^q_t \neq 0\} \cup \{ t\geq 0: \Delta M^a_t \neq 0\},
\]
\[
\{ t\geq 0: \Delta N_t \neq 0\} = \{ t\geq 0: \Delta N^q_t \neq 0\} \cup \{ t\geq 0: \Delta N^a_t \neq 0\},
\]
and
\[
 \{ t\geq 0: \Delta M^q_t \neq 0\} \cap \{ t\geq 0: \Delta M^a_t \neq 0\} = \varnothing,
\]
\[
\{ t\geq 0: \Delta N^q_t \neq 0\} \cap \{ t\geq 0: \Delta N^a_t \neq 0\} = \varnothing.
\]
Thus in order to show \eqref{eq:probofsupofjumpsoftang} it is sufficient to prove that
\begin{equation}\label{eq:probofsupofjumpsoftangQLC}
\mathbb P\bigl((\Delta N^q)^* \geq \lambda \bigr) \leq 4 \mathbb P\bigl((\Delta M^q)^* \geq \lambda \bigr),\;\;\; \lambda>0,
\end{equation}
\begin{equation}\label{eq:probofsupofjumpsoftangAJ}
\mathbb P\bigl((\Delta N^a)^* \geq \lambda \bigr) \leq 2 \mathbb P\bigl((\Delta M^a)^* \geq \lambda \bigr),\;\;\; \lambda>0.
\end{equation}
First notice that \eqref{eq:probofsupofjumpsoftangAJ} follows from a standard discrete approximation of purely discontinuous martingales with accessible jumps (see e.g.\ the proof of Proposition \ref{prop:pdwajjusttang} and Subsection \ref{subsec:appforMarAppPDMAJ}) and \cite[Lemma 2.3.3]{dlPG}. Let us show \eqref{eq:probofsupofjumpsoftangQLC}. Without loss of generality we may assume that $M+M^T$ and $N = N^T$ for some foxed $T>0$. Let us define for any $n\geq 1$
\[
d_n^k :=M^q_{Tk/n} - M^q_{T(k-1)/n},\;\;\;e_n^k :=N^q_{Tk/n} - N^q_{T(k-1)/n}, \;\;\; k=1,\ldots, n.
\]
For each $n\geq 1$, let $(\tilde d_n^k)_{k=1}^n$ be a decoupled tangent sequence of $(d_n^k)_{k=1}^n$ and $(\tilde e_n^k)_{k=1}^n$ be a decoupled tangent sequence of $(e_n^k)_{k=1}^n$.
Then by \cite[Lemma 2.3.3]{dlPG} we have that
\begin{equation}\label{eq:stepapproxandfortangent}
\begin{split}
\mathbb P\Bigl(\sup_{k=1}^n\|e_{k}^n\|\geq \lambda \Bigr) &\leq 2 \mathbb P\Bigl(\sup_{k=1}^n\| \tilde e_{k}^n\|\geq \lambda \Bigr),\;\;\;\;\; \lambda>0, \\
\mathbb P\Bigl(\sup_{k=1}^n\|\tilde d_{k}^n\|\geq \lambda \Bigr) &\leq 2 \mathbb P\Bigl(\sup_{k=1}^n\|  d_{k}^n\|\geq \lambda \Bigr),\;\;\;\;\; \lambda>0.
\end{split}
\end{equation}
Let $\widetilde M^q$ be a local martingale decoupled tangent to both $M^q$ and $N^q$. As $M^q$, $N^q$, and $\widetilde M^q$ have c\`adl\`ag trajectories (see Subsection \ref{subsec:BSvmart}), we have the following convergences
\[
\mathbb P-\lim_{n\to \infty}\sup_{k=1}^n\|  d_{k}^n\| = \sup_{0\leq t\leq T}\|\Delta M^q_t\|, \;\;\; \mathbb P-\lim_{n\to \infty}\sup_{k=1}^n\|  e_{k}^n\|=  \sup_{0\leq t\leq T}\|\Delta N^q_t\|,
\]
\[
\mathbb P\Bigl(\sup_{k=1}^n\| \tilde  d_{k}^n\|>\lambda\Bigr) , \mathbb P\Bigl(\sup_{k=1}^n\| \tilde e_{k}^n\|>\lambda\Bigr) \to \mathbb P\Bigl( \sup_{0\leq t\leq T}\|\Delta \widetilde M_t\|>\lambda\Bigr),\;\;\; n\to \infty,\;\; \lambda >0,
\]
where the latter follows from Theorem \ref{thm:condtrstepsnsidtrtothepledone}; thus by \eqref{eq:stepapproxandfortangent} we have that
\[
\mathbb P\bigl((\Delta N^q)^* \geq \lambda \bigr)\leq 2 \mathbb P\bigl((\Delta \widetilde M^q)^* \geq \lambda \bigr) \leq 4 \mathbb P\bigl((\Delta M^q)^* \geq \lambda \bigr),\;\;\; \lambda>0,
\]
so \eqref{eq:probofsupofjumpsoftangQLC} (and consequently \eqref{eq:probofsupofjumpsoftang}) follows.
\end{proof}

\begin{proof}[Proof of Theorem \ref{thm:tangentphiPDQLCtratata}]
First we prove the conditional symmetric case, and then the general case.

{\em Step 1: conditionally symmetric case.} Let $M$ and $N$ be conditionally symmetric. Then
by Lemma \ref{lem:goodlgivesphi} and Proposition \ref{prop:goodlambdapdQlc}  we have that 
\[
\mathbb E \phi(N^*) \lesssim_{\phi, X} \mathbb E \phi( \Delta M^* \vee\Delta N^* \vee M^*).
\]
As $\phi$ has a moderate growth, we have that
\[
 \mathbb E \phi( \Delta M^* \vee\Delta N^* \vee M^*) \eqsim_{\phi}  \mathbb E \phi( \Delta M^*) +  \mathbb E \phi(\Delta N^*) +  \mathbb E \phi(M^*),
\]
where
\begin{equation}\label{eq:phiDeltaM*leqphiM*}
  \mathbb E \phi( \Delta M^*) \lesssim_{\phi}  \mathbb E \phi( M^*),
\end{equation}
as $ \Delta M^* \leq 2 M^*$, and 
$$
\mathbb E \phi( \Delta N^*) \leq 6\mathbb E \phi( \Delta M^*) \lesssim_{\phi}  \mathbb E \phi( M^*),
$$ 
by Lemma \ref{lem:MNtangentthenthejumpsaresmall}, since $\mathbb E \phi(\xi) = \int_{\mathbb R_+} \mathbb P(\xi>\lambda) \ud \phi(\lambda)$ for any random variable $\xi:\Omega \to \mathbb R_+$ and since $\phi(0)=0$, and by \eqref{eq:phiDeltaM*leqphiM*}. Thus we have that $ \mathbb E \phi(N^*) \lesssim_{\phi}  \mathbb E \phi(M^*)$; the converse follows similarly.

{\em Step 2: general case.} First of all, it is sufficient to assume that $N$ is a decoupled tangent martingale to $M$. Let $N'$ be another decoupled tangent martingale to $M$ conditionally independent of $N$ given $\mathcal F$. Then $M-N'$ and $N-N'$ are tangent martingales which are conditionally symmetric, and thus 
$$
\mathbb E \phi(M^*)\stackrel{(i)} \leq\mathbb E \phi\bigl((M-N')^*\bigr) \stackrel{(ii)}\lesssim_{\phi, X} \mathbb E \phi\bigl((N-N')^*\bigr) \stackrel{(iii)}\lesssim_{\phi} \mathbb E \phi(N^*),
$$
where $(i)$ holds by the fact that a conditional expectation is a contraction and by the fact that $\phi$ is convex, $(ii)$ follows from Step 1, and $(iii)$ follows by the fact that $\phi$ is convex of moderate growth and that  $N$ and $N'$ are conditionally independent given $\mathcal F$ and equidistributed. 

\smallskip

Let us show that
\begin{equation}\label{eq:EphiN*leqEphiM*forQLCPDMART}
\mathbb E \phi(N^*)\lesssim_{\phi, X}  \mathbb E \phi(M^*).
\end{equation} 
Without loss of generality by the dominated convergence theorem we may assume that $M_t = M_T$ and $N_t= N_T$ for some fixed $T>0$ and any $t\geq T$. By Theorem \ref{thm:condtrstepsnsidtrtothepledone} there exist pure jump processes $(M^n)_{n\geq 1}$ and $(N^n)_{n\geq 1}$ such that
\begin{enumerate}[\rm (A)]
\item for each $n\geq 1$, $M^n$ and $N^n$ have jumps at $\{\tfrac{T}{n},\ldots,\tfrac{T(n-1)}{n}, T\}$,
\item for each $n\geq 1$, $(M^n_{Tk/n} - M^n_{T(k-1)/n})_{k=1}^n$ and $(N^n_{Tk/n} - N^n_{T(k-1)/n})_{k=1}^n$ are martingale difference sequences with respect to the enlarged filtration $(\overline{\mathcal F}_{Tk/n})_{k=1}^n$ (which enlarges $({\mathcal F}_{Tk/n})_{k=1}^n$) such that $(N^n_{Tk/n} - N^n_{T(k-1)/n})_{k=1}^n$ is a decoupled tangent martingale difference sequence to $(M^n_{Tk/n} - M^n_{T(k-1)/n})_{k=1}^n$,
\item $N^n$ converges to $N$ in distribution as random variables with values in the Skorokhod space $\mathcal D([0,T], X)$ as $n\to \infty$,
\item $M^n$ converges to $M$ a.s.\ as $n\to \infty$, and, moreover, $(M^n)^* \nearrow M^*$ a.s.
\end{enumerate}
By $\rm (B)$ and Theorem \ref{thm:Ephifordiscretetantdandee} we have that
\begin{equation}\label{eq:EphiNn*eqEphiMn*forQLCPDMART}
\mathbb E \phi\bigl((M^n)^*\bigr) \eqsim_{\phi, X}\mathbb E \phi\bigl((N^n)^*\bigr),
\end{equation}
for any $n\geq 1$. On the other hand we have that $\mathbb E \phi\bigl((M^n)^*\bigr) \nearrow \mathbb E \phi(M^*) $ by the dominated convergence theorem and $\rm (D)$. Therefore \eqref{eq:EphiN*leqEphiM*forQLCPDMART} follows from \eqref{eq:EphiNn*eqEphiMn*forQLCPDMART}, $\rm (C)$, and Fatou's lemma.
\end{proof}

\begin{remark}\label{rem:predboundjumpsforgoodL}
Note that if $M$ and $N$ have predictably bounded jumps, i.e.\ there exists a predictable increasing process $A:\mathbb R_+ \times \Omega \to \mathbb R_+$ such that $\|\Delta M_t\|, \|\Delta N_t\| \leq A_t$ a.s.\ for any $t\geq 0$, then there is no need in conditional symmetry in the proof of  Proposition \ref{prop:goodlambdapdQlc}, and hence there is no need in  using Section \ref{sec:appJKW}  in order to prove Theorem \ref{thm:tangentphiPDQLCtratata} (see e.g.\ the proof of Proposition \ref{prop:XUMDiffdecouplforPois}).
\end{remark}

Let us eventually show Theorem \ref{thm:XUMDiffphimaxtangent}. By Proposition \ref{prop:candecandphiarecorrespondest} it is sufficient to show that
\begin{equation}\label{eq:phiMcNcasffda}
\mathbb E \phi\bigl((M^c)^*\bigr) \eqsim_{\phi, X}\mathbb E \phi\bigl((N^c)^*\bigr),
\end{equation}
\begin{equation}\label{eq:phiMqNqasffda}
\mathbb E \phi\bigl((M^q)^*\bigr) \eqsim_{\phi, X}\mathbb E \phi\bigl((N^q)^*\bigr),
\end{equation}
\begin{equation}\label{eq:phiMaNaasffda}
\mathbb E \phi\bigl((M^a)^*\bigr) \eqsim_{\phi, X}\mathbb E \phi\bigl((N^a)^*\bigr).
\end{equation}
The inequality \eqref{eq:phiMcNcasffda} follows from Theorem \ref{thm:candecfortangaretang} and \ref{thm:MNconttantgethenphiinrq}, \eqref{eq:phiMqNqasffda} follows from Theorem \ref{thm:candecfortangaretang} and \ref{thm:tangentphiPDQLCtratata}, and finally \eqref{eq:phiMaNaasffda} holds by Theorem \ref{thm:candecfortangaretang} and \ref{thm:Ephifordiscretetantdandee}, and the approximation argument from the proof of Proposition \ref{prop:pdwajjusttang} and \ref{prop:MmapproxMinLpforacccase}.

\subsection{Not convex functions}

What is of a big interest is whether it is possible to have an analogue of Theorem \ref{thm:XUMDiffphimaxtangent} for a general $\phi$ of moderate growth (e.g.\ $\phi(t) = \sqrt t$), as it was done in the conditionally symmetric case in \cite[Theorem 4.1]{Kal17}. In our case this is possible due to the following theorem.

\begin{theorem}\label{thm:gencondeemartgenphimosdopfew}
 Let $X$ be a UMD Banach space, $\phi:\mathbb R_+ \to \mathbb R_+$ be an increasing function of moderate growth such that $\phi(0)=0$. Then for any tangent conditionally symmetric martingales $M, N:\mathbb R_+ \times \Omega \to X$ we have that
 \[
  \mathbb E \phi(M^*) \eqsim_{\phi, X} \phi(N^*).
 \]
\end{theorem}

Note that $M$ is conditionally symmetric if and only if $M^d$ is conditionally symmetric, where $M=M^c+M^d$ is the Meyer-Yoeurp decomposition of $M$ (see Remark \ref{rem:MYdecBanach}).

\begin{proof}[Proof of Theorem \ref{thm:gencondeemartgenphimosdopfew}]
 First note that one can show Proposition \ref{prop:goodlambdapdQlc} and Lemma \ref{lem:smallerjumpsonlyforgoodl} for general conditionally symmetric $M$ and $N$ (with modifying $\widehat M$ and $\widehat N$ by adding to them $M^c_{\tau\wedge \rho\wedge t} - M^c_{\tau\wedge\sigma\wedge  \rho}$ and $N^c_{\tau\wedge \rho\wedge t} - N^c_{\tau\wedge\sigma\wedge  \rho}$ respectively). Then thanks to \eqref{eq:goodlforstochintwwrtrm} we get that for any $\lambda>0$, $\delta>0$, and $\beta>1+\delta$
 \begin{align*}
    \mathbb P(N^*>\beta\lambda)& - \mathbb P( \Delta M^* \vee\Delta N^* \vee M^* > \delta \lambda)\\
    &\leq \mathbb P(N^*>\beta\lambda, \Delta M^* \vee\Delta N^* \vee M^* \leq \delta \lambda) \lesssim_{p,X}\frac{ \delta^p}{(\beta-\delta - 1)^p}\mathbb P(N^*>\lambda),
 \end{align*}
so by fixing $p\geq 1$, $\delta=2$, and $\beta>4$ we derive for some fixed $C_{p,X}>0$ by \eqref{eq:probofsupofjumpsoftang}
\begin{align*}
 \mathbb P(N^*>\beta\lambda) &\leq C_{p,X} \frac{2}{(\beta-3)^p} \mathbb P(N^*>\lambda) + \mathbb P( \Delta M^* >2\lambda) + \mathbb P(\Delta N^* >2\lambda) + \mathbb P( M^* > 2\lambda)\\
 &\leq C_{p,X} \frac{2}{(\beta-3)^{p}} \mathbb P(N^*>\lambda) + 8\mathbb P( M^* > \lambda).
\end{align*}
Then
\begin{equation}\label{eq:phiN*formviaPN*fncik}
 \mathbb E \phi(N^*) = \int_{\mathbb R_+} \mathbb P(\phi(N^*)>\lambda) \ud \lambda = \int_{\mathbb R_+} \mathbb P(N^*>\lambda) (\phi^{-1})'(\lambda) \ud \lambda,
\end{equation}
consequently in particular if $\phi(ex) \leq \alpha \phi(x)$ for some $\alpha>1$ (hence $\phi(\beta x) \leq \beta^{\ln \alpha +1} \phi(x)$ for $\beta$ big enough), which holds as $\phi$ is of moderate growth, then analogously to \cite[p.\ 38]{Kal17} by using the fact that $\phi$ is increasing so $(\phi^{-1})'$ is nonnegative a.s.\ on $\mathbb R_+$
\begin{align*}
 \Bigl(\beta^{-\ln \alpha-1} &-   C_{p,X} \frac{2}{(\beta-3)^p}  \Bigr) \mathbb E \phi(N^*) \leq E \phi(N^*/\beta) -  C_{p,X} \frac{2}{(\beta-3)^p}   \mathbb E \phi(N^*)\\
 &=  \int_{\mathbb R_+} \Bigl(\mathbb P(N^*/\beta>\lambda) -C_{p,X} \frac{2}{(\beta-3)^p}   \mathbb  P(N^*>\lambda)\Bigr) (\phi^{-1})'(\lambda) \ud \lambda\\
 &\leq  \int_{\mathbb R_+} \mathbb P(M^*>\lambda) (\phi^{-1})'(\lambda) \ud \lambda = \mathbb E \phi(M^*),
\end{align*}
so the desired follows by choosing $p>\ln \alpha +2$ and $\beta$ big enough.
\end{proof}

Unfortunately the author does not know whether Theorem \ref{thm:gencondeemartgenphimosdopfew} holds for general martingales. Nonetheless, the following upper estimate can be shown.

\begin{theorem}
 Let $X$ be a UMD Banach space, $\phi:\mathbb R_+ \to \mathbb R_+$ be an increasing function of moderate growth such that $\phi(0)=0$. Then for any local martingale $M:\mathbb R_+ \times \Omega \to X$ with a decoupled tangent local martingale $N$ we have that
 \[
  \mathbb E \phi(M^*) \lesssim_{\phi, X} \mathbb E  \phi(N^*).
 \]
\end{theorem}

\begin{proof}
 Let $\widetilde N$ be another decoupled tangent local martingale to $M$ which is conditionally independent of $N$ given the local characteristics of $M$. Then by Theorem \ref{thm:gencondeemartgenphimosdopfew} we have that 
 $$
 \mathbb E \phi\bigl((M-\widetilde N)^*\bigr) \eqsim_{\phi, X} \mathbb E\phi\bigl((N-\widetilde N)^*\bigr).
 $$
 It remains to notice that $\mathbb P(M^* \geq \lambda) \leq \mathbb P\bigl((M-\widetilde N)^* \geq \lambda/2\bigr) + \mathbb P\bigl((\widetilde N)^* \geq \lambda/2\bigr)$, so $\mathbb E \phi(M^*) \lesssim_{\phi} \mathbb E \phi((M-\widetilde N)^*) + \mathbb E \phi(N^*)$ by \eqref{eq:phiN*formviaPN*fncik}, the fact that $\phi$ has a moderate growth, and due to the equidistribution of $N$ and $\widetilde N$, and to note that by the fact that $N-\widetilde N$ has increments which are independent symmetric  given the local characteristics of $M$ we have that thanks to \eqref{eq:phiN*formviaPN*fncik} and \cite[Proposition 6.1.12]{HNVW2}
 \[
  \mathbb E\phi\bigl((N-\widetilde N)^*\bigr) \eqsim \mathbb E\phi(\|N_{\infty}-\widetilde N_{\infty}\|) \stackrel{(*)}\lesssim_{\phi}  \mathbb E\phi(\|N_{\infty}\|) \lesssim \mathbb E\phi(N_{\infty}^*)
 \]
where $(*)$ follows from \eqref{eq:phiN*formviaPN*fncik}, the fact that $\mathbb P(\|N_{\infty}-\widetilde N_{\infty}\| >\lambda) \leq \mathbb P(\|N_{\infty}\| >\lambda/2) + \mathbb P(\|\widetilde N_{\infty}\| >\lambda/2)$, and the fact that $\phi$ has a moderate growth.
\end{proof}

\section{Integration with respect to a general martingale}\label{sec:intwrtgenmart}

The present section is devoted to new estimates for stochastic integrals with values in UMD Banach spaces. These are so-called {\em predictable estimates} as we will have a {\em predictable} process on the right-hand side since this process depends only on the corresponding local characteristics and thus it is  predictable. In particular, these estimates extend sharp bounds for a stochastic integral with respect to a cylindrical Brownian motion obtained by van Neerven, Veraar, and Weis in \cite{NVW,NVW15} (see also \cite{Ver,VY16} for continuous martingale case). On the other hand, this section in some sense extends a recent work \cite{DY17} by Dirksen and the author on stochastic integration in $L^q$-spaces, though the latter publication provides precise formulas for the right-hand side of \eqref{eq:stochintestwithpredRHS}, i.e.\ formulas that do not depend on the decoupled tangent martingale or the corresponding Cox process, but only on $\nu^M$. We also wish to note that the obtained below estimates are very different from those proven in \cite[Subsection 7.1]{Y18BDG}: estimates \eqref{eq:stochintestwithpredRHS} are more in the spirit of  works of Novikov \cite{Nov75}, Burkholder \cite{Bur73}, and Rosenthal \cite{Ros70}, while \cite[Subsection 7.1]{Y18BDG} is based on Burkholder-Davis-Gundy inequalities, which are similar to square function estimates (see e.g.\ also \cite{VY18}).

\begin{theorem}
Let $H$ be a Hilbert space, $X$ be a UMD Banach space, $\widetilde  M:\mathbb R_+\times \Omega \to H$ be a local martingale. Then for any elementary predictable $\Phi:\mathbb R_+ \times \Omega \to \mathcal L(H, X)$ and for any $1\leq p<\infty$ one has that
\begin{align}\label{eq:stochintestwithpredRHS}
 \mathbb E \sup_{t\geq 0} \Bigl\| \int_0^t \Phi \ud \widetilde M \Bigr\|^p &\eqsim_{p, X} \mathbb E \|\Phi q_{\widetilde  M^c}^{1/2}\|_{\gamma(L^2(\mathbb R_+, [ M^c]; H), X)}^p\nonumber\\ 
 &\quad\quad \quad+ \mathbb E \mathbb E_{\rm Cox} \Bigl\|\int_{\mathbb R_+ \times H}\Phi(s) h \ud \bar \mu^{\widetilde M^q}_{\rm Cox}(s, h) \Bigr\|^p\\
 &\quad \quad \quad \quad\quad \quad+ \mathbb E \mathbb E_{\rm CI} \Bigl\|\sum_{0\leq t < \infty} \Phi \Delta\widetilde  N_t^{a}\Bigr\|^p,\nonumber
\end{align}
where $\widetilde  M=\widetilde   M^c + \widetilde  M^q + \widetilde  M^a$ is the canonical decomposition, $q_{\widetilde M^c}$ is as defined in Subsection \ref{subsec:quadrvar}, $\mu^{\widetilde M^q}_{\rm Cox}$ is a Cox process directed by $\nu^{M^q}$, and $\widetilde N^{a}$ is a decoupled tangent martingale to $\widetilde  M^a$ constructed in Theorem \ref{thm:detangmartforMXVpdwithaccjumps}, while $\mathbb E_{\rm Cox}$ denotes an expectation for a fixed $\omega\in \Omega$ in a probability space generated by $\mu_{\rm Cox}$, and $\mathbb E_{\rm CI}$ denotes an expectation for a fixed $\omega \in \Omega$ in a probability space generated by $\widetilde N^a$ (see Subsection~\ref{subsec:ConexponPSCondProbCondIndep}).
\end{theorem}

\begin{proof}
The theorem follows from Theorem \ref{thm:candecXvalued}, the continuous case \cite[Example 3.19 and Theorem 4.1]{VY16}, Theorem \ref{thm:mumuCoxcomparable}, and Proposition \ref{prop:pdwajjusttang}.
\end{proof}

\begin{remark}
 As on both the right- and the left-hand sides of \eqref{eq:stochintestwithpredRHS} we have norms (strictly speaking, seminorms, but we can consider a quotient space and make these expressions norms), analogously to \cite[Subsection 7.1]{Y18BDG} we can extend the definition of a stochastic integral to {\em any} strongly predictable $\Phi:\mathbb R_+ \times \Omega \to X$ with
 \begin{multline*}
  \mathbb E \|\Phi q_{\widetilde M^c}^{1/2}\|_{\gamma(L^2(\mathbb R_+, [\widetilde  M^c]; H), X)}+ \mathbb E \mathbb E_{\rm Cox} \Bigl\|\int_{\mathbb R_+ \times H}\Phi(s) h \ud \bar \mu^{\widetilde M^q}_{\rm Cox}(s, h) \Bigr\|\\
  + \mathbb E \mathbb E_{\rm CI} \Bigl\|\sum_{0\leq t < \infty} \Phi \Delta\widetilde  N_t^{a}\Bigr\|<\infty.
 \end{multline*}
\end{remark}

\begin{remark}
 Due to standard Lenglart's trick \cite[Corollaire II]{Len77} we can extend the upper bounds of \eqref{eq:stochintestwithpredRHS} to $p\in (0,1)$ in the following way
\begin{align}\label{eq:stochintestwithpredRHSpin01}
 \mathbb E \sup_{t\geq 0} \Bigl\| \int_0^t \Phi \ud \widetilde M \Bigr\|^p &\lesssim_{p, X} \mathbb E \|\Phi q_{\widetilde  M^c}^{1/2}\|_{\gamma(L^2(\mathbb R_+, [\widetilde  M^c]; H), X)}^p\nonumber\\ 
 &\quad\quad \quad+ \mathbb E \Bigl (\mathbb E_{\rm Cox} \Bigl\|\int_{\mathbb R_+ \times H}\Phi(s) h \ud \bar \mu^{\widetilde M^q}_{\rm Cox}(s, h) \Bigr\|\Bigr)^p\\
 &\quad \quad \quad \quad\quad \quad+ \mathbb E \Bigl(\mathbb E_{\rm CI} \Bigl\|\sum_{0\leq t < \infty} \Phi \Delta\widetilde  N_t^{a}\Bigr\|\Bigr)^p.\nonumber
\end{align}
\end{remark}

\begin{remark}
 Why expressions on the right-hand sides of \eqref{eq:stochintestwithpredRHS} and \eqref{eq:stochintestwithpredRHSpin01} can be useful? First, if one fixes $\omega\in \Omega$, then these expressions become stochastic integrals with respect to martingales with independent increments, which it is easier to work with. Second, if we are in the quasi-left continuous setting (i.e.\ $M^a=0$ and we have only Poisson-like jumps), then we end up with $\gamma$-norms and the norms generated by Cox processes, which might be of $\gamma$-radonifying nature but with the Poisson distribution (see Remark \ref{rem:CoxidlikePoisson}).
\end{remark}

\begin{remark}\label{rem:stochinforgenmartfordecpropet}
 Thanks to Theorem \ref{thm:deopropertandLpboubdsfordecoupltangma} both \eqref{eq:stochintestwithpredRHSpin01} and the upper bound of \eqref{eq:stochintestwithpredRHS} hold true if $X$ has the decoupling property.
\end{remark}

\section{Weak tangency versus tangency}\label{sec:WTversusT}

The natural question is raising up while working with local characteristics in infinite dimensions: given a Banach space $X$ (perhaps, not UMD) and an \mbox{$X$-va}\-lued martingale $M$. {\em Can we have results of the form \eqref{eq:Lpboundsforgentangmarraz} for more general Banach spaces by using a family of local characteristics $([\langle M, x^* \rangle^c], \nu^{\langle M, x^* \rangle})_{x^*\in X^*}$ instead of local characteristics discovered in Section \ref{sec:tangmartconttime} (note that the latter even might not exist by Theorem \ref{thm:tangentgencaseUMDuhoditrotasoldat})? And how do these weak local characteristics correspond to the those defined in Section \ref{sec:tangmartconttime}?} Let us answer these questions. First we will need the following definitions.

\begin{definition}
Let $X$ be a Banach space, $M:\mathbb R_+ \times \Omega \to X$ be local martingale. Then the family $([\langle M, x^* \rangle^c], \nu^{\langle M, x^* \rangle})_{x^*\in X^*}$ is called {\em weak local characteristics}.
\end{definition}

\begin{definition}
Let $X$ be a Banach space, $M, N:\mathbb R_+ \times \Omega \to X$. Then $M$ and $N$ are {\em weakly tangent} if $\langle M, x^*\rangle$ and $\langle N, x^*\rangle$ are tangent for any $x^*\in X^*$, i.e.\ if $M$ and $N$ have the same weak local characteristics.
\end{definition}

Here we show that weak tangency coincides with tangency in the UMD case, so this approach cannot extend Theorem \ref{thm:tangentgencaseUMDuhoditrotasoldat} in the UMD setting.

\begin{theorem}\label{thm:tangiffwtangUMDcase}
Let $X$ be a Banach space, $M, N:\mathbb R_+ \times \Omega \to X$ be local martingales which have local characteristics. Then $M$ and $N$ are tangent if and only if they are weakly tangent. 
\end{theorem}

For the proof we will need the following lemma.

\begin{lemma}\label{lem:[[M]]isweka*consda}
 Let $X$ be a Banach space, $M:\mathbb R_+ \times \Omega \to X$ be a martingale. Assume that $M$ has a covariation bilinear form $[\![M]\!]$. Then for any $t\geq 0$ we have that
 \[                                                                                                                                                                   
X^* \to L^0(\Omega),\;\;\; x^* \mapsto  [\![M]\!]_t(x^*, x^*),                                                                                                                                                                               \]
is continuous for $X^*$ endowed with the weak$^*$ topology.
\end{lemma}

\begin{proof}
 By a stopping time argument and by Remark \ref{rem:locmartislocallyL1DR+X} we may assume that $\mathbb E \sup_{t\geq 0} \|M_t\|<\infty$. Let $(x_n^*)_{n\geq 1}$ be a weak$^*$ Cauchy sequence with the limit $x^*$. By the definition of weak$^*$ convergence we have that $\langle x, x_n^*\rangle \to \langle x, x^*\rangle$ for any $x\in X$. Thus by \cite[Theorem 26.6]{Kal} a.s.\
 \begin{multline*}
  \bigl|[\![M]\!]_t(x^*, x^*) - [\![M]\!]_t(x^*_n, x^*_n) \bigr|= \bigl|[\langle M ,x^*\rangle]_t - [\langle M, x_n^*\rangle]_t\bigr|\\
  = \bigl|2[\langle M ,x^*\rangle, \langle M ,x_n^*-x^*\rangle]_t + [ \langle M ,x^*-x_n^*\rangle]_t \bigr|\\
  \leq 2 \sqrt{[ \langle M ,x^*\rangle]_t}\sqrt{[ \langle M ,x^*-x_n^*\rangle]_t} + [ \langle M ,x^*-x_n^*\rangle]_t,
 \end{multline*}
so we have that
\begin{align*}
 \|[\![M]\!]_t(x^*, x^*) &- [\![M]\!]_t(x^*_n, x^*_n)\|_{L^{1/2}(\Omega)}\\
 &\leq \|2 \sqrt{[ \langle M ,x^*\rangle]_t}\sqrt{[ \langle M ,x^*-x_n^*\rangle]_t} + [ \langle M ,x^*-x_n^*\rangle]_t\|_{L^{1/2}(\Omega)}\\
 &\lesssim \|\sqrt{[ \langle M ,x^*\rangle]_t}\sqrt{[ \langle M ,x^*-x_n^*\rangle]_t}\|_{L^{1/2}(\Omega)} + \|[ \langle M ,x^*-x_n^*\rangle]_t\|_{L^{1/2}(\Omega)}\\
 &\leq \|\sqrt{[ \langle M ,x^*\rangle]_t}\|_{L^{1/2}(\Omega)}^{1/2}\|\sqrt{[ \langle M ,x^*-x_n^*\rangle]_t}\|_{L^{1/2}(\Omega)}^{1/2} +  \|[ \langle M ,x^*-x_n^*\rangle]_t\|_{L^{1/2}(\Omega)},
\end{align*}
and thus it is enough to show that $\|[ \langle M ,x^*-x_n^*\rangle]_t\|_{L^{1/2}(\Omega)} \to 0$ as $n\to \infty$, which follows from the fact that  by Burkholder-Davis-Gundy inequalities \cite[Theorem 26.12]{Kal}
\[
 \|[ \langle M ,x^*-x_n^*\rangle]_t\|_{L^{1/2}(\Omega)}  \eqsim \bigl(\mathbb E \sup_{0\leq s\leq t}\|\langle M ,x^*-x_n^*\rangle\|\bigr)^2 \to 0,\;\; n\to \infty,
\]
as $\langle M ,x^*-x_n^*\rangle \to 0$ a.s.\ and $\langle M ,x^*-x_n^*\rangle$ are uniformly bounded by Banach--Steinhaus theorem.
\end{proof}

\begin{proof}[Proof of Theorem \ref{thm:tangiffwtangUMDcase}]
It is clear that tangency implies weak tangency. Let us show the converse. Assume that $M$ and $N$ are weakly tangent.
Let $M = M^c + M^d$ and $N = N^c + N^d$ be the Meyer-Yoeurp decompositions (see Remark \ref{rem:MYdecBanach}; this decomposition exists as $M$ and $N$ have local characteristics). First notice that for any $t\geq 0$ and for any $x^*\in X^*$ a.s.\
\begin{multline*}
[\![M^c ]\!]_t(x^*, x^*) = [\langle M^c, x^* \rangle]_t \stackrel{(*)}= [\langle M, x^* \rangle^c]_t\\
 = [\langle N, x^* \rangle^c]_t \stackrel{(*)}=  [\langle N^c, x^* \rangle]_t =[\![N^c ]\!]_t(x^*, x^*),
\end{multline*}
where $(*)$ follows from the fact that $\langle M^c, x^* \rangle = \langle M, x^* \rangle^c$ and $\langle N^c, x^* \rangle = \langle N, x^* \rangle^c$ a.s.\ (see \cite{Y17GMY,Y17MartDec,DY17}). Therefore $[\![M^c]\!]_t(x^*) = [\![N^c]\!]_t(x^*)$ for any $x^*\in X^*$ a.s., so we can set $[\![M^c]\!]_t = [\![N^c]\!]_t$ a.s.\ on $\Omega$ as $X$ can be assumed separable by the Pettis measurability theorem \cite[Theorem 1.1.20]{HNVW1}, so $[\![M^c]\!]_t$ and $[\![N^c]\!]_t$ coincide a.s.\ on weak$^*$ dense subset of $X^*$ which can be assumed countable by  the sequential Banach--Alaoglu theorem, and thus they coincide on the whole $X^*$ by Lemma \ref{lem:[[M]]isweka*consda}.

Now let us show that $\nu^M = \nu^N$ a.s.\ Fix Borel sets $B\subset X$ and $A \subset \mathbb R_+$. It is sufficient to show that a.s.\
\begin{equation}\label{eq:nuMandnuNcoincideonAtimesB}
 \nu^M(A\times B) = \nu^N(A\times B)
\end{equation}
As $X$ is separable, the Borel $\sigma$-algebra of $X$ is generated by cylinders (see e.g.\ \cite[Section 2.1]{BGaus}), we may assume that $B$ is a cylinder as well, i.e.\ there exist linear functions $x^*_1,\ldots, x^*_m \in X^*$ and a Borel set $\widetilde B\in \mathbb R^m$ such that $\mathbf 1_{B}(x) = \mathbf 1_{\widetilde B}(\langle x, x^*_1\rangle, \ldots, \langle x, x^*_m\rangle)$. Let $Y = \text{span}(x^*_1,\ldots, x^*_m)$, $(y_n)_{n\geq 1}$ be a dense sequence of $Y$. Then by the assumption of the theorem there exists $\Omega_0 \subset \Omega$ of full measure such that on $\Omega_0$
\[
 \nu^{\langle M, y_n\rangle} = \nu^{\langle N, y_n\rangle},\;\;\; n\geq 1,
\]
and as $(y_n)_{n\geq 1}$ is dense in $Y$ by a continuity argument we have that on $\Omega_0$
\begin{equation}\label{eq:nuMandnuNanregoodonfunctionals}
  \nu^{\langle M, y\rangle} = \nu^{\langle N, y\rangle},\;\;\; y\in Y.
\end{equation}
Let $P:X \to \mathbb R^m$ be such that $Px = (\langle x, x^*_1\rangle, \ldots, \langle x, x^*_m\rangle) \in \mathbb R^m$ for any $x\in X$. Then by \eqref{eq:nuMandnuNanregoodonfunctionals}, by Lemma \ref{lem:lintransfofmart-->transfofnuM}, and by the Cram\'er-Wold theorem (see e.g.\ \cite[Theorem 29.4]{Bill95} and \cite{BMR97}) we have that on $\Omega_0$
\[
 \nu^M(A \times B) = \nu^{ P M}(A \times \widetilde B) = \nu^{P N} (A \times \widetilde B) = \nu^N(A \times B),
\]
and thus \eqref{eq:nuMandnuNcoincideonAtimesB} follows. Consequently, $M$ and $N$ have the same local characteristics, and thus they are tangent.
\end{proof}

Assume now that inequalities of the form \eqref{eq:Lpboundsforgentangmarraz} hold for some Banach space $X$ for some $1\leq p <\infty$ for all weakly tangent martingales. Then in particular  for any independent Brownian motions $W$ and $\widetilde W$ and for any elementary predictable $\Phi: \mathbb R_+ \times \Omega \to X$ for martingales
$M:= \int \Phi \ud W$ and $N:= \int \Phi \ud \widetilde W$
we have that by \eqref{eq:qvvarofstintwrtHcylbrmot} for any $t\geq 0$ a.s.\
\[
[\![M]\!]_t(x^*, x^*) = [\langle M, x^*\rangle]_t = \int_0^t| \Phi^*x^*|^2 \ud s = [\langle N, x^*\rangle]_t =[\![N]\!]_t(x^*, x^*),\;\;\; t\geq 0,
\]
so $M$ and $N$ are weakly tangent and thus by \eqref{eq:Lpboundsforgentangmarraz}
\[
\mathbb E \sup_{t\geq 0} \Bigl\|\int_0^t \Phi \ud W \Bigr\|^p \eqsim_{p, X}\mathbb E \sup_{t\geq 0} \Bigl\|\int_0^t \Phi \ud \widetilde W \Bigr\|^p,
\]
which implies UMD e.g.\ by \cite{Gar85,HNVW1,CG} and by good-$\lambda$ inequalities \eqref{eq:goodlforconttangmaty}.

\section{Decoupled tangent martingales and the recoupling property}\label{sec:LCharandRecPro}

An interesting question is the following. Thanks to Theorem \ref{thm:mainforDTMgencasenananana} we know that for any given UMD space $X$ any $X$-valued local martingale has a decoupled tangent local martingale. Using Section \ref{sec:WTversusT} one can try to extend the notion of a decoupled tangent local martingale via exploiting weak local characteristics in Definition \ref{def:dectanglocmartcontimecased}, i.e.\ for a general Banach space $X$ a local martingale $N$ defined on enlarged probability space and filtration is called {\em decoupled tangent} to a local martingale $M$ if $M$ is a local martingale with respect to the enlarged filtration having the same weak local characteristics $([\langle M, x^* \rangle^c], \nu^{\langle M, x^* \rangle})_{x^*\in X^*}$ so that $N(\omega)$ is a local martingale with independent increments with weak local characteristics $([\langle M, x^* \rangle^c](\omega), \nu^{\langle M, x^* \rangle}(\omega))_{x^*\in X^*}$ for a.e.\ $\omega$ from the original probability space. {\em For which Banach space $X$ we can guarantee existence of such an object?}

In order to answer this question we need the recoupling property, which is dual to the decoupling property (see Definition \ref{def:decproperty}).

\begin{definition}\label{def:recproperty}
Let $X$ be a Banach space. Then $X$ is said to have the {\em recoupling property} if for some  $1\leq p <\infty$, for any $X$-valued martingale difference sequence $(d_n)_{n\geq 1}$ and for a decoupled tangent martingale difference sequence $(e_n)_{n\geq 1}$ one has that
\begin{equation}\label{eq:defofrqwecpropqdsao}
\mathbb E  \Bigl\| \sum_{n=1}^{\infty} e_n \Bigr\|^p \lesssim_{p, X} \mathbb E \sup_{N\geq 1}\Bigl\| \sum_{n=1}^{N} d_n\Bigr\|^p.
\end{equation}
\end{definition}

Let us first show the following elementary proposition demonstrating that we can assume any $1\leq p <\infty$ in Definition \ref{def:recproperty}.
\begin{proposition}\label{prop:5defofrecouplproepnmfoac}
 Let $X$ be a Banach space, $(d_n)$ be an arbitrary $X$-valued martingale difference sequence, $(e_n)$ be its decoupled tangent martingale sequence. Then the following are equivalent.
 \begin{enumerate}[(i)]
  \item $\mathbb E  \bigl\| \sum_{n=1}^{\infty} e_n \bigr\|^p \lesssim_{p, X} \mathbb E \sup_{N\geq 1} \bigl\| \sum_{n=1}^{N} d_n\bigr\|^p$ for any $1\leq p<\infty$, any $(d_n)$ and $(e_n)$.
    \item $\mathbb E   \bigl\| \sum_{n=1}^{\infty} e_n\bigr\|^p \lesssim_{p, X} \mathbb E \sup_{N\geq 1} \bigl\| \sum_{n=1}^{N} d_n\bigr\|^p$ for some $1\leq p<\infty$, any $(d_n)$ and $(e_n)$.
      \item $\mathbb E  \mathbb E \Bigl(\bigl\| \sum_{n=1}^{\infty} e_n \bigr\|\Big|\mathcal F\Bigr)^p \lesssim_{p, X} \mathbb E \sup_{N\geq 1} \bigl\| \sum_{n=1}^{N} d_n\bigr\|^p$ for some $1\leq p<\infty$, any $(d_n)$ and $(e_n)$.
            \item There exists a constant $C_X$ such that if $\mathbb E \Bigl( \bigl\| \sum_{n=1}^{\infty} e_n \bigr\|\Big|\mathcal F\Bigr)>1$ a.s.\ then we have that $ \mathbb E\sup_{N\geq 1} \bigl\| \sum_{n=1}^{N} d_n\bigr\|>C_X$.
\item There exists a constant $C_X$ such that if $ \bigl\| \sum_{n=1}^{\infty} e_n \bigr\|>1$ a.s.\ then we have that $ \mathbb E\sup_{N\geq 1} \bigl\| \sum_{n=1}^{N} d_n\bigr\|>C_X$.
 \end{enumerate}
\end{proposition}

\begin{proof}
 The proof in analogous to the one of \cite[Theorem 4.1]{Gar90}. $(i)\Rightarrow(ii)\Rightarrow(iii)$ follow for an obvious reason and by Jensen's inequality. $(iii)\Rightarrow(iv)$ follows similarly to \cite[p.\ 999]{Burk81}. $(iv)\Rightarrow(v)$ follows from the fact that if $ \bigl\| \sum_{n=1}^{\infty} e_n \bigr\|>1$ a.s., then $\mathbb E \Bigl( \bigl\| \sum_{n=1}^{\infty} e_n \bigr\|\Big|\mathcal F\Bigr)>1$ a.s., and thus the desired holds. Now let us show $(v)\Rightarrow(i)$. Fix $1\leq p<\infty$. Let $G_p(d_n):= \mathbb E \Bigl( \bigl\| \sum_{n=1}^{\infty} e_n \bigr\|^p \Big| \mathcal F \Bigr)$. Then
 \begin{equation}\label{eq:Gpd_nbddbufweLpnoreiofdn}
   \sup_{\lambda>0}\lambda \mathbb P (G_p(d_n)>\lambda) \stackrel{(*)}\leq \sup_{\lambda>0} \lambda\mathbb P \Bigl(\bigl\| \sum_{n=1}^{\infty} e_n \bigr\|^p \geq \lambda \Bigr) \stackrel{(**)}\lesssim_{p, X} \mathbb E \sup_{N\geq 1} \bigl\| \sum_{n=1}^{N} d_n\bigr\|^p,
 \end{equation}
where $(*)$ follows from the fact that averaging operators are contractions on Lorentz spaces (see e.g.\ \cite{JLhbBs}) and $(**)$ can be derived from $(v)$ analogously \cite[p.\ 1000]{Burk81}. Now applying \cite[pp.\ 999--1000]{Burk81} one can conclude from \eqref{eq:Gpd_nbddbufweLpnoreiofdn} that
\[
 \mathbb E\Bigl\| \sum_{n=1}^{\infty} e_n \Bigr\|^p = \mathbb E G_p(d_n) \lesssim_{p, X}E \sup_{N\geq 1} \Bigl\| \sum_{n=1}^{N} d_n\Bigr\|^p,
\]
so the proposition follows.
\end{proof}

Recall that $X$ is called a {\em UMD$^+$ Banach space} if for some (equivalently, for all)
$p \in [1,\infty)$,
every martingale
difference sequence $(d_n)^{\infty}_{n=1}$ in $L^p(\Omega; X)$, and every independent Rademacher sequence
$(r_n)^{\infty}_{n=1}$ one has that
\begin{equation}\label{eq:UMD+defeqwo23}
 \mathbb E \Bigl\| \sum^{\infty}_{n=1} r_n d_n\Bigr\|^p
\lesssim_{p, X} \mathbb E\sup_{N=1}^{\infty} \Bigl \| \sum^N_{n=1}d_n\Bigr\|^p
\end{equation}
(see \cite{Gar90,HNVW1,CG,CV,Ver07}).

\begin{remark}\label{rem:recouplbasicproperta}
The recoupling property immediately yields the UMD$^+$ property for  Paley-Walsh and Gaussian martingales (see \cite[pp.\ 498--500]{HNVW1}, \cite[Section 4.2]{OY18}, and \cite{Gar90}) and hence any Banach space $X$ with the recoupling property is supereflexive, has finite cotype (see \cite[Theorem 3.2]{Gar90}) and nontrivial type (due to \cite[Theorem 7.3.8]{HNVW2} and the supereflexivity of $X$). It remains open whether the recoupling property implies UMD (note that the recoupling property is equivalent to UMD if $X$ is a Banach lattice thanks to \cite[Theorem 8.4]{KLW19}; see also the discussion in \cite{Geiss99} and \cite[Section O]{HNVW1}). Nonetheless, one can show that the recoupling property is in fact equivalent to UMD$^+$ for general martingales, which is in some sense a dual result to \cite[Theorem 6.6(iii)]{CG}.
\end{remark}

\begin{proposition}\label{prop:recouplequivUMD+}
Let $X$ be a Banach space. Then 
$X$ has the recoupling property if and only if it is UMD$^+$.
\end{proposition}

\begin{proof}
{\em The ``only if'' part.}
Assume that $X$ has the recoupling property.
Let $(d_n)^{\infty}_{n=1}$ be an $X$-valued martingale difference sequence, $(e_n)_{n=1}^{\infty}$ be a corresponding decoupled tangent sequence on an enlarged filtration. Then by Definition \ref{def:recproperty}
\[
 \mathbb E \sup_{N\geq 1}\Bigl\| \sum_{n=1}^{N} d_n-e_n\Bigr\|^p\lesssim_{p, X} \mathbb E \sup_{N\geq 1}\Bigl\| \sum_{n=1}^{N} d_n\Bigr\|^p.
\]
Note that $(d_n-e_n)_{n\geq 1}$ is conditionally symmetric. By an approximation argument (by adding a sequence $(\eps x r_n'')$ for some $x\in X$, independent Rademacher sequence $(r_n'')$, and some small enough $\eps > 0$) we may assume that $\mathbb P(d_n-e_n = 0) = 0$. Moreover, for the same approximation argument we may assume that there exists $x^*\in X^*$ such that $\langle d_n-e_n, x^* \rangle \neq 0$ a.s. Let $r_n' := \langle d_n-e_n, x^* \rangle/|\langle d_n-e_n, x^* \rangle|$, $\xi_n := (d_n-e_n)/r_n'$. Let us show that $r_n'$ is independent of $\sigma(\xi_n, \mathcal F_{n-1})$. Indeed, for any $A\subset X$ with $\langle x, x^*\rangle>0$ for any $x\in A$ and any $B\in \mathcal F_{n-1}$ by the conditional symmetry we have that
\begin{multline*}
\mathbb P(\{r_n'=1\}\cap \{\xi_n\in A\}\cap B) = \mathbb P(\{d_n-e_n\in A\}\cap B) \\
= \mathbb P(\{d_n-e_n\in -A\}\cap B) = \mathbb P(\{r_n'=-1\}\cap \{\xi_n\in A\}\cap B),
\end{multline*}
so
\[
\mathbb P(\{r_n'=\pm1\}\cap \{\xi_n\in A\}\cap B) = \frac{1}{2}\mathbb P( \{\xi_n\in A\}\cap B) = \mathbb P(r_n'=\pm1) \mathbb P( \{\xi_n\in A\}\cap B).
\]
 Therefore by setting a new filtration $(\mathcal  G_n)_{n\geq 1} := (\sigma(\mathcal  F_{n}, \xi_{n+1}))_{n\geq 1}$ and noticing that $(r_n' \xi_n)$ is a martingale difference sequence w.r.t.\ this filtration we can deduce from \eqref{eq:defofrqwecpropqdsao}, Example \ref{ex:tangstandarddecxinvn-->xi'nvn}, and the fact that a product of two Rademachers is a Rademacher that for any independent sequence $(r_n)_{n\geq 1}$ of Rademachers
 \begin{multline*}
\mathbb E \Bigl\| \sum_{n=1}^{\infty} r_n(d_n-e_n)\Bigr\|^p = \mathbb E \Bigl\| \sum_{n=1}^{\infty}r_n r_n' \xi_n\Bigr\|^p\\
\lesssim_{p, X} \mathbb E \sup_{N\geq 1} \Bigl\| \sum_{n=1}^{N} r_n' \xi_n\Bigr\|^p = \mathbb E \sup_{N\geq 1} \Bigl\| \sum_{n=1}^{N} d_n-e_n\Bigr\|^p.
\end{multline*}
Finally, by applying a conditional expectation w.r.t.\ $\sigma(\mathcal F, (r_n))$ and Jensen's inequality we know that $\mathbb E \| \sum_{n=1}^{\infty} r_nd_n\|^p \leq \mathbb E \| \sum_{n=1}^{\infty} r_n(d_n-e_n)\|^p$. By combining all the inequalities above \eqref{eq:UMD+defeqwo23} follows.

 {\em The ``if'' part.} 
 Let $X$ be UMD$^+$, $(d_n)_{n\geq 1}$ be an $X$-valued martingale difference sequence, $(e_n)_{n\geq 1}$ be a decoupled tangent sequence. Given the UMD$^+$ property we need to show \eqref{eq:defofrqwecpropqdsao} for some $p\geq 1$. Without loss of generality we may assume that the filtration is generated by $(d_n)$ and the enlarged filtration is generated by $(d_n)$ and $(e_n)$. First assume that $(d_n)_{n\geq 1}$ is conditionally symmetric. In this case for any Borel $A\subset X$ we have that
 \begin{equation}\label{eq:dncpaspsymmUMD+recpi}
    \mathbb P(\mathbf 1_{A}(d_n)|\mathcal F_{n-1}) =   \mathbb P(\mathbf 1_{-A}(d_n)|\mathcal F_{n-1}),\;\;\; n\geq 1.
 \end{equation}
Let $(x_m)_{m=1}^M\in X\setminus \{0\}$ be disjoint and let balls $(B_m)_{m=1}^M$ be disjoint with $x_m \in B_m$ and $\{0\}\notin B_m$ so $B_m \cap -B_m = \varnothing$. Then for any $n\geq 1$ we can approximate $d_n$ in $L^p$ for any $p<\infty$ by $\sum_{m=1}^M x_m (\mathbf 1_{B_m}(d_n) - \mathbf 1_{-B_m}(d_n))$ by taking big enough set $(x_m)_{m=1}^M\in X$ and by considering smaller balls $(B_m)_{m=1}^M$. Let $A_m^n:= \{d_n \in B_m\} \cup \{d_n \in -B_m\}$ and let $r_m^n := \mathbf 1_{B_m}(d_n) - \mathbf 1_{-B_m}(d_n) + r_m^{n'} \mathbf 1_{\Omega\setminus A_m^n}$ for some independent $r_m^{n'}$ (one may need to enlarge the probability space and filtration). By \eqref{eq:dncpaspsymmUMD+recpi} and independence of $(r_m^{n'})$ we conclude that
\begin{multline*}
 \mathbb P (r_m^n = \pm 1|\mathcal F_{n-1}) = \mathbb P(d_n \in \pm B_m|\mathcal F_{n-1}) + \mathbb P(\Omega \setminus A_m^n \cap \{r_m^{n'} = \pm 1\}|\mathcal F_{n-1})\\
 =\frac{1}{2}\mathbb P(A_m^n|\mathcal F_{n-1}) + \frac{1}{2}\mathbb P(\Omega \setminus A_m^n|\mathcal F_{n-1}) = \frac{1}{2},
\end{multline*}
so $(r_m^n)$ is conditionally independent of $\mathcal F_{n-1}$. For the same reason $(r_m^n)_{m=1}^M$ are independent Rademachers which are mutually independent of $\sigma(\mathcal F_{n-1}, A_1^n, \ldots, A_M^n)$. Thus we can approximate $d_n$ by $\sum_{m=1}^M r_m^n \mathbf 1_{A_m^n}$, so by assuming $d_n=\sum_{m=1}^M r_m^n \mathbf 1_{A_m^n}$, by \eqref{eq:UMD+defeqwo23}, and by the fact that $A_m^n$ are disjoint for every $n\geq 1$ (which guarantees that the corresponding suprema coincide) we get
\begin{equation}\label{eq:recoupropforrnmrmgfsw''da}
   \mathbb E \Bigl\| \sum^{\infty}_{n=1} \sum_{m=1}^M  r_m^{n''} \mathbf 1_{A_m^n}\Bigr\|^p
\lesssim_{p, X} \mathbb E\sup_{N=1}^{\infty} \Bigl \| \sum^N_{n=1}\sum_{m=1}^M  r_m^{n} \mathbf 1_{A_m^n}\Bigr\|^p,
\end{equation}
where $(r_m^{n''})_{n=1, m=1}^{\infty, M}$ is an independent sequence of Rademachers. For each $n\geq 1$ let $e_n:= \sum_{m=1}^M  r_m^{n''} \mathbf 1_{A_m^n}$ and let us show that $(e_n)_{n\geq 1}$ is a decoupled tangent sequence to $(d_n)_{n\geq 1}$. First both $(d_n)$ and $(e_n)$ take values only in $(\pm x_m)_{m=1}^M$, and for any $m=1, \ldots, M$ by the independence of $r_m^{n''}$ and by \eqref{eq:dncpaspsymmUMD+recpi}
\begin{multline*}
 \mathbb P(e_n = \pm x_m|\mathcal F_{n-1}) =  \mathbb P(\{r_m^{n''} = \pm 1\}\cap A_m^n|\mathcal F_{n-1}) = \frac{1}{2}\mathbb P( A_m^n|\mathcal F_{n-1})\\
 = \mathbb P(\{r_m^{n} = \pm 1\}\cap A_m^n|\mathcal F_{n-1}) =  \mathbb P(d_n = \pm x_m|\mathcal F_{n-1}),
\end{multline*}
so $(d_n)$ and $(e_n)$ are tangent. Next, for any fixed $\omega\in \Omega$, $(e_n(\omega))$ are independent, hence $(e_n)$ is decoupled. Thus \eqref{eq:defofrqwecpropqdsao} follows from \eqref{eq:recoupropforrnmrmgfsw''da}.

Now let $(d_n)$ be general. Then for any sequence of independent Rademachers $(r_n)$ by \eqref{eq:UMD+defeqwo23} and \cite[Proposition 6.1.12]{HNVW2}
\[
 \mathbb E \sup_{N=1}^{\infty}\Bigl \| \sum^N_{n=1}r_nd_n\Bigr\|^p\eqsim \mathbb E \Bigl\| \sum^{\infty}_{n=1} r_n d_n\Bigr\|^p
\lesssim_{p, X}\mathbb E  \sup_{N=1}^{\infty}\Bigl \| \sum^N_{n=1}d_n\Bigr\|^p.
\]
Let $(e_n)$ be a decoupled tangent sequence to $(d_n)$. Then $(r_n e_n)$ is a decoupled tangent sequence to $(r_n d_n)$ (see \eqref{eq:rndnandrnenaretagicdnenadfosim}), which is conditionally symmetric, and hence by the conditional symmetric case we get
\[
 \mathbb E \Bigl\| \sum^{\infty}_{n=1} r_n e_n\Bigr\|^p \lesssim_{p, X}\mathbb E \sup_{N=1}^{\infty}\Bigl \| \sum^N_{n=1}r_nd_n\Bigr\|^p.
\]
It remains to show that $\mathbb E \| \sum^{\infty}_{n=1} e_n\|^p \eqsim_{p}\mathbb E \| \sum^{\infty}_{n=1} r_n e_n\|^p$, which follows directly from the conditional independence and  a randomization argument (see e.g.\ \cite[Lemma 6.3]{LT11}. 
\end{proof}

\begin{remark}
The same proof yields that the {\em UMD$^-$ property} (the property which is inverse to UMD$^+$, see e.g.\ \cite[Chapter 4]{HNVW1}) implies the decoupling property. The converse statement can be shown for conditionally symmetric martingale difference sequences (which is a weaker form of \cite[Theorem 6.6(iii)]{CG}), but unfortunately a similar technique seems to be not able to provide an extension to general martingales as UMD$^-$ constants heavily dominate the decoupling constants in the real-valued case (see the discussion in \cite[pp.\ 346--348]{CV}). The equivalence of UMD$^-$ and decoupling remains unknown for the author.
\end{remark}

The following theorem is the main result of the section.

\begin{theorem}\label{thm:existofdecoupltagivenerqw}
 Let $X$ be a Banach space. Then the following are equivalent:
 \begin{enumerate}[(i)]
  \item $X$ has the recoupling property.
  \item Any $X$-valued martingale $M$ has 
  a decoupled tangent local martingale $N$.
 \end{enumerate}
Moreover, if this is the case, then for any $1\leq p <\infty$
\begin{equation}\label{eq:UMD+NtbddbyMtforpgeq1}
\mathbb E \sup_{t\geq 0}\|N_t\|^p \lesssim_{p, X} \mathbb E \sup_{t\geq 0}\|M_t\|^p.
\end{equation}
\end{theorem}

\begin{proof}
 {\em Part 1: $(i) \Rightarrow(ii)$.} Assume that $X$ has the recoupling property. Let $M$ be an $X$-valued martingale. First let us show that $[\![M]\!]$ exists, c\`adl\`ag, and $\gamma([\![M]\!]_t)<\infty$ a.s.\ for any $t\geq 0$ (see Remark \ref{rem:ifUMDthencovbilform} and \cite{Y18BDG}). To this end note that by Proposition \ref{prop:recouplequivUMD+} one can repeat the proof of \cite[Theorem 2.1]{Y18BDG} showing only the upper estimate $\gtrsim_{p, X}$ in \cite[(2.2) and (2.3)]{Y18BDG}. This upper estimate together with the superreflexivity of $X$ (see Remark \ref{rem:recouplbasicproperta}) is enough to repeat the proof of \cite[Theorem 5.1]{Y18BDG} again with providing only 
 \begin{equation}\label{eq:gammaMforUMD+isbddbyM}
   \mathbb E \gamma([\![M]\!]_{\infty})^p\lesssim_{p, X} \mathbb E \sup_t \|M_t\|^p, \;\; 1\leq p<\infty,
 \end{equation} 
 in \cite[(5.1)]{Y18BDG}. For the same reason it follows that $\gamma([\![M]\!]_t)<\infty$ a.s.\ for any $t\geq 0$ and that it is locally integrable. Existence of a c\`adl\`ag version of $[\![M]\!]$ holds similarly \cite[Proposition 5.5]{Y18BDG}.
 
 Now, as $[\![M]\!]$ is c\`adl\`ag and nondecreasing (see Subsection \ref{subsec:quadrvar} and \cite{Y18BDG}), analogously to \cite[Section 25, 26]{Kal} it has a continuous part $[\![M]\!]^c$ and a pure jump part $[\![M]\!]^d$ (shortly, one can simply define $[\![M]\!]^d_t:= \sum_{0\leq s\leq t} \Delta [\![M]\!]_s$ and $[\![M]\!]^c := [\![M]\!] - [\![M]\!]^d$). Moreover, by \cite[Lemma 3.9, 3.10]{Y18BDG}, by the definition of $[\![M]\!]$, and by \cite[Theorem 26.14]{Kal} we have that $\gamma([\![M]\!]) \eqsim \gamma([\![M]\!]^c) + \gamma([\![M]\!]^d)$ a.s., so in particular by \cite[(3.2)]{Y18BDG} $[\![M]\!]^c$ is bounded and locally integrable.

 Let us show that there exists a decoupled tangent martingale $N$. We will construct separately $N^c$, $N^q$, and $N^a$, and then sum them up. For $N^c$ let us consider  $[\![M]\!]^c$. We know that $\gamma([\![M]\!]^c)$ is finite and locally integrable. Therefore analogously to the proof of Theorem \ref{thm:DTMforcontcase} there a.s.\ exist an invertible time-change $(\tau_s)_{s\geq 0}$ with an inverse time-change $(A_t)_{t\geq 0}$ (see the proof of Theorem \ref{thm:DTMforcontcase}), a Hilbert space $H$, $\Phi\in \gamma(L^2(\mathbb R_+; H), X)$ predictable with respect to $(\mathcal F_{\tau_s})_{s\geq 0}$, and a cylindrical Wiener process $W_H$ such that for any $x^* \in X^*$ a.s.\
 \[
 ( \langle M, x^*\rangle^c \circ \tau)_t= \int_0^t \Phi^* x^*\ud W_H,\;\;\;\;\; t\geq 0.
 \]
As $\Phi\in \gamma(L^2(\mathbb R_+; H), X)$, by \cite[Subsection 3.2]{Y18BDG} (see also \cite{NVW}) $\Phi$ is a.s.\ integrable with respect to an independent cylindrical Brownian motion $W_H'$ (with $(\mathbb E_{W_H'} \|\Phi\cdot W_H'\|^2)^{\frac 12} = \|\Phi\|_{\gamma(L^2(\mathbb R_+; H), X)} $ a.s.), so we can define $N^c_t:= \int_0^{A_t} \Phi \ud W_H'$. Moreover, by \eqref{eq:gammaMforUMD+isbddbyM}, Remark \ref{rem:recouplbasicproperta}, and \cite[Lemma 3.10 and Proposition 6.8]{Y18BDG} 
\begin{equation}\label{eq:NcUMD+bddbyNfotop}
\begin{split}
 \mathbb E \sup_{t\geq 0} \|N^c_{t}\|^p &\eqsim_{p, X} \mathbb E \gamma([\![N^c]\!]_{\infty})^p = \mathbb E \gamma([\![M^c]\!]_{\infty})^p\\
 & \leq \mathbb E \gamma([\![M]\!]_{\infty})^p \lesssim_{p, X}  \mathbb E \sup_{t\geq 0} \|M_{t}\|^p.
 \end{split}
\end{equation}

Let us now construct $N^q$. Analogously to \cite[Proposition 25.17]{Kal}, $[\![M]\!]^d$ can be decomposed into $[\![M]\!]^q$, which is quasi-left continuous, and $[\![M]\!]^a$, which is purely discontinuous with accessible jumps. Let $\mu^M$ be defined by \eqref{eq:defofmuM}, let $\mu$ be the quasi-left continuous component of $\mu^M$ (see Lemma \ref{lem:decompofmeasuresontwoparts}), and let $\nu$ be the compensator of $\mu$, $\bar{\mu} = \mu - \nu$. Then we have that for a.e.\ $\omega\in \Omega$ $[\![M]\!]^q_{\infty}$ collects quasi-left continuous jumps of $M$ and in particular by \cite{Y18BDG}, $\gamma([\![M]\!]^q_{\infty}) = (\mathbb E_{\gamma} \| \int_{\mathbb R_+ \times X} \gamma_t x\ud {\mu}(t, x)\|^2)^{\frac 12}$, where $(\gamma_t)_{t\geq 0}$ is a family of independent Gaussians which can be considered countable a.s.\ as $M$ has countably many jumps a.s. Let $\mu_{\rm Cox}$ be a Cox process directed by $\nu$, $\bar{\mu}_{\rm Cox} = \mu_{\rm Cox}-\nu$. Let us show that $x$ is integrable with respect to $\bar{\mu}_{\rm Cox}(\omega)$ for a.e.\ $\omega\in \Omega$. To this end notice that by \cite[Proposition 6.8 and Subsection 7.2]{Y18BDG} and by the fact that $\bar{\mu}_{\rm Cox}(\omega)$ is a Poisson random measure (so it is a random measure with independent increments) $x$ is integrable with respect to $\bar{\mu}_{\rm Cox}(\omega)$ if $\mathbb E_{\rm Cox}\|x\|_{\gamma(L^2(\mathbb R_+ \times X; {\mu}_{\rm Cox}(\omega)), X)} <\infty$, which is a.s.\ satisfied as for any $p\geq 1$
\begin{align}\label{eq:xja,,danomdasbddbyEsum|Mdaao|}
\mathbb E\|x\|_{\gamma(L^2(\mathbb R_+ \times X; {\mu}_{\rm Cox}), X)}^p & \stackrel{(i)}= \mathbb E \Bigl(\mathbb E_{\gamma} \Bigl\|\int_{\mathbb R_+ \times X} \gamma_t x \ud  {\mu}_{\rm Cox}(\omega)(t,x) \Bigr\|^2\Bigr)^{\frac p2}\nonumber\\
& \stackrel{(ii)}= \lim_{n\to \infty}\mathbb E \Bigl(\mathbb E_{\gamma} \Bigl\|\int_{\mathbb R_+ \times X} \mathbf 1_{A_n}(t, x)\gamma_{n(t)} x \ud  {\mu}_{\rm Cox}(t,x) \Bigr\|^2\Bigr)^{\frac p2}\nonumber\\
& \stackrel{(iii)} \lesssim_{p, X}\lim_{n\to \infty}\mathbb E \Bigl(\mathbb E_{\gamma} \Bigl\|\int_{\mathbb R_+ \times X} \mathbf 1_{A_n}(t, x)\gamma_{n(t)} x \ud  {\mu}(t,x) \Bigr\|^2\Bigr)^{\frac p2} \\
& \stackrel{(iv)}=\mathbb E \Bigl(\mathbb E_{\gamma} \Bigl\|\int_{\mathbb R_+ \times X} \gamma_t x \ud  {\mu}(t,x) \Bigr\|^2\Bigr)^{\frac p2}\nonumber\\
&\stackrel{(v)}=\mathbb E\gamma([\![M]\!]^q_{\infty})^p \lesssim \mathbb E \sup_{t\geq 0} \|M_t\|^p,\nonumber
\end{align}
where $(A_n)_{n\geq 1}$ are defined analogously to the proof of Theorem \ref{thm:genformofmartwithindince} for quasi-left continuous jumps of $M$, $(\gamma_t)_{t\geq 0}$ are independent Gaussians, a step nondecreasing function $t\mapsto n(t)$ is a discretization so that $(\gamma_{n(t)})_{t\geq 0}$ includes finitely many Gaussians and $n(t)\to t$ as $n\to \infty$, $(i)$ and $(v)$ follow from \cite[Section 9]{HNVW2} and the fact that ${\mu}_{\rm Cox}$ and ${\mu}$ are a.s.\ atomic measures, approximations $(ii)$ and $(iv)$ follow analogously \cite[Lemma 3.11]{Y18BDG}, and finally $(iii)$ holds due to \eqref{eq:defofrqwecpropqdsao}, \cite[Lemma 6.3]{LT11}, \cite[Proposition 6.3.1]{HNVW2}, and the fact that $t\mapsto \int_{[0,t] \times X} \mathbf 1_{A_n}(t, x)\gamma_{n(t)} x \ud  {\mu}_{\rm Cox}(t,x)$ is {\em approximately} a decoupled tangent martingale to $t\mapsto \int_{[0,t] \times X} \mathbf 1_{A_n}(t, x)\gamma_{n(t)} x \ud  {\mu}(t,x)$ (in fact $t\mapsto \int_{[0,t] \times X} \gamma_t x \ud  {\mu}_{\rm Cox}(t,x)$ is a decoupled tangent martingale to $t\mapsto \int_{[0,t] \times X} \gamma_t x \ud  {\mu}(t,x)$, but the filtration generated by $(\gamma_t)_{t\geq 0}$ is not countably generated, so an approximation needed; such an approximation can be done analogously Section \ref{sec:appJKW}). Therefore $t\mapsto N^q_t:= \int_{[0, t]\times X} x \ud \bar{\mu}_{\rm Cox}$ is a well-defined purely discontinuous quasi-left continuous martingale which has independent increment for a.e.\ $\omega\in \Omega$ and thanks to \eqref{eq:xja,,danomdasbddbyEsum|Mdaao|} and \cite[Proposition 6.8 and Subsection 7.2]{Y18BDG} we have that
\begin{equation}\label{eq:NqUMD+bddbyMpgeq1omc}
\mathbb E \sup_{t\geq 0} \|N^q_t\|^p \eqsim_{p,X}\mathbb E\|x\|_{\gamma(L^2(\mathbb R_+ \times X; {\mu}_{\rm Cox}), X)}^p \lesssim_{p, X}\mathbb E \sup_{t\geq 0} \|M_t\|^p.
\end{equation}

Finally, let us turn to $N^a$. First let $(\tau_n)_{n\geq 1}$ be a sequence of predictable stopping time exhausting predictable jumps of $M$ (equivalently, all the jumps of $[\![M]\!]^a$, see also Lemma \ref{lem:PDmAJhasjumpsatprsttimes} and \cite[Proposition 25.4]{Kal}). Let $N^{a,m}$ be constructed for any $m\geq 1$ similarly to \eqref{eq:defofNmapproxNDTMforPDWAJ}. Let us show that $N^{a,m}$ converges in strong $L^1(\Omega; X)$ as $m\to 0$, i.e.\ for any $n\geq m\geq 1$
\[
\mathbb E \sup_{t\geq 0} \|N^{a,n}_t- N^{a,m}_t\|\to 0,\;\;\; n\geq m\to \infty.
\]
Note that by its construction $N^{a,n}(\omega)- N^{a,m}(\omega)$ has independent increments, hence by \cite[Subsection 6.2]{Y18BDG} and Remark \ref{rem:recouplbasicproperta}
\begin{equation}\label{eq:NamNqnto0forUMD+eoqm}
\begin{split}
\mathbb E \sup_{t\geq 0} \|N^{a,n}_t- N^{a,m}_t\|&\eqsim_{X} \mathbb E \gamma([\![N^{a,n}- N^{a,m}]\!]_{\infty}) \eqsim \mathbb E \Bigl\| \sum_{k=m+1}^{n}\gamma_k \Delta N_{\tau_k} \Bigr\|\\
& \stackrel{(*)}\lesssim_{p, X}\mathbb E \Bigl\| \sum_{k=m+1}^{n}\gamma_k \Delta M_{\tau_k} \Bigr\| \stackrel{(**)}\to 0 \;\;\; \text{as}\;\;\; n\geq m\to \infty,
\end{split}
\end{equation}
 where $(*)$ follows from \eqref{eq:defofrqwecpropqdsao}, \cite[Lemma 6.3]{LT11}, and the fact that $(\gamma_k\Delta N_{\tau_k})_{k=m+1}^n$ is a decoupled tangent sequence to $(\gamma_k\Delta M_{\tau_k})_{k=m+1}^n$ (this follows analogously Theorem \ref{thm:detangmartforMXVpdwithaccjumps}, where one needs to reorder $(\tau_k)_{k=m+1}^n$ making it increasing as it was done in the proof of Proposition \ref{prop:pdwajjusttang}), and $(**)$ holds true similarly to \cite[Theorem 7.14]{Y18BDG}. Thus $N^a := \lim_{n\to \infty} N^{a, n}$ is a well-defined purely discontinuous martingale with accessible jumps and has independent increments for a.e.\ fixed $\omega\in \Omega$ due to its construction (see the proof of Theorem \ref{thm:detangmartforMXVpdwithaccjumps}), and analogously $(*)$ in \eqref{eq:NamNqnto0forUMD+eoqm} (which holds for any power $p\geq 1$) and due to \cite[Proposition 6.8 and Subsection 7.2]{Y18BDG} and \eqref{eq:gammaMforUMD+isbddbyM}
 \begin{equation}\label{eq:UMD+NabddbyMdpdq}
 \mathbb E \sup_{t\geq 0} \|N^a_t\|^p \eqsim_{p, X} \mathbb E \gamma([\![N^a]\!]_{\infty})^p \lesssim_{p, X} \mathbb E \gamma([\![M]\!]_{\infty})^p \lesssim_{p, X} \mathbb E \sup_{t\geq 0}\|M_t\|^p.
 \end{equation}
 It remains to let $N:= N^c + N^q + N^a$ with \eqref{eq:UMD+NtbddbyMtforpgeq1} holding by \eqref{eq:NcUMD+bddbyNfotop}, \eqref{eq:NqUMD+bddbyMpgeq1omc}, and \eqref{eq:UMD+NabddbyMdpdq}. The fact that $N$ is a decoupled tangent martingale to $M$ follows similarly Subsection \ref{subsec:CIprocess}.
  
  {\em Part 2: $(ii) \Rightarrow(i)$.}  Let $X$ fail the recoupling property. It is sufficient to show that there exists a martingale without a decoupled tangent local one. First let us construct an $X$-valued martingale difference sequence $(d_n)_{n\geq 1}$ such that $\sum_{n\geq 1} d_n$ converges and bounded a.s.\ and $\sum_{\geq 1}e_n$ diverges a.s., where $(e_n)_{n\geq 1}$ is a decoupled tangent sequence to $(d_n)_{n\geq 1}$. To this end we will apply Proposition \ref{prop:5defofrecouplproepnmfoac}. Due to Proposition \ref{prop:5defofrecouplproepnmfoac}(v) for any $m\geq 1$ we can inductively construct an $X$-valued martingale difference sequence $(d_n^m)_{m=1}^{M_m}$ with a decoupled tangent martingale difference sequence $(e_n^m)_{m=1}^{M_m}$ such that
  \begin{enumerate}[(i)]
  \item $\mathbb E \sup_{N=1}^{M_m} \|\sum_{n=1}^N d_n^m\| \leq \frac{1}{2^m}$, and
   \item $\mathbb P( \|\sum_{n=1}^{M_m} e_n^m\|>C_{m-1})>1-\frac{1}{2^m}$, where $C_0=1$ and for any $m\geq 1$ $C_m$ is such a constant bigger then $2C_{m-1}$ that $\mathbb P( \|\sum_{n=1}^{M_m} e_n^m\|>C_{m})<\frac{1}{2^m}$
  \end{enumerate}
(see e.g.\ \cite[Subsection 4.4]{Y17GMY}). Then we can set 
\[
 d_{M_1 + \ldots + M_{m-1} + n}:= d_n^m,\;\;\;m\geq 1,\;\; 1\leq n\leq M_m,
\]
\[
 e_{M_1 + \ldots + M_{m-1} + n}:= e_n^m,\;\;\;m\geq 1,\;\; 1\leq n\leq M_m,
\]
and the pair $(d_n)_{n\geq 1}$ and $(e_n)_{n\geq 1}$ would satisfy the desired properties. Now let us construct a martingale $M:\mathbb R_+ \times \Omega \to X$ without a decoupled tangent one. To this end let the filtration $(\mathcal F_{t})_{t\geq 0}$ be generated by  $(d_n)_{n\geq 1}$ and $(e_n)_{n\geq 1}$ in the following way: let $\mathcal F_{1-1/2^n} := \sigma(d_1, \ldots, d_n, e_1, \ldots, e_n)$ and let $\mathcal F_t := \mathcal F_{1-1/2^n}$ for any $n\geq 1$ and $t\in[1-1/2^n, 1-1/2^{n+1})$. Set 
$$
M_t:= \sum_{n:1-1/2^n\leq t} d_n,\;\;\;N_t:= \sum_{n:1-1/2^n\leq t} e_n,\;\;\;0\leq t<1.
$$ 
First of all notice that $N$ by its definition is a decoupled tangent martingale to $M$ on $[0, 1-1/2^n]$ for any $n\geq 1$. But $\lim_{t\nearrow 1}N_t = \sum_{n\geq 1}e_n$ does not exists due to the construction, so it is not a local martingale and thus not a decoupled tangent local martingale to $M$. Assume that $M$ has some decoupled tangent martingale $\widetilde N$. Then by Remark \ref{rem:defofdectanglmartidiscreteandfac,q} for each $\omega\in \Omega$ we have that $(\Delta N_t(\omega))_{t\geq 0}$ and $(\Delta\widetilde N_t(\omega))_{t\geq 0}$ are equidistributed and independent, so $N$ and $\widetilde N$ are equidistributed, hence $\widetilde N$ is not a local martingale, so the desired holds true.
\end{proof}

\begin{remark}\label{rem:decweaktangmartopenq}
Note that thanks to the proof of Theorem \ref{thm:existofdecoupltagivenerqw} any UMD$^+$-valued martingale $M$ has a covariation bilinear form $[\![M]\!]$, and more importantly, this covariation bilinear form has a continuous part $[\![M]\!]^c$ so that the weak local characteristics of $M$ are generated by $([\![M]\!]^c, \nu^M)$. It remains unknown for the author whether one can characterize the UMD$^+$ property of $X$ via existence of the pair $([\![M]\!]^c, \nu^M)$ (or, equivalently, only via existence of $[\![M]\!]^c$ as $\nu^M$ is always well-defined, see Subsection \ref{subsec:ranmeasures}) for any $X$-valued martingale $M$.
\end{remark}

\section{Independent increments}\label{sec:indincrements}

The present section is devoted to martingales with independent increments. As we will see below, in this case one could avoid the UMD assumption in order to show existence of local characteristics. Moreover, in Subsection \ref{subsec:martwithindincrgenform} we will show that such martingales have a precise form in terms of stochastic integrals with respect to cylindrical Brownian motions and Poisson random measures. Recall that we will be talking about martingales with independent increments without the localization assumption which can be omitted due to Remark \ref{rem:locmartwithIIisamartwithII}.

\subsection{Weak local characteristics and independent increments}\label{subsec:weakloccharandMII}

As it was originally shown by Grigelionis in \cite{Grig77} (see also a multidimensional version \cite[p.\ 106]{JS}), a local martingale has independent increments if and only if its local characteristics are deterministic. Let us extend this result to infinite dimensions via using weak local characteristics.

\begin{theorem}\label{thm:Mhasindinciffweallocchararedeterms}
 Let $X$ be a Banach space, $M:\mathbb R_+ \times \Omega \to X$ be a local martingale. Then $M$ has independent increments if and only if its weak local characteristics are deterministic.
\end{theorem}

\begin{proof}
The ``only if'' part is simple and follows directly from the real-valued case \cite{Grig77} and the fact that if $M$ has independent increments then $\langle M, x^*\rangle$ has independent increments for any $x^*\in X^*$ as well.

Let us show the ``if'' part.
First we reduct to the finite dimensional case.
 By the Pettis measurability theorem \cite[Theorem 1.1.20]{HNVW1} we may assume that $X$ is separable. Let $(x_n)_{n\geq 1}$ be a dense sequence in $X\setminus \{0\}$, $(x_n^*)_{n\geq 1}$ be a norming sequence, i.e.\ $\langle x_n, x_n^*\rangle = \|x_n\|$ and $\|x_n^*\|= 1$ for any $n\geq 1$ (such linear functionals exist by the Hahn-Banach theorem). For each $m\geq 1$ define $Y_m := \text{span}(x_1^*, \ldots x_m^*)$ and let $P_m: Y_m\to X^*$ be the corresponding inclusion operator. Then by the definition of $(x_n)_{n\geq 1}$ and $(x_n^*)_{n\geq 1}$ we have that the Borel $\sigma$-algebra of $X$ is generated by $(x_n^*)_{n\geq 1}$ (e.g.\ $x$ in the unit ball of $X$ if and only if $|\langle x, x_n^* \rangle|\leq 1$ for all $n\geq 1$), and so by the definition of $P_m$ we have that $M$ has independent increments if and only if $P_m^*M$ has independent increments for any $m\geq 1$. So we need to prove the theorem for any $m\geq 1$, which is equivalent to proving it for finite dimensional case as $P_m^*M$ takes values in a finite dimensional space $\text{ran}(P_m^*)$.
 
 Now let $X$ be finite dimensional. Then the theorem follows from \cite[Theorem II.4.15]{JS}.
\end{proof}

\subsection{General form of a martingale with independent increments}\label{subsec:martwithindincrgenform}

Now we are going to show that any martingale with independent increments (with values in {\em any} Banach space) has local characteristics, so there is no need in weak local characteristics. Moreover, any such a martingale has a very specific form outlined in Theorem \ref{thm:genformofmartwithindince}. Recall that a vector-valued stochastic integral of a deterministic function with respect to a compensated Poisson random measure was defined in Definition~\ref{def:defofstochintwrtPoisspoirnprod}.

\begin{theorem}\label{thm:genformofmartwithindince}
 Let $X$ and $M$ be as in Theorem \ref{thm:Mhasindinciffweallocchararedeterms}. Assume additionally that $M_0=0$. Then $M$ has the canonical decomposition $M = M^c + M^q + M^a$, where martingales $M^c$, $M^q$, and $M^a$ are independent and have independent increments, and for any $\phi:\mathbb R_+ \to \mathbb R_+$ with moderate growth and with $\phi(0)=0$ we have that
 \begin{equation}\label{eq:phiineqforindincrcandec}
    \mathbb E \sup_{t\geq 0} \phi\bigl (\|M_t\|\bigr) \eqsim_{\phi}   \mathbb E\sup_{t\geq 0} \phi\bigl ( \|M^c_t\|\bigr) +  \mathbb E\sup_{t\geq 0} \phi \bigl(\|M^q_t\|\bigr) +  \mathbb E\sup_{t\geq 0} \phi \bigl( \|M^a_t\|\bigr).
 \end{equation}
 Moreover, there exist a cylindrical Brownian motion $W_H$, $\Phi\in \gamma(L^2(\mathbb R_+;H), X)$ locally, and a deterministic time-change $\tau^c$ such that $M^c \circ \tau^c = \Phi \cdot W_H$, there exists a Poisson random measure $N_{\nu_{na}}$ on $\mathbb R_+ \times X$ with a compensator $\nu_{na}$ (which is a non-atomic in time part of $\nu^M$) such that $M^q  = \int_{[0, \cdot] \times J} x \ud \widetilde N_{\nu_{na}}$ (where $\widetilde N_{\nu_{na}} := N_{\nu_{na}} - \nu_{na}$), and finally $M^a$  is a martingale which has fixed jump times.
\end{theorem}

Recall that $\nu_{na}$ was defined in Lemma \ref{lem:decompofmeasuresontwoparts} since a measure is quasi-left continuous if and only if the corresponding compensator is non-atomic in time by Remark \ref{rem:whichwhatpartofrmmeans} (see also \cite[Theorem 9.22]{KalRM}). In order to prove Theorem \ref{thm:genformofmartwithindince} we will use these lemmas.

\begin{lemma}\label{lem:MhasIIthemathbbEsupphiMeqsimphiEphiM}
 Let $X$ be a Banach space, $M:\mathbb R_+ \times \Omega \to X$ be a martingale with independent increments. Then we have that for any $\phi:\mathbb R_+ \to \mathbb R_+$ with moderate growth and with $\phi(0) = 0$
 \begin{equation}\label{eq:EsupphiMIIMeqsimEphiM}
    \mathbb E \sup_{0\leq t\leq T} \phi\bigl(\|M_t\|\bigr) \eqsim_{\phi} \mathbb E  \phi\bigl(\|M_T\|\bigr),\;\;\; T \geq 0.
 \end{equation}
\end{lemma}

\begin{proof}
 Let $\widetilde M$ be an independent copy of $M$. Then $M-\widetilde M$ is a symmetric martingale with independent increments, so by \cite[Proposition 1.1.2]{dlPG}
 \begin{equation}\label{eq:EsupphiMIIMeqsimEphiMforMsymm}
  \begin{split}
   \mathbb E \sup_{0\leq t\leq T} \phi \bigl(\|M_t - &\widetilde M_t\|\bigr) = \int_{\lambda>0} \mathbb P\Bigl(\sup_{0\leq t\leq T} \|M_t- \widetilde M_t\|>\lambda\Bigr) \ud \phi(\lambda) \\
 &\leq 2\int_{\lambda>0} \mathbb P( \|M_{T}- \widetilde M_T\|>\lambda) \ud \phi(\lambda)  = 2 \mathbb E \phi\bigl(\|M_{T}- \widetilde M_T\|\bigr).
  \end{split}
 \end{equation}
Moreover, by conditional Jensen's inequality \cite[Proposition 2.6.29]{HNVW1}, by a triangle inequality, by the fact that $\phi$ has moderate growth, and by the fact that $M$ and $\widetilde M$ are equidistributed we have that
\begin{equation}\label{eq:symmMMIIisalsomdthesameasMwithfsojaJen}
 \begin{split}
  \mathbb E \sup_{0\leq t\leq T} \phi \bigl(\|M_t - \widetilde M_t\|\bigr) &\eqsim_{\phi} \mathbb E \sup_{0\leq t\leq T} \phi \bigl(\|M_t\|\bigr),\\
  \mathbb E \phi\bigl(\|M_{T}- \widetilde M_T\|\bigr)& \eqsim_{\phi}\mathbb E \phi\bigl(\|M_{T}\|\bigr),
 \end{split}
\end{equation}
so \eqref{eq:EsupphiMIIMeqsimEphiM} follows from \eqref{eq:EsupphiMIIMeqsimEphiMforMsymm} and \eqref{eq:symmMMIIisalsomdthesameasMwithfsojaJen}.
\end{proof}

\begin{lemma}\label{lem:McontpdqlcaccjiffMxn*forxxn*norming}
Let $X$ be a separable Banach space, $(x_n)_{n\geq 1}$ be a dense subset of $X$, $(x_n^*)_{n\geq 1}$ be a norming sequence, i.e.\ $\langle x_n, x_n^*\rangle = \|x_n\|$ and $\|x_n^*\|=1$ for any $n\geq 1$. Let $M:\mathbb R_+ \times \Omega \to X$ be a local martingale. Then
\begin{enumerate}[\rm (I)]
\item $M$ is continuous if and only if $\langle M, x_n^*\rangle$ is continuous for any $n\geq 1$,
\item $M$ is purely discontinuous if and only if $\langle M, x_n^*\rangle$ is purely discontinuous for any $n\geq 1$,
\item $M$ is quasi-left continuous if and only if $\langle M, x_n^*\rangle$ is quasi-left continuous for any $n\geq 1$,
\item $M$ is with accessible jumps if and only if $\langle M, x_n^*\rangle$ is with accessible jumps for any $n\geq 1$.
\end{enumerate}
\end{lemma}

\begin{proof}
The ``only if'' part of each of the statements is obvious. Let us show the ``if'' part. First let us start with $\rm (I)$. Assume that $M$ is not continuous. Then there exists a stopping time $\tau$ such that $\mathbb P(\Delta M_{\tau} \neq 0)>0$. Without loss of generality by multiplying $M$ by a constant we may assume that $\mathbb P(\|\Delta M_{\tau}\| >1)>0$. Fix $\eps<1/2$. Then, as $(x_n)_{n\geq 1}$ is dense in $X$, there exists $n\geq 1$ such that $\|x_n\|>1-\eps$ and for a ball $B$ with centre in $x_n$ and radius $\eps$ we have that $\mathbb P(\Delta M_{\tau} \in B)>0$ (such a ball exists as $X$ can be covered by countably many such balls). Then
\begin{multline*}
\mathbb P(\langle\Delta M_{\tau}, x_n^* \rangle \neq 0) \geq \mathbb P\bigl(\langle\Delta M_{\tau}, x_n^* \rangle \in [\|x_n\|-\eps, \|x_n\|+\eps]\bigr) \geq \mathbb P(\Delta M_{\tau} \in B)>0,
\end{multline*}
so $\langle M, x_n^*\rangle$ is not continuous and the desired follows.

\smallskip

Now let us turn to $\rm (II)$. Assume that $M$ is not purely discontinuous. By \cite[Subsection 2.5]{Y17MartDec} (see also \cite[Subsection 5.2]{DY17}) it is analogous to the fact that there exists a continuous uniformly bounded martingale $N:\mathbb R_+ \times\Omega \to \mathbb R$ such that $N_0=0$ and $MN$ is not a martingale. Moreover, by exploiting the proof of \cite[Proposition 2.10]{Y17MartDec} we even can find such $N$ that $\mathbb E M_tN_t \neq 0$ for some $t\geq 0$. On the other hand  if $\langle M, x_n^*\rangle$ is purely discontinuous for any $n\geq 1$, then by \cite[Proposition~2.10]{Y17MartDec} $\langle M, x_n^* \rangle N$ is a martingale starting in zero, so
\[
 \langle \mathbb E M_tN_t, x_n^* \rangle =  \mathbb E\langle M_tN_t, x_n^* \rangle  = \mathbb E \langle M_t, x_n^* \rangle N_t = 0,
\]
consequently $\mathbb E M_tN_t = 0$ as $(x_n^*)_{n\geq 1}$ is a norming sequence, and thus $M$ is purely discontinuous.

\smallskip

Let us show $\rm(III)$. Let $\tau$ be a predictable stopping time. Then it can be shown that $\Delta M_{\tau} = 0$ a.s.\  analogously $\rm (I)$, so $M$ is quasi-left continuous. $\rm (IV)$ follows similarly.
\end{proof}

\begin{corollary}\label{cor:candeciffcandecforanormingseq}
Let $X$, $(x_n)_{n\geq 1} \in X$ and $(x_n^*)_{n\geq 1} \in X^*$ be as in Lemma \ref{lem:McontpdqlcaccjiffMxn*forxxn*norming}. Let $M, M^c, M^q, M^a:\mathbb R_+ \times \Omega \to X$ be local martingales. Then $M = M^c + M^q + M^a$ is the canonical decomposition of $M$ if and only if $\langle M, x_n^* \rangle= \langle M^c, x_n^* \rangle + \langle M^q, x_n^* \rangle + \langle M^a, x_n^* \rangle$ is the canonical decomposition of $\langle M, x_n^* \rangle$ for any $n\geq 1$.
\end{corollary}

Eventually we are going to show Theorem \ref{thm:genformofmartwithindince}.

\begin{proof}[Proof of Theorem \ref{thm:genformofmartwithindince}]
Without loss of generality assume that $M_0=0$ a.s. We can also set that there exists $T>0$ such that $M_t= M_T$ for $t\geq T$, so $\mathbb E\sup_{t\geq 0} \|M_t\| \eqsim \mathbb E \|M_T\|<\infty$ (see Remark \ref{rem:locmartwithIIisamartwithII} and Lemma \ref{lem:MhasIIthemathbbEsupphiMeqsimphiEphiM}). First of all let us prove the first part of the proposition in the finite dimensional case, and then treat the whole proposition in infinite dimensions.

 {\em Step 1. $X$ is finite dimensional.}
 First assume that $X$ is finite dimensional. Then the existence of the canonical decomposition is guaranteed by Theorem \ref{thm:candecXvalued}. Let us show that $M^c$, $M^q$, and $M^a$ are independent and have independent increments. By Proposition \ref{prop:Grigcharcontcase} $M^c$ has local characteristics $([\![M^c]\!], 0)$. Further, by Lemma \ref{lem:nuMisnonatomiffMqlc}, Proposition \ref{prop:GigcharforPDQLC}, \ref{prop:GrigcharforPDmAJ} and \ref{prop:Grigcharforcandechowdoesitlooklike} $M^q$ has local characteristics $(0, \nu^M_{na})$, where $\nu^M_{na}$ is the nonatomic part of $\nu^M$, and $M^a$ has local characteristics $(0, \nu^M_{a})$, where $\nu^M_{a}$ is the atomic part of $\nu^M$. Each of three local characteristics are deterministic, so by Theorem \ref{thm:Mhasindinciffweallocchararedeterms} each of $M^c$, $M^q$, and $M^a$ has independent increments.
Let us show that  $M^c$, $M^q$, and $M^a$ are independent. By the L\'evy-Khinchin-type formula \cite[II.4.16]{JS} (see also Remark \ref{rem:LKformdoeinqaindwoe}) and by the fact that $M$, $M^c$, $M^q$, and $M^a$ have independent increments we have that for any $t_0 < t_1 <\ldots <t_N$,
for any numbers $(a_n)_{n=1}^N$, $(b_n)_{n=1}^N$, and $(c_n)_{n=1}^N$, and for any vectors $(x^*_n)_{n=1}^N, (y^*_n)_{n=1}^N, (z^*_n)_{n=1}^N\in X^*$
 \begin{align*}
   \mathbb E &\exp\Bigl\{\sum_{n=1}^N \langle x^*_n, M^c_{t_n} - M^c_{t_{n-1}}\rangle + \langle y^*_n, M^q_{t_n} - M^q_{t_{n-1}}\rangle + \langle z^*_n, M^a_{t_n} - M^a_{t_{n-1}}\rangle\Bigr\}\\
  &\stackrel{(i)}= \Pi_{n=1}^N \mathbb E \exp\bigl\{\langle x^*_n, M^c_{t_n} - M^c_{t_{n-1}}\rangle + \langle y^*_n, M^q_{t_n} - M^q_{t_{n-1}}\rangle + \langle z^*_n, M^a_{t_n} - M^a_{t_{n-1}}\rangle\bigr\}\\
  &\stackrel{(ii)}=\Pi_{n=1}^N \mathbb E e^{\langle x^*_n, M^c_{t_n} - M^c_{t_{n-1}}\rangle} \mathbb E e^{\langle y^*_n, M^q_{t_n} - M^q_{t_{n-1}}\rangle} \mathbb E e^{\langle z^*_n, M^a_{t_n} - M^a_{t_{n-1}}\rangle}\\
  &\stackrel{(iii)}= \mathbb E e^{\sum_{n=1}^N \langle x^*_n, M^c_{t_n} - M^c_{t_{n-1}}\rangle} \mathbb E e^{\sum_{n=1}^N \langle y^*_n, M^q_{t_n} - M^q_{t_{n-1}}\rangle} \mathbb E e^{\sum_{n=1}^N \langle z^*_n, M^a_{t_n} - M^a_{t_{n-1}}\rangle},
 \end{align*}
where $(i)$ follows from the fact that $M$, $M^c$, $M^q$, and $M^a$ have independent increments (so that $M^i_{t_n} - M^i_{t_{n-1}}$ is independent of $M^j_{t_{m}} - M^j_{t_{m-1}}$ for any $i, j \in \{c,q, a\}$, $i\neq j$, and any $n\neq m$, which can be shown by the L\'evy-Khinchin-type formula \cite[II.4.16]{JS} and by \cite[Theorem II.12.4]{ShPr2e}),
$(ii)$ follows by \cite[II.4.16]{JS}, 
and $(iii)$ follows analogously $(i)$. Now the desired independence follows from \cite[Theorem II.12.4]{ShPr2e}.

Now let us show \eqref{eq:phiineqforindincrcandec}. By Lemma \ref{lem:MhasIIthemathbbEsupphiMeqsimphiEphiM} it is sufficient to show that
 \begin{equation}\label{eq:padhfiineqforindincrcandecwithoutsup}
    \mathbb E  \phi\bigl (\|M_t\|\bigr) \eqsim_{\phi}   \mathbb E \phi\bigl ( \|M^c_t\|\bigr) +  \mathbb E \phi \bigl(\|M^q_t\|\bigr) +  \mathbb E \phi \bigl( \|M^a_t\|\bigr),\;\;\; t\geq 0.
 \end{equation}
First $M^c_t$, $M^q_t$, and $M^a_t$ are independent mean zero, so for any $i\in \{c,q,a\}$ we have that $M^i_t = \mathbb E (M_t|\sigma(M^i_t))$, consequently by Jensen's inequality
\[
 \mathbb E \phi(\|M^i_t\|) = \mathbb E \phi\bigl ( \bigl\| \mathbb E (M_t|\sigma(M^i_t)) \bigr\| \bigr) \leq \mathbb E\Bigl[ \mathbb E \bigl(\phi(\|M_t\|)| \sigma(M^i_t) \bigr)\Bigr] = \mathbb E  \phi(\|M_t\|),
\]
so $\gtrsim$ in \eqref{eq:padhfiineqforindincrcandecwithoutsup} follows.
On the other hand as $\phi$ has moderate growth and as $M = M^c + M^q +M^a$ we have $\lesssim_{\phi}$ of \eqref{eq:padhfiineqforindincrcandecwithoutsup}.

\smallskip

{\em Step 2. $X$ is general.}
Now let $X$ be general. We will constrict each part of the canonical decomposition separately and show that each of them has the form predicted in the second part of the theorem.

{\em Step 2. Part 1. Construction of $M^a$.}
Let $(t_m)_{m\geq 1} \subset \mathbb R_+$ be such that $\mathbb P(\Delta M_{t_m}\neq 0) >0$ for any $m\geq 1$ (recall that c\`adl\`ag processes have at most countably many jumps, so the set of such $t_m$'s is at most countable). For each $t\geq 0$ define 
\[
\mathcal F^a := \sigma(\Delta M_{t_m}:m\geq 1),\;\;\; M^a_t := \mathbb E (M_t |\mathcal F^a),\;\;\; t\geq 0.
\]
Let us show that $M^a$ is an $\mathbb F$-martingale. Let 
$$
\mathcal F^a_t := \sigma(\Delta M_{t_m}:0\leq t_m\leq t),\;\;\; \mathcal F^a_{>t} := \sigma(\Delta M_{t_m}:t_m>t),\;\;\; t\geq 0.
$$
Then for any $t\geq 0$
$$
\mathbb E (M_t |\mathcal F^a) \stackrel{(*)}= \mathbb E (M_t |\mathcal F^a_t \otimes \mathcal F^a_{>t}) \stackrel{(**)}=  \mathbb E (M_t |\mathcal F^a_t),
$$
where $(*)$ holds from the fact that $\mathcal F^a_t$ and $\mathcal F^a_{>t}$ are independent, and $(**)$ holds since $M_t$ is independent of $\mathcal F^a_{>t}$. Therefore $M_t^a$ is $\mathcal F_t$-measurable.
Now fix $t\geq s\geq 0$. Then $M_t - M_s$ is independent of $\mathcal F^a_{s}$ and $\mathcal F^a_{>s}$ is independent of $\mathcal F_s$, and thus 
\begin{multline*}
\mathbb E (M^a_t - M^a_s|\mathcal F_s) = \mathbb E \bigl( \mathbb E (M_t -M_s|\mathcal F^a_s \otimes \mathcal F^a_{>s})|\mathcal F_s\bigr)\\
 = \mathbb E \bigl( \mathbb E (M_t -M_s| \mathcal F^a_{>s})|\mathcal F_s\bigr)  = \mathbb E (M_t-M_s)=0,
\end{multline*}
consequently, $M^a$ is a martingale. Let us show that $M^a$ is a purely discontinuous martingale with accessible jumps. To this end it is sufficient to notice that $M^a$ is an ${\mathbb F}^a$-martingales (where ${\mathbb F}^a=  ({\mathcal F}^a_t)_{t\geq 0}$), so for any $x^*\in X^*$ a process $t\mapsto \langle M^a_t, x^*\rangle$ is ${\mathbb F}^a$-adapted, and hence it is purely discontinuous with jumps in the set $(t_m)_{m\geq 1}$ because of the structure ${\mathbb F}^a$. Indeed, let $L:= \langle M^a, x^*\rangle$. For each $k\geq 1$ define 
$$
L^k_t := \sum_{m\leq k} \mathbb E (L_t| \Delta \mathcal F^a_{t_m}),\;\;\; t\geq 0,
$$
where $ \Delta \mathcal F^a_{t_m} := \sigma(\Delta M_{t_m})$. Then $L^k$ is an ${\mathbb F}^a$-martingale which has jumps of the size $(\langle \Delta M_{t_m}, x^* \rangle)_{n=1}^m$ at $(t_m)_{m=1}^k$, and $L^k$ converges to $L$ in $L^1(\Omega)$ by \cite[Theorem 3.3.2]{HNVW1}, the definition of $L$ and $M^a$, and by the fact that $\mathcal F^a = \otimes_n \Delta \mathcal F^a_{t_n}$, so 
$$
L_t- L^k_t = \mathbb E (L_t |\otimes_{m>k} \Delta\mathcal F^a_{t_m}),\;\;\; t\geq 0.
$$
 Thus $L$ is purely discontinuous with jumps of size $(\langle \Delta M_{t_m}, x^* \rangle)_{m\geq 1}$ at $(t_m)_{m\geq 1}$ by the fact that purely discontinuous martingales with accessible jumps form a closed subspace of $L^1(\Omega)$, see e.g.\ \cite[Proposition 3.30]{Y17MartDec} or \cite{Kal}, and hence $M^a$ is purely discontinuous with jumps of the size $( \Delta M_{t_m})_{m\geq 1}$ at $(t_m)_{m\geq 1}$, so it has accessible jumps.
 
 {\em Step 2. Part 2. Construction of $M^c$.} Let us now construct $M^c$. By the Pettis measurability theorem \cite[Theorem 1.1.20]{HNVW1} $X$ can be presumed separable. Let $(x_n)_{n\geq 1}$ be a dense sequence in $X$. Let $(x_n^*)_{n\geq 1}$ be a norming sequence in $X^*$, i.e.\ $\|x^*_n\|=1$ and $\langle x_n^*, x_n\rangle=\|x_n\|$ for any $n\geq 1$. For each $n\geq 1$ let $M^n:= \langle M, x_n^*\rangle$. Let $M^n = M^{n, c} + M^{n, q} + M^{n, a}$ be the corresponding canonical decomposition. By a stopping time argument, by a rescaling argument, and by Lemma \ref{lem:MhasIIthemathbbEsupphiMeqsimphiEphiM} we may assume that $\mathbb E \sup_{t\geq 0} \|M_t\| \leq 1$. Then by \eqref{eq:candecstrongLpestmiaed}, \cite[Theorem 26.12 and 26.14]{Kal} we have that
 \begin{equation}\label{eq:sum1/n2[Mcn]isboddasa}
  \begin{split}
       \mathbb E \sum_{n=1}^{\infty} \frac 1{n^2} [M^{n, c}]_{\infty}^{1/2}& \leq   \mathbb E \sum_{n=1}^{\infty} \frac 1{n^2} [M^{n}]_{\infty}^{1/2} \eqsim \sum_{n=1}^{\infty} \frac 1{n^2}\mathbb E \sup_{t\geq 0} |M^n_t|\\
    &\leq \sum_{n=1}^{\infty} \frac 1{n^2}\mathbb E \sup_{t\geq 0} |\langle M_t, x_n^*\rangle| \leq \sum_{n=1}^{\infty} \frac 1{n^2}\mathbb E \sup_{t\geq 0} \| M_t\| \leq \frac{\pi^2}{6},
  \end{split}
 \end{equation}
so $([M^{n, c}]_{\infty}/n^4)_{n\geq 1}$ are uniformly bounded a.s. Note that $M$ has independent increments, so by Theorem \ref{thm:Mhasindinciffweallocchararedeterms} it has deterministic weak local characteristics, and thus $t\mapsto [M^{n, c}]_{t}/n^4$ equals a finite deterministic constant a.s.\ for any $t\geq 0$ and $n\geq 1$. Therefore without loss of generality we may assume that all the processes $t\mapsto [M^{n, c}]_{t}/n^4$, $t\geq 0$, are continuous (see \cite[Theorem 26.14]{Kal}), deterministic, and uniformly bounded by \eqref{eq:sum1/n2[Mcn]isboddasa}. Let us then define a deterministic function
\[
 A_t := \sum_{n=1}^{\infty}\frac{1}{n^6}[M^{n, c}]_t + t,\;\;\;; t\geq 0,
\]
and let $(\tau^c_s)_{s\geq 0}$ be a deterministic time change defined by $\tau^c_s := \inf\{t\geq 0:A_t = s\}$ for all $s\geq 0$. Then by Lemma \ref{lem:brrepres} there exist a Hilbert space $H$, an enlarged probability space $(\overline{\Omega}, \overline{\mathcal F}, \overline{\mathbb P})$ with a cylindrical Brownian motion $W_H$ living on this space (here we set the enlargement filtration to be $\overline {\mathbb F} = (\overline{\mathcal F}_{t})_{t\geq 0}$ is defined by $\overline{\mathcal F}_{t} := \sigma(\mathcal F_t, W_H|_{[0, A_{t}]})$), and a set of {\em deterministic} functions $f_n:\mathbb R_+ \to H$ (note that by Lemma \ref{lem:brrepres} $(f_n)_{n\geq 1}$ depends on $([M^{n, c}, M^{m, c}])_{m, n\geq 1}$ which are deterministic as $M$ has deterministic weak local characteristics by Theorem \ref{thm:Mhasindinciffweallocchararedeterms}) such that $M^{n, c}\circ \tau^c := f_n \cdot W_H$. Let $M^c_t := \mathbb E (M_t|\sigma(W_H))$ for every $t\geq 0$. Let us show that $M^c$ is a continuous martingale. First $M^c$ is a martingale as we have that for any $t\geq s\geq 0$ the martingale difference $M_t- M_s$ is independent of $\sigma(W_H|_{[0, A_s]})$ as $M$ is  a martingale with independent increments and by the construction of $W_H$, so as $\sigma(W_H|_{[A_s, \infty)})$ is independent of $\overline{\mathcal F}_s$ we have that (here we for simplicity write $\sigma(W_H|_{[A_s, \infty)})$ instead of $\sigma((W_H - W_H(A_s))|_{[A_s, \infty)}) =\sigma(\dd W_H|_{[A_s, \infty)}) $)
\begin{multline*}
 \mathbb E(M^c_t-M^c_s|\overline{\mathcal F}_s) = \mathbb E \bigl( \mathbb E(M_t - M_s|\sigma(W_H)) \big|\overline{\mathcal F}_s \bigr) \\
 = \mathbb E \bigl( \mathbb E(M_t - M_s|\sigma(W_H|_{[A_s, \infty)})) \big|\overline{\mathcal F}_s \bigr) = 0,
\end{multline*}
hence $M^c$ is a martingale. Let us show that it is continuous. As $(x_n^*)_{n\geq 1}$ is a norming sequence,  by Lemma \ref{lem:McontpdqlcaccjiffMxn*forxxn*norming} it is sufficient to show that $\langle M^c, x_n^*\rangle$ is continuous for any $n\geq 1$, so it is enough to prove that $\langle M^c, x_n^*\rangle = M^{n, c}$. First notice that by the construction of $W_H$ in Lemma \ref{lem:brrepres} the latter depends only on $(M^{n, c})_{n\geq 1}$. Next note that the families $(M^{n, q})_{n\geq 1}$ and $(M^{n, a})_{n\geq 1}$ are independent of $(M^{n, c})_{n\geq 1}$ which follows from Step 1 of the present proof (Step 1 proves the independence directly for $(M^{n, q})_{n= 1}^N$, $(M^{n, a})_{n= 1}^N$, and $(M^{n, c})_{n=1}^N$ for any $N\geq 1$, and the desired independence follows by letting $N\to \infty$). Finally, we have that for any $n\geq 1$ and for any $t\geq 0$ a.s.\
\begin{multline*}
 \langle M^c_t, x_n^*\rangle = \langle \mathbb E (M_t|\sigma({W_H})), x_n^*\rangle = \mathbb E (\langle M_t, x_n^*\rangle|\sigma({W_H})) \\
 = \mathbb E (M^{n, c} + M^{n, q} + M^{n, a}|\sigma({W_H})) \stackrel{(*)}= M^{n, c},
\end{multline*}
where $(*)$ follows from the fact that $W_H$ may be assumed to depend only on $(M^{n, c})_{n\geq 1}$ and the fact that $M^{n, q}$ and $M^{n, a}$ are independent of $(M^{n, c})_{n\geq 1}$.

Now let us show that there exists $\Phi \in \gamma(L^2(\mathbb R_+; H), X)$ such that $M^c \circ \tau^c = \Phi \cdot W_H$. First notice that for any $x^*\in X^*$ a martingale $\langle M^c, x^* \rangle \circ \tau^c$ is adapted with respect to the filtration $\mathbb G := (\mathcal G_s)_{s\geq 1}$ generated by $W_H$. Therefore by the martingale representation theorem (see \cite[\S V.3]{RY} for the case of finite dimensional $H$, the infinite dimensional case can be shown analogously) there exists a $\mathbb G$-predictable process $f^{x^*}:\mathbb R_+ \times \Omega \to H$ such that $\langle M^c, x^* \rangle \circ \tau^c = f^{x^*}\cdot W_H$. Note that $f^{x^*}$ is deterministic. Indeed, as $X$ can be assumed separable, the unit ball of $X^*$ is sequentially weak$^*$ compact by sequential Banach-Alaoglu theorem, so we may assume that $(x_n^*)_{n\geq 1}$ is  weak$^*$ dense in the unit sphere of $X^*$. So for a sequence $(y_m)_{m\geq 1} \subset (x_n^*)_{n\geq 1}$ weak$^*$ converging to $x^*$ we have that by Burkholder-Davis-Gundy inequalities \cite[Theorem 26.12]{Kal}, by Lemma \ref{lem:MhasIIthemathbbEsupphiMeqsimphiEphiM}, and by the dominated convergence theorem
\begin{align*}
\mathbb E \Bigl(\int_0^{A_t} \|f^{x^*}(s) - f^{y_m}(s)\|^2 \ud s\Bigr)^{1/2} &\eqsim \mathbb E \sup_{0\leq t\leq T}| \langle M_t, x^* - y_m\rangle| \\
&\eqsim \mathbb E| \langle M_T, x^* - y_m\rangle|  \to 0,\;\;\; m\to \infty,
\end{align*}
so $f^{x^*}$ is deterministic as the limit of $f^{y_m}$ which are deterministic.
Also note that by our assumption from the very beginning of the proof $\mathbb E \|M_{\infty}\| =\mathbb E \|M_{T}\| <\infty$ for some fixed $T>0$. Therefore as we have that $\langle M^c_{\infty}, x^* \rangle = \int_0^{A_T}f^{x^*}\ud W_H$ is a Gaussian random variable for any $x^*\in X^*$ (since $f^{x^*}$ is deterministic), $M^c_{\infty}$ is a Gaussian random variable itself, so by Fernique's inequality \cite[Theorem 2.8.5]{BGaus} we have that $\mathbb E \|M^c_T\|^2 <\infty$. Let $\Phi \in \gamma(L^2(\mathbb R_+; H), X)$ be defined in the following way: $\Phi f = \mathbb E M^c_{\infty}N_{\infty}$, where $N = f\cdot W_H$ for any step deterministic $f\in L^2(\mathbb R_+; H)$. This $\Phi$ is bounded as by It\^o's isometry \cite[Proposition 4.13]{DPZ} (see e.g.\ \cite[Lemma 3.1.5]{Oks98} for the finite dimensional version) and by H\"older's inequality for any $f\in L^2(\mathbb R_+; H)$ step deterministic
\[
\|\Phi f\| = \|\mathbb E M^c_{\infty}N_{\infty}\| \leq (\mathbb E \|M^c_T\|^2)^{1/2} (\mathbb E |N_{\infty}|^2)^{1/2} = (\mathbb E \|M^c_T\|^2)^{1/2} \|f\|_{L^2(\mathbb R_+; X)},
\]
and $\gamma$-radonifying by \cite[Subsection 3.2]{Y18BDG} since 
$$
\mathbb E \langle M^c_{T}, x^*\rangle\langle M^c_{T}, y^*\rangle = \int_0^{\infty} \langle f^{x^*}(s), f^{y^*}(s) \rangle \ud s \stackrel{(*)}= \langle \Phi^*x^*, \Phi^*y^*\rangle,\;\;\; x^*,y^*\in X^*,
$$ 
is a covariation bilinear form of a Gaussian random variable $M^c_T$, where $(*)$ follows from the fact that by It\^o's isometry \cite[Proposition 4.13]{DPZ} and by the definition of~$f^{x^*}$
\begin{equation}\label{eq:Phi*x*fx*svabodu}
\begin{split}
 \langle \Phi^*x^*, f\rangle = \langle x^*, \Phi f\rangle& =\Bigl\langle x^*, \mathbb E M^c_t \int_0^{\infty} f \ud W_H \Bigr\rangle  = \mathbb E \langle x^*, M^c_t \rangle \int_0^{\infty} f \ud W_H \\
 & = \mathbb E \int_0^{\infty} f^{x^*} \ud W_H \int_0^{\infty} f\ud W_H\\
 &  =  \int_0^{\infty} \langle f^{x^*}, f \rangle \ud s,\;\;\; f\in L^2(\mathbb R_+; H), \;\; x^*\in X^*. 
 \end{split}
 \end{equation}
Now in order to show that $\Phi \cdot W_H$ coincides with $M^c \circ \tau^c$ it is sufficient to notice that  by \eqref{eq:Phi*x*fx*svabodu} $\Phi^*x^* = f^{x^*}$, so
$$
(\Phi^*x^*) \cdot W_H = f^{x^*} \cdot W_H = \langle M^c, x^* \rangle \circ \tau^c = \langle M^c\circ\tau^c, x^* \rangle,\;\;\; x^*\in X^*,
$$ 
and thus the desired follows from \cite[Theorem 6.1]{vNW08}.

  {\em Step 2. Part 3. Construction of $M^q$.} 
   Now let us show that $M^q:= M -M^c- M^a$ is quasi-left continuous purely discontinuous and has the following form $M^q = \int x \ud \bar{\mu}^{M^q} = \int x \ud \widetilde N_{\nu_{na}}$ for some Poisson random measure $N_{\nu_{na}}$ with a compensator $\nu_{na}$. First notice that $M^q$ is purely discontinuous quasi-left continuous by Corollary \ref{cor:candeciffcandecforanormingseq} as we have that for $(x_n^*)_{n\geq 1}$ exploited in Step 2. Part 2 $\langle M^c, x_n^* \rangle$ is the continuous part of $\langle M, x_n^* \rangle$ for any $n\geq 1$. Moreover, $\langle M^a, x_n^* \rangle$ is the purely discontinuous with accessible jumps part of $\langle M, x_n^* \rangle$ as $M^a$ collects all the deterministic-time jumps of $M$, and since by Theorem \ref{thm:Mhasindinciffweallocchararedeterms} $\nu^{\langle M, x_n^* \rangle}$ is deterministic for any $n\geq 1$, its atomic part $\nu^{\langle M, x_n^* \rangle}_a$ (which coincides with $\nu^{\langle M, x_n^* \rangle^a}$ by Proposition \ref{prop:Grigcharforcandechowdoesitlooklike} and Remark \ref{rem:whichwhatpartofrmmeans}) has a deterministic support, which is a subset of $(t_m)_{m\geq 1}$ presented in Step 2. Part 1 as if $\mathbb P(\Delta {\langle M, x_n^* \rangle}_t \neq 0)>0$ for some $t\geq 0$, then $\mathbb P(\Delta M_t \neq 0)>0$, so the jump times of $\langle M, x_n^* \rangle^a$ are covered by and coincide with the jump times of $\langle M^a, x_n^* \rangle$, consequently $\langle M^a, x_n^* \rangle$ is the purely discontinuous with accessible jumps part of $\langle M, x_n^* \rangle$ for any $n\geq 1$, and thus $M^q$ is the purely discontinuous quasi-left continuous part of the canonical decomposition of $M$.
   
   Next let us show that $\mu^{M^q}$ is a Poisson random measure with a compensator $\nu^{M^q} = \nu_{na}$ (the letter equality follows from Proposition \ref{prop:Grigcharforcandechowdoesitlooklike}, Lemma \ref{lem:decompofmeasuresontwoparts}, and Remark \ref{rem:whichwhatpartofrmmeans}). First note that $M^q$ is independent of $M^c$ and $M^a$ and that $M^q$ has independent increments. This follows from a standard finite dimensional argument (see the proof of Theorem \ref{thm:tangiffwtangUMDcase}), Step 1, and the Cram\'er-Wold theorem (see \cite[Theorem 29.4]{Bill95}). Now let us fix disjoint cylindrical sets $B_1, \ldots, B_K \in \mathcal B(X)$ (see the proof of Theorem \ref{thm:tangiffwtangUMDcase}) satisfying ${\rm dist}(B_k, \{0\}) >\eps$ for any $k=1,\ldots, K$ for some fixed $\eps>0$. Then for any stopping time $\tau$ we have that
   \begin{equation}\label{eq:B_karelocallyfiniteinnu-na}
   \mathbb E \int_{[0, \tau]\times B_k} 1 \ud \nu^{M^q}  =   \mathbb E \int_{[0, \tau]\times B_k} 1 \ud \mu^{M^q} = \mathbb E \sum_{0\leq s \leq \tau} \mathbf 1_{B_k}(\Delta M^q_s),
   \end{equation}
    and the latter is locally finite if one chooses $\tau$ to be the time of $n$th jump of $M^q$ of value more than $\eps$. Therefore we can define point processes $L^1, \ldots,L^K :\mathbb R_+ \times \Omega \to \mathbb N_0$ satisfying $L^k_t = \mu^{M^q}({[0, t]\times B_k})$ for any $k=1,\ldots, K$ for any $t\geq 0$. But then by \cite[Corollary 25.26]{Kal} and Step 1 these processes are times-changed Poissons, where the time-changes are deterministic as processes $\nu_{na}({[0, t]\times B_k})$ are deterministic since $\nu_{na}$ is so. Therefore $\mu^{M^q}|_{\mathbb R_+ \times X\setminus B(0, \eps)}$ is a Poisson random measure with the compensator $\nu_{na}|_{\mathbb R_+ \times X\setminus B(0, \eps)}$ (here $B(0, \eps) \subset X$ is the ball in $X$ with the radius $\eps$ and the centre in $0$), and then $\mu^{M^q}$ is Poisson as we can send $\eps\to 0$ and use the fact that by \eqref{eq:defofmuM} we have that $\mu^{M^q}(\mathbb R_+ \times \{0\}) = 0$ a.s. Therefore we can set $N_{\nu_{na}}:= \mu^{M^q}$ and $\widetilde N_{\nu_{na}} = \bar{\mu}^{M^q}$.
    
 Finally, let us prove that $M^q = \int x \ud \bar{\mu}^{M^q} = \int x \ud \widetilde N_{\nu_{na}}$. Recall that  the definition of such an integrability was discussed in Subsection \ref{subsec:PoissRMprelim}. Let us show that there exist an increasing family $(A_n)_{n\geq 1}$ of elements of $\mathcal B(\mathbb R_+)\otimes \mathcal B(X)$ such that $\cup_{n} A_n = \mathbb R_+\times X$, $\int_{A_n} \|x\| \ud \nu_{na}<\infty$ for any $n\geq 1$, and $\int_{A_n} x \ud \widetilde N_{\nu_{na}}$ converges in $L^1(\Omega) $ to $M^q_{\infty} = M^q_T$. For every $k\in \mathbb Z$ let $B_k := B(0, 2^k)\setminus B(0, 2^{k-1})$. By \eqref{eq:B_karelocallyfiniteinnu-na} and the discussion thereafter we have that 
 \begin{equation}\label{eq:int||x||wrtnunaontmaptoB_kfiniteandcaos}
 t\mapsto \int_{[0, t] \times B_k} \|x\| \ud \nu_{na} \leq 2^k \int_{[0, t] \times B_k} 1\ud \nu_{na} <\infty,\;\;\; t\geq 0.
 \end{equation}
  Moreover, the process \eqref{eq:int||x||wrtnunaontmaptoB_kfiniteandcaos} is continuous as $\nu_{na}$ is nonatomic in time. Thus for any $n\geq 1$ there exists $t_n^k$ such that $ \int_{[0, t_n^k] \times B_k} \|x\| \ud \nu_{na} \leq  n2^{-k}$. Without loss of generality we may assume that $(t_n^k)_{n\geq 1}$ is an increasing sequence. Moreover,  we may assume that $t_n^k \to \infty$ as $n\to \infty$ for any $k\geq 1$. For each $n\geq 1$ let us set
  \begin{equation}\label{eq:AnforangeneraMqfortheproofofII}
  A_n:=\bigl (\mathbb R_+ \times \{0\} \bigr) \cup_{k\geq 1} \bigl([0, t_n^k] \times B_k\bigr).
  \end{equation}
  Then by the construction of $(t_n^k)_{n, k\geq 1}$ and by the fact that $\nu_{na}(\mathbb R_+ \times \{0\} ) = \mu^{M^q}(\mathbb R_+ \times \{0\})=0 $ a.s.\ by \eqref{eq:defofmuM} we have that $\int_{A_n}\|x\| \ud \nu_{na}<\infty$ by \eqref{eq:int||x||wrtnunaontmaptoB_kfiniteandcaos} and \eqref{eq:AnforangeneraMqfortheproofofII}. Let $\xi_n := \int_{A_n}x \ud \widetilde N_{\nu_{na}}$ for every $n\geq 1$ (see Remark \ref{rem:defofdtochintwrtpoisrmforFoffinitenumes}). Let $\xi := M^q_{\infty}$. By \cite[Theorem 3.3.2]{HNVW1} in order to show that $\xi_n \to \xi$ in $L^1(\Omega; X)$ it is sufficient to prove that 
  \begin{equation}\label{eq:xiNiscondexpxiwithNonA_n}
  \xi_n = \mathbb E (\xi|\sigma(N_{\nu_{na}}|_{A_n})),\;\;\;n\geq 1.
  \end{equation}
   Fix $n\geq 1$. To this end it is enough to show that $\langle \xi_n, x^* \rangle =  \mathbb E (\langle \xi, x^* \rangle|\sigma(N_{\nu_{na}}|_{A_n}))$ for any $x^*\in X^*$. Fix $x^*\in X^*$. Then $\langle \xi, x^* \rangle = \int_{\mathbb R_+ \times X} \langle x, x^* \rangle \ud  \widetilde N_{\nu_{na}}(\cdot, x)$ as by Burkholder-Davis-Gundy inequalities \cite[Theorem 26.12]{Kal} and by the dominated convergence theorem
\begin{multline*}
\mathbb E \sup_{t\geq 0} \Bigl| \langle M^q_t, x^* \rangle - \int_{[0, t]\times X \cap A_n} \langle x, x^* \rangle  \ud  \widetilde N_{\nu_{na}}(\cdot, x) \Bigr|\\
 \eqsim \mathbb E \Bigl( \sum_{t\geq 0} \mathbf 1_{\overline A_n}(t, \Delta M^q)|\langle\Delta M^q, x^* \rangle |^2\Bigr)^{1/2} \to 0\;\;\; n\to \infty,
\end{multline*}
where $\overline A_n \subset \mathbb R_+ \times X$ is the completion of $A_n$. Therefore
\begin{align*}
\mathbb E (\langle \xi, x^* \rangle|\sigma(N_{\nu_{na}}|_{A_n})) &= \mathbb E \Bigl(\int_{\mathbb R_+ \times X} \langle x, x^* \rangle \ud  \widetilde N_{\nu_{na}}(\cdot, x)\Big|\sigma(N_{\nu_{na}}|_{A_n})\Bigl)\\
&=  \mathbb E \Bigl(\int_{A_n} + \int_{\overline A_n} \langle x, x^* \rangle \ud  \widetilde N_{\nu_{na}}(\cdot, x)\Big|\sigma(N_{\nu_{na}}|_{A_n})\Bigl)\\
&\stackrel{(*)}= \int_{A_n}  \langle x, x^* \rangle \ud  \widetilde N_{\nu_{na}}(\cdot, x) = \langle \xi_n, x^* \rangle,
\end{align*}
where $(*)$ holds from the fact that $N$ is a Poisson random measure so $N_{\nu_{na}}|_{A_n}$ and $N_{\nu_{na}}|_{\overline A_n}$ are independent, and the fact that $\mathbb E \int_{\overline A_n} \langle x, x^* \rangle \ud  \widetilde N_{\nu_{na}}(\cdot, x) = 0$. Therefore \eqref{eq:xiNiscondexpxiwithNonA_n} holds true, and thus $\xi_n \to \xi$ in $L^1(\Omega; X)$ by the It\^o-Nisio theorem \cite[Theorem 6.4.1]{HNVW2}, so $M^q= \int x \ud \widetilde N_{\nu_{na}}(\cdot, x)$.

{\em Step 3. Proving \eqref{eq:phiineqforindincrcandec}}. Finally let us show \eqref{eq:phiineqforindincrcandec}. This estimates follow analogously finite dimensional case proven in Step 1, with exploiting the fact that $M^c$, $M^q$, and $M^a$ are independent by the Cram\'er-Wold theorem (see \cite[Theorem 29.4]{Bill95}) and by Step 1.
\end{proof}

 \begin{remark}
Note that $\Phi$ is locally in $\gamma(L^2(\mathbb R_+; H), X)$, so by \cite[Subsection 3.2]{Y18BDG} and by \eqref{eq:[[Mc]]isdefbyPhiforMcwthidnsdincs} we have that $\gamma([\![M^c]\!]_t) = \|\Phi\|_{\gamma(L^2([0, t]; H), X)}<\infty$.
\end{remark}

The following corollary is an extension of the famous result of Grigelionis \cite{Grig77} (see also \cite[p.\ 106]{JS}) to infinite dimensions.

\begin{corollary}\label{cor:GrigcharformartwithIIXgeneral}
 Let $X$ be a Banach space, $M:\mathbb R_+ \times \Omega \to X$ be a martingale. Then $M$ has independent increments of and only if it has local characteristics which are deterministic. 
\end{corollary}

\begin{proof}
The ``if'' part follows from Theorem \ref{thm:Mhasindinciffweallocchararedeterms}. Let us show the ``only if'' part.
The fact that $M$ admits the Meyer-Yoeurp decomposition and that $\nu^M$ is deterministic was shown in Theorem \ref{thm:genformofmartwithindince}. Let us show that $[\![M^c]\!]$ exists. This follows from the fact that for any $t\geq 0$ a.s.\ for any $x^*, y^*\in X^*$ by the Kazamaki theorem \cite[Theorem 17.24]{Kal} and by \eqref{eq:qvvarofstintwrtHcylbrmot} (recall that $\Phi^*x^*$ is locally in $L^2(\mathbb R_+; H)$ for any $x^*\in X^*$)
\begin{equation}\label{eq:[[Mc]]isdefbyPhiforMcwthidnsdincs}
\begin{split}
|[\![M^c]\!]_t(x^*, y^*)|  &= \bigl|[\langle M^c, x^* \rangle,\langle M^c, y^* \rangle]_t\bigr| \\
&= \Bigl|\Bigl[\int \Phi^* x^* \ud W_H, \int \Phi^* y^* \ud W_H \Bigr]_{A_t} \Bigr|\\
& = \Bigl| \int_0^{A_t} \langle \Phi x^*(s) ,\Phi^*y^*(s) \rangle \ud s \Bigr| \leq \|\Phi\|_{\mathcal L(L^2([0, A_t]), X)}^2\|x^*\| \|y^*\|,
\end{split}
\end{equation}
so $[\![M^c]\!]_t$ is a.s.\ a bounded bilinear form. 
\end{proof}

\begin{corollary}\label{cor:ifMIIthendirstofMisundetbylocchas}
Let $X$ be a Banach space, $M:\mathbb R_+ \times \Omega \to X$ be a martingale with independent increments satisfying $M_0=0$. Then the distribution of $M$ is uniquely determined by its local characteristics.
\end{corollary}

\begin{proof}
By Theorem \ref{thm:genformofmartwithindince} it is sufficient to show that each part of the canonical decomposition $M = M^c + M^q + M^a$ is uniquely determined by its characteristics. To this end it is enough to notice that the distribution of $M^c$ depends only on $\Phi$ and $(A_t)_{t\geq 0}$ which depends only  on $[\![M^c]\!]$, the distribution of $M^q$ depends only  on $\nu_{na}$, and the  distribution of $M^a$ depends only  on $\nu_{a}$.
\end{proof}

\section{The approach of Jacod, Kwapie\'{n}, and Woyczy\'{n}ski}\label{sec:appJKW}

In the present section we discover the infinite dimensional analogue of the celebrated result of Jacod \cite{Jac84} and Kwapie\'{n} and Woyczy\'{n}ski \cite{KwW91}, which says that if one discretize a real-valued quasi-left continuous martingale $M$ by creating a sequence $d^n = (d^n_k)_{k= 1}^n = (M_{Tk/n} - M_{T(k-1)/n})_{k=1}^n$, and if one considers a decoupled tangent martingale difference sequence $\tilde d^n = (\tilde d^n_k)_{k= 1}^n$ to $d^n$, then $\tilde d^n$ converges in distribution to a decoupled tangent martingale $\widetilde M$. The goal of the present section is to extend this statement to UMD-valued general local martingales.

\bigskip

Before stating the main theorem of the section we will need the following definitions.
 First recall that $\mathcal D([0, T], X)$ denotes the {\em Skorokhod space} of all $X$-valued c\`adl\`ag functions on $[0, T]$ (see Definition~\ref{def:defofskotokhodspaxc}). Throughout this section we will assume the Skorokhod space to be endowed with the {\em Skorokhod topology} (instead of the sup-norm topology, see Remark \ref{rem:supnormisbadforJKW}) which is generated by the Skorokhod metric which has the following form. Let $F, G\in \mathcal D([0, T], X)$. Then
 \[
  d_{J_1}(F,G) := \inf_{\lambda} \Bigl( \sup_{0\leq t\leq T}|\lambda(t) - t| + \sup_{0\leq t\leq T} \bigl\|F(t) -G(\lambda(t))\bigr\| \Bigr),
 \]
where the infimum is taken over all nondecreasing functions $\lambda:[0, T] \to [0, T]$. Note that
\begin{equation}\label{eq:skorleqsup}
  d_{J_1}(F,G)  \leq \|F-G\|_{\infty}.
\end{equation}
We refer the reader to \cite{Skor56,Jak86,Bill95,Wh02,Bill68} for further information on Skorokhod spaces.

\begin{definition}\label{def:convindisrtandweakconv}
Let $(\mathcal D, d)$ be a metric space. A sequence of $\mathcal D$-valued random variables $(\xi_n)_{n\geq 1}$ converges to an $\mathcal D$-valued random variable $\xi$ {\em in distribution} if the distributions of $(\xi_n)_{n\geq 1}$ {\em converge weakly} to the distribution of $\xi$, i.e.\ $\mathbb E f(\xi_n) \to \mathbb E f(\xi)$ as $n\to \infty$ for any bounded continuous function $f:\mathcal D\to \mathbb R$.
\end{definition}

We refer the reader to  \cite{ShPr2e,BMT07} for further details on convergence in distribution and on weak convergence.

\begin{remark}\label{rem:donvindistrcanbecheckedforLip}
Assume that $\mathcal D$ in Definition \ref{def:convindisrtandweakconv} is a locally convex space. Without loss of generality by \cite[Remark 8.3.1 and the proof of Theorem 8.2.3]{BMT07} we may assume that $f$ in the definition above is Lipschitz. Moreover, by multiplying $f$ by a constant we may assume that $\|f\|_{\infty} \leq 1$ and that 
$$
|f(x) - f(y)| \leq d(x, y),\;\;\; x, y\in X.
$$
\end{remark}

\smallskip
 
Let $X$ be a Banach space, $M:\mathbb R_+ \times \Omega \to X$ be a local martingale. Fix $T>0$ and for each $n\geq 1$ define
\begin{equation}\label{eq:dedefofdknforJKW!!}
 d_k^n :=M_{Tk/n} - M_{T(k-1)/n},\;\;\; k=1,\ldots, n.
\end{equation}
For each $n\geq 1$, let $(\tilde d_k^n)_{k=1}^n$ be a decoupled tangent sequence. Let
\begin{equation}\label{eq:defofMnforapproxbydiscretedecoupl}
\widetilde M^n_t := \sum_{k: Tk/n \leq t}\tilde d_k^n,\;\;\; t\geq 0.
\end{equation}

First start with a classical real-valued result.  The following theorem was shown by Jacod in \cite{Jac83,Jac84} (see also Kwapie\'{n} and Woyczy\'{n}ski \cite[p.\ 176]{KwW91}, and \cite[Chapter VI--VIII]{JS}).

\begin{theorem}\label{thm:condtrstepsnsidtrtothepledone}
Let $X$ be a finite dimensional Banach space, $M:\mathbb R_+ \times \Omega \to X$ be a quasi-left continuous local martingale that starts at zero, $T>0$, $(\widetilde M^n)_{n\geq 1}$ be defined by \eqref{eq:defofMnforapproxbydiscretedecoupl}. Then $(\widetilde M^n)_{n\geq 1}$ converges in distribution as random variables with values in $\mathcal D([0,T], X)$ to a decoupled tangent process $\widetilde M$ of $M$.
\end{theorem}

The goal of the present subsection is to extend Theorem \ref{thm:condtrstepsnsidtrtothepledone} to infinite dimensions and to general local martingales.

\begin{theorem}\label{thm:condtrstepsnsidtrtothepledoneUMDcase}
Let $X$ be a UMD Banach space, $M:\mathbb R_+ \times \Omega \to X$ be a local martingale that starts at zero, $T>0$, $(\widetilde M^n)_{n\geq 1}$ be defined by \eqref{eq:defofMnforapproxbydiscretedecoupl}. Then $(\widetilde M^n)_{n\geq 1}$ converges in distribution as random variables with values in $\mathcal D([0,T], X)$ to a decoupled tangent process $\widetilde M$ of $M$.
\end{theorem}

In order to prove Theorem \ref{thm:condtrstepsnsidtrtothepledoneUMDcase} we will need several intermediate steps outlined here as lemmas.

\begin{lemma}\label{lem:tangdecofsumissumoftangdec}
Let $X$ be a UMD Banach space, $M^1, \ldots, M^n:\mathbb R_+\times \Omega \to X$ be local martingales. Then there exist corresponding decoupled tangent local martingales $\widetilde M^1,\ldots,  \widetilde M^n$ such that $\alpha^1 \widetilde M^1 + \ldots +  \alpha^n \widetilde M^n$ is a decoupled tangent local martingale to $\alpha^1 M^1 + \ldots +  \alpha^n M^n$ for any $\alpha^1,\ldots,\alpha^n\in \mathbb R$.
\end{lemma}

\begin{proof}
First notice that the $n$th power $X \times \dots \times X $ of $X$ (endowed with the $\ell^p_n$ product norm for any $1< p< \infty$) is a UMD Banach space (see e.g.\ \cite[Proposition 4.2.17]{HNVW1}). Then it is sufficient to consider an $X \times \dots \times X $-valued local martingale $(M^1, \ldots, M^n)$ and set $\widetilde M^1, \ldots, \widetilde M^n$ to be such that $(\widetilde M^1, \ldots, \widetilde M^n)$ is a decoupled tangent local martingale to $(M^1, \ldots, M^n)$. Then the lemma follows from Theorem \ref{thm:MNtangent==>TMTNaretangentforanylinop} and the fact that 
$$
(x^1,\ldots,  x^n) \mapsto \alpha^1 x^1 +\ldots+ \alpha^n x^n,\;\;\;\; x^1, \ldots, x^n \in X,
$$ 
is a bounded linear operator from $X \times \dots \times X$ to $X$.
\end{proof}

\begin{lemma}\label{lem:approxMbyMmthesameforDTM}
Let $X$ be a UMD Banach space, $M$, $(M_m)_{m\geq 1}$ be local $X$-valued martingales such that $\mathbb E \sup_{t\geq 0}\|M_t - (M_m)_t\| \to 0$ as $m\to \infty$. Let $T>0$, and let $\widetilde M^n$ and $(\widetilde M^n_m)_{m\geq 1}$ be defined analogously to \eqref{eq:defofMnforapproxbydiscretedecoupl}. Assume additionally that $\widetilde M^n_m$ converges in distribution to a local martingale $\widetilde M_m$ which is decoupled tangent to $M_m$ for any $m\geq 1$. Then $\widetilde M^n$ converges in distribution to a local martingale $\widetilde M$ which is decoupled tangent to $M$.
\end{lemma}

\begin{proof}
Fix a continuous bounded function $f:\mathcal D([0,T], X) \to \mathbb R$. Our goal is to show that $\mathbb E f(\widetilde M) = \lim_{n\to \infty} \mathbb E f(\widetilde M^n)$. Thanks to Remark \ref{rem:donvindistrcanbecheckedforLip} we may assume that $f$ is Lipschitz, $\|f\|_{\infty}\leq 1$, and that by \eqref{eq:skorleqsup} $|f(x) - f(y)| \leq d_{J_1}(x, y) \leq  \|x-y\|$ for any $x, y\in \mathcal D([0,T], X)$.

Assume the converse, i.e.\ that $\mathbb E f(\widetilde M) \neq \lim_{n\to \infty} \mathbb E f(\widetilde M^n)$. Then we can identify our sequence with a subsequence such that
\begin{equation}\label{eq:EfMlessthenEfMn-delta}
\mathbb E f(\widetilde M) \leq  \mathbb E f(\widetilde M^n) - \delta,
\end{equation}
for some fixed $\delta>0$ for any $n\geq 1$ (we may use $-f$ instead of $f$ if needed). Fix $\eps>0$. Let $m\geq 1$ be such that $\mathbb E \sup_{0\leq t\leq T}\|M_t - (M_m)_t\| <\eps$. Then by Lemma \ref{lem:tangdecofsumissumoftangdec} and by \eqref{eq:Lpboundsforgentangmarraz} we have that 
\begin{equation}\label{eq:estforMandMmforlipschfunc}
\mathbb E \sup_{0\leq t\leq T}\|\widetilde M_t - ( \widetilde M_m)_t\| < C_X\eps
\end{equation} 
for some universal constant $C_X >0$. For this $m$ by the assumption of the theorem we can find $n\geq 1$ such that 
\begin{equation}\label{eq:Ef(tildeMm)geqEf(tileMmn)}
\mathbb E f(( \widetilde M_m)_t) \geq \mathbb E f(  (\widetilde M_m^n)_t) - \eps.
\end{equation}
Finally, by the construction \eqref{eq:defofMnforapproxbydiscretedecoupl} of $\widetilde M_m^n$ and $\widetilde M^n$ and by Lemma \ref{lem:tangdecofsumissumoftangdec} we may assume that $\widetilde M^n-\widetilde M_m^n$ is a discrete decoupled tangent martingale to $M_t - M_m$ constructed by \eqref{eq:defofMnforapproxbydiscretedecoupl} so analogously to \eqref{eq:estforMandMmforlipschfunc} we have that 
\begin{equation}\label{eq:ineqfortildeMmnandMn}
\mathbb E \sup_{0\leq t\leq T}\| \widetilde M^n_t-(\widetilde M_m^n)_t \|  \leq C_X\mathbb E \sup_{0\leq t\leq T}\|M_t - ( M_m)_t\| < C_X\eps.
\end{equation}
Therefore summing up \eqref{eq:estforMandMmforlipschfunc}, \eqref{eq:Ef(tildeMm)geqEf(tileMmn)}, and \eqref{eq:ineqfortildeMmnandMn} together with the Lipschitzivity of $f$ we have that $\mathbb E f(\widetilde M) \geq  \mathbb E f(\widetilde M^n) - \eps - 2C_X \eps$, so \eqref{eq:EfMlessthenEfMn-delta} does not hold if $\eps$ is small enough.
\end{proof}

\begin{lemma}\label{lem:contmartandmuareindepaftertimehacge}
Let $M^c:\mathbb R_+ \times \Omega \to \mathbb R$ be a continuous local martingale, let $n\geq 1$, $J=\{1,\ldots, n\}$, and let $\mu$ be an integer-valued optional random measure on $\mathbb R_+ \times J$ with a compensator $\nu$. Assume that there exists a random time-change $\tau = (\tau_t)_{t\geq 0}$ such that $[M^c]_{\tau_t} = t$, and random time-changes $(\tau^j_t)_{t\geq 0}$, $j=1,\ldots, n$ such that
\[
\nu\bigl([0, \tau^j_t]\times \{j\}\bigr) = t,\;\;\; t\geq 0,\;\;\; j=1,\ldots,n.
\]
Then the process $t\mapsto M^c_{\tau_t}$ and random measures $(t, j)\mapsto  \tilde {\mu}^j([0, t]\times \{j\}) :=  \mu([0, \tau^j_t]\times \{j\})$, $j=1,\ldots,n$, are mutually independent.
\end{lemma}

\begin{proof}
The fact that $(\tilde{\mu}^j)_{j=1}^n$ are independent standard Poisson random measures follows from \cite[Corollary 25.26]{Kal}. Let us show that $W := M^c\circ \tau$ is independent of $(\tilde{\mu}^j)_{j=1}^n$. Without loss of generality assume that $n=1$ (the proof for $n>1$ is analogous). Let $\tilde \mu = \tilde \mu^1$, $\widetilde N_t = \tilde \mu([0, t]\times \{1\}) - t$, $t\geq 0$. Then $W$ is a standard Brownian motion by \cite[Theorem 18.3 and 18.4]{Kal} and $\widetilde N$ is a standard compensated Poisson process by \cite[Corollary 25.26]{Kal}, and for any $0\leq t_1 < \ldots < t_N$, $0\leq s_1 < \ldots < s_N$ and for any numbers  $\alpha_1, \ldots, \alpha_N$ and $ \beta_1, \ldots, \beta_N$ we have that (here we set $t_0=s_0 = \alpha_0 = \beta_0=0$ and $d\alpha_k := \alpha_k-\alpha_{k-1}$, $d\beta_k := \alpha_k-\beta_{k-1}$ for any $k=1,\ldots, N$)
\begin{multline*}
e^{-\frac 12 \sum_{k=1}^N \bigl([d\alpha_k M^c]_{\tau_{t_k}} -  [d\alpha_k M^c]_{\tau_{t_{k-1}}}\bigr) + \sum_{k=1}^N\int_{\bigl(\tau^1_{s_{k-1}},\tau^1_{s_k}\bigr]\times \{1\}} e^{id\beta_k} -1 -id\beta_k \ud \nu}\\
 = e^{-\frac 12 \sum_{k=1}^N d\alpha_k^2 (t_k-t_{k-1})  +  \sum_{k=1}^N(e^{id\beta_k} -1 -id\beta_k)(s_k-s_{k-1})}
\end{multline*}
is a constant a.s., so by a stopping time argument and by \cite[II.4.16]{JS} we have that
\begin{align*}
\mathbb E e^{i(\sum_{k=1}^N\alpha_k W_{t_k} + \sum_{k=1}^N\beta_k \widetilde N_{s_k})} &= e^{-\frac 12 \sum_{k=1}^N d\alpha_k^2 (t_k-t_{k-1})  +  \sum_{k=1}^N(e^{id\beta_k} -1 -id\beta_k)(s_k-s_{k-1})}\\
&=  e^{-\frac 12 \sum_{k=1}^N d\alpha_k^2 (t_k-t_{k-1})} e^{ \sum_{k=1}^N(e^{id\beta_k} -1 -id\beta_k)(s_k-s_{k-1})}\\
&=\mathbb E e^{i \sum_{k=1}^N\alpha_k W_{t_k}} \mathbb E e^{\sum_{k=1}^N\beta_k \widetilde N_{s_k}} ,
\end{align*}
so $(W_{t_k})_{k=1}^N$ and $(\widetilde N_{s_k})_{k=1}^N$ are independent by \cite[Theorem II.12.4]{ShPr2e} as $\alpha_1, \ldots, \alpha_N$ and $ \beta_1, \ldots, \beta_N$ are arbitrary. As $t_1, \ldots, t_N$ and $s_1, \ldots, s_N$ are arbitrary, $W$ and $\widetilde N$ (hence, $W$ and $\tilde \mu$) are independent.
\end{proof}

Let us finally prove Theorem \ref{thm:condtrstepsnsidtrtothepledoneUMDcase}.

\begin{proof}[Proof of Theorem \ref{thm:condtrstepsnsidtrtothepledoneUMDcase}] First, by a stopping time argument and by Theorem \ref{thm:MNtangent==>MtauNtautangent+ifNCIINtauCII} we may assume that $\mathbb E \sup_{0\leq t\leq T} \|M_t\| <\infty$.
By Lemma \ref{lem:approxMbyMmthesameforDTM} and by a stopping time argument it is sufficient to show  Theorem \ref{thm:condtrstepsnsidtrtothepledoneUMDcase} for a martingale $M$ from a dense subset of martingales. In particular, by \cite[Subsection 7.5]{Y18BDG} we may assume that $X$ is finite dimensional. Let $M = M^c + M^q + M^a$ be the canonical decomposition. We will approximate each part of the canonical decomposition separately (we are allowed to do so by \eqref{eq:candecstrongLpestmiaed}). Our goal is to exploit Lemma~\ref{lem:tangdecofsumissumoftangdec} and approximate $M$ in such a way that $\widetilde M^n$ almost coincides with $\widetilde M$ for any $n$ big enough.

\smallskip

{\em Step 1: approximation of $M^c$.} By the proof of Theorem \ref{thm:DTMforcontcase} we may assume that there exists an invertible time-change $(\tau_s)_{s\geq 0}$ with an inverse time-change $(A_t)_{t\geq 0}$, a separable Hilbert space $H$, an elementary $(\mathcal F_{\tau_s})_{s\geq 0}$-predictable process $\Phi:\mathbb R_+ \times \Omega \to\mathcal L(H,X)$, and an $(\mathcal F_{\tau_s})_{s\geq 0}$-adapted cylindrical Brownian motion $W_H$ such that $M^c \circ \tau = \Phi \cdot W_H$, or, in other words, $M^c_t = \int_0^{A_t} \Phi \ud W_H$. By an approximation argument and the definition of a stochastic integral (see Subsection \ref{subsec:prelimstint}) we may assume that  $\Phi$ is elementary $(\mathcal F_{\tau_s})_{s\geq 0}$-predictable with respect to the mesh 
\begin{equation}\label{eq:meshforMCN0}
\{A_0, A_{T/N_0},\ldots, A_{T(N_0-1)/N_0}, A_{T}\}
\end{equation}
so $(\Phi\circ A)_t$ is $(\mathcal F_{T(k-1)/{N_0}})_{t\geq 0}$-measurable for any $\tfrac{k-1}{N_0}T \leq t <\tfrac{k}{N_0}T$ for some fixed big natural number $N_0$. Moreover, for $N_0$ big enough we can approximate $M^c$ by a continuous martingale (adapted to another filtration) $M^{c, N_0}$ in the following way.

By a stopping time argument and by the fact that $A$ is continuous we may assume that $A_T \leq C$ a.s\ for some fixed $C>0$. Let $\overline{W}_H$ be a cylindrical Brownian motion such that 
$$
\overline{W}_H|_{[C(k-1), C(k-1) + A_{Tk/{N_0}} - A_{T(k-1)/{N_0}}]}=W_H|_{[A_{T(k-1)/{N_0}},A_{Tk/{N_0}}]},\;\;\; k=1,\ldots, N_0.
$$
Set $A':\mathbb R_+ \times \Omega \to \mathbb R_+$ to be
$$
A'_t :=
\begin{cases}
 C(k-1) + A_{t} - A_{T(k-1)/{N_0}},\;\;\; &\tfrac{k-1}{N_0}T \leq t <\tfrac{k}{N_0}T,\;\; k=1,\ldots, N_0,\\
 C(N_0) + A_{T} - A_{T(N_0-1)/{N_0}},\;\;\; &t\geq T,
\end{cases}
$$
and set $\Phi':\mathbb R_+ \times \Omega \to \mathcal L(H, X)$ to be such that 
$$
\Phi'|_{(C(k-1), C(k-1) + A_{Tk/{N_0}} - A_{T(k-1)/{N_0}}]} = \Phi|_{(A_{T(k-1)/{N_0}},A_{Tk/{N_0}}]},\;\;\; k=1,\ldots, N_0,
$$
and zero otherwise. (Recall that we assumed $\Phi$ to be elementary predictable with respect to the mesh \eqref{eq:meshforMCN0}, so $\Phi$ is a.s.\ a constant on $(A_{T(k-1)/{N_0}},A_{Tk/{N_0}}]$, and thus $\Phi'$ is a.s.\ a constant on $(C(k-1), C(k-1) + A_{Tk/{N_0}} - A_{T(k-1)/{N_0}}]$). By the definition of $\Phi'$ and $\overline W_H$ we have that $M^c_t = \int_0^{A'_t} \Phi' \ud \overline W_H$ for any $t\geq 0$. Now let us construct the desired $M^{c, N_0}$ in the following way. Let $\overline A:\mathbb R_+ \times \Omega \to \mathbb R_+$ be defined by 
$$
\overline A_t :=
\begin{cases}
 C(k-1) + \mathbb E \bigl( A_{t} - A_{T(k-1)/{N_0}}\big|\mathcal F_{T(k-1)/{N_0}}\bigr),\;\; &\tfrac{k-1}{N_0}T \leq t <\tfrac{k}{N_0}T,\\
 &\quad\quad\quad\quad k=1,\ldots, N_0,\\
 C(N_0) +  \mathbb E \bigl( A_{T} - A_{T(N_0-1)/{N_0}}\big|\mathcal F_{T(k-1)/{N_0}}\bigr),\;\; &t\geq T,
\end{cases}
$$
and let $\overline {\Phi}:\mathbb R_+ \times \Omega \to \mathcal L(H, X)$ be
$$
\overline {\Phi}_s :=
\begin{cases}
\Phi'_{C(k-1)},\;\; &C(k-1) \leq s \leq C(k-1) + \mathbb E \bigl( A_{t} - A_{T(k-1)/{N_0}}\big|\mathcal F_{T(k-1)/{N_0}}\bigr),\\
 &\quad\quad\quad\quad\quad\quad\quad\quad\quad\quad\quad\quad\quad\quad\quad\quad\quad\quad\quad\quad\quad k=1,\ldots, N_0,\\
0,\;\; &\text{otherwise}.
\end{cases}
$$
Let $M^{c, N_0}_t:= \int_0^{\overline A_t}\overline {\Phi} \ud \overline W_H$, $t\geq 0$. Let us show that
\begin{equation}\label{eq:supMcN0-Mcto0ueue}
 \mathbb E \sup_{k=1}^{N_0}\|M^c_{Tk/N_0} - M^{c, N_0}_{Tk/N_0}\| \to 0,\;\;\; N_0 \to \infty.
\end{equation}
(Though the limit here cannot be considered literally as $\Phi \circ A$ is step with respect to the mesh \eqref{eq:meshforMCN0} for some fixed $N_0$, but not for any big $N_0$, we still can send $N_0$ to infinity, as for any $N \gg N_0$ we have that $\Phi \circ A$ is {\em almost} step with respect to the mesh $\{A_0, A_{T/N},\ldots, A_{T(N-1)/N}, A_{T}\}$ as for such big $N$ by the boundedness of $\Phi$ the process $\Phi$ can simply be assumed zero on $(A_{T(k-1)/N}, A_{Tk/N}]$ if one of the knots of \eqref{eq:meshforMCN0} turned out to be in $(A_{T(k-1)/N}, A_{Tk/N}]$. This does not change much the norm of the stochastic integral by \cite[Theorem 26.12]{Kal} and \eqref{eq:qvvarofstintwrtHcylbrmot}). 

Let $L^c_t:= \int_0^t \overline {\Phi} a^{c, N_0} \ud \overline W_H$, where 
$$
a^{c, N_0}(t) := \sum_{k=1}^{N_0} \mathbf 1_{[C(k-1) + \mathbb E ( A_{Tk/N_0} - A_{T(k-1)/{N_0}}|\mathcal F_{T(k-1)/{N_0}}), C(k-1) + A_{Tk/N_0} - A_{T(k-1)/{N_0}}]}(t),
$$
for any $t\geq 0$.
Then
\begin{align*}
  \mathbb E \sup_{t\geq 0} \|L^c_t\| & \stackrel{(*)}\lesssim_{X, \Phi} \mathbb E\Big( \sum_{k=1}^{N_0} \big|A_{Tk/N_0} - A_{T(k-1)/{N_0}} \\
  &\quad\quad\quad\quad\quad\quad\quad\quad\quad\quad- \mathbb E ( A_{Tk/N_0} - A_{T(k-1)/{N_0}}|\mathcal F_{T(k-1)/{N_0}})\big|\Big)^{1/2}\\
  &=\mathbb E\Big( \sum_{k=1}^{N_0} \bigl|A_{Tk/N_0}  - \mathbb E ( A_{Tk/N_0} |\mathcal F_{T(k-1)/{N_0}})\bigr|\Big)^{1/2}  \stackrel{(**)}\to 0,\;\;\; N_0 \to \infty,
\end{align*}
where $(*)$ holds by \eqref{eq:LpboundsforstochintwrtHcylbrm}, the fact that $X$ is finite dimensional, and the fact that $\Phi$ is elementary predictable, and hence bounded, while $(**)$ follows from Lemma \ref{lem:Lipshfuncapprox}.
Moreover, if for any $t\geq 0$ we set
\[
b^{c, N_0}(t) := \sum_{k=1}^{N_0} \mathbf 1_{[C(k-1), C(k-1) + \mathbb E ( A_{Tk/N_0} - A_{T(k-1)/{N_0}}|\mathcal F_{T(k-1)/{N_0}}) \wedge (A_{Tk/N_0} - A_{T(k-1)/{N_0}})]}(t),
\]
then
\begin{align*}
 \mathbb E \sup_{k=1}^{N_0}\bigl\|M^c_{Tk/N_0} &- (M^{c, N_0}_{Tk/N_0}-L^c_{TK/N_0})\bigr\| \leq \mathbb E \sup_{t\geq 0} \Bigl\|\int_0^t \Phi' (1-b^{c, N_0}) \ud \overline W_H \Bigr\| \\
 & \stackrel{(*)}\lesssim_{X, \Phi} \mathbb E \Big(\sum_{k=1}^{N_0} \big|A_{Tk/N_0} \wedge \mathbb E ( A_{Tk/N_0} |\mathcal F_{T(k-1)/{N_0}})\\
 &\quad\quad\quad\quad\quad\quad\quad\quad\quad\quad\quad\quad\quad\quad -\mathbb E ( A_{Tk/N_0} |\mathcal F_{T(k-1)/{N_0}} )\big|\Big)^{1/2}\\
 & \leq \mathbb E\Big( \sum_{k=1}^{N_0} \bigl|A_{Tk/N_0}  - \mathbb E ( A_{Tk/N_0} |\mathcal F_{T(k-1)/{N_0}})\bigr|\Big)^{1/2} \stackrel{(**)}\to 0,\;\;\; N_0 \to \infty,
\end{align*}
where $(*)$ follows from \eqref{eq:LpboundsforstochintwrtHcylbrm}, and $(**)$ follows from Lemma \ref{lem:Lipshfuncapprox}. Thus \eqref{eq:supMcN0-Mcto0ueue} follows.

\smallskip

{\em Step 2: approximation of $M^q$.} By Theorem \ref{thm:XisUMDiffMisintxwrtbarmuMvain} and by Lemma \ref{lem:approxMbyMmthesameforDTM} we may assume that $M^q = \int_{[0,\cdot]\times J} F \ud \bar{\mu}$ for some finite set $J = \{1,\ldots,n\}$, for some elementary predictable $F:\mathbb R_+ \times \Omega \times J \to X$, and for some quasi-left continuous random measure $\mu$ on $\mathbb R_+ \times J$ with a compensator $\nu$ so that $\nu([0, t] \times \{j\})<\infty$ a.s.\ for any $t\geq 0$ and $\nu(\mathbb R_+ \times \{j\}) = \infty$ a.s.\ for any $j\in J$ (the latter can be done e.g.\ similarly to Step 4 of the proof of Theorem \ref{thm:mumuCoxcomparable}). By a stopping time argument and by Lemma \ref{lem:approxMbyMmthesameforDTM} we may assume that $\nu([0, T] \times J)<C$ a.s.\ for the same constant $C$ as in Step 1 of the present proof. Analogously to Step 1 of the proof of Theorem \ref{thm:mumuCoxcomparable} we may construct an approximation $M^{q, N_0}$ of $M^q$ such that $\nu^{M^{q, N_0}}([0,t]\times \{j\})$ is $\mathcal F_{T({k-1})/{N_0}}$-measurable for any $\tfrac{k-1}{N_0}T \leq t < \tfrac{k}{N_0}T$. Indeed, as it was done in the proof of Theorem \ref{thm:mumuCoxcomparable}, we can approximate $M^q$ with a process $M^{q,N_0}$ which is constructed analogously to \eqref{eq:defofLapproxMq} in the following way. Let $N_0$ be so big that $M^q$ has the form (which is analogous to the form \eqref{eq:Mqformhalohalolastline})
\begin{equation*}
 M^q_t =  \sum_{k=1}^{N_0} \sum_{j=1}^n \bar{\eta}\Bigl(\bigl[C(k-1),  C(k-1)+ \nu^j[\tfrac{T(k-1)}{N_0} \wedge t, \tfrac{TK}{N_0} \wedge t)\bigr) \times \{j\}\Bigr) \xi_{k, j},\;\;\; t\geq 0,
\end{equation*}
where $\eta$ is some standard Poisson random measure on $\mathbb R_+ \times J$ with a compensator $\nu_{\eta}$, $\bar{\eta} := \eta - \nu_{\eta}$, $\xi_{k, j}$ is $\mathcal F_{{T(k-1)/N_0}}$-measurable and simple. For any $j=1,\ldots, n$ let $\nu^j$ be defined by \eqref{eq:defofnujddsd} and let 
\begin{multline*}
 M^{q, N_0}_t =  \sum_{k=1}^{N_0} \sum_{j=1}^n \bar{\eta}\Big(\big[C(k-1),\\
 C(k-1)
 +\mathbb E ( \nu^j[\tfrac{T(k-1)}{N_0} \wedge t, \tfrac{Tk}{N_0}\wedge t)| \mathcal F_{T(k-1)/N_0})\big) \times \{j\}\Big) \xi_{k, j},\;\;\; t\geq 0.
\end{multline*}
It remains to notice that 
\begin{equation}\label{eq:supMqN0-Mqto0ueuqqqe}
 \mathbb E \sup_{k=1}^{N_0}\|M^q_{Tk/N_0} - M^{q, N_0}_{Tk/N_0}\| \to 0,\;\;\; N_0 \to \infty.
\end{equation}
which follows analogously \eqref{eq:L1approxintFwrtmubar} and \eqref{eq:L2issmall<delta}. (Here we are allowed to send $N_0$ to infinity for the same reason as in Step 1).

\smallskip

{\em Step 3: approximation of $M^a$.} By Lemma \ref{lem:approxMbyMmthesameforDTM} and by \eqref{eq:approxfactforMmforPDwAJcase} we may assume that $M^a$ has its jumps in predictable stopping times $0\leq \sigma_1 < \ldots < \sigma_m \leq T$. Moreover, by Lemma \ref{lem:approxMbyMmthesameforDTM}, Lemma \ref{lem:predsttimetang}, and the proof of Proposition \ref{prop:pdwajjusttang} we may assume that $\Delta M_{\sigma_1},\ldots,\Delta M_{\sigma_m}$ are simple in $L^{\infty}(\Omega; X)$ and thus have values in a finite dimensional subspace of $X$. In addition assume that there exists $\delta>0$ such that for all differences we have that a.s.\
\begin{equation}\label{eq:sigmasaredeltafarfromeachoth}
\sigma_2 - \sigma_1, \ldots, \sigma_m - \sigma_{m-1}\geq \delta.
\end{equation}
We can assume so as we can approximate $M^q$ by the following martingale
\[
 M^{a, \delta}_t:= \Delta M_{\sigma_{1}} \mathbf 1_{[0, t]}(\sigma_{1}) + \sum_{\ell = 2}^m \Delta M_{\sigma_{\ell}} \mathbf 1_{[0, t]}(\sigma_{\ell}) \mathbf 1_{\sigma_{2} \geq \sigma_{1 } + \delta, \ldots, \sigma_{\ell} \geq \sigma_{\ell -1 } + \delta},\;\;\; t\geq 0,
\]
which is a martingale by Lemma \ref{lem:DeltaMtaugiventau-=0}, and set $\sigma_1' := \sigma_1$ and
\[
\sigma_{\ell}':=
\begin{cases}
\sigma_{\ell},\;\;\;\; & \text{if}\; \sigma_{2} \geq \sigma_{1 } + \delta, \ldots, \sigma_{\ell} \geq \sigma_{\ell -1 } + \delta,\\
\sigma_{\ell-1}' + \delta,\;\;\;\;&\text{otherwise},
\end{cases}
\]
for any $\ell = 2, \ldots, m$.
Indeed, as $M$ takes valued in a finite dimensional subspace of $X$, $X$ can be assumed finite dimensional, so by the finite dimensional version of \cite[Theorem 26.12]{Kal}
\[
 \mathbb E \sup_{t\geq 0}\|M^a_t -M^{a, \delta}_t\| \eqsim \mathbb E \Bigl( \sum_{\ell=1}^m \| \Delta M_{\sigma_{\ell}}\|^2 \mathbf 1_{\sigma_{2} \geq \sigma_{1 } + \delta, \ldots, \sigma_{\ell} \geq \sigma_{\ell -1 } + \delta} \Bigr)^{1/2} \to 0,\;\;\; \delta \to 0,
\]
where the latter holds by the dominated convergence theorem and by the fact that~a.s.\
\[
\max\{\sigma_2-\sigma_1, \ldots, \sigma_m - \sigma_{m-1}\}>0.
\]
Thus we can redefine $M^{a} := M^{a, \delta}$ and $\{\sigma_1, \ldots, \sigma_m\} := \{\sigma_1', \ldots, \sigma_m'\}$.

\smallskip

{\em Step 4: the desired convergence in distribution.}
For any $n\geq 1$ let $\widetilde M^n$ be defined by \eqref{eq:defofMnforapproxbydiscretedecoupl}. Let $f:\mathcal D([0,T], X) \to [-1,1]$ be a Lipschitz function such that $|f(x) - f(y)| \leq d_{J_1}(x, y)$ for any $x, y\in \mathcal D([0,T], X)$ (see Remark \ref{rem:donvindistrcanbecheckedforLip}). Our goal is to show that $\mathbb E f(\widetilde M^n)\to \mathbb E f (\widetilde M)$ as $n\to \infty$.

By Step 3 we may assume that there exists $m\geq 1$ and predictable stopping times $0\leq \sigma_1 < \ldots < \sigma_m \leq T$ such that $M^a$ has its jumps in $\{\sigma_1 , \ldots , \sigma_m \}$ satisfying \eqref{eq:sigmasaredeltafarfromeachoth}, and $\Delta M_{\sigma_1},\ldots,\Delta M_{\sigma_m}$ are simple in $L^{\infty}(\Omega; X)$.

Let $M^{c, N_0}$ and $M^{q, N_0}$ be as in Step 1 and 2. These processes are martingales with respect to different filtrations $\mathbb F^{c, N_0}$ and $\mathbb F^{q, N_0}$ (generated by time-changed $\overline W_H$ and $\eta$ respectively). Nonetheless, if we set 
\[
 \mathcal F^{N_0}_k := \sigma(\mathcal F_{Tk/N_0}, \mathcal F^{c, N_0}_{Tk/N_0}, \mathcal F^{q, N_0}_{Tk/N_0}),\;\;\; k=1,\ldots,N_0,
\]
then $(d^{N_0}_{k})_{k=1}^{N_0}$ defined by \eqref{eq:dedefofdknforJKW!!} and
\begin{align*}
 (e^{c,N_0}_{k})_{k=1}^{N_0} &:= \bigl(M^{c, N_0}_{Tk/N_0}  - M^{c, N_0}_{T{k-1}/N_0}\bigr)_{k=1}^{N_0},\;\;\; k=1,\ldots N_0\\
 (e^{q,N_0}_{k})_{k=1}^{N_0} &:= \bigl(M^{q, N_0}_{Tk/N_0}  - M^{q, N_0}_{T{k-1}/N_0}\bigr)_{k=1}^{N_0},\;\;\; k=1,\ldots N_0\\
 (e^{a,N_0}_{k})_{k=1}^{N_0} &:= \bigl(M^{a}_{Tk/N_0}  - M^{a}_{T{k-1}/N_0}\bigr)_{k=1}^{N_0},\;\;\; k=1,\ldots N_0
\end{align*}
are martingale difference sequences with respect to $\mathbb F^{N_0} := (\mathcal F^{N_0}_k)_{k=0}^{N_0}$ (since $\overline W_H|_{[0, Ck]}$ and $\eta|_{[0, Ck]}$ are independent of $\overline W_H|_{[Ck, \infty)} - \overline W_H(Ck)$ and $\eta|_{[Ck, \infty)} - \eta(Ck)$). Let $e^{N_0} := (e^{N_0}_k)_{k=1}^{N_0}$ be a martingale difference sequence defined by
\[
 e^{N_0}_k:=e^{c, N_0}_k + e^{q, N_0}_k+ e^{a, N_0}_k,\;\;\; k=1,\ldots, N_0.
\]
Fix $\eps>0$. Let us show that for $N_0$ big enough  we have that for a decoupled tangent martingale difference sequence $\tilde e^{N_0} := (\tilde e^{N_0}_k)_{k=1}^{N_0}$ the following holds true
\begin{equation}\label{eq:tildeMN0NN0dsa}
 |\mathbb E f(\widetilde M^{N_0}) - \mathbb E f(\widetilde N^{N_0}) | \lesssim_{X} \eps,
\end{equation}
\begin{equation}\label{eq:tildeMNN0dssdssa}
|\mathbb E f(\widetilde M) - \mathbb E f(\widetilde N^{N_0}) |\lesssim_{X} \eps,
\end{equation}
so $|\mathbb E f(\widetilde M) - \mathbb E f(\widetilde M^{N_0}) |\lesssim_{X} \eps$ by a triangle inequality, where
\begin{equation}\label{eq:defoftildeNN_)dsajdfs}
  \widetilde N^{N_0}:= \sum_{k: Tk/N_0 \leq t}\tilde e_k^{N_0},\;\;\; t\geq 0.
\end{equation}

First let us show \eqref{eq:tildeMN0NN0dsa}. As $f$ is Lipschitz and as $M^{N_0}$ and $N^{N_0}$ are pure jump with jumps at $\{T/N_0, \ldots, T(N_0-1)/N_0\}$, by \eqref{eq:skorleqsup} it is sufficient to show that  there exist a version of $\widetilde M^{N_0}$ and a version of the decoupled tangent martingale difference sequence $\tilde e^{N_0} := (\tilde e^{N_0}_k)_{k=1}^{N_0}$ satisfying
\begin{equation}\label{eq:tildeMN_0-tildeNN_0||||leqXepsps}
 \mathbb E \sup_{k\geq 1}^{N_0} \bigl\|\widetilde M^{N_0}_{Tk/N_0} - \widetilde N^{N_0}_{Tk/N_0}\bigr\| \lesssim_X \eps.
\end{equation}
Let $d^{c,N_0}$, $d^{q, N_0}$, and $d^{a, N_0}$ be defined analogously to \eqref{eq:dedefofdknforJKW!!} for martingales $M^c$, $M^q$, and $M^a$ respectively (note that $d^{a, N_0} = e^{a, N_0}$). Let $\tilde d^{c, N_0}$, $\tilde d^{q, N_0}$, and $\tilde d^{a, N_0}$ be the corresponding decoupled tangent martingales. Let $\tilde e^{c, N_0}$, $\tilde e^{q, N_0}$, and $\tilde e^{a, N_0}$ be decoupled tangent martingales to $e^{c, N_0}$, $e^{q, N_0}$, and $e^{a, N_0}$ respectively. 
Then by Lemma \ref{lem:tangdecofsumissumoftangdec}, Theorem \ref{thm:intromccnnll}, \eqref{eq:supMcN0-Mcto0ueue}, and \eqref{eq:supMqN0-Mqto0ueuqqqe} we may find such suitable versions of $\tilde d^{c, N_0}$, $\tilde d^{q, N_0}$, $\tilde e^{c, N_0}$, and $\tilde e^{q, N_0}$, and assume that $N_0$ is so big that
\begin{equation}\label{eq:tildeecN-0isalsmodstildedcN)_0}
 \mathbb E \sup_{k=1}^{N_0} \Bigl\|\sum_{\ell =1}^{k} \tilde e^{c, N_0}_{\ell} - \tilde d^{c, N_0}_{\ell}\Bigr\|\lesssim_X \mathbb E \sup_{k=1}^{N_0} \Bigl\|\sum_{\ell =1}^{k}e^{c, N_0}_{\ell} -  d^{c, N_0}_{\ell}\Bigr\| \leq \eps,
\end{equation}
\begin{equation}\label{eq:tildeeqN_)iosalmstildedqN_0}
 \mathbb E \sup_{k=1}^{N_0} \Bigl\|\sum_{\ell =1}^{k} \tilde e^{q, N_0}_{\ell} - \tilde d^{q, N_0}_{\ell}\Bigr\|\lesssim_X \mathbb E \sup_{k=1}^{N_0} \Bigl\|\sum_{\ell =1}^{k}e^{q, N_0}_{\ell} -  d^{q, N_0}_{\ell}\Bigr\| \leq \eps.
\end{equation}
It remains to show that for $N_0$ big enough there exist such versions of $\tilde d^{a, N_0}$ and $\tilde e^{a, N_0}$ that
\begin{equation}\label{eq:tildeeaN_0====tildedaN_0}
 \mathbb E \sup_{k=1}^{N_0} \Bigl\|\sum_{\ell =1}^{k} \tilde e^{a, N_0}_{\ell} - \tilde d^{a, N_0}_{\ell}\Bigr\| = 0,
\end{equation}
which follows directly from the fact that $d^{a, N_0} = e^{a, N_0}$. Thus \eqref{eq:tildeMN_0-tildeNN_0||||leqXepsps} follows from Lemma \ref{lem:tangdecofsumissumoftangdec} (so that we can set $\tilde d^{N_0} := \tilde d^{c, N_0} + \tilde d^{q,N_0} + \tilde d^{a, N_0}$ and $\tilde e^{N_0} := \tilde e^{c, N_0} + \tilde e^{q,N_0} + \tilde e^{a, N_0}$), \eqref{eq:tildeecN-0isalsmodstildedcN)_0}, \eqref{eq:tildeeqN_)iosalmstildedqN_0}, and \eqref{eq:tildeeaN_0====tildedaN_0}.

\smallskip

Let us show \eqref{eq:tildeMNN0dssdssa}. First note that for any $N_0$ big enough we have that
\begin{equation}\label{eq:EftildemisalmostEfLN_0}
 |\mathbb E f(\widetilde M) - \mathbb E f(L^{N_0}) |\lesssim_{X} \eps,
\end{equation}
where $L^{N_0}$ is a discretization of $\widetilde M$ defined by
\[
  L^{N_0}_t := \sum_{k:Tk/N_0 \leq t} \widetilde M_{Tk/N_0} - \widetilde M_{T(k-1)/N_0},\;\;\; t\geq 0.
\]
Indeed, it is sufficient to notice that a.s.\ as $f$ is Lipschitz
\begin{multline*}
  f(\widetilde M) -  f(L^{N_0}) \leq d_{J_1}(\widetilde M, L^{N_0}) \\
 \stackrel{(*)}\leq \sup_{t\in[0, T]}\bigl|\lambda^{N_0}(t) - t\bigr| + \sup_{k=1}^{N_0}\sup_{t=0}^{T/N_0}\bigl| \widetilde M_{Tk/N_0} - \widetilde M_{Tk/N_0-t} \bigr| \stackrel{(**)}\to 0,\;\;\; N_0 \to \infty,
\end{multline*}
where 
$$
\lambda^{N_0}(t) := \tfrac{T}{N_0}\lceil \tfrac{t N_0}{T}\rceil,\;\;\; t\in[0, T]
$$ 
(here $\lceil a \rceil$ is the smallest integer bigger than $a\in  \mathbb R $),
$(*)$ follows by \eqref{eq:skorleqsup}, and $(**)$ follows from the fact that $\widetilde M$ has c\`adl\`ag paths and the fact that 
$$
\sup_{t\in[0, T]}\bigl|\lambda^{N_0}(t) - t\bigr| \leq T/N_0.
$$
Hence \eqref{eq:EftildemisalmostEfLN_0} follows by the dominated convergence theorem and therefore by \eqref{eq:EftildemisalmostEfLN_0}, \eqref{eq:skorleqsup}, and the fact that $f$ is Lipschitz it is sufficient to show that for any $N_0$ big enough there exist such a version of $\tilde e^{N_0}$ that
\begin{equation}\label{eq:fNN_0-fLNPoidsmall}
  |\mathbb E f(\widetilde N^{N_0}) - \mathbb E f(L^{N_0}) | \leq \mathbb E \sup_{k=1}^{N_0}\|\widetilde N^{N_0}_{Tk/N_0} - \widetilde M_{Tk/N_0}\|\lesssim_{X} \eps.
\end{equation}
By Step 1 and 2, by Lemma \ref{lem:contmartandmuareindepaftertimehacge}, and by the fact that $\mathbb E (A_{Tk/N_0} - A_{T({k-1})/N_0}|\mathcal F_{T({k-1})/N_0})$ and 
$$
\mathbb E\bigl(\nu^j({T({k-1})/N_0},{Tk/N_0}]|\mathcal F_{T({k-1})/N_0}\bigr),\;\;\;j=1,\ldots, n,
$$ 
are $\mathcal F_{T({k-1})/N_0}$-measurable for any $k=1, \ldots, N_0$, we may assume that for any $k=1, \ldots, N_0$
$$
\tilde e^{c, N_0}_k = \int_{C(k-1)}^{C(k-1) + \mathbb E (A_{Tk/N_0} - A_{T({k-1})/N_0}|\mathcal F_{T({k-1})/N_0})} \Phi (T({k-1})/N_0) \ud \widetilde W_H,
$$ 
$$
\tilde e^{q, N_0}_k = \int_{\cup_{j}[C(k-1), C(k-1) + \mathbb E(\nu^j({T({k-1})/N_0},{Tk/N_0}]|\mathcal F_{T({k-1})/N_0})]\times \{j\}}  \ud \overline {\eta}_{\rm ind},
$$ 
where $\widetilde W_H$ and $\overline {\eta}_{\rm ind}$ are independent copies of $W_H$ and $\overline{\eta}$ respectively. Moreover, by Lemma \ref{lem:tangdecofsumissumoftangdec} and by the proofs of Theorem \ref{thm:DTMforcontcase} and \ref{thm:mumuCoxcomparable} we can set for any $1\leq k\leq N_0$
\[
 \widetilde M^c_{Tk/N_0} - \widetilde M^c_{T(k-1)/N_0} := \int_{C(k-1)}^{C(k-1) + A_{Tk/N_0} - A_{T({k-1})/N_0}} \Phi (T({k-1})/N_0) \ud \widetilde W_H,
\]
\[
  \widetilde M^q_{Tk/N_0} - \widetilde M^q_{T(k-1)/N_0} :=\int_{\cup_{j}[C(k-1), C(k-1) + \mathbb \nu^j({T({k-1})/N_0},{Tk/N_0}]]\times \{j\}}  \ud \overline {\eta}_{\rm ind}.
\]
Then similarly to \eqref{eq:supMcN0-Mcto0ueue} and \eqref{eq:supMqN0-Mqto0ueuqqqe} we have that for any $N_0$ big enough (here $\widetilde N^{c, N_0}$ and $\widetilde N^{q, N_0}$ are defined analogously to \eqref{eq:defoftildeNN_)dsajdfs})
\begin{equation}\label{eq:NctildeapproxMctildeindiscr}
  \mathbb E \sup_{k=1}^{N_0}\|\widetilde N^{c,N_0}_{Tk/N_0} - \widetilde M^c_{Tk/N_0}\|\lesssim_{X} \eps,
\end{equation}
\begin{equation}\label{eq:NqtildeapproxMqtildeindiscr}
  \mathbb E \sup_{k=1}^{N_0}\|\widetilde N^{q,N_0}_{Tk/N_0} - \widetilde M^q_{Tk/N_0}\|\lesssim_{X} \eps.
\end{equation}
It remains to show that for $N_0$ big enough there exists a version $\tilde e^{a, N_0}$ such that
\begin{equation}\label{eq:NatildeapproxMatildeindiscr}
  \mathbb E \sup_{k=1}^{N_0}\|\widetilde N^{a,N_0}_{Tk/N_0} - \widetilde M^a_{Tk/N_0}\| = 0.
\end{equation}
To this end notice that we can choose $T$ and $N_0$ big enough so that $T/N_0 \ll \delta$, and hence 
$$
d^{a, N_0}_k = \Delta M_{\sigma_{\ell}},\;\;\; k=1,\ldots, N_0, \;\; \ell \; \text{is such that}\; \sigma_{\ell}\in\bigl(T(k-1)/N_0, Tk/N_0\bigr]. 
$$
For each $k=1, \ldots, N_0$ set
\[
\rho_k:=
\begin{cases}
\sigma_{\ell}, \;\;\; &\text{if there exists}\; \ell\; \text{such that} \sigma_{\ell} \in \bigl(T(k-1)/N_0, Tk/N_0\bigr],\\
Tk/N_0,\;\;\; &\text{otherwise}.
\end{cases}
\]
Then $(\rho_k)_{k=1}^{N_0}$ are predictable stopping times and $d^{a, N_0}_k = e^{a, N_0}_k = \Delta M_{\rho_{k}}$ a.s.\ for any $k=1, \ldots, N_0$.
Let us show that we can set $(\tilde e^{a, N_0}_k)_{k=1}^{N_0}$ and  $(\Delta \widetilde M_{\rho_k})_{k=1}^{N_0}$ have the same distribution (and thus can be assumed to coincide). First notice that $(\Delta \widetilde M_{\rho_k})_{k=1}^{N_0}$ are independent given $\mathcal F$. Moreover, as by \cite[Lemma 25.2]{Kal} $\mathcal F_{T(k-1)/N_0} \subset \mathcal F_{\rho_k} \subset \mathcal F_{Tk/N_0}$ and as $e^{a, N_0}_k = \Delta M_{\rho_k}$  for any $k=1,\ldots, N_0$, for any Borel set $B \in X$ we have that
\begin{align*}
\mathbb P \bigl(\Delta  M_{\rho_k}\big|\mathcal F_{T(k-1)/N_0}\bigr)(B) &= \mathbb E\bigl (\mathbf 1_{B}(\Delta  M_{\rho_k})\big|\mathcal F_{T(k-1)/N_0}\bigr) \\
&=\mathbb E \Bigl(\mathbb E\bigl (\mathbf 1_{B}(\Delta  M_{\rho_k})\big|\mathcal F_{\rho_k}\bigr) \Big|\mathcal F_{T(k-1)/N_0}\Bigr)\\
& \stackrel{(*)}=\mathbb E \Bigl(\mathbb E\bigl (\mathbf 1_{B}(\Delta \widetilde M_{\rho_k})\big|\mathcal F_{\rho_k}\bigr) \Big|\mathcal F_{T(k-1)/N_0}\Bigr)\\
&=  \mathbb E\bigl (\mathbf 1_{B}(\Delta \widetilde  M_{\rho_k})\big|\mathcal F_{T(k-1)/N_0}\bigr)= \mathbb P \bigl(\Delta \widetilde M_{\rho_k}\big|\mathcal F_{T(k-1)/N_0}\bigr)(B),
\end{align*}
where $(*)$ follows from Lemma \ref{lem:predsttimetang}. Therefore by Subsection \ref{subsec:tanmartdiscase} we can set $\tilde e^{a, N_0}_k :=  \Delta\widetilde M_{\rho_k}$ for any $k=1,\ldots, N_0$, and \eqref{eq:NatildeapproxMatildeindiscr} follows.

 \eqref{eq:tildeMNN0dssdssa} follows from \eqref{eq:EftildemisalmostEfLN_0} and  \eqref{eq:fNN_0-fLNPoidsmall}, where \eqref{eq:fNN_0-fLNPoidsmall} holds by Lemma \ref{lem:tangdecofsumissumoftangdec},
\eqref{eq:NctildeapproxMctildeindiscr}, \eqref{eq:NqtildeapproxMqtildeindiscr}, \eqref{eq:NatildeapproxMatildeindiscr}, and the triangle inequality.
\end{proof}

\begin{remark}
What if we endow our Skorokhod space $\mathcal D([0, T]; X)$ with a different Skorokhod topology, say, with the $J_2$, $M_1$, or $M_2$ topology (see \cite{Wh02,Bj14} for the definition)? Then we will still have convergence in distribution in Theorem  \ref{thm:condtrstepsnsidtrtothepledoneUMDcase} as $J_1$ as a topology is stronger than any one of the aforementioned (see e.g.\ \cite[Subsection 11.5.2]{Wh02}).
\end{remark}

\begin{remark}\label{rem:supnormisbadforJKW}
 Note that Theorem \ref{thm:condtrstepsnsidtrtothepledoneUMDcase} does not hold if one changes the topology of the Skorokhod space to the one generated by the sup-norm. If this is the case, then $\mathcal D([0, T], X)$ becomes nonseparable. In particular, if we set $T=1$ and $X = \mathbb R$, and if we choose our martingale $M$ to be a compensated standard Poisson process, then $\widetilde M$ has the same distribution as $M$. Let $f:\mathcal D([0, T], \mathbb R) \to [-1,1]$ be a function such that 
 $$
 f(F) := \max_{t\in [0,T]\setminus \mathbb Q} |\Delta F(t)| \wedge 1,\;\;\; F \in \mathcal D([0, T], \mathbb R).
 $$ 
 Then $f$ is Lipschitz if $\mathcal D([0, T], \mathbb R)$ is endowed with the sup-norm, $\mathbb P(f(\widetilde M) = 1)>0$ as $\widetilde M$ has jumps of size $1$ and as the first jump time has the exponential distribution (recall that this distribution has a density on $\mathbb R_+$ thus $\mathbb Q$ has zero measure with this distribution), but $f(\widetilde M^n) = 0$ as $\widetilde M^n$ has jumps only in rational points, so $\lim_{n\to \infty}\mathbb E f(\widetilde M^n) \neq \mathbb E f(\widetilde M)$.
\end{remark}

\section{Exponential formula}\label{sec:LevyHinchinFormula}

In the present section we are going to provide another elementary characterization of local characteristics. Namely, we will be generalizing a L\'evy-Khinchin-type result for general martingales which is of the form \cite[Theorem II.2.47]{JS} (see also \cite{Jac83}). First recall that 
for a given predictable stopping time $\tau$ a process $V$ is called a {\em local martingale on $[0,\tau)$} if $V^{\tau_n}$ is a local martingale for any $n\geq 1$ and for any announcing sequence $(\tau_n)_{n\geq 1}$ of $\tau$ (see Subsection \ref{subsec:prelimsttimes} and \cite[Definition II.2.46]{JS}). Then the main result of the section is as follows.

\begin{theorem}\label{thm:LevyHinchin}
Let $X$ be a UMD Banach space, $M:\mathbb R_+ \times \Omega \to X$ be a local martingale. Let $V$ be a c\`adl\`ag bilinear form-valued predictable process starting in zero, $\nu$ be a predictable random measure on $\mathbb R_+ \times X$. Then $(V, \nu)$ are the local characteristics $([\![M^c]\!], \nu^M)$ of $M$ if and only if for any $x^*\in X^*$ there exists a process
\begin{equation}\label{eq:defofAtfpwoiqemLKd}
 A_t(x^*) := - \frac 12 V_t(x^*, x^*) + \int_{[0,t]\times X}(e^{i\langle x, x^*\rangle} - 1 - i\langle x, x^*\rangle) \ud \nu(s, x),\; t\geq 0,
\end{equation}
such that for a process $G(x^*):\mathbb R_+ \times \Omega \to \mathbb R$ defined by
\begin{equation}\label{eq:defofG(x^*)vasqf}
 G_t(x^*) = \mathcal E(A(x^*))_t := e^{A_t(x^*)}\Pi_{0\leq s\leq t}(1 + \Delta A_s(x^*)) e^{-\Delta A_s(x^*)},\;\;\; t\geq 0,
\end{equation}
and for a predictable stopping time $\tau_{G(x^*)} := \inf \{t\geq 0: G_t(x^*) = 0\} =  \inf \{t\geq 0: \Delta A_t(x^*) = -1\}$, we have that
\begin{equation}\label{eq:LHforgenmartstrangeform}
t\mapsto e^{i\langle M_t, x^*\rangle} / {G_t(x^*)},\;\;\; t\geq 0,
\end{equation}
is a local martingale on $[0,\tau_{G(x^*)})$.
\end{theorem}

Why is Theorem \ref{thm:LevyHinchin} connected to the L\'evy-Khinchin formula? Assume for a moment that $M$ is quasi-left continuous. Then $\nu^M$ is non-atomic in time, so $A$ does not have jumps and $G(x^*) = e^{A(x^*)}$ for any $x^*\in X^*$, and thus by Theorem \ref{thm:LevyHinchin} we have that $\tau_{G(x^*)} = \infty$, hence 
\begin{multline*}
t\mapsto e^{i\langle M_t, x^*\rangle} /e^{A_t(x^*)} \\
= e^{i\langle M_t, x^*\rangle} /e^{- \frac 12 [\![M^c]\!]_t(x^*, x^*) + \int_{[0,t]\times X}(e^{i\langle x, x^*\rangle} - 1 - i\langle x, x^*\rangle) \ud \nu^M(s, x)},\;\;\; t\geq 0
\end{multline*}
 is a local martingale. Furthermore, if $M$ additionally has independent increments and if $M_0=0$, then by Corollary \ref{cor:GrigcharformartwithIIXgeneral} $([\![M^c]\!], \nu^M)$ are deterministic, hence $A_t$ is deterministic, so we have that $e^{A_t(x^*)}$ is deterministic for any $t\geq 1$ and for any $x^*\in X^*$, and as $e^{i\langle M_t, x^*\rangle}$ is integrable and uniformly bounded, we have that $t\mapsto e^{i\langle M_t, x^*\rangle}/e^{A_t(x^*)}$ is a martingale, so for any $t\geq 0$ and $x^*\in X^*$
 \begin{equation}\label{eq:LKhformdorqlcmarindawinfi}
   \mathbb E e^{i\langle M_t, x^*\rangle} = e^{- \frac 12 [\![M^c]\!]_t(x^*, x^*) + \int_{[0,t]\times X}(e^{i\langle x, x^*\rangle} - 1 - i\langle x, x^*\rangle) \ud \nu^M(s, x)},
 \end{equation}
 which is the L\'evy-Khinchin formula (see e.g.\ \cite{JS,Sato}).
 
 \smallskip
 
 Let us shortly recall to the reader the idea of the proof of Theorem \ref{thm:LevyHinchin} in the real-valued setting (for the full proof we refer the reader to \cite[\S II.2d]{JS}). We start with proving  the ``only if'' part, i.e.\ first we show that \eqref{eq:LHforgenmartstrangeform} is a local martingale given the corresponding local characteristics. Fix a local martingale $M:\mathbb R_+\times \Omega \to \mathbb R$ with the local characteristics $([M^c], \nu^M)$ and fix $u\in \mathbb R$. Then by It\^o's formula \cite[Theorem 26.7]{Kal} for any $t\geq 0$ we have that
 \begin{equation}\label{eq:Itosfornforeiumfrq}
  \begin{split}
     e^{iuM_t} = 1 &+ iu \int_0^t e^{iuM_{s-}} \ud M_s - \frac{1}{2} u^2 \int_0^t e^{iuM_{s-}} \ud [M^c]_s\\
     &+ \sum_{0\leq s\leq t} \Delta  e^{iuM_s} - iu e^{iuM_{s-}} \Delta M_s.
  \end{split}
 \end{equation}
Note that 
\begin{equation}\label{eq:deltaeifwmeiooviamuM}
 \sum_{0\leq s\leq t} \Delta  e^{iuM_s} = \int_{\mathbb R_+ \times R} e^{iuM_{s-}}(e^{iux}-1) \ud \mu^M(s, x)
\end{equation}
and 
\begin{equation}\label{iuweqwfdeltaMiicsjcsviamyuM}
 \sum_{0\leq s\leq t} iu e^{iuM_{s-}} \Delta M_s =  \int_{\mathbb R_+ \times R}iu e^{iuM_{s-}}x \ud \mu^M(s, x),
\end{equation}
so as $\nu^M$ is a compensator of $\mu^M$, \eqref{eq:Itosfornforeiumfrq}, \eqref{eq:deltaeifwmeiooviamuM}, and \eqref{iuweqwfdeltaMiicsjcsviamyuM} yield that
\[
 e^{iuM_t} + \frac{1}{2} u^2 \int_0^t e^{iuM_{s-}} \ud [M^c]_s -  \int_{\mathbb R_+ \times R} (e^{iux} - 1- iux )e^{iuM_{s-}} \ud \nu^M(s, x),\;\;\; t\geq 0,
\]
is a local martingale. Denote this local martingale by $N$ and note that 
\begin{equation}\label{eq:N=Y-Y_cdotA}
 N_t = e^{iuM_t}- \int_0^t e^{iuM_{s-}} \ud A_s,\;\;\; t\geq 0,
\end{equation}
where $A = A(u)$ is defined by \eqref{eq:defofAtfpwoiqemLKd}. For simplicity also denote $Y_t:= e^{iuM_t}$ and $G_t := G_t(u)$ for any $t\geq 0$. By a stopping time argument and thanks to the fact that $G$ defined by \eqref{eq:defofG(x^*)vasqf} is predictable, c\`adl\`ag, and starts in 1, we may assume that it almost never vanishes, so $\tau_{G(u)} = \infty$, and thus we need to show that $t \mapsto Y_t/G_t$ is a local martingale in $t\geq 0$. It\^o's formula \cite[Theorem 26.7]{Kal} yields that for any $t\geq 0$
\begin{equation*}
 \begin{split}
   \ud \frac{G_t}{Y_t} &= \frac{1}{G_{t-}}\ud Y_t-\frac{Y_{t-}}{G_{t-}^2} \ud G_t + \Delta \frac{Y_t}{G_t} - \frac{1}{G_{t-}}\Delta Y_t + \frac{Y_{t-}}{G_{t-}^2} \Delta G_t\\
   &\stackrel{(i)}=\frac{1}{G_{t-}}\ud N_t  + \Delta \frac{Y_t}{G_t} - \frac{1}{G_{t-}}\Delta Y_t + \frac{Y_{t-}}{G_{t-}^2} \Delta G_t\\
      &\stackrel{(ii)}=\frac{1}{G_{t-}}\ud N_t  + \frac{1}{G_{t-}(1+A_t)}\Delta N_t \Delta A_t
 \end{split} 
\end{equation*}
where $(i)$ follows from \eqref{eq:N=Y-Y_cdotA} (so $\ud Y_t = \ud N_t + Y_{t-} \ud A_t$) and the fact that $\ud G_t = G_{t-}\ud A_t$ due to the definition of a stochastic exponential (see \cite[Theorem I.4.61]{JS}), while $(ii)$ follows for the similar reason and a direct computation (one needs to simply set $\Delta G_t = G_{t-} \Delta A_t$ and $\Delta Y_t = \Delta N_t + Y_{t-} \Delta A_t$). Therefore we have that $Y_t/G_t = 1 + \int_0^t \tfrac{1}{G_{s-}}\ud N_s + \sum_{0\leq s \leq t} \tfrac{1}{G_{s-}(1+A_s)}\Delta N_s \Delta A_s$, which is a local martingale as $N$ is a local martingale and thanks to \cite[Proposition I.4.49]{JS}. 

The ``if'' part of Theorem \ref{thm:LevyHinchin} follows from the fact that as $Z := Y/G$ is a local martingale, and as $G$ is predictable, $Y = GZ = G_{-} \cdot Z + Z_{-}\cdot G + [Z, G]=G_{-} \cdot Z + Y_{-}\cdot A + [Z, G]$ (here we omit the difficulty with complex values, one should consider real and imaginary parts separately), where $G_{-} \cdot Z$ is a local martingale as a stochastic integral w.r.t.\ a local martingale and $[Z, G]$ is a local martingale by \cite[Proposition I.4.49]{JS}, therefore $Y -  Y_{-}\cdot A $ is a local martingale, and hence $A$ has the form \eqref{eq:defofAtfpwoiqemLKd} with the desired $V$ and $\nu$ due to the uniqueness of a predictable finite variation compensator (see \cite[Theorem I.3.18]{JS}) and the fact that $Y_{-}$ has absolute value 1, so $A$ can be reconstructed from $Y_{-}\cdot A $.

 \smallskip
 
Let us finally prove Theorem \ref{thm:LevyHinchin}. For the proof we will need the following lemma.

\begin{lemma}\label{lem:nuisweaklygood}
Let $M:\mathbb R_+ \times \Omega \to \mathbb R$ be a local martingale. Then we have that for any $t\geq 0$ a.s.
\begin{equation}\label{eq:nuisweaklylevy}
\int_{[0, t]\times \mathbb R} |x|^2\wedge|x|\ud \nu^M(x, s), \int_{[0, t]\times \mathbb R} |x|^2\wedge|x|\ud \mu^M(x, s)<\infty.
\end{equation}
\end{lemma}

\begin{proof}
By a stopping time argument and by Doob's maximal inequality \eqref{eq:DoobsineqXBanach} we may assume that $\mathbb E \sup_{t\geq 0} |M_t| <\infty$ and that for some constant $C$ there is a.s.\ at most one jumps exceeding $C$ by the absolute value e.g.\ by setting $M := M^{\tau_C}$ where $\tau_C := \inf\{t\geq 0: |\Delta M_t|>C\}$. First let us show that
\begin{equation}\label{eq:MRvaluemarand|x|2CisintwermuMnuM}
\int_{[0, t]\times \mathbb R} |x|^2\mathbf 1_{|x|\leq C}\ud \nu^M(x, s), \int_{[0, t]\times \mathbb R} |x|^2\mathbf 1_{|x|\leq C}\ud \mu^M(x, s)<\infty,
\end{equation}
which follows from the fact that 
$$
t \mapsto A_t := \int_{[0, t]\times \mathbb R} |x|^2\mathbf 1_{|x|\leq C}\ud \mu^M(x, s) = \sum_{0\leq s\leq t} |\Delta M_s|^2 \leq [M]_t <\infty,
$$
and $A_t$ locally has the first moment (so \eqref{eq:MRvaluemarand|x|2CisintwermuMnuM} for $\nu^M$ follows from \eqref{eq:defofcompofrandsamdsaer}) as $A_t$ has jumps of at most value $C^2$.

Let us show that
\begin{equation}\label{eq:MRvaluemarand|x|1CsaisintwermuMnuM}
\int_{[0, t]\times \mathbb R} |x|\mathbf 1_{|x|> C}\ud \nu^M(x, s), \int_{[0, t]\times \mathbb R} |x|\mathbf 1_{|x|> C}\ud \mu^M(x, s)<\infty,
\end{equation}
which follows from our assumption on $M$, the fact that consequently 
\[
\mathbb E\int_{[0, t]\times \mathbb R} |x|\mathbf 1_{|x|> C}\ud \mu^M(x, s) = \mathbb E \sum_{0\leq s\leq t} |\Delta M_s| \mathbf 1_{|\Delta M_s| >C} \leq 2 \mathbb E\sup_{t\geq 0} |M_t| <\infty,
\]
and from \eqref{eq:defofcompofrandsamdsaer}. \eqref{eq:nuisweaklylevy} follows from \eqref{eq:MRvaluemarand|x|2CisintwermuMnuM} and \eqref{eq:MRvaluemarand|x|1CsaisintwermuMnuM}.
\end{proof}

\begin{proof}[Proof of Theorem \ref{thm:LevyHinchin}]
Let us start with the ``only if'' part. Let $(V, \nu) = ([\![ M^c]\!], \nu^M)$.
Notice that for any $x^*\in X^*$ by Taylor's formula
\[
e^{i\langle x, x^*\rangle} - 1 - i\langle x, x^*\rangle \approx \frac{1}{2}| \langle x, x^*\rangle|^2,\;\;\; x\in X,
\]
for small $\langle x, x^*\rangle$ and
\[
|e^{i\langle x, x^*\rangle} - 1 - i\langle x, x^*\rangle| \leq |\langle x, x^*\rangle| + 2 \lesssim  |\langle x, x^*\rangle|, \;\;\;x\in X,
\]
for big $\langle x, x^*\rangle$, so by Lemma \ref{lem:nuisweaklygood} and \ref{lem:lintransfofmart-->transfofnuM} the integral
\begin{equation}\label{eq:AfromLKHmakesence}
 \int_{[0,t]\times X}e^{i\langle x, x^*\rangle} - 1 - i\langle x, x^*\rangle \ud \nu^M(s, x),\;\;\; t\geq 0,
\end{equation}
is well defined. Consequently, $A$ is well defined.

Now let us show that \eqref{eq:LHforgenmartstrangeform} is a local martingale. As by Definition \ref{def:prelimdefcovbilform}  and by Remark \ref{rem:candec==>candecforanyx*} we have that a.s.
\[
 [\![ M^c]\!](x^*, x^*) = [\langle M^c, x^*\rangle]_t = [\langle M, x^*\rangle^c]_t,\;\;\; t\geq 0,
\]
and since by Lemma \ref{lem:lintransfofmart-->transfofnuM} we have that
\[
 \int_{[0,t]\times X}e^{i\langle x, x^*\rangle} - 1 - i\langle x, x^*\rangle \ud \nu^M(s, x) =  \int_{[0,t]\times \mathbb R}(e^{ir} - 1 - ir) \ud \nu^{\langle M, x^*\rangle}(s, r) ,\;\;\; t\geq 0,
\]
we can restrict ourselves to the one dimensional setting. Let $X = \mathbb R$, $M$ be one dimensional. We will be using the setting of \cite[Section II.2]{JS}. Let $h(r) = r$, $r\in \mathbb R$ (though $h$ is assumed to be bounded in \cite[Section II.2]{JS}, in our case both $\int(e^{ir} - 1 - ih(r)) \ud \nu^{\langle M, x^*\rangle}(s, r) $ and $\int h\ud \bar{\mu}^M$ are well defined by \eqref{eq:AfromLKHmakesence} and by Theorem \ref{thm:XisUMDiffMisintxwrtbarmuMvain}, so we can set $h$ to be as defined). Then we have that the representation of $M$ given by formula \cite[II.2.35]{JS} is
\begin{equation*}\label{eq:reprofMII235}
 M = M_0 + M^c + \int h \ud \bar{\mu}^M + \int_{[0, \cdot] \times \mathbb R} r - h(r) \ud \mu^M(s, r) + B = M^c + \int h \ud \bar{\mu}^M,
\end{equation*}
so by the uniqueness of this representation and by \cite[Theorem II.2.47]{JS} the desired follows.

\smallskip

Let us now show the ``if'' part. It is sufficient to show that $V(x^*, x^*) = [\![M^c]\!](x^*, x^*)$ a.s.\ for any $x^*\in X^*$, and analogously to the proof of Theorem \ref{thm:tangiffwtangUMDcase} it is sufficient to show that $\nu^{x^*}= \nu^{\langle M, x^*\rangle}$ for any $x^*\in X^*$, where $\nu^{x^*}$ is defined as a unique predictable random measure on $\mathbb R_+ \times \mathbb R$ such that for any elementary predictable $F:\mathbb R_+\times \Omega \times \mathbb R \to \mathbb R_+$ one has that a.s.\
\[
 \int_{\mathbb R_+ \times Y} F(s, \cdot, y) \ud \nu^{x^*}(s, \cdot, y) = \int_{\mathbb R_+ \times X} F(s, \cdot, \langle x, x^*\rangle) \ud \nu(s, \cdot, x).
\]
In order to construct such a random measure it is sufficient to set a.s.\
$$
 \int \mathbf 1_{A}(s, \cdot) \mathbf 1_{B}(y) \ud \nu^{x^*}(s, \cdot, y) :=  \int \mathbf 1_{A}(s, \cdot) \mathbf 1_{B}(\langle x, x^*\rangle) \ud \nu(s, \cdot, x)
$$
for any $A \in \mathcal P$ and $B\in \mathcal B(\mathbb R)$. To this end note that $V(x^*, x^*) = [\![M^c]\!](x^*, x^*)$ a.s.\ for any $x^*\in X^*$ as well as $\nu^{x^*} = \nu^{\langle M, x^*\rangle}$ a.s.\ by the fact that \eqref{eq:LHforgenmartstrangeform} is a local martingale and by \cite[Theorem II.2.47 and II.2.49]{JS}.
\end{proof}

\begin{remark}
 Note that $\tau_{G(\eps x^*)} \to \infty$ a.s.\ as $\eps\to 0$ for any $x^*\in X^*\setminus\{0\}$. Indeed, due to the definition of $\tau_{G(\eps x^*)}$ is it sufficient to show that 
 \[
 \int_{[0,t]\times X}|e^{i\langle x,\eps x^*\rangle} - 1 - i\langle x, \eps x^*\rangle| \ud \nu^M(s, x) \to 0,\;\;\; \eps \to 0,
 \]
which follows from the fact that for $\eps<1$
$$
|e^{i\langle x,\eps x^*\rangle} - 1 - i\langle x, \eps x^*\rangle| \lesssim  |\langle x,\eps x^*\rangle|^2 \wedge |\langle x,\eps x^*\rangle| \lesssim \eps \bigl(|\langle x,x^*\rangle|^2 \wedge |\langle x, x^*\rangle|\bigr)
$$ 
 and from Lemma \ref{lem:nuisweaklygood}. It remains unknown for the author whether $\tau_{G( x^*)}\to \infty$ as $x^*\to 0$.
\end{remark}

\begin{remark}\label{rem:LKformdoeinqaindwoe}
 Let $M$ have independent increments and let $M_0=0$. In this case one can show that analogously to \eqref{eq:LKhformdorqlcmarindawinfi}
 \begin{equation}\label{eq:LKhfoemrioforMIIinsgenfoqw}
  \mathbb E e^{i\langle M_t, x^*\rangle} = G_t(x^*) = e^{A_t(x^*)}\Pi_{0\leq s\leq t}(1 + \Delta A_s(x^*)) e^{-\Delta A_s(x^*)},
 \end{equation}
for any $t\geq 0$ and $x^*\in X^*$ even if $t>\tau_{G(x^*)}$. First note that $\tau_{G(x^*)}$ in this case is a deterministic stopping time as $G(x^*)$ is deterministic by Corollary \ref{cor:GrigcharformartwithIIXgeneral}. Second, in order to prove \eqref{eq:LKhfoemrioforMIIinsgenfoqw} for $t\geq \tau_{G(x^*)}$ (the case $t< \tau_{G(x^*)}$ follows from Theorem \ref{thm:LevyHinchin}) let $r = \tau_{G(x^*)}$. Then by the discussions in Subsection \ref{subsec:dectanPDwithAJ}
\begin{align*}
 0 = 1 + \Delta A_r(x^*) &= 1 + \int_{\mathbb R_+\times X} \mathbf 1_{\{r\}}(s) (e^{i\langle x, x^*\rangle} - 1 - i\langle x, x^*\rangle) \ud \nu^M(s, x) \\
&=1 + \mathbb E (e^{i\langle \Delta M_r, x^*\rangle} - 1 - i\langle \Delta M_r, x^*\rangle|\mathcal F_{r-})\\
&=\mathbb E e^{i\langle \Delta M_r, x^*\rangle} ,
\end{align*}
where the latter follows as $M$ has independent increments and as $M$ is a martingale (so $\mathbb E (\langle \Delta M_r, x^*\rangle|\mathcal F_{r-})=0$). Therefore $\mathbb E e^{i\langle \Delta M_r, x^*\rangle} =0$ and as $M$ has independent increments
\[
 \mathbb E e^{i\langle M_t, x^*\rangle} = \mathbb E e^{i\langle M_{r-}, x^*\rangle} \mathbb E e^{i\langle \Delta M_r, x^*\rangle} \mathbb E e^{i\langle (M_t-M_r), x^*\rangle}=0,
\]
so \eqref{eq:LKhfoemrioforMIIinsgenfoqw} is satisfied.
Note that thanks to the techniques from Section \ref{sec:indincrements} one can omit the UMD assumption.
\end{remark}

\section{Characteristic subordination and characteristic domination\\ of martingales}\label{sec:charsubandchardomofmart}

In the present section we prove basic $L^p$-estimates concerning characteristic subordination and characteristic domination.

\subsection{Characteristic subordination}\label{subsec:charsub}

Let $X$ be a Banach space, $M, N:\mathbb R_+ \times \Omega \to X$ be local martingales. Then $N$ is called to be {\em weakly differentially subordinate} to $M$ if $|\langle N_0, x^* \rangle| \leq |\langle M_0, x^* \rangle|$ a.s.\ and if $[\langle M, x^* \rangle]_t-[\langle N, x^* \rangle]_t$ is nondecreasing a.s.\ for any $x^*\in X^*$. Weak differential subordination was intensively studied during past two years in \cite{Y17FourUMD,Y17MartDec,OY18,Y18BDG,YPhD,Y19weakL1}. In particular, it was shown in \cite[Subsection 7.4]{Y18BDG} that if this is the case and $X$ is UMD then
\begin{equation}\label{eq:WDSstrongLpineqRG}
\mathbb E \sup_{t\geq 0}\|N_t\|^p \lesssim_{p, X} \mathbb E \sup_{t\geq 0}\|M_t\|^p.
\end{equation} 
In the present section we will consider a predictable analogue of weak differential subordination which exploits local characteristics -- {\em characteristic subordination}.

\begin{definition}\label{def:charsubdefinb}
Let $X$ be a Banach space, $M, N:\mathbb R_+ \times \Omega \to X$ be local martingales. Then $N$ is {\em characteristicly subordinate to} $M$ if for any $x^*\in X^*$ a.s.\
\begin{enumerate}[\rm (A)]
\item $|\langle N_0, x^* \rangle| \leq |\langle M_0, x^* \rangle|$,
\item $[\langle N, x^*\rangle^c]_t - [\langle N, x^*\rangle^c]_s \leq [\langle M, x^*\rangle^c]_t - [\langle M, x^*\rangle^c]_s$,
\item $\nu^{\langle N, x^*\rangle} \leq \nu^{\langle M, x^*\rangle}$.
\end{enumerate}
\end{definition}

(Recall that though $M$ and $N$ take their values in a general (not necessarily UMD) Banach space $X$ (so $M^c$ and $N^c$ may not have sense), $\langle M, x^*\rangle$ and $\langle N, x^*\rangle$ are real-valued martingales, so $\langle M, x^*\rangle^c$ and $\langle N, x^*\rangle^c$ exist, see Theorem \ref{thm:candecXvalued}).

Note that if $M$ and $N$ are continuous, then $N$ is characteristically subordinate to $M$ if and only if $N$ is weakly differentially subordinate to $M$. The following two propositions show that weak differential subordination coincides with characteristic subordination only in continuous case.

\begin{proposition}
Weak differential subordination does not imply characteristic subordination.
\end{proposition}

\begin{proof}
Let $M$ be a purely discontinuous nonzero martingale  with an a.s.\ finite measure $\nu^M$ (e.g.\ a compensated standard Poisson process which stopped at time point $1$) and $N = \tfrac 12 M$. Then $N$ is weakly differentially subordinate to $M$, but is not  characteristically subordinate. Indeed, by Lemma \ref{lem:lintransfofmart-->transfofnuM} we have that a.s.\ for any Borel $B\in X$
\begin{equation}\label{eq:nu12MB=nuM12B}
\nu^N([0,t]\times B) = \nu^M([0,t]\times 2 B),\;\;\; t\geq 0.
\end{equation}
Thus we have that a.s.\ $\nu^M([0,t]\times X \setminus\{0\})=\nu^N([0,t]\times X \setminus\{0\})$, and as $\nu^M$ is finite, by \eqref{eq:nu12MB=nuM12B} we have that $\nu^N \nleq \nu^M$ on a set of positive probability, as if we assume the converse, then for the sets $C_n = 2^nB \setminus 2^{n-1}B$, $-\infty<n<\infty$, where $B\in X$ is the unit ball,  we have that $C_n = 2C_{n-1}$, and hence by \eqref{eq:nu12MB=nuM12B} for any $-\infty<n<\infty$
\[
\nu^M([0,t]\times C_n) \geq \nu^N([0,t]\times C_n) = \nu^M([0,t]\times C_{n+1}),\;\;\; t\geq 0,
\]
so $\nu^M$ is infinite (as $C_n$'s are disjoint, $\cup C_n = X \setminus \{0\}$, and as $\nu^M \neq 0$, there exists $n$ and $t$ such that $ \nu^M([0,t]\times C_{n}) >0$), which contradicts our assumption.
\end{proof}

\begin{proposition}
Characteristic subordination does not imply weak differential subordination.
\end{proposition}

\begin{proof}
It is sufficient to consider two independent compensated standard Poisson processes $\widetilde N_1$ and $\widetilde N_2$, as they are characteristically subordinate to each other (because they have the same local characteristics), but they are not weakly differentially subordinate to each other as they have jumps at different times a.s., i.e.\ $\Delta \widetilde N_1 \neq 0 \Rightarrow \Delta \widetilde N_2 = 0$ and $\Delta \widetilde N_2 \neq 0 \Rightarrow \Delta \widetilde N_1 = 0$ a.s.\ for any $t\geq 0$.
\end{proof}

\begin{remark}
What do the aforementioned examples demonstrate? These examples show that $N$ is weakly differentially subordinate to $M$ if $N$ {\em has smaller jumps than} $M$, and $N$ is characteristically subordinate to $M$ if $N$ {\em has the same jumps as} $M$ {\em but these jumps occur less often}.
\end{remark}

Let us now formulate the main theorem of the present section.

\begin{theorem}\label{thm:charsubforXvalued}
Let $X$ be a Banach space. Then $X$ is UMD if and only if for any $1\leq p<\infty$ and for any local martingales $M,N:\mathbb R_+\times \Omega \to X$ such that $N$ is characteristicly subordinate to $M$ one has that
\[
\mathbb E \sup_{t\geq 0}\|N_t\|^p \lesssim_{p, X}\mathbb E \sup_{t\geq 0}\|M_t\|^p.
\]
\end{theorem}

The proof of the theorem is based on the canonical decomposition (see Subsection \ref{subsec:candec}) and treating each case of the canonical decomposition separately. Therefore we will need the following propositions.

\begin{proposition}\label{prop:charsubordforrandommeasures}
Let $X$ be a UMD Banach space, $(J, \mathcal J)$ be a measurable space, $\mu$ and $\mu'$ be quasi-left continuous optional random measures on $\mathbb R_+ \times J$ such that for the corresponding compensators $\nu$ and $\nu'$ we have that $\nu' \leq \nu$ a.s. Then for any elementary predictable $F:\mathbb R_+ \times \Omega \times J \to X$ and for any $1\leq p<\infty$ we have that
\[
\mathbb E \sup_{t\geq 0} \Bigl\| \int_{[0,t]\times J} F \ud \bar {\mu}' \Bigr\|^p \lesssim_{p, X} \mathbb E \sup_{t\geq 0} \Bigl\| \int_{[0,t]\times J} F \ud \bar {\mu} \Bigr\|^p,
\]
where $\bar {\mu} := \mu - \nu$ and $\bar {\mu}' := \mu' - \nu'$.
\end{proposition}

\begin{proof}
Let $\mu_{\rm Cox}$ and $\mu'_{\rm Cox}$ be Cox processes directed by $\nu$ and $\nu'$ respectively, and set $\bar{\mu}_{\rm Cox} := \mu_{\rm Cox} - \nu$, $\bar{\mu}'_{\rm Cox} := \mu'_{\rm Cox} - \nu'$. Then by Theorem \ref{thm:mumuCoxcomparable}
\[
\mathbb E \sup_{t\geq 0} \Bigl\| \int_{[0,t]\times J} F \ud \bar {\mu} \Bigr\|^p \eqsim_{p, X} \mathbb E  \mathbb E_{\rm Cox} \Bigl\| \int_{\mathbb R+\times J} F \ud \bar {\mu}_{\rm Cox}\Bigr\|^p,
\]
\[
\mathbb E \sup_{t\geq 0} \Bigl\| \int_{[0,t]\times J} F \ud \bar {\mu}' \Bigr\|^p \eqsim_{p, X} \mathbb E \mathbb E_{\rm Cox}  \Bigl\| \int_{\mathbb R_+\times J} F \ud \bar {\mu}'_{\rm Cox}\Bigr\|^p,
\]
where $ \mathbb E_{\rm Cox}$ is defined in Example \ref{ex:defofExiforrvxis}. Thus it is sufficient to show that for a.e.\ fixed $\omega\in \Omega$
\begin{equation}\label{eq:domCoxmeasurepointwise}
 \mathbb E_{\rm Cox} \Bigl\| \int_{\mathbb R_+\times J} F \ud \bar {\mu}'_{\rm Cox}\Bigr\|^p \leq  \mathbb E_{\rm Cox} \Bigl\| \int_{\mathbb R_+\times J} F \ud \bar {\mu}_{\rm Cox}\Bigr\|^p.
\end{equation}
Let us now show \eqref{eq:domCoxmeasurepointwise}. Fix $\omega\in \Omega$ such that $\nu'(\omega) \leq \nu(\omega)$. Then both $\mu_{\rm Cox}$ and $\mu_{\rm Cox}'$ are time changed Poisson. Let $\nu'' = \nu-\nu'$, $\mu_{\rm Cox}''$ be the Cox process directed by $\nu''$. As $\omega$ is fixed, $\mu'_{\rm Cox}$ and $\mu''_{\rm Cox}$ are independent and $\mu'_{\rm Cox} + \mu''_{\rm Cox}$ has the same compensator and hence coincides in distribution with $\mu_{\rm Cox}$ so we can set $\mu_{\rm Cox} = \mu'_{\rm Cox} + \mu''_{\rm Cox}$, $\bar{\mu}_{\rm Cox} =\bar{\mu}'_{\rm Cox} + \bar{\mu}''_{\rm Cox}$. Therefore \eqref{eq:domCoxmeasurepointwise} follows from the fact that for a fixed $\omega\in \Omega$ the process $F$ is deterministic, the fact that a  conditional expectation operator is a contraction (see \cite[Section 2.6]{HNVW1}), and
\[
\int_{\mathbb R_+\times J} F \ud \bar {\mu}'_{\rm Cox} = \mathbb E \Bigl(\int_{\mathbb R_+\times J} F \ud \bar {\mu}_{\rm Cox}\Big| \sigma(\bar {\mu}'_{\rm Cox})\Bigr),
\]
as $ \mu'_{\rm Cox}$ and $\mu''_{\rm Cox}$ are independent for any fixed $\omega\in \Omega$.
\end{proof}

We will also need the following proposition which is some sense extends stochastic domination inequality \cite[Theorem 2]{PrA97} (see also \cite{MSP01}).

\begin{proposition}\label{prop:indeprvcharsubimpliesLp}
Let $X$ be a Banach space, $(\xi_n)_{n=1}^N$ and $(\xi'_n)_{n=1}^N$ be independent $X$-valued symmetric random variable such that for any Borel set $A \subset X \setminus \{0\}$ and for any $n=1,\ldots,N$ one has that $\mathbb P(\xi'_n\in A) \leq \mathbb P(\xi_n\in A)$. Then for any convex symmetric function $\phi:X \to \mathbb R_+$ one has that
\begin{equation}\label{eq:compforsymindrvoneisbigge}
\mathbb E \phi\Bigl( \sum_{n=1}^N \xi'_n\Bigr) \leq \mathbb E \phi\Bigl( \sum_{n=1}^N \xi_n\Bigr).
\end{equation}
\end{proposition}

\begin{proof}
As $(\xi_n)_{n=1}^N$ and $(\xi'_n)_{n=1}^N$ are symmetric, \eqref{eq:compforsymindrvoneisbigge} is equivalent to
\begin{equation}\label{eq:compforsymindrvoneisbiggewithrademach}
\mathbb E \phi\Bigl( \sum_{n=1}^N r_n \xi'_n\Bigr) \leq \mathbb E \phi\Bigl( \sum_{n=1}^N r_n\xi_n\Bigr),
\end{equation}
where $(r_n)_{n=1}^N$ is an independent sequence of i.i.d.\ Rademachers (see Definition \ref{def:ofRadRV}). Thus it is sufficient to show \eqref{eq:compforsymindrvoneisbiggewithrademach}. By an approximation argument we may assume that $(\xi_n)_{n=1}^N$ and $(\xi'_n)_{n=1}^N$ take finitely many values. By the assumption of the proposition we have that for any $n=1,\ldots,N$ the random variable $\xi'_n$ has the same distribution as $\eta_n(\xi_n)\xi_n$, where for any $x\in X\setminus \{0\}$ we define a random variable $\eta_n(x)$ on an independent probability space $(\widetilde {\Omega}, \widetilde {\mathcal F}, \widetilde{\mathbb P})$ to be such that $\eta_n(x) \in \{0, 1\}$ a.s.\ and $\mathbb E \eta_n(x) = \tfrac{\mathbb P(\xi'_n=x)}{\mathbb P(\xi_n=x)}$, where we set $\tfrac{0}{0}:=0$. Fix $\omega\in \Omega$ and $\widetilde \omega \in \widetilde {\Omega}$. Then in order to show \eqref{eq:compforsymindrvoneisbiggewithrademach} it remains to prove that
\[
\mathbb E_r \phi\Bigl( \sum_{n=1}^N r_n\eta_n(x)(\widetilde \omega) \xi_n(\omega)\Bigr) \leq \mathbb E \phi\Bigl( \sum_{n=1}^N r_n\xi_n(\omega)\Bigr),
\]
and as all the coefficients $(\eta_n(x)(\widetilde \omega))_{n=1}^N$ are either 0 or 1, the latter follows from Jensen's inequality (see \cite[Proposition 2.6.29]{HNVW1}) as $ \sum_{n=1}^N r_n\eta_n(x)(\widetilde \omega) \xi_n(\omega)$ is just a conditional expectation of $ \sum_{n=1}^N r_n\xi_n(\omega)$ given $\sigma (r_n\eta_n(x)(\widetilde \omega))_{n=1}^N$.
\end{proof}

\begin{proof}[Proof of Theorem \ref{thm:charsubforXvalued}]
Without loss of generality (as $\|N_0\| \leq \|M_0\|$, see \cite[Lemma 3.6]{Y17FourUMD}) we can set $M_0 = N_0 = 0$. Let $M = M^c + M^q + M^a$ and $N=N^c + N^q + N^a$ be the canonical decompositions. Note that due to Definition \ref{def:charsubdefinb} and  Subsection \ref{subsec:charandthecandecbe} we have that $\nu^{\langle N^q, x^*\rangle}\leq \nu^{\langle M^q, x^*\rangle}$ and $\nu^{\langle N^a, x^*\rangle}\leq \nu^{\langle M^a, x^*\rangle}$ a.s.\ for any $x^*\in X^*$, so $N^{i}$ is characteristically subordinate to $M^i$ for any $i\in \{c,q,a\}$. By \eqref{eq:candecstrongLpestmiaed} it is sufficient to show that
\begin{equation}\label{eq:chardomforcandeccontcase}
\mathbb E \sup_{t\geq 0} \|N^c_t\|^p \lesssim_{p, X}\mathbb E \sup_{t\geq 0} \|M^c_t\|^p,
\end{equation}
\begin{equation}\label{eq:chardomforcandecpdqlccase}
\mathbb E \sup_{t\geq 0} \|N^q_t\|^p \lesssim_{p, X}\mathbb E \sup_{t\geq 0} \|M^q_t\|^p,
\end{equation}
\begin{equation}\label{eq:chardomforcandecpdwajcase}
\mathbb E \sup_{t\geq 0} \|N^a_t\|^p \lesssim_{p, X}\mathbb E \sup_{t\geq 0} \|M^a_t\|^p.
\end{equation}
First of all, \eqref{eq:chardomforcandeccontcase} follows from \eqref{eq:WDSstrongLpineqRG}. \eqref{eq:chardomforcandecpdqlccase} follows from Proposition \ref{prop:charsubordforrandommeasures}, the fact that $M^q = \int x \ud \bar{\mu}^{M^q}$ and $N^q = \int x \ud \bar{\mu}^{N^q}$ by Theorem \ref{thm:XisUMDiffMisintxwrtbarmuMvain}, and the fact that $N^q$ is characteristically subordinate to $M^q$. Finally, \eqref{eq:chardomforcandecpdwajcase} follows from a standard approximation argument (see e.g.\ Proposition \ref{prop:MmapproxMinLpforacccase}), the fact that any purely discontinuous martingale with finitely many predictable jumps has a discrete representation (see e.g.\ the proof of Proposition \ref{prop:pdwajjusttang}), Proposition \ref{prop:indeprvcharsubimpliesLp}, the construction of a decoupled tangent martingale from the proof of Theorem \ref{thm:detangmartforMXVpdwithaccjumps}, the symmetrization argument (see the proof of Theorem \ref{thm:Ephifordiscretetantdandee}), and the fact that $N^a$ is characteristically subordinate to $M^a$. 
\end{proof}

\subsection{Characteristic domination}\label{subsec:chardom}

We can straighten characteristic subordination in the following way. Let $X$ be a Banach space, $M$ and $N$ be $X$-valued martingales. Then $N$ is {\em characteristically dominated} by $M$ if a.s.\ $|\langle N_0, x^* \rangle| \leq |\langle M_0, x^* \rangle|$ for any $x^*\in X^*$, $[\![N^c]\!]_{\infty} \leq [\![M^c]\!]_{\infty}$ and $\nu^N(\mathbb R_+ \times \cdot)\leq \nu^M(\mathbb R_+ \times \cdot)$. Then the following theorem 
holds true.

\begin{theorem}\label{thm:chardomqlccase}
 Let $X$ be a Banach space. Then $X$ has the UMD property if and only if for any (equivalently, for some) $1\leq p<\infty$ and for any $X$-valued quasi-left continuous local martingales such that $N$ is characteristically dominated by $M$ one has that
 \begin{equation}\label{eq:chardomqlccasemamamamari}
    \mathbb E \sup_{0\leq t <\infty} \|N_t\|^p \lesssim_{p, X}   \mathbb E \sup_{0\leq t <\infty} \|M_t\|^p.
 \end{equation}
\end{theorem}

For the proof we will need the following proposition.

\begin{proposition}\label{prop:chardomforrandommeasures}
Let $X$ be a UMD Banach space, $(J, \mathcal J)$ be a measurable space, $\mu$ and $\mu'$ be optional quasi-left continuous random measures on $\mathbb R_+ \times J$ such that for the corresponding compensators $\nu$ and $\nu'$ we have that $\nu' (\mathbb R_+ \times A) \leq \nu(\mathbb R_+ \times A) <\infty$ a.s.\ for any $A\in \mathcal J$. Then for any elementary $\mathcal B(\mathbb R_+) \otimes \mathcal F_0 \otimes J$-measurable $F:\mathbb R_+ \times \Omega \times J \to X$ and for any $1\leq p<\infty$ we have that
\[
\mathbb E \sup_{t\geq 0} \Bigl\| \int_{[0,t]\times J} F \ud \bar {\mu}' \Bigr\|^p \lesssim_{p, X} \mathbb E \sup_{t\geq 0} \Bigl\| \int_{[0,t]\times J} F \ud \bar {\mu} \Bigr\|^p,
\]
where $\bar {\mu} := \mu - \nu$ and $\bar {\mu}' := \mu' - \nu'$.
\end{proposition}

\begin{proof}
Without loss of generality we may assume that $t\mapsto F(t, \cdot, \cdot)$ is a constant a.e.\ on $ \Omega \times J$  as otherwise we just approximate $F$ by a step $\mathcal F_0$-measurable function and apply the whole proof below for each step of $F$ separately.

The proposition follows analogously Proposition \ref{prop:charsubordforrandommeasures}, but then we need to show \eqref{eq:domCoxmeasurepointwise} in a difference way. Fix $\omega\in \Omega$ such that 
\begin{equation}\label{eq:nu'ismallerthennuintheend}
 \nu' (\mathbb R_+ \times A) \leq \nu(\mathbb R_+ \times A) <\infty,\;\;\;\;A\in \mathcal J.
\end{equation}
Then by the definition of a stochastic integral \eqref{eq:stochintwrtranmeasdefof}, by the fact that $F$ is elementary predictable, and by the definition of a Cox process (see Subsubsection \ref{subsubsec:Coxprocess}) there exist $x_1, \ldots, x_M\in X$, independent Poisson random variables $(\xi_m)_{m=1}^M$ with parameters $(\lambda_m)_{m=1}^M$ and independent Poisson random variables $(\xi'_m)_{m=1}^M$ with parameters $(\lambda'_m)_{m=1}^M$ satisfying $\lambda'_m \leq \lambda_m$ for any $m=1,\ldots, M$ by \eqref{eq:nu'ismallerthennuintheend} such that
\begin{align*}
 \int_{[0,t]\times J} F \ud \bar {\mu}_{\rm Cox} &= \sum_{m=1}^M x_m (\xi_m-\lambda_m),\\
  \int_{[0,t]\times J} F \ud \bar {\mu}'_{\rm Cox} &= \sum_{m=1}^M x_m (\xi'_m-\lambda'_m)
\end{align*}
(see e.g.\ \cite[p.\ 88]{Kal} or \cite[Section 23]{Bill95} for details on Poisson distributions). Now by the fact that sum of two independent Poisson random variable is again has Poisson distribution with parameter being the sum of the corresponding parameters and by independence of all $\xi_m$'s and $\xi_m'$'s we can assume that there exists a sequence of independent Poisson random variables $(\xi''_m)_{m=1}^M$ with parameters $(\lambda''_m)_{m=1}^M = (\lambda_m -\lambda'_m)_{m=1}^M$ such that $\xi_m-\lambda_m = (\xi'_m-\lambda'_m) + (\xi''_m-\lambda''_m)$, and then the desired follows from the same conditional expectation trick used in the end of the proof of Proposition \ref{prop:charsubordforrandommeasures} and Theorem \ref{thm:mumuCoxcomparable}.
\end{proof}

\begin{proof}[Proof of Theorem \ref{thm:chardomqlccase}]
The ``if'' part follows directly from Proposition \ref{prop:domcontcase} as the latter is a particular case of Theorem \ref{thm:chardomqlccase}. 

Let us show the ``only if'' part. By Proposition \ref{prop:domcontcase} we have that for the continuous terms of the Meyer-Yoeurp decompositions $M=M^c + M^d$ and $N=N^c + N^d$ (see Remark \ref{rem:MYdecBanach}) 
\begin{equation}\label{eq:chardomcontproofofgeneralfascase}
  \mathbb E \sup_{0\leq t <\infty} \|N^c_t\|^p \lesssim_{p, X}   \mathbb E \sup_{0\leq t <\infty} \|M^c_t\|^p.
\end{equation}
Also note that as $M$ and $N$ are quasi-left continuous, $M^d = M^q$ and $N^d=N^q$ (see Subsection \ref{subsec:candec} for the definition of $M^q$ and $N^q$). Thus by Theorem \ref{thm:candecXvalued}, by Proposition \ref{prop:domcontcase}, and by the considerations above it is sufficient to show \eqref{eq:chardomqlccasemamamamari} for $M$ and $N$ being purely discontinuous quasi-left continuous.
First let us show that~if
\[
   \mathbb E \sup_{0\leq t <\infty} \|M_t\|^p <\infty,
\]
then $\nu^M (\mathbb R_+ \times \overline B)<\infty$ a.s.\ for any centered ball $B \subset X$ (here $\overline B$ is the complement of $B$). Indeed, as $M$ is purely discontinuous, then by the fact that any UMD Banach space has a finite Gaussian cotype $q\geq 2$ (see \cite[Definition 7.1.17, Corollary 7.2.11, and Proposition 7.3.15]{HNVW2}) and by Burkholder-Davis-Gundy inequalities \cite[Subsection 6.1]{Y18BDG} for a family $(\gamma_t)_{t\geq 0}$ of i.i.d.\ standard Gaussians and for any $\delta>0$ we have that
\begin{align*}
 \mathbb E \sup_{0\leq t <\infty} \|M_t\|^p& \eqsim_{p, X}\mathbb E \mathbb E_{\gamma} \Bigl\| \sum_{t\geq 0} \gamma_t \Delta M_t \Bigr\|^p \geq \mathbb E \mathbb E_{\gamma} \Bigl\| \sum_{t\geq 0} \gamma_t \Delta M_t \mathbf 1_{\|\Delta M_t\|>\delta}\Bigr\|^p,\\
 &\stackrel{(i)}\eqsim_{p,q}\mathbb E \Bigl(\mathbb E_{\gamma} \Bigl\| \sum_{t\geq 0} \gamma_t \Delta M_t \mathbf 1_{\|\Delta M_t\|>\delta}\Bigr\|^q\Bigr)^{p/q}\\
 &\stackrel{(ii)}=\mathbb E \Bigl(\mathbb E_{\gamma} \Bigl\| \int_{\mathbb R_+\times X}\gamma_t x \mathbf 1_{\|x\|>\delta}\ud \mu^M(t,x)\Bigr\|^q\Bigr)^{p/q}\\
 &\stackrel{(iii)}\gtrsim_{X}\mathbb E \Bigl( \int_{\mathbb R_+\times X} \|x\|^q \mathbf 1_{\|x\|>\delta}\ud \mu^M(t,x)\Bigr)^{p/q}\\
 &\gtrsim_{\delta}\mathbb E \Bigl( \int_{\mathbb R_+\times X} \mathbf 1_{\|x\|>\delta}\ud \mu^M(t,x)\Bigr)^{p/q},
\end{align*}
where $\mathbb E_{\gamma}$ is defined by Example \ref{ex:defofExiforrvxis}, $(i)$ follows from Kahane-Khinchin inequalities \cite[Theorem 6.2.6]{HNVW2}, $(ii)$ holds by the definition of $\mu^M$ (see \eqref{eq:defofmuM}), and $(iii)$ follows from the definition of a Gaussian cotype \cite[Definition 7.1.17]{HNVW2}. Therefore we have that 
\begin{equation}\label{eq:int>deltawrtmuMisfinitea.s.}
 \int_{\mathbb R_+\times X} \mathbf 1_{\|x\|>\delta}\ud \mu^M(t,x) \;\; \text{is finite a.s.,}
\end{equation}
and hence its compensator, 
\begin{equation}\label{eq:int>deltawrtnuMisfinitea.s.votvet}
 \int_{\mathbb R_+\times X} \mathbf 1_{\|x\|>\delta}\ud \nu^M(t,x) = \nu^M (\mathbb R_+ \times \overline B)\;\; \text{is finite a.s.,}
\end{equation}
as well because for stopping times
\[
 \tau_n:= \inf\Bigl\{\int_{[0,t]\times X} \mathbf 1_{\|x\|>\delta}\ud \mu^M(\cdot,x)>n\Bigr\}
\]
we have that by \eqref{eq:int>deltawrtmuMisfinitea.s.} $\{\tau_n = \infty\}\nearrow \Omega$ up to a negligible set, and then \eqref{eq:int>deltawrtnuMisfinitea.s.votvet} follows by \cite[Theorem I.3.17]{JS}.

Now let $M^m$ and $N^m$ be as defined by \eqref{eq:approxforMqlcpdbigjumps}. Then by Proposition \ref{prop:chardomforrandommeasures}, by the definition of the characteristic domination, by Theorem \ref{thm:XisUMDiffMisintxwrtbarmuMvain},
and by \eqref{eq:int>deltawrtnuMisfinitea.s.votvet} we have that
\[
  \mathbb E \sup_{0\leq t <\infty} \|N^m_t\|^p \lesssim_{p, X}   \mathbb E \sup_{0\leq t <\infty} \|M^m_t\|^p,
\]
and thus
\begin{equation}\label{eq:chardompdqsdqproofofgeneralfascase}
   \mathbb E \sup_{0\leq t <\infty} \|N^m_t\|^p \lesssim_{p, X}   \mathbb E \sup_{0\leq t <\infty} \|M^m_t\|^p,
\end{equation}
holds true by Proposition \ref{prop:approxforPDQLC} and by letting $m\to \infty$. Then \eqref{eq:chardomqlccasemamamamari} follows from \eqref{eq:chardomcontproofofgeneralfascase}, \eqref{eq:chardompdqsdqproofofgeneralfascase}, \eqref{eq:candecstrongLpestmiaed}, and the fact that $M$ and $N$ are quasi-left continuous.
\end{proof}

\begin{remark}\label{rem:chardomfordiscretecase}
 It is not known whether Theorem \ref{thm:chardomqlccase} holds for general local martingales. By Theorem \ref{thm:detangmartforMXVpdwithaccjumps} and by Proposition \ref{prop:MmapproxMinLpforacccase} the main issue here is in proving a similar statement for discrete martingales with independent increments. Let us state this problem here as open. Let $X$ be a Banach space, $1\leq p<\infty$. Let $(\xi_n)_{n\geq 1}$ and $(\eta_n)_{n\geq 1}$ be $X$-valued mean-zero independent random variables such that for any Borel $B \in X\setminus\{0\}$
 \[
 \sum_{n} \mathbb P(\xi_n \in B) \leq  \sum_{n} \mathbb P(\eta_n \in B).
 \]
 Does there exists a constant $C$ (perhaps depending on $p$ and $X$) such that
 \[
 \mathbb E \Bigl\| \sum_n \xi_n \Bigr\|^p \leq C \mathbb E \Bigl\| \sum_n \eta_n \Bigr\|^p?
 \]
 By a standard symmetrization trick \cite[Lemma 6.3]{LT11} one can assume that $\xi_n$'s and $\eta_n$'s are symmetric. But even the symmetric case is not known for the author.
\end{remark}

\appendix

\section{Tangency under linear operators}\label{sec:tangunderlinoper}

The goal of this section is to show that $TM$ and $TN$ are tangent for any linear operator $T$ from a certain family given $M$ and $N$ are tangent. Let us start with bounded linear operators between Banach spaces.

\begin{theorem}\label{thm:MNtangent==>TMTNaretangentforanylinop}
Let $X$ be a Banach space, $M, N:\mathbb R_+ \times \Omega \to X$ be tangent local martingales. Let $Y$ be a Banach space, $T \in \mathcal L(X, Y)$. Then $TM$ and $TN$ have local characteristics and are tangent. Moreover, if $N$ is a decoupled tangent local martingale to $M$, then $TN$ is a decoupled tangent local martingale to $TM$.
\end{theorem}

The proof needs the following lemma.

\begin{lemma}\label{lem:lintransfofmart-->transfofnuM}
Let $X$, $Y$, and $Z$ be Banach spaces, $M:\mathbb R_+ \times \Omega \to X$ be a local martingale, $T \in \mathcal L(X, Y)$. Then for any predictable function $F:\mathbb R_+ \times \Omega \times Y \to Z$ we have that
\[
\int_{[0, \cdot]\times Y} \|F(s, \cdot, y)\| \ud \mu^{TM}(s, y),
\]
is locally finite if and only if
\[
\int_{[0, \cdot]\times X} \|F(s, \cdot, Tx)\| \ud \mu^{M}(s, x),
\]
and if this is the case then
\begin{equation}\label{eq:F(y)dmuTM=F(TX)dmuM}
\int_{[0, \cdot]\times X} F(s, \cdot, y) \ud \mu^{TM}(s, y) = \int_{[0, \cdot]\times X} F(s, \cdot, Tx) \ud \mu^{M}(s, x).
\end{equation}
Moreover, if 
$$
\mathbb E \int_{[0, t]\times Y} \|F(s, \cdot, y)\| \ud \mu^{TM}(s, y) <\infty
$$ 
or, equivalently, 
$$
\mathbb E \int_{[0, t]\times X} \|F(s, \cdot, Tx)\| \ud \mu^{M}(s, x) < \infty
$$ 
for any $t\geq 0$, then
\begin{equation}\label{eq:F(y)dnuTM=F(TX)dnuM}
\int_{[0, t]\times Y} F(s, \cdot, y) \ud \nu^{TM}(s, y) = \int_{[0, t]\times X} F(s, \cdot, Tx) \ud \nu^{M}(s, x)<\infty,\;\;\; t\geq 0.
\end{equation}
\end{lemma}

\begin{proof}
The first part of the lemma follows directly from the definition of $\mu^M$ and $\mu^{TM}$ (see \eqref{eq:defofmuM}). \eqref{eq:F(y)dmuTM=F(TX)dmuM} follows for a similar reason. \eqref{eq:F(y)dnuTM=F(TX)dnuM} follows from \eqref{eq:F(y)dmuTM=F(TX)dmuM}, the definition of a compensator random measure (see Subsection \ref{subsec:ranmeasures}), the definition of a compensator process \cite[Theorem I.3.17]{JS}, and the uniqueness of the compensator process.
\end{proof}

\begin{proof}[Proof of Theorem \ref{thm:MNtangent==>TMTNaretangentforanylinop}]
Let us start with the first part of the theorem. We need to show that $TM$ and $TN$ have local characteristics which coincide. First let us show that $TM$ and $TN$ have the Meyer-Yoeurp decomposition. As $M$ and $N$ are tangent, they have the Meyer-Yoeurp decomposition (see Subsection \ref{subsec:loccharandtangds}). Let $M= M^c + M^d$ and $N=N^c + N^d$ be this decomposition. Then $TM = TM^c + TM^d$ and $TN = TN^c + TN^d$ are the Meyer-Yoeurp decomposition as well since $TM^c$ and $TN^c$ are continuous and $TM^d$ and $TN^d$ are purely discontinuous as for any $y^*\in Y^*$ we have that both $\langle TM^d, y^*\rangle = \langle M^d, T^*y^*\rangle$ and $\langle TN^d, y^*\rangle = \langle N^d, T^*y^*\rangle$ are purely discontinuous (see Definition \ref{def:purelydiscmart}). Let us show that both $[\![TM^c]\!]$ and $[\![TN^c]\!]$ exist and coincide. To this end it is sufficient to notice that for any $y^*\in Y^*$ we have that a.s.\
\begin{equation}\label{eq:[[TMc]]y*idbssby[[Mc]]Ty*}
\begin{split}
[\![TM^c]\!]_t(y^*, y^*) = [\langle &TM^c, y^*\rangle]_t  = [\langle M^c, T^*y^*\rangle]_t \\
&\leq \|[\![M^c]\!]_t\|\|T^*y^*\|^2 \leq \|[\![M^c]\!]_t\|\|T\|^2\|y^*\|^2,\;\;\; t\geq 0,
\end{split}
\end{equation}
where we define $\|V\| := \sup_{z^*\in Z^*, \|z^*\| \leq 1}V(z^*,z^*)$ for any symmetric bilinear form $V:Z^*\times Z^* \to \mathbb R$ for any Banach space $Z$. Therefore $\|[\![TM^c]\!]_t\| \leq \|T\|^2 \|[\![M^c]\!]_t\| $ for any $t\geq 0$, and $[\![TM^c]\!]_t$ defines a bounded bilinear form. The same holds for $[\![TN^c]\!]_t$. Equality $[\![TM^c]\!]_t = [\![TN^c]\!]_t$ follows directly from the fact that for any $y^*\in Y^*$ a.s.\ for any $t\geq 0$ by \eqref{eq:[[TMc]]y*idbssby[[Mc]]Ty*}
\[
[\![TM^c]\!]_t(y^*, y^*) = [\![M^c]\!]_t(T^*y^*, T^*y^*) = [\![N^c]\!]_t(T^*y^*, T^*y^*) = [\![TN^c]\!]_t(y^*,y^*).
\]

The fact that $\nu^{TM} = \nu^{TN}$ a.s.\ follows from Lemma \ref{lem:lintransfofmart-->transfofnuM}. Therefore $TM$ and $TN$ have the same local characteristics, and thus are tangent.

\smallskip

Let us show the second part of the theorem. This part follows from Definition \ref{def:dectanglocmartcontimecased} and the fact that action of a bounded linear operator does not ruin independence and martingality (so if $N(\omega)$ is a martingale with independent increments, $TN(\omega)$ is so as well).
 \end{proof}

Another important type of operators are stopping time operators. Apparently, they also preserve tangency.

\begin{theorem}\label{thm:MNtangent==>MtauNtautangent+ifNCIINtauCII}
Let $X$ be a Banach space, $M, N:\mathbb R_+ \times \Omega \to X$ be tangent local martingales. Then $M^{\tau}$ and $N^{\tau}$ are tangent. Moreover, if $N$ is a decoupled tangent local martingale to $M$, and if $\tau$ is an $\mathbb F$-stopping time (where $\mathbb F$ is the original filtration where $M$ used to live), then $N^{\tau}$ is a decoupled tangent local martingale to $M^{\tau}$.
\end{theorem}

\begin{proof}
Let $M= M^c + M^d$ and $N = N^c + N^d$ be the Meyer-Yoeurp decompositions. Then $M^{\tau} = (M^c)^{\tau} + (M^d)^{\tau}$ is the Meyer-Yoeurp decomposition as $\langle M^{\tau}, x^* \rangle = \langle (M^c)^{\tau}, x^* \rangle + \langle (M^d)^{\tau}, x^* \rangle$ is the Meyer-Yoeurp decomposition since $\langle (M^c)^{\tau}, x^* \rangle = \langle M^c, x^* \rangle^{\tau}$ is continuous and since $\langle (M^d)^{\tau}, x^* \rangle = \langle M^d, x^* \rangle^{\tau}$ is purely discontinuous by \cite[Theorem 26.6]{Kal}. For the same reason $N^{\tau} = (N^c)^{\tau} + (N^d)^{\tau}$ is the Meyer-Yoeurp decomposition as well. Moreover,  by \cite[Theorem 26.6]{Kal} we have that $[\![ (M^c)^{\tau}]\!] = [\![ M^c]\!]^{\tau} =[\![ N^c]\!]^{\tau} = [\![ (N^c)^{\tau}]\!]  $ a.s. It remains to show that $\nu^{M^{\tau}} = \nu^{N^{\tau}}$. To this end it is sufficient to notice that $\mu^{M^{\tau}} = \mathbf 1_{[0, \tau]} \mu^{M}$ and $\mu^{N^{\tau}} = \mathbf 1_{[0, \tau]} \mu^{N}$, so by \cite[Proposition II.1.30]{JS} (see also \cite[Subsection 5.4]{DY17}) we have that $\nu^{M^{\tau}} = \mathbf 1_{[0,\tau]} \nu^M = \mathbf 1_{[0,\tau]} \nu^N  =  \nu^{N^{\tau}}$, so $M^{\tau}$ and $N^{\tau}$ are tangent.

Let us show the second part. First recall that by Definition \ref{def:dectanglocmartcontimecased} $N$ is a decoupled tangent local martingale if and only if $N(\omega)$ is a martingale with independent increments with local characteristics $([\![M^c(\omega)]\!], \nu^M(\omega))$ for any $\omega\in \Omega$ (where $\Omega, \mathcal F, \mathbb P$ is a probability space where $M$ lives). As $\tau$ is an $\mathbb F$-stopping time, it depends only on $\omega$, $N^{\tau}(\omega) = N(\omega)^{\tau(\omega)}$ is a martingale with independent increments having $([\![M^c(\omega)]\!]^{\tau(\omega)}, \mathbf 1_{[0, \tau(\omega)]}\nu^M(\omega))$ as its local characteristics, so the desired holds true.
\end{proof}

\section{Martingale approximations}\label{sec:appmartapprox}

Here we present certain martingales approximation techniques shown in \cite{Y18BDG}. Recall that a function $\phi:\mathbb R_+ \to \mathbb R_+$ is called to have a {\em moderate growth} if there exists $\alpha >0$ such that $\phi(2t)\leq \alpha \phi(t)$ for any $t\geq 0$.

\subsection{Purely discontinuous martingales with accessible jumps}\label{subsec:appforMarAppPDMAJ}

Let $X$ be a Banach space, $M:\mathbb R_+ \times \Omega\to X$ be a purely discontinuous martingale with accessible jumps. Then by Lemma \ref{lem:PDmAJhasjumpsatprsttimes} there exist finite predictable stopping times $(\tau_n)_{n\geq 1}$ with disjoint graphs which exhaust jumps of $M$. For any $m\geq 1$ let us define
\begin{equation}\label{eq:defofMmforaccjase}
M^m_t := \sum_{n=1}^m \Delta M_{\tau_n} \mathbf 1_{[\tau_n,\infty)} (t)\;\;\; t\geq 0.
\end{equation}
Then due to Lemma \ref{lem:DeltaMtaugiventau-=0} $M^m$ is a local martingale for any $m\geq 1$ and by \cite[Subsubsection 7.5.2]{Y18BDG} the following proposition holds true.

\begin{proposition}\label{prop:MmapproxMinLpforacccase}
Let $X$ be a UMD Banach space, $\phi:\mathbb R_+ \to \mathbb R_+$ be a convex function of moderate growth with $\phi(0)=0$, $M:\mathbb R_+ \times \Omega \to X$ be a purely discontinuous martingale with accessible jumps such that
\begin{equation}\label{eq:phi(supM)isfiniteforPDwAJ}
\mathbb E \sup_{t\geq 0}\phi\bigl( \|M_t\|\bigr) <\infty.
\end{equation}
For any $m\geq 1$ let $M^m$ be defined by \eqref{eq:defofMmforaccjase}. Then
\begin{equation}\label{eq:approxofPDwAJmartbyjumps}
\mathbb E \sup_{t\geq 0} \phi\bigl(\|M^m_t\|\bigr) <\infty,\;\;\;\; m\geq 1,
\end{equation}
and moreover
\begin{equation}\label{eq:approxfactforMmforPDwAJcase}
\mathbb E \sup_{t\geq 0} \phi\bigl(\|M_t - M^m_t\|\bigr) \to 0,\;\;\; m\to \infty.
\end{equation}
\end{proposition}

\begin{proof}
The case of $\phi(t) = t^p$, $1\leq p<\infty$, was covered  \cite[Subsubsection 7.5.2]{Y18BDG}. The general case follows analogously. Indeed, first notice that $[\![M^m]\!]_{\infty} \leq [\![M]\!]_{\infty}$ a.s.\ for any $m\geq 1$ by \cite[Theorem 26.6 and Corollary 26.15]{Kal}. Thus $\gamma([\![M^m]\!]_{\infty}) \leq \gamma([\![M]\!]_{\infty})$ by \cite[Subsection 3.2]{Y18BDG}, so by \cite[Section 5]{Y18BDG} we have that
\begin{multline*}
\mathbb E \sup_{t\geq 0} \phi\bigl(\|M^m_t\|\bigr) \eqsim_{\phi, X} \mathbb E \phi\bigl(\gamma([\![M^m]\!]_{\infty})\bigr) \\
\leq \mathbb E \phi\bigl(\gamma([\![M]\!]_{\infty})\bigr) \eqsim_{p, X}\mathbb E \sup_{t\geq 0} \phi\bigl(\|M_t\|\bigr) <\infty,
\end{multline*}
and \eqref{eq:approxofPDwAJmartbyjumps} holds true. Moreover, by \cite[Subsubsection 7.5.2]{Y18BDG} we know that $[\![M- M^m]\!]_{\infty} \to 0$ monotonically a.s., so by \cite[Subsection 3.2]{Y18BDG} we have that $\gamma([\![M-M^m]\!]_{\infty}) \to 0$ monotonically a.s., and hence by the dominated convergence theorem, the fact that $[\![M- M^m]\!]_{\infty} \leq [\![M]\!]_{\infty}$ a.s., by \eqref{eq:phi(supM)isfiniteforPDwAJ}, and \cite[Section 5 and Subsection 3.2]{Y18BDG} we have that
\[
\mathbb E \sup_{t\geq 0} \phi\bigl(\|M - M^m_t\|\bigr) \eqsim_{\phi, X} \mathbb E \phi\bigl(\gamma([\![M-M^m]\!]_{\infty})\bigr) \to 0,\;\;\; m\to \infty,
\]
so \eqref{eq:approxfactforMmforPDwAJcase} follows.
\end{proof}

\subsection{Purely discontinuous quasi-left continuous martingales}

Now let $M:\mathbb R_+ \times \Omega\to X$ be a purely discontinuous quasi-left continuous martingale. Then by Theorem \ref{thm:XisUMDiffMisintxwrtbarmuMvain}   (see also \cite[Subsection 5.4]{DY17} and \cite[Subsection 7.2 and 7.5]{Y18BDG}) $M_t = \int_{[0, t] \times X} x \ud \bar{\mu}^M$, where $\mu^M$ is as defined by \eqref{eq:defofmuM}, $\nu^M$ is the corresponding compensator, $ \bar{\mu}^M = \mu^M - \nu^M$. For each $m\geq 1$ let
\begin{equation}\label{eq:approxforMqlcpdsmalljumps}
 M^m_t := \int_{[0, t] \times X} x \mathbf 1_{[0, m]}(\|x\|)\ud \bar{\mu}^M,\;\;\; t\geq 0,
\end{equation}
or
\begin{equation}\label{eq:approxforMqlcpdbigjumps}
 M^m_t := \int_{[0, t] \times X} x \mathbf 1_{[1/m, \infty)}(\|x\|)\ud \bar{\mu}^M,\;\;\; t\geq 0.
\end{equation}
Then due to \cite[Subsubsection 7.5.1]{Y18BDG} and \cite{DY17} $M^m$ is a local martingale and the following proposition holds true by \cite[Subsubsection 7.5.1]{Y18BDG}.

\begin{proposition}\label{prop:approxforPDQLC}
Let $X$ be a UMD Banach space,  $\phi:\mathbb R_+ \to \mathbb R_+$ be a convex function of moderate growth with $\phi(0)=0$, $M:\mathbb R_+ \times \Omega \to X$ be a purely discontinuous quasi-left continuous martingale such that
\[
\mathbb E \sup_{t\geq 0}\phi\bigl( \|M_t\|\bigr) <\infty.
\]
For any $m\geq 1$ let $M^m$ be defined by \eqref{eq:approxforMqlcpdsmalljumps} or \eqref{eq:approxforMqlcpdbigjumps}. Then
\[
\mathbb E \sup_{t\geq 0} \phi\bigl(\|M^m_t\|\bigr) <\infty,\;\;\;\; m\geq 1,
\]
and moreover
\[
\mathbb E \sup_{t\geq 0} \phi\bigl(\|M_t - M^m_t\|\bigr) \to 0,\;\;\; m\to \infty.
\]
\end{proposition}

\begin{proof}
The proof is fully analogous to the proof of Proposition \ref{prop:MmapproxMinLpforacccase}.
\end{proof}

\bibliographystyle{plain}

\def\cprime{$'$} \def\polhk#1{\setbox0=\hbox{#1}{\ooalign{\hidewidth
  \lower1.5ex\hbox{`}\hidewidth\crcr\unhbox0}}}
  \def\polhk#1{\setbox0=\hbox{#1}{\ooalign{\hidewidth
  \lower1.5ex\hbox{`}\hidewidth\crcr\unhbox0}}} \def\cprime{$'$}

\end{document}